\documentclass[10pt]{amsart}

\usepackage{etoolbox}

\makeatletter
\let\old@tocline\@tocline
\let\section@tocline\@tocline
\newcommand{\subsection@dotsep}{4.5}
\newcommand{\subsubsection@dotsep}{4.5}
\patchcmd{\@tocline}
  {\hfil}
  {\nobreak
     \leaders\hbox{$\m@th
        \mkern \subsection@dotsep mu\hbox{.}\mkern \subsection@dotsep mu$}\hfill
     \nobreak}{}{}
\let\subsection@tocline\@tocline
\let\@tocline\old@tocline

\patchcmd{\@tocline}
  {\hfil}
  {\nobreak
     \leaders\hbox{$\m@th
        \mkern \subsubsection@dotsep mu\hbox{.}\mkern \subsubsection@dotsep mu$}\hfill
     \nobreak}{}{}
\let\subsubsection@tocline\@tocline
\let\@tocline\old@tocline

\let\old@l@subsection\l@subsection
\let\old@l@subsubsection\l@subsubsection

\def\@tocwriteb#1#2#3{%
  \begingroup
    \@xp\def\csname #2@tocline\endcsname##1##2##3##4##5##6{%
      \ifnum##1>\c@tocdepth
      \else \sbox\z@{##5\let\indentlabel\@tochangmeasure##6}\fi}%
    \csname l@#2\endcsname{#1{\csname#2name\endcsname}{\@secnumber}{}}%
  \endgroup
  \addcontentsline{toc}{#2}%
    {\protect#1{\csname#2name\endcsname}{\@secnumber}{#3}}}%

\newlength{\@tocsectionindent}
\newlength{\@tocsubsectionindent}
\newlength{\@tocsubsubsectionindent}
\newlength{\@tocsectionnumwidth}
\newlength{\@tocsubsectionnumwidth}
\newlength{\@tocsubsubsectionnumwidth}
\newcommand{\settocsectionnumwidth}[1]{\setlength{\@tocsectionnumwidth}{#1}}
\newcommand{\settocsubsectionnumwidth}[1]{\setlength{\@tocsubsectionnumwidth}{#1}}
\newcommand{\settocsubsubsectionnumwidth}[1]{\setlength{\@tocsubsubsectionnumwidth}{#1}}
\newcommand{\settocsectionindent}[1]{\setlength{\@tocsectionindent}{#1}}
\newcommand{\settocsubsectionindent}[1]{\setlength{\@tocsubsectionindent}{#1}}
\newcommand{\settocsubsubsectionindent}[1]{\setlength{\@tocsubsubsectionindent}{#1}}

\renewcommand{\l@section}{\section@tocline{1}{\@tocsectionvskip}{\@tocsectionindent}{}{\@tocsectionformat}}%
\renewcommand{\l@subsection}{\subsection@tocline{2}{\@tocsubsectionvskip}{\@tocsubsectionindent}{}{\@tocsubsectionformat}}%
\renewcommand{\l@subsubsection}{\subsubsection@tocline{3}{\@tocsubsubsectionvskip}{\@tocsubsubsectionindent}{}{\@tocsubsubsectionformat}}%
\newcommand{\@tocsectionformat}{}
\newcommand{\@tocsubsectionformat}{}
\newcommand{\@tocsubsubsectionformat}{}
\expandafter\def\csname toc@1format\endcsname{\@tocsectionformat}
\expandafter\def\csname toc@2format\endcsname{\@tocsubsectionformat}
\expandafter\def\csname toc@3format\endcsname{\@tocsubsubsectionformat}
\newcommand{\settocsectionformat}[1]{\renewcommand{\@tocsectionformat}{#1}}
\newcommand{\settocsubsectionformat}[1]{\renewcommand{\@tocsubsectionformat}{#1}}
\newcommand{\settocsubsubsectionformat}[1]{\renewcommand{\@tocsubsubsectionformat}{#1}}
\newlength{\@tocsectionvskip}
\newcommand{\settocsectionvskip}[1]{\setlength{\@tocsectionvskip}{#1}}
\newlength{\@tocsubsectionvskip}
\newcommand{\settocsubsectionvskip}[1]{\setlength{\@tocsubsectionvskip}{#1}}
\newlength{\@tocsubsubsectionvskip}
\newcommand{\settocsubsubsectionvskip}[1]{\setlength{\@tocsubsubsectionvskip}{#1}}

\patchcmd{\tocsection}{\indentlabel}{\makebox[\@tocsectionnumwidth][l]}{}{}
\patchcmd{\tocsubsection}{\indentlabel}{\makebox[\@tocsubsectionnumwidth][l]}{}{}
\patchcmd{\tocsubsubsection}{\indentlabel}{\makebox[\@tocsubsubsectionnumwidth][l]}{}{}

\newcommand{\@sectypepnumformat}{}
\renewcommand{\contentsline}[1]{%
  \expandafter\let\expandafter\@sectypepnumformat\csname @toc#1pnumformat\endcsname%
  \csname l@#1\endcsname}
\newcommand{\@tocsectionpnumformat}{}
\newcommand{\@tocsubsectionpnumformat}{}
\newcommand{\@tocsubsubsectionpnumformat}{}
\newcommand{\setsectionpnumformat}[1]{\renewcommand{\@tocsectionpnumformat}{#1}}
\newcommand{\setsubsectionpnumformat}[1]{\renewcommand{\@tocsubsectionpnumformat}{#1}}
\newcommand{\setsubsubsectionpnumformat}[1]{\renewcommand{\@tocsubsubsectionpnumformat}{#1}}
\renewcommand{\@tocpagenum}[1]{%
  \hfill {\mdseries\@sectypepnumformat #1}}

\let\oldappendix\appendix
\renewcommand{\appendix}{%
  \leavevmode\oldappendix%
  \addtocontents{toc}{%
    \protect\settowidth{\protect\@tocsectionnumwidth}{\protect\@tocsectionformat\sectionname\space}%
    \protect\addtolength{\protect\@tocsectionnumwidth}{2em}}%
}
\makeatother

\makeatletter
\settocsectionnumwidth{2em}
\settocsubsectionnumwidth{2.5em}
\settocsubsubsectionnumwidth{3em}
\settocsectionindent{1pc}%
\settocsubsectionindent{\dimexpr\@tocsectionindent+\@tocsectionnumwidth}%
\settocsubsubsectionindent{\dimexpr\@tocsubsectionindent+\@tocsubsectionnumwidth}%
\makeatother

\settocsectionvskip{10pt}
\settocsubsectionvskip{0pt}
\settocsubsubsectionvskip{0pt}
    
\settocsectionformat{\bfseries}
\settocsubsectionformat{\mdseries}
\settocsubsubsectionformat{\mdseries}
\setsectionpnumformat{\bfseries}
\setsubsectionpnumformat{\mdseries}
\setsubsubsectionpnumformat{\mdseries}

\let\oldtableofcontents\tableofcontents
\renewcommand{\tableofcontents}{%
  \vspace*{-\linespacing}
  \oldtableofcontents}

\usepackage[dvipsnames]{xcolor}
\usepackage[hyperfootnotes=false]{hyperref}
\usepackage{enumerate}
\usepackage{tikz}
\usetikzlibrary{trees}
\usepackage{tikz-cd}
\usetikzlibrary{patterns}
\usepackage{hyperref}
\usepackage{amsthm,amsfonts,amsmath,amscd,amssymb}
\usepackage{xcolor,float}
\usepackage[all]{xy}
\usepackage{mathrsfs}
\usepackage{graphics,graphicx}
\usepackage{caption}
\usepackage{trimclip}

\usepackage{comment}

\usetikzlibrary{decorations.markings}

\let\oldmarginpar\marginpar
\renewcommand\marginpar[1]{\-\oldmarginpar[\raggedleft\footnotesize #1]%
	{\raggedright\footnotesize #1}}
\theoremstyle{plain}
\newtheorem{thm}{Theorem}[section]
\newtheorem{lemma}[thm]{Lemma}

\newtheorem{prop}[thm]{Proposition}

\newtheorem{cor}[thm]{Corollary}

\newtheorem{conj}[thm]{Conjecture}

\theoremstyle{definition}

\newtheorem{definition}[thm]{Definition}
\newtheorem{remark}[thm]{Remark}

\newtheorem{ex}[thm]{Example}
\theoremstyle{remark}

\numberwithin{equation}{section}

\renewcommand{\S}{\mathbb{S}}
\renewcommand{\P}{\mathbb{P}}

\newcommand{\F}{\mathbb{F}}

\newcommand{\N}{\mathbb{N}}
\newcommand{\Z}{\mathbb{Z}}

\newcommand{\R}{\mathbb{R}}
\newcommand{\C}{\mathbb{C}}
\newcommand{\SA}{\mathcal{A}}

\newcommand{\SM}{\mathcal{M}}

\newcommand{\SB}{\mathcal{B}}

\newcommand{\La}{\Lambda}

\newcommand{\ww}{\mathfrak{w}}

\newcommand{\FT}{\Delta^2}
\newcommand{\Tr}{\mathrm{Tr}}
\newcommand{\diag}{\mathrm{diag}}
\newcommand{\bt}{\mathbf{t}}

\newcommand{\dd}{\partial}

\newcommand{\sse}{\subseteq}
\newcommand{\lr}{\longrightarrow}

\newcommand{\GL}{\operatorname{GL}}

\newcommand{\Aug}{\operatorname{Aug}}
\newcommand{\Spec}{\operatorname{Spec}}

\newcommand{\wt}{\widetilde}

\newcommand{\st}{\text{st}}

\newcounter{daggerfootnote}

\usepackage[dvipsnames]{xcolor}

\newcommand{\cB}{\mathcal{B}}

\newcommand{\cF}{\mathcal{F}}

\newcommand{\cS}{\mathcal{S}}

\newcommand{\Br}{\mathrm{Br}}
\newcommand{\Hom}{\mathrm{Hom}}

\newcommand{\ft}{\mathfrak{t}}

\newcommand{\Fl}{\mathcal{F}\ell} 
	
\DeclareMathOperator{\OBS}{OBS} 

\DeclareMathOperator{\dft}{def}
\DeclareMathOperator{\rot}{rot}

\newcommand{\we}[1]{\mathtt{wt}(#1)}

\newcommand{\be}{\mathbf{e}}

\newcommand{\dec}{decomposition}
\newcommand{\piece}{piece}

\begin{document}
	
\title{Algebraic Weaves and Braid Varieties}

\author[R. Casals]{Roger Casals}
\address{Dept. of Mathematics\\ University of California, Davis\\ One Shields Avenue, Davis, CA 95616}
\email{casals@math.ucdavis.edu}
\author[E. Gorsky]{Eugene Gorsky}
\address{Dept. of Mathematics\\ University of California, Davis\\ One Shields Avenue, Davis, CA 95616}
\email{egorskiy@math.ucdavis.edu}
\author[M. Gorsky]{Mikhail Gorsky}
\address{Department of Mathematics,
University of Vienna,
Oscar-Morgenstern Platz~1,
1090 Vienna,
Austria}
\email{mikhail.gorskii@univie.ac.at}
\author[J. Simental]{Jos\'e Simental}
\address{Instituto de Matem\'aticas\\ Universidad Nacional Aut\'onoma de M\'exico\\ Ciudad Universitaria, Mexico City, Mexico 04510}
\email{simental@im.unam.mx}
	
\maketitle

\vspace{-1cm}
\begin{abstract}
In this manuscript we study braid varieties, a class of affine algebraic varieties associated to positive braids. Several geometric constructions are presented, including certain torus actions on braid varieties and holomorphic symplectic structures on their respective quotients. We also develop a diagrammatic calculus for correspondences between braid varieties and  use these correspondences to obtain interesting \dec s of braid varieties and their quotients. It is shown that the maximal charts of these \dec s are exponential Darboux charts for the holomorphic symplectic structures, and we relate these charts to exact Lagrangian fillings of Legendrian links. 
\end{abstract}

\setcounter{tocdepth}{2}
\tableofcontents


\section{Introduction}\label{sec:intro}

This article studies braid varieties, a class of affine algebraic varieties associated to positive braids, and their relation to contact and symplectic geometry. First, the geometric properties of braid varieties are studied, including the construction of torus actions and holomorphic symplectic structures on their quotients. Then, we construct correspondences between these braid varieties by using certain moduli spaces associated to weaves, a class of labeled planar diagrams. These geometric correspondences are shown to induce valuable \dec s for braid varieties and their quotients, also unifying known constructions of P. Boalch and A. Mellit, in the case of character varieties, and M. Henry and D. Rutherford, in the case of augmentation varieties.

The diagrammatic calculus based on weaves, presented in Section \ref{sec: alg weaves}, allows for direct and explicit computations, and we provide new constructions of embedded exact Lagrangian fillings for Legendrian links through combinatorial methods. The main results of the article are Theorems \ref{thm:main1}, \ref{thm:main2} and \ref{thm:main3}, and several detailed examples are provided throughout the manuscript. In particular, we believe that  the construction of a holomorphic symplectic structure on augmentation varieties, developed in Section \ref{sec:symplecticform}   is of value for contact and symplectic geometry.


\subsection{Context} Legendrian links in contact 3--manifolds \cite{ArnoldSing,Bennequin83,Geiges} are central in contact and symplectic geometry. Legendrian fronts, immersed planar cuspidal curves, arise in topology, as Cerf diagrams \cite{Arnold80,BST15,Cerf70}, in differential equations, as Stokes data for irregular singularities \cite{Berry88,STZ,Stokes}, and in analysis, as wavefront sets \cite{Ho2,Ho1,KKP08}. In this article, we use that a positive braid $\beta$ gives rise to a Legendrian link $\La(\beta)\sse(\R^3,\xi_\st)$, cf.~\cite[Section 2.2]{CN} or \cite{CG,Geiges}.

Associated to a Legendrian link $\La\sse(\R^3,\xi_\st)$, there exist two geometrically defined moduli spaces: the moduli space of microlocal sheaves in $\R^2$ microlocally supported at $\La$, cf.~\cite{CW,GKS,KS82,KS90}, and the moduli space of exact Lagrangian fillings $L\sse(\R^4,\omega_\st)$, with boundary $\dd L=\La$, cf.~\cite{Arnold80,CG,CN,Geiges}. Note that the latter can be understood as the (geometric part of the) moduli space of objects of the Fukaya category of $(\R^4,\omega_\st)$ partially wrapped at $\La$. For the Legendrian links $\La(\beta)\sse(\R^3,\xi_\st)$, these moduli are algebraic stacks and, when appropriately decorated, smooth algebraic varieties. The present manuscript studies a collection of algebraic varieties associated to a positive braid $\beta$, including and generalizing these two moduli spaces, and new correspondences between them. These algebraic correspondences are often induced by geometric exact Lagrangian cobordisms between Legendrian links, and can in general be described with a diagrammatic calculus, as we will show, building on work of the first author with E. Zaslow \cite{CZ}.

In summary, we introduce the class of {\it braid varieties}, study torus actions and their quotients, construct correspondences and morphisms between them, and develop a diagrammatic calculus associated to these correspondences. As we establish these results, we prove several theorems of interest, including the fact that the augmentation variety associated to $\La(\beta)$ admits a holomorphic symplectic structure, and explain the relation between A. Mellit's \dec\, of character varieties \cite{Mellit} and the ruling stratification of the augmentation variety \cite{HR14,HR15}. Note that holomorphic symplectic structures play a central role in the study of moduli spaces of connections \cite{Boalch01,Boalch20}, and there ought to be a relation to their symplectic structures through understanding the moduli stack of objects in the $\mbox{Aug}_+$-category \cite{NRSSZ} as a wild character variety \cite{Boalch00,STWZ}. It should be noted that our diagrammatic calculus, which we refer to as {\it algebraic weaves}, provides a combinatorial and explicit approach to these \dec s. In addition, the \piece s are compatible with the holomorphic symplectic structure, the open toric charts admitting (exponential) holomorphic Darboux coordinates.


\subsection{Main Results} Let us define our main object of study, the braid matrices and braid varieties. In order to do this, let us fix $n > 0$. For each $i = 1, \dots, n-1$, we  consider the \emph{braid matrix}
$B_i(z)\in \GL(n,\C[z])$ defined by:
\[
(B_i(z))_{jk} := \begin{cases} 1 & j=k \text{ and } j\neq i,i+1 \\
1 & (j,k) = (i,i+1) \text{ or } (i+1,i) \\
z & j=k=i+1 \\
0 & \text{otherwise;}
\end{cases},\qquad\mbox{i.e.}\quad
B_i(z):=\left(\begin{matrix}
1 & \cdots  & & & \cdots & 0\\
\vdots & \ddots & & & & \vdots\\
0 & \cdots & 0 & 1 & \cdots & 0\\
0 & \cdots & 1 & z & \cdots & 0\\
\vdots &  & & &\ddots & \vdots\\
0 & \cdots & & & \cdots & 1\\
\end{matrix}\right).
\]

Note that the only the non-trivial $(2\times 2)$-block is at $i$th and $(i+1)$st rows. Braid matrices have appeared in a range of areas, starting with L. Euler's continuants \cite{Euler64}, G. Stokes' study of irregular singularities \cite{Stokes} (see P. Boalch's \cite{Boalch02,Boalch20}), M. Brou\'e and J. Michel's work on Deligne-Lusztig varieties \cite{BM}, P. Deligne's braid invariants \cite{Deligne}, and more recently in T. K\'alm\'an's study of the Legendrian Contact DGA \cite{Kalman} (see also \cite{CN}) and A. Mellit's results on the curious Lefschetz property for character varieties \cite{Mellit}, among others.

Let $\gamma$ be a positive $n$-braid word $[\gamma]\in \Br^+_n$, $\gamma=\sigma_{i_1}\cdots \sigma_{i_\ell}$. We consider the following matrix $B_{\gamma}(z_1,\ldots,z_{\ell})\in \GL(n,\C[z_1,\ldots,z_\ell])$, we which define to be the matrix product
$$
B_{\gamma}(z_1,\ldots,z_{\ell}):=B_{i_1}(z_1)\cdots B_{i_\ell}(z_{\ell}).
$$
Finally, if $\pi\in \GL(n,\C)$ is a permutation matrix we consider the {\em braid variety}
$$
X_0(\gamma;\pi):=\left\{(z_1,\ldots,z_\ell)\ :\ B_\gamma(z_1,\ldots,z_{\ell})\pi\ \text{is upper-triangular}\right\}\sse \C^{\ell}.
$$
Note that this is an affine algebraic variety, given by the vanishing of $\binom{n}{2}$ polynomial equations in the variables $z_1, \dots, z_{\ell}$.

From our definition above, it is simple to see that $X_0(\gamma;\pi)$ is isomorphic to $X_0(\gamma';\pi)$ if $[\gamma]=[\gamma']\in \Br_n$, i.e. if two positive words $\gamma,\gamma'$ represent the same $n$-braid, the resulting braid varieties are isomorphic, hence the name. In the course of the article, the permutation (matrix) $\pi$ will often be the identity $\pi=\mathrm{Id}=e\in S_n$ or the longest element $\pi=w_0=(n\ n-1\ \ldots\ 1)\in S_n$. Let $\Delta\in \Br^+_n$ be a positive braid lift of the permutation $w_0$, i.e. $\Delta$ will be a braid word for the half-twist. See Example \ref{ex: half twist} for our specific choice of positive braid word for $\Delta$.

The first result of the article establishes geometric properties of braid varieties, including the existence of a torus action and their relation to the Floer-theoretically defined augmentation varieties \cite{Bour,Che,NRSSZ}. It reads as follows.

\begin{thm}\label{thm:main1} Let $\gamma$ be a positive $n$-braid word $[\gamma]\in \Br^+_n$. Then the following statements hold:

\begin{itemize}
	\item[(i)] $X_0(\gamma\Delta;1)\simeq X_0(\gamma;w_0)\times \C^{\binom{n}{2}}$, and $X_0(\gamma;w_0)$ is non-empty if and only if the Demazure product of $\gamma$ equals $w_0$. In this case, $X_0(\gamma;w_0)$ is an irreducible complete intersection of dimension $\ell(\gamma)-\binom{n}{2}$, and $X_0(\gamma\Delta;1)$ is an irreducible complete intersection of dimension $\ell(\gamma)$.\\
	
\noindent	Suppose that there exists a positive $n$-braid word $\beta$ such that $\gamma=\beta\Delta$. Then:\\
	
	\item[(ii)] The braid variety $X_0(\beta\Delta;w_0)$, and thus $X_0(\beta\Delta^2;1)$, is smooth.\\
	
	\item[(iii)] There exists a free torus action on $X_0(\beta\Delta;w_0)$ such that the quotient algebraic variety $X_0(\beta\Delta;w_0)/T$ is smooth and holomorphic symplectic.\\
	
	\item[(iv)] There exists an isomorphism between $X_0(\beta\Delta;w_0)/T$ and an augmentation variety $\Aug(\beta)$ associated to the Legendrian link $\La(\beta)$. In particular, $\Aug(\beta)$ is a holomorphic symplectic $($smooth$)$ affine variety.

\item[(v)] The open Bott-Samelson variety $\OBS(\beta)$ associated to $\beta$ is isomorphic to the quotient
$$\OBS(\beta)\cong\left(\GL(n,\C)\times X_0(\beta;1)\right)/\SB,$$
where $\SB\sse \GL(n,\C)$ is the Borel subgroup of upper-triangular matrices. 
\end{itemize}
\end{thm}

In Theorem \ref{thm:main1}.(iii), the dimension of the torus $T$ does depend on the number of components in the closure of $\beta$, see Section \ref{sec:braid} for details. See Section \ref{sect:augmentations} for the details on marked points used to define the objects in \ref{thm:main1}.(iv). The different varieties and the torus action featured in Theorem \ref{thm:main1} are presented in the course of the article, and the proof of this theorem is obtained by gathering some the results we develop, such as Theorem \ref{thm:aug vs braid}, Theorem \ref{thm: smooth}, Theorem \ref{thm: symplectic} and Corollary \ref{cor: non empty}.  See also Section \ref{section:Dem_prod} for the definition of Demazure product, and note that the Demazure product of $\beta\Delta$ equals $w_0$ for any $\beta$.

Theorem \ref{thm:main1} discusses the absolute aspects of braid varieties. The study of such varieties also relies crucially on their relative geometry: morphisms between different such braid varieties and, more generally, correspondences, yield interesting (and useful) results. In order to study this relative setting, we develop the diagrammatic calculus of {\it weaves}, which we summarize as follows.

Let $\mathfrak{W}_n$ be the category defined as:

\begin{itemize}
	\item[-] {\bf Objects}: $\mbox{Ob}(\mathfrak{W}_n)$ are arbitrary positive braid words $\gamma=\sigma_{i_1}\cdots \sigma_{i_\ell}$, $[\gamma]\in Br^+_n$,
	\item[-] {\bf Morphisms}: $\mbox{Hom}_{\mathfrak{W}_n}(\gamma,\gamma')$ are compositions of the following six elementary moves, starting at $\gamma$ at the top and ending at $\gamma'$ at the bottom. The moves are
	$$\sigma_i\sigma_{i}\rightarrow \sigma_i,\qquad \sigma_i\sigma_{i+1}\sigma_i\leftrightarrow \sigma_{i+1}\sigma_i\sigma_{i+1},\qquad\sigma_i\sigma_j\rightarrow \sigma_{j}\sigma_i\ (|i-j|>1),\quad\mbox{ and }\quad\sigma_i\sigma_i\leftrightarrow 1.$$
	
\end{itemize}
We will declare some of the morphisms to be equivalent, see section \ref{sec:CombinatoricsWeaves}.

The morphisms in $\mathfrak{W}_n$ will be represented diagrammatically as certain planar graphs with edges decorated by simple transpositions $s_i$. (Namely, $s_i$ are the Coxeter projections of the Artin braid generators $\sigma_i$, $1\leq i\leq n$.) These planar graphs are referred to as {\it weaves}, following \cite[Section 2]{CZ}, and $\mathfrak{W}_n$ will be called the {\it category of weaves}. The elementary moves above, i.e. the building blocks for morphisms, can be drawn as follows:
\begin{center}
	\includegraphics[scale=0.3]{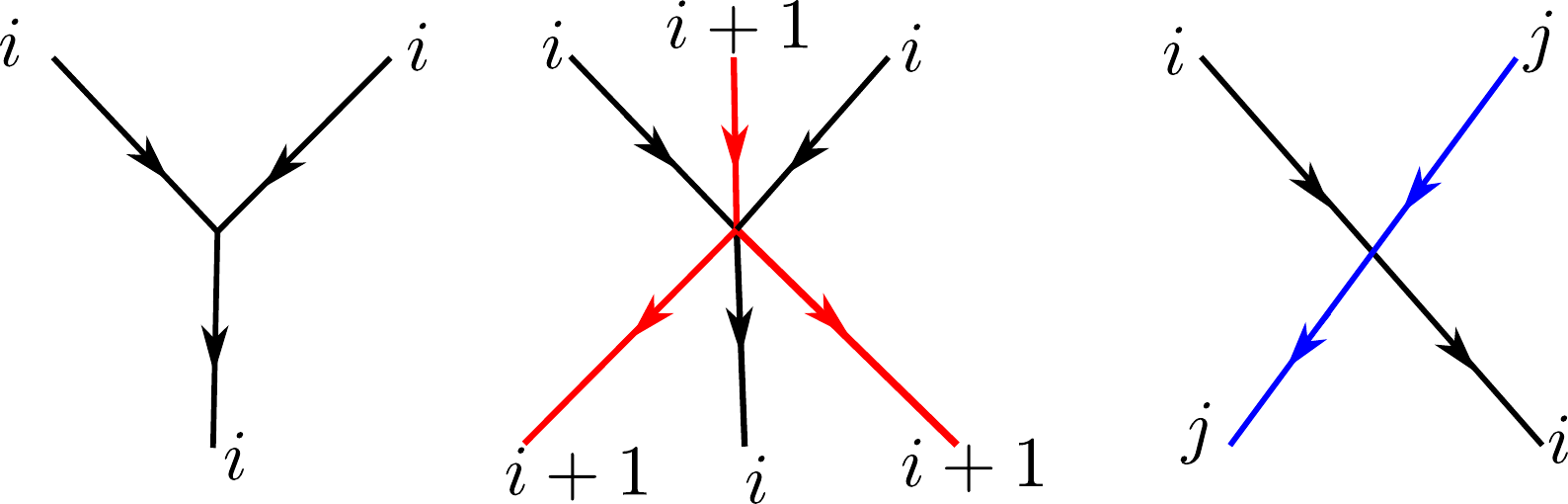}
	
	\includegraphics[scale=0.3]{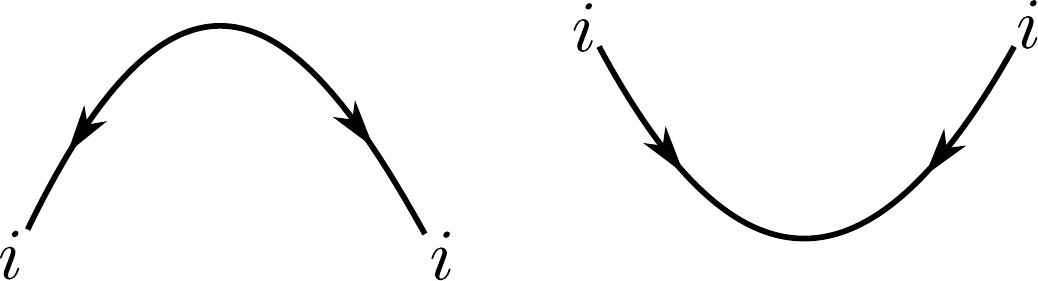}
\end{center}
There is also a dual 6-valent vertex corresponding to $\sigma_i\sigma_{i-1}\sigma_i\to \sigma_{i-1}\sigma_i\sigma_{i-1}$ which we do not draw here but is equally allowed.
An algebraic weave, obtained by vertically and horizontally concatenating the models above (plus additional decorations), represents a morphism from the braid word on the top to the braid word on the bottom. The composition of weaves
$$\mbox{Hom}_{\mathfrak{W}_n}(\gamma',\gamma'')\times\mbox{Hom}_{\mathfrak{W}_n}(\gamma,\gamma')\lr\mbox{Hom}_{\mathfrak{W}_n}(\gamma,\gamma'')$$
is given by vertical stacking of these weave diagrams. See Figure \ref{fig:IntroWeaveExample} for an instance of a morphism. 
\begin{center}
	\begin{figure}[h!]
		\centering
		\includegraphics[scale=0.8]{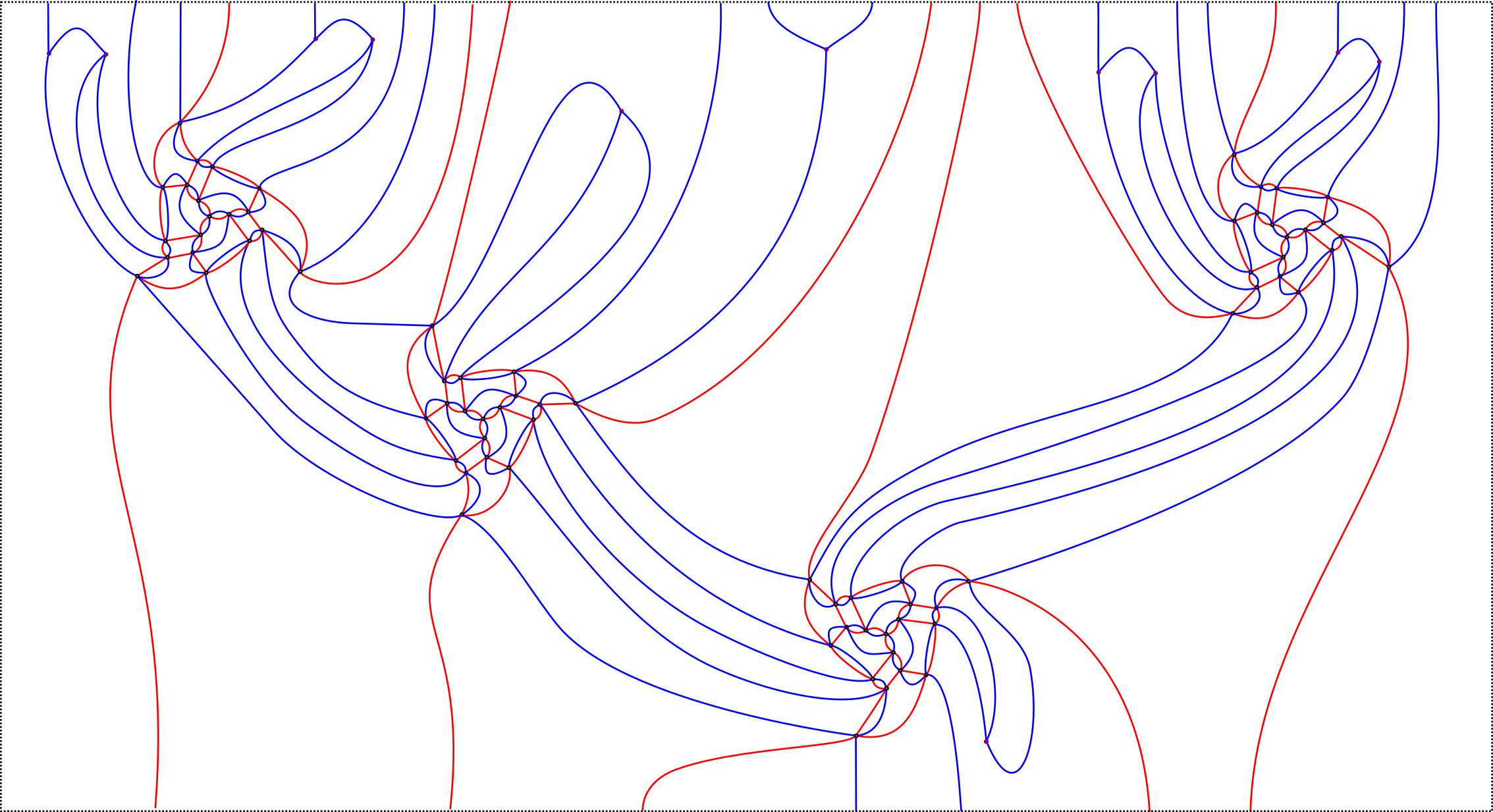}
		\caption{An algebraic weave in $\mbox{Hom}_{\mathfrak{W}_n}(\gamma,\gamma')$ between the two positive $3$-braids $\gamma=\sigma_1^3\sigma_2\sigma_1^3\sigma_1^2\sigma_1^3\sigma_2^3\sigma_1^3\sigma_2\sigma_1^3$, at the top, and $\gamma'=\sigma_2^3\sigma_1^2\sigma_2^2$, at the bottom. The color code is that \color{blue} blue \color{black} is labeled with the transposition $s_1$ and \color{red} red \color{black} is labeled with $s_2$. For readability we omit the (downward pointing) orientations. The upside-down trivalent vertices are defined using the usual trivalent vertices and cups, see Section \ref{sec: nonstandard 3v}.}
		\label{fig:IntroWeaveExample}
	\end{figure}
\end{center}

\begin{remark}
Note that this diagrammatic category is in part similar to the categories appearing in Soergel calculus \cite{EK,EW}, but differs in several key aspects. In particular, in the category of algebraic weaves there is no requirement that the two ways of getting from $\sigma_i\sigma_i\sigma_i$ to $\sigma_i$, via the moves $\sigma_i\sigma_i\rightarrow\sigma_i$, are equivalent:
\begin{center}
	\includegraphics[scale=0.3]{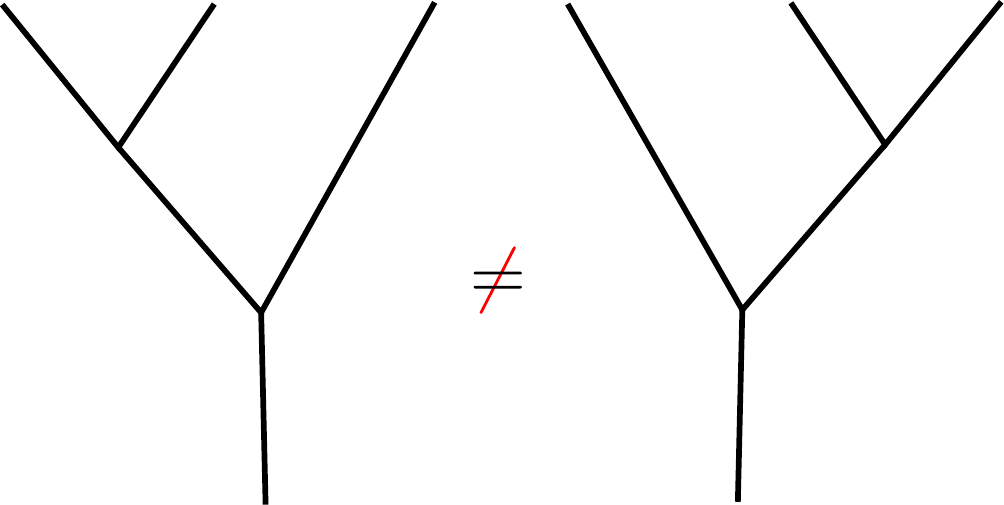}
\end{center}
The difference between these diagrams will be referred to as a {\it weave mutation}.
\end{remark}

Let $\mathfrak{C}$ be the category of algebraic varieties whose morphisms are correspondences. That is, a morphism $X \to Y$ consists of a pair of morphisms $X \leftarrow Z \rightarrow Y$, and composition corresponds to the fiber product. The second result in this manuscript shows that braid varieties and their correspondences provide a realization of the weave category $\mathfrak{W}_n$, as follows.

\begin{thm}\label{thm:main2}
	There exists a functor $\mathfrak{X}_0:\mathfrak{W}_n\lr\mathfrak{C}$ such that:
	\begin{itemize}
		\item [(a)] {\bf Objects}: For a positive braid word $\gamma\in\mbox{Ob}(\mathfrak{W}_n)$, the functor $\mathfrak{X}_0$ associates the braid variety $\mathfrak{X}_0(\gamma):=X_0(\beta;w_0)$.\\
		
		\item [(b)] {\bf Morphisms}: For a weave $\ww\in \Hom_{\mathfrak{W}_n}(\beta_2,\beta_1)$, the functor associates a correspondence $\mathfrak{X}_0(\ww)$ between 
		$X_0(\beta_2;w_0)$ and $X_0(\beta_1;w_0)$, such that correspondences $\mathfrak{X}_0(\ww)$ and $\mathfrak{X}_0(\ww')$ associated to equivalent weaves with no caps $\ww,\ww'$ are isomorphic. $($The algebraic variety $\mathfrak{X}_0(\ww)$ is in fact described as a certain moduli space governed by the weave $\ww$.$)$\\
		
		\item [(c)] {\bf Composition}: Let $\ww_1\in\Hom_{\mathfrak{W}_n}(\beta_1,\beta_0)$, $\ww_2\in\Hom_{\mathfrak{W}_n}(\beta_2,\beta_1)$, and consider their composition 
        $\ww = \ww_1 \circ \ww_2 \in \Hom_{\mathfrak{W}_n}(\beta_2,\beta_0)$,
        which is obtained by vertical concatenation of $\ww_2$, at the top, and $\ww_1$, at the bottom. Then the composition of weaves under $\mathfrak{X}_0$ corresponds to the diagram:
		\begin{center}
			\begin{tikzcd}
			& & \mathfrak{X}_0(\ww) \arrow{dl} \arrow{dr} & & \\
			& \mathfrak{X}_0(\ww_1) \arrow{dl} \arrow{dr}  & & \mathfrak{X}_0(\ww_2) \arrow{dl} \arrow{dr}  & \\
			X_0(\beta_0;w_0) &  & X_0(\beta_1;w_0) & & X_0(\beta_2;w_0),\\
			\end{tikzcd} 
		\end{center}
		where the middle square is Cartesian.\\
		
		\item [(d)] Let $\ww\in\Hom_{\mathfrak{W}_n}(\beta_2,\beta_1)$ be a weave with no caps, $a$ cups and $b$ trivalent vertices. Then the correspondence $\mathfrak{X}_0(\ww)$ defines an injective map
		$$
		\mathfrak{X}_0(\ww):\C^a\times (\C^*)^b\times X_0(\beta_1;w_0)\hookrightarrow X_0(\beta_2;w_0).
		$$
	\end{itemize}
Furthermore, the correspondences $\mathfrak{X}_0(\ww)$ are equivariant with respect to appropriate torus actions and, using Theorem \ref{thm:main1}.$($iv$)$, yield correspondences between augmentation varieties.
\end{thm}

The proof of Theorem \ref{thm:main2} occupies the majority of Section \ref{sec: alg weaves}, the equivariance statement being discussed in Subsection \ref{seq: equivatiant weaves}. The statements in Theorem \ref{thm:main2}.(a)-(c) are the algebraic analogues of the symplectic geometric results obtained in \cite{CZ}. Note that the algebraic variety $X_0(\Delta;w_0)$ is a point, and thus Theorem \ref{thm:main2} implies the following.

\begin{cor}
	\label{cor: strata intro}
	Let $\ww\in\Hom_{\mathfrak{W}_n}(\gamma,\Delta)$ be a weave with no caps, $a$ cups and $b$ trivalent vertices, $a,b\in\N$. Then the correspondence $\mathfrak{X}_0(\ww)$ yields an injective map 
	$$\mathfrak{X}_0(\ww):\C^a\times (\C^*)^b\hookrightarrow  X_0(\gamma;w_0),\ 2a+b=\ell(\gamma)-\binom{n}{2}.$$
\end{cor}

Corollary \ref{cor: strata intro} provides a unifying framework for many known \dec s, including the ruling stratification in augmentation varieties \cite{Fuchs,HR14,HR15}, and the \dec \, by walks in character varieties \cite{Mellit}. First, we will see that any weave $\ww \in \Hom_{\mathfrak{W}_n}(\gamma, \Delta)$ with no cups or caps yields an open  (algebraic) torus $(\C^*)^{\ell(\gamma)-\binom{n}{2}}\subset X_0(\gamma;w_0)$. Fixing such a weave, we will see that its complement can be further decomposed with weaves (that will now include cups). The weaves with no caps or cups, that will be of primary importance in this article, will be referred to as {\it Demazure weaves}. Note that different Demazure weaves define (a priori) different \dec s of $X_0(\gamma; w_0)$. 

\begin{remark} The manuscript also includes a new construction of weaves, coming from a class of labeled triangulations. This construction, described in Section~\ref{sec: admissible}, uses Demazure products in a crucial manner and, together with results of \cite{CZ}, provides a systematic (and combinatorial) mechanism to construct embedded exact Lagrangian fillings for Legendrian links in $(\R^3,\xi_\st)$ which are obtained as closures of a positive braid $\beta$. Specifically, the points in the braid variety $X_0(\beta;w_0)$ correspond to fillings of the $(-1)$-closure of $\beta\Delta$, cf.~\cite[Section 2]{CN}.
\end{remark}

Finally, complementing Theorem \ref{thm:main1} and Theorem \ref{thm:main2}, we give a geometric interpretation to these toric charts associated to Demazure weaves $\ww\in\mbox{Hom}_{\mathfrak{W}_n}(\beta\Delta,\Delta)$, as follows.

First, we show in Section \ref{sec: opening} that these charts can be combinatorially obtained by opening the crossings of the positive braid $\beta$. Indeed, Section \ref{sec: opening} shows that there is an injective map 
$$
X_0(\beta'\Delta; w_0)\times \C^*\hookrightarrow X_0(\beta\Delta; w_0),
$$
if the positive braid word $\beta'$ is obtained from $\beta$ by removing exactly one crossing. Therefore, opening the crossings in $\beta$ one by one, in some order, yields a toric chart in $X_0(\beta\Delta; w_0)$. Different orders might yield identical or different toric charts. For instance, for a 2-strand braid $\beta=\sigma_1^n$, there are $n!$ possible orderings and one obtains a Catalan number $C_n$ of toric charts. In particular, in this correspondence, each toric chart is obtained by exactly one 312-pattern avoiding permutation.

\begin{definition}
Throughout the paper, if $\beta$ is a braid word of length $\ell$, we denote by $S_{\ell}$ the set of orderings on the crossings of $\beta$ (which is in bijection with the symmetric group in $\ell$ letters). 
\end{definition}

Regarding this relation, between toric charts and openings of crossings, we show the following.

\begin{thm}\label{thm:main3} Let $[\beta]\in Br^+_n$ be a positive braid and $\beta=\sigma_{i_1}\cdot\sigma_{i_2}\cdot\ldots\cdot \sigma_{i_l}$ a positive braid word. Consider an ordering $\rho\in S_{l}$ for the crossings of $\beta$. Then:

\begin{itemize}
\item[(i)] There exists a $($Demazure$)$ weave $\ww_\rho$ such that the sequence of crossing openings according to $\rho$ is realized by the weave $\ww_\rho$. Conversely, any Demazure weave $\ww$ is equivalent to opening crossings for some ordering $\rho\in S_l$, i.e. there exists $\rho\in S_l$ such that $\ww$ is equivalent to $\ww_\rho$.\\
	
\item[(ii)] Two toric charts $C_1,C_2\sse X_0(\beta\Delta;w_0)$ associated to different orderings of the crossings are represented by weaves $\ww_1,\ww_2$ such that $\ww_1,\ww_2$ are related by a sequence of mutations. In addition, the union of all such toric charts covers $X_0(\beta\Delta;w_0)$ up to codimension 2.
\end{itemize}
\end{thm}

The first item is proven in Lemma \ref{lem: open weave} and Theorem \ref{th: weave opening}, and the proof of the codimension-2 cover is established in Theorem \ref{thm: codim 2}. The mutation equivalence of any weaves yielding toric charts  follows from the more general Theorem \ref{thm: ben}, which states that, under technical conditions that are satisfied in the weaves pertaining to Theorem \ref{thm:main3}, any two Demazure weaves between the same two braid words are related by a sequence of equivalence moves and mutations. Note that Theorem \ref{thm: ben} is a translation of a result of B. Elias \cite{Elias} to our weaves framework.

\begin{remark} Note that both the openings of crossings and mutations can be described in terms of braid words. Indeed, consider a braid word $\sigma_iu\sigma_j$ with $\sigma_iu=u\sigma_j$, i.e. $(\sigma_i,\sigma_j)$ is a deletion pair in the notation of \cite{Hersh}. Then, we can consider two different weaves:
\begin{center}
	\begin{tikzcd}
	& \sigma_iu\sigma_j \arrow{dl} \arrow{dr}& \\
	\sigma_i\sigma_iu  \arrow{d}& & u\sigma_j\sigma_j\arrow{d}\\
	\sigma_iu \arrow{rr}& & u\sigma_j 
	\end{tikzcd}
\end{center}
In this diagram, the left weave $\sigma_iu\sigma_j\to u\sigma_j$
corresponds to the opening of a crossing $\sigma_i$, and the right weave 
$\sigma_iu\sigma_j\to u\sigma_j$
to opening a crossing $\sigma_i$. Theorem \ref{thm:main3} implies that the two weaves are always related by a sequence of equivalence moves and mutations. For example, for $i=j$ and $u=\sigma_i$ we get a mutation, while for $i=1,j=2$ and $u=\sigma_2\sigma_1$ we get an equivalence (see Section \ref{sec: 1212}):
\begin{center}
	\begin{tikzcd}
	& \sigma_1(\sigma_2\sigma_1)\sigma_2 \arrow{dl} \arrow{dr}& \\	\sigma_1\sigma_1(\sigma_2\sigma_1)  \arrow{d}& & (\sigma_2\sigma_1)\sigma_2\sigma_2\arrow{d}\\
	\sigma_1(\sigma_2\sigma_1) \arrow{rr}& & (\sigma_2\sigma_1)\sigma_2 
	\end{tikzcd}
\end{center}
\end{remark}

\subsection{Related developments}

In this section we comment on some recent results and developments which were completed after the first version of this paper was posted on arXiv.

We further studied certain classes braid varieties in \cite{CGGS2}. In particular, all positroid varieties \cite{KLS} in the Grassmannian $\mathrm{Gr}(k,n)$ were shown to be isomorphic to braid varieties for several different braids, both on $n$ and on $k$ strands. The paper \cite{CGGS2} also gives a precise relation between braid varieties, 
\emph{subword complexes} and brick polytopes  \cite{BS, CLS,Escobar,G1,G2, JLS, KM1, KM2, PS}. The faces of a subword complex for a braid word $\gamma$ correspond to all possible subwords of $\gamma$ such that the Demazure product of their complements equals $w_0$. Subword complexes were introduced by Knutson and Miller \cite{KM1, KM2} in the context of Gr\"obner geometry of Schubert polynomials. Knutson and Miller proved that subword complexes are homeomorphic to balls or spheres. Pilaud and Stump found polytopal realizations of spherical subword complexes and called them \emph{brick polytopes}.  Results of \cite{G1,G2,G3} describe the behavior of subword complexes under braid moves and moves $s_is_i\to s_i$ in  $\gamma$. Ceballos, Labb\'e and Stump \cite{CLS} proved that certain brick polytopes are generalized associahedra, thus relating subword complexes to the theory of cluster algebras. See also more recent works of Brodsky and Stump \cite{BS} and of Jahn, L\"owe and Stump \cite{JLS} further exploring this relation. Using the work of Escobar \cite{Escobar}, we also show in \cite{CGGS2} that a braid variety admits a smooth compactification by the so-called brick manifold. The combinatorics of the boundary divisor agree with the dual subword complex.

Finally, there was a recent increase of interest relating weaves, braid varieties and cluster algebras. In particular, in a joint work with I. Le and L. Shen \cite{CGGLSS}, we show that any braid variety admits a cluster structure. This result was also proven in \cite{GLSBS1,GLSBS2,SBS} by different methods, and we expect the two cluster structures to be closely related. The above results resolve a long-standing conjecture of Leclerc \cite{Leclerc} on the existence of cluster structure on open Richardson varieties. In particular, a cluster structure guarantees the existence of a collection of open tori which correspond to Demazure weaves as in Corollary \ref{cor: strata intro}. On such a torus, \cite{CGGLSS} defines a collection of cluster coordinates using the combinatorics of a weave. We refer to \cite{CGGLSS,GLSBS1,GLSBS2,SBS,SSW,SW} and references therein for all definitions and details.\\

\noindent{\bf Acknowledgements}: This paper started as a joint project with our colleague Monica Vazirani, and would not have happened without her; we thank her warmly. We also thank Ben Elias, Honghao Gao, Tamas K\'alm\'an, Anton Mellit, Lenny Ng, Minh-Tam Trinh and Daping Weng for useful discussions, and Johan Asplund for comments on an earlier version of the paper. E.~Gorsky would like to thank Yuri Chekanov for the lifetime influence. 
We would like to thank the anonymous referee for their very thorough reading of the paper and many helpful suggestions.

R.~Casals is supported by the NSF grant DMS-1841913, the NSF CAREER grant DMS-1942363, the Alfred P. Sloan Foundation and a UC Davis College of L\&S Dean's Fellowship. The work of E.~Gorsky was partially supported by the NSF grants DMS-1700814, DMS-1760329 and DMS-2302305. Parts of this work were done when M.~Gorsky's participated in the junior trimester program ``New Trends in Representation Theory'' at the Hausdorff Institute for Mathematics in Bonn. M.~Gorsky was partially supported by the French ANR grant CHARMS (ANR-19-CE40-0017) and received funding from the European Research Council (ERC) under the European Union’s Horizon 2020 research and innovation programme (grant agreement No. 101001159). J.S. was partially supported by CONAHCyT project CF-2023-G-106.
\hfill$\Box$


\section{Braid Varieties and Augmentation Varieties}\label{sec:braid}

In this section we introduce and start studying braid varieties. Part of Theorem \ref{thm:main1} is proven in this section, with the holomorphic symplectic structure being discussed in Section \ref{sec:symplecticform}. This section also discusses the torus actions on braid varieties and their quotients, which relate to augmentation varieties.

{\bf Notations for the braid group}. Let $n\in\N$. The braid group $\Br_n$ on $n$-strands is presented with $n-1$ generators $\sigma_i$, $i\in[1,n-1]$, and relations
\begin{equation}
\label{eq:braid}
\sigma_i\sigma_{i+1}\sigma_i=\sigma_{i+1}\sigma_{i}\sigma_{i+1},\,  \text{for} \, i = 1, \dots, n-2, \qquad \sigma_i\sigma_j=\sigma_j\sigma_i\ \text{for}\ |i-j|\ge 2, i,j\in[1,n-1].
\end{equation}
In this article, we mainly work with the positive braid monoid $\Br^+_n\sse \Br_n$ generated by the nonnegative powers of the generators $\sigma_i$, $i\in[1,n-1]$. By definition, a (positive) braid word is a product expression of non-negative powers of the generators $\sigma_i$ where no relations are being applied. For instance, the two braid words $\sigma_1\sigma_2\sigma_3\sigma_1$ and $\sigma_2\sigma_1\sigma_2\sigma_3$ are distinct as braid words and represent the same element $[\sigma_1\sigma_2\sigma_3\sigma_1]=[\sigma_2\sigma_1\sigma_2\sigma_3]\in\Br^+_4$.

The symmetric group $S_n$ is the Coxeter group associated to $\Br_n$: it is generated by the transpositions $s_i=(i\ i+1)$, subject to relations \eqref{eq:braid} above and the additional relation $s_i^2=1$, for all $i\in[1,n-1]$. By definition, a reduced expression for a permutation $w\in S_n$ is a minimal length expression for the element $w$ as a product of the generators $s_i$, $i\in[1,n-1]$; the length $\ell(w)$ is defined as the length of such reduced expression. It is well-known that any two reduced expressions are related by a sequence of braid moves \eqref{eq:braid}. Therefore, one can define a  positive braid lift of $w\in S_n$ to $\Br^+_n$ by choosing an arbitrary reduced expression and replacing each generator $s_i$ with the generator $\sigma_i$, for each $i\in[1,n-1]$. We will refer to such positive braid lifts as reduced braid words. To ease notation, we interchangeably use $\sigma_i, \, s_i,$ and sometimes simply $i$ for the braid group generators, $i\in[1,n-1]$.


\subsection{Braid matrices and braid varieties}\label{sec: braid varieties}
Braid varieties are affine algebraic varieties cut out by matrix equations. Their definition relies on the following notion.

\begin{definition}\label{def:braidmatrix}
Let $n\in\N$, $i\in[1,n-1]\in\N$ and $z$ a (complex) variable. The braid matrix $B_i(z)\in \GL(n,\C[z])$ is defined as
$$
(B_i(z))_{jk} := \begin{cases} 1 & j=k \text{ and } j\neq i,i+1 \\
1 & (j,k) = (i,i+1) \text{ or } (i+1,i) \\
z & j=k=i+1 \\
0 & \text{otherwise;}
\end{cases},\qquad\mbox{i.e.}\quad
B_i(z):=\left(\begin{matrix}
1 & \cdots  & & & \cdots & 0\\
\vdots & \ddots & & & & \vdots\\
0 & \cdots & 0 & 1 & \cdots & 0\\
0 & \cdots & 1 & z & \cdots & 0\\
\vdots &  & & &\ddots & \vdots\\
0 & \cdots & & & \cdots & 1\\
\end{matrix}\right).
$$
Given a positive braid word $\beta=\sigma_{i_1}\cdots\sigma_{i_r} \in \mbox{Br}_{n}^{+}$  and $z_{1}, \dots, z_{r}$ complex variables, we define the braid matrix $B_{\beta}(z_1,\ldots,z_r)\in\GL(n,\C[z_1,\ldots,z_r])$ to be the product
$$
B_{\beta}(z_1,\ldots,z_r)=B_{i_1}(z_1)\cdots B_{i_r}(z_r).
$$
\end{definition}

\noindent For instance, it follows from Definition \ref{def:braidmatrix} that $B_{\beta}(0,\dots, 0)$ is simply the permutation matrix associated to the Coxeter projection $\pi(\beta)\in S_n$. Thus, in a sense, braid matrices are deformations of permutation matrices. It is a simple computation to verify the two relations:
\begin{equation}\label{equation: reidemeister3 matrices}
B_{i}(z_1)B_{i+1}(z_2)B_{i}(z_3)=B_{i+1}(z_3)B_{i}(z_2-z_1z_3)B_{i+1}(z_1),\quad \forall i\in[1,n-2],
\end{equation}

\noindent and 

\begin{equation}\label{eqn: commutation rel}
B_{i}(z_{1})B_{j}(z_{2}) = B_{j}(z_{2})B_{i}(z_{1}) \, \text{for} \,  |i - j|\geq2.
\end{equation}

Here are a few useful examples. 

\begin{ex}\label{ex: half twist}
Let us first consider (a lift of) the Coxeter element $\sigma_1\sigma_2\cdots \sigma_{n-1} \in \Br^+_n$. Induction on $n$ shows that
\begin{equation}\label{eqn: braid matrix cox element}
B_{\sigma_1\sigma_2\cdots \sigma_{n-1}}(z_1, \dots, z_{n-1}) =  \left(\begin{matrix} 0 & 0 & \ldots & 0 & 1 \\ 1 & 0 & \ldots &0 & z_1 \\ 0 & 1 & \ldots &  0 & z_2  \\ \vdots & \vdots & \ddots & \vdots & \vdots \\  0 & 0 & \ldots & 1 & z_{n-1}\end{matrix}\right).
\end{equation}
Now we consider the positive braid word $\Delta = (\sigma_{1}\sigma_{2}\cdots\sigma_{n-1})(\sigma_{1}\cdots\sigma_{n-2})\cdots(\sigma_{1}\sigma_{2})\sigma_{1}$, which represents a half-twist. It follows from \eqref{eqn: braid matrix cox element} that its associated braid matrix is
\begin{equation}\label{eqn: braid matrix half twist}
B_{\Delta}\left(z_{1}, \dots, z_{\binom{n}{2}}\right) = \left(\begin{matrix} 0 & 0 & \ldots & 0 & 1\\  0 & 0 &  \ldots & 1 & z_{1} \\ 0 &  0 & \ldots & z_{n} & z_2 \\  \vdots & \vdots & \ddots & \vdots & \vdots \\ 1 & z_{\binom{n}{2}} & \ldots & z_{2n-3} & z_{n-1}
\end{matrix}\right)
\end{equation}

\noindent  Let $\Delta'\in \Br^+_n$ be \emph{any} positive braid lift of the longest element $w_{0}$ of $S_{n}$. It then follows from the braid relation \eqref{equation: reidemeister3 matrices} that
\begin{equation}\label{eqn: braid matrix general half twist}
B_{\Delta'}\left(z_{1}, \dots, z_{\binom{n}{2}}\right) = \left(\begin{matrix} 0 & 0 & \ldots & 0 & 1\\  0 & 0 &  \ldots & 1 & z_{2,n} \\ 0 &  0 & \ldots & z_{3,n-1} & z_{3,n} \\  \vdots & \vdots & \ddots & \vdots & \vdots \\ 1 & z_{n,2} & \ldots & z_{n,n-1} & z_{n,n}
\end{matrix}\right),
\end{equation}
\noindent where the $z_{i,j}$ are algebraically independent generators of $\C\left[z_{1}, \dots, z_{\binom{n}{2}}\right]$.
\end{ex}

\begin{lemma}
\label{lem: full twist}
Let $\Delta^{2}\in \Br^+_n$ represent the full-twist braid, i.e. the square of the positive braid lift of $w_0\in S_n$ to the braid group. Then its braid matrix can be decomposed as
$$
B_{\Delta^{2}}\left(z_{1}, \dots, z_{\binom{n}{2}}, w_{1}, \dots w_{\binom{n}{2}}\right) =LU =\left(\begin{matrix} 1 & 0 & \ldots & 0\\ c_{21} & 1 & \ldots & 0 \\
\vdots & \cdots & \ddots & 0\\
c_{n1} & \cdots & \cdots & 1\\
\end{matrix}\right)\left(\begin{matrix} 1 & u_{12} & \ldots & u_{1n}\\ 0 & 1 & \ldots & u_{2n} \\
0 & \cdots & \ddots & u_{n-1,n}\\
0 & \cdots & \cdots & 1\\
\end{matrix}\right),
$$
where $c_{ij} \in \C\left[z_{1}, \dots, z_{\binom{n}{2}}\right]$ and $u_{ij} \in \C\left[w_{1}, \dots, w_{\binom{n}{2}}\right]$ are algebraically independent generators.
\end{lemma}

\begin{proof}
By Example \ref{ex: half twist}, $B_{\Delta}= L w_0=w_0U$. Hence $B_{\Delta^{2}}=B_{\Delta}B_{\Delta}=Lw_0w_0U=LU$.
\end{proof}

Let us now use braid matrices to define the central object of interest in this manuscript.
 
\begin{definition}
Let $\beta=\sigma_{i_1}\cdots\sigma_{i_r} \in \mbox{Br}_{n}^{+}$ be a positive braid word. The braid variety $X_{0}(\beta)\sse\C^r$ associated to $\beta$ is the affine closed subvariety given by
$$
X_{0}(\beta):=\left\{(z_1,\ldots,z_r) : B_{\beta}(z_1,\ldots,z_r)\mbox{ is upper-triangular}\right\}\sse\C^{r}.
$$
Let $\pi \in S_{n}$ be considered as a permutation matrix. We define the braid variety $X_{0}(\beta; \pi) \sse \C^{r}$ as
$$
X_{0}(\beta; \pi) :=\left\{(z_1, \ldots, z_r) : B_{\beta}(z_1,\ldots,z_r)\pi \mbox{ is upper-triangular}\right\}\sse\C^{r}.$$
It follows from the braid relation \eqref{equation: reidemeister3 matrices} that different presentations of the same braid $[\beta]\in \Br_n$ yield algebraically isomorphic braid varieties.
\end{definition}

Let us give some simple examples of braid varieties.

\begin{ex}\label{ex: braid varieties}
Consider the positive braid associated to the full twist $\beta = \FT$. Lemma \ref{lem: full twist} implies that $X_{0}(\FT)$ is given by the equations $c_{ij} = 0$, and thus the braid variety is the affine space $X_{0}(\FT) \cong \C^{\binom{n}{2}}$, with coordinates being $u_{ij}$. Similarly, Example \ref{ex: half twist} implies that the braid variety $X_{0}(\Delta; w_{0}) = \{\operatorname{pt}\}$ is a point.
\end{ex}

The computation in Example \ref{ex: half twist} shows that $X_{0}(\beta; w_{0})$ admits a closed embedding into $X_{0}(\beta\cdot\Delta)$:
\[
\iota: X_0(\beta; w_0) \to X_0(\beta\cdot\Delta), \qquad (z_1, \dots, z_{\ell}) \mapsto (z_{1}, \dots, z_{\ell}, 0, 0, \dots, 0)
\]
where there are $\binom{n}{2}$ zeroes in $(z_1, \dots, z_{\ell}, 0, \dots, 0)$. In general, if $\Pi \in \mbox{Br}^{+}_{n}$ is a positive lift of a permutation $\pi \in S_{n}$ then $X_{0}(\beta; \pi)$ embeds into $X_{0}(\beta\cdot\Pi)$. Let us now establish the general dimension and smoothness for braid varieties.

\begin{thm}
\label{thm: smooth}
Let $\beta\in \Br_n^+$ be a positive braid of length $\ell(\beta)$. Then, the braid varieties $X_{0}(\beta\cdot \Delta; w_0)$ and $X_{0}(\beta\cdot\FT)$ are smooth of dimension $\ell(\beta)$ and $\ell(\beta)+\binom{n}{2}$, respectively. In addition, $X_{0}(\beta\cdot\FT)\simeq X_{0}(\beta\cdot \Delta; w_0)\times \C^{\binom{n}{2}}$.
\end{thm}

\begin{proof}
The variety $X_0(\beta\cdot \FT)$ is defined by the condition that $B_{\beta\cdot \FT}$ is an upper triangular matrix. 
By Lemma \ref{lem: full twist}, we get 
$$
B_{\beta\cdot \FT} = B_{\beta}B_{\FT}=B_{\beta}LU=B_{\beta \cdot \Delta}w_0U.
$$
This is upper-triangular if and only if $B_{\beta \cdot \Delta}w_0$ is upper-triangular, which is precisely the condition defining $X_{0}(\beta\cdot \Delta; w_0)$. Therefore, $X_{0}(\beta\cdot \FT)\simeq X_{0}(\beta \cdot \Delta; w_0)\times \C^{\binom{n}{2}}$, with $\C^{\binom{n}{2}}$ being the coordinates on the upper unitriangular matrix $U$. Now, by \eqref{eqn: braid matrix half twist} we have that $B_{\Delta}w_0$ is a lower unitriangular matrix, and thus we can write 
\begin{equation}\label{eqn: beta delta dec}
B_{\beta \cdot \Delta}w_0 = B_{\beta}L
\end{equation}
where $L$ is lower unitriangular. If we have a point in $X(\beta \cdot \Delta; w_0)$ we  obtain $B_{\beta \cdot \Delta}w_0 = U'$, an upper triangular matrix. Together with \eqref{eqn: beta delta dec} we obtain
$$
B_{\beta}^{-1}=L(U')^{-1}.
$$
Note that the existence of an $LU$ decomposition $M=LU''$ is an open condition on $M$, namely the non-vanishing of principal minors; also, if an $LU$ decomposition exists, it is unique provided that $L$ has 1s on the diagonal. Therefore $X_{0}(\beta\cdot \Delta; w_{0})$ is isomorphic to an open subset in the affine space $\C^{\ell(\beta)}$. Hence, it is smooth of dimension $\dim X_{0}(\beta\cdot\Delta; w_0) = \ell(\beta)$, and $X_0(\beta\cdot \FT)$ is also smooth of 
dimension $\dim X_0(\beta\cdot \FT)=\ell(\beta)+\binom{n}{2}$, as required. 
\end{proof}

In the proof of the previous result we obtained that $X_0(\beta\cdot \Delta; w_0)$ is open in the affine space $\C^{\ell(\beta)}$. Since this will be used again later, let us state it as a separate result.

\begin{lemma}\label{lem:dbs open}
Let $\beta \in \Br_n^+$ be a positive braid of length $\ell(\beta)$. Then, the braid variety $X(\beta\cdot \Delta; w_0)$ is open in the affine space $\C^{\ell(\beta)}$, and it is given by the non-vanishing of the leading principal minors of the matrix $B^{-1}_{\beta}(z_1, \dots, z_{\ell})$.
\end{lemma}

Lemma \ref{lem:dbs open} implies that the braid variety $X(\beta\cdot \Delta; w_0)$ is isomorphic to the (half-decorated) double Bott-Samelson cell studied in \cite{SW}. Note that a similar smoothness result was proved in \cite[Theorem 2.30]{SW}. The braid varieties associated to 2-stranded braids $\beta\in \Br^+_2$ are smooth varieties whose equations closely relate to Euler's continuants \cite{Euler64}. 

\begin{ex}\label{ex:trefoil} 
Consider $\beta=\sigma_1^3\in Br^+_2$, the braid variety $X_{0}(\sigma_1^5) = X_{0}(\sigma_{1}^{3} \cdot \Delta^2)$ is defined by the equation:
$$
B(z_1)B(z_2)B(z_3)B(z_4)B(z_5)\mbox{ is upper-triangular}.
$$
This condition can written as (cf. Lemma \ref{lem: full twist})
$$
B(z_1)B(z_2)B(z_3)\left(\begin{matrix} 1 & 0\\ z_4 & 1\end{matrix}\right)\left(\begin{matrix} 1 & z_5\\ 0 & 1\end{matrix}\right)\mbox{ is upper-triangular},
$$
and equivalently
$$
B(z_1)B(z_2)B(z_3)\left(\begin{matrix} 1 & 0\\ z_4 & 1\end{matrix}\right)=\left(\begin{matrix} z_2+(z_2z_3+1)z_4 & z_2z_3+1\\ z_1z_2+(z_1+(z_1z_2+1)z_3)z_4+1 & z_1+(z_1z_2+1)z_3\end{matrix}\right)\mbox{ is upper-triangular}.
$$

Note that we have
$$\left(\begin{matrix} 1 & 0\\ z_4 & 1\end{matrix}\right)=\left(\begin{matrix} 0 & 1\\ 1 & z_4\end{matrix}\right)\left(\begin{matrix} 0 & 1\\ 1 & 0\end{matrix}\right)=\left(\begin{matrix} 0 & 1\\ 1 & z_4\end{matrix}\right)w_0,$$
and thus the condition above, in the coordinates $(z_1,z_2,z_3,z_4)\in\C^4$, is in fact the equation for $X_{0}(\sigma_1^3\cdot \Delta; w_{0})$. This implies that $X_{0}(\sigma_1^{3}\cdot \FT)$ is isomorphic to $X_{0}(\sigma_1^3\cdot \Delta; w_{0})$ times an affine line $\C=\Spec(\C[z_5])$. This is proven in general in Theorem \ref{thm: smooth}. It thus suffices to understand $X_{0}(\sigma_1^3\cdot \Delta; w_{0})$. For that, consider the equation above:
\begin{equation}
\label{eq:char trefoil}
X_{0}(\sigma_1^3\cdot \Delta; w_{0})=\{(z_1,z_2,z_3,z_4)\in\C^4:1+z_1z_2+z_4(z_1+z_3+z_1z_2z_3)=0\}\sse\C^4,
\end{equation}
which cuts out a hypersurface, and should be smooth according to Theorem \ref{thm: smooth}. Indeed, note that we must have $z_1+z_3+z_1z_2z_3\neq0$, otherwise the defining Equation \ref{eq:char trefoil} would imply $1+z_1z_2=0$, and in these constraints $z_1=z_1+z_3(z_1z_2+1)=z_1+z_3+z_1z_2z_3=0$. This is a contradiction; thus, $z_1+z_3+z_1z_2z_3\neq0$ in $X_{0}(\sigma_1^3\cdot \Delta; w_{0})$. In consequence, $X_{0}(\sigma_1^3\cdot \Delta; w_{0})$ is isomorphic to the open subset
$$X_{0}(\sigma_1^3\cdot \Delta; w_{0})=\{(z_1,z_2,z_3)\in\C^3:(z_1+z_3+z_1z_2z_3)\neq0\}\sse\C^3,$$
since the coordinate $z_4$ can be obtained uniquely from any points $(z_1,z_2,z_3)\in\C^3$ in this subset. This shows that $X_{0}(\sigma_1^3\cdot \Delta; w_{0})$ is smooth.

In fact, this provides a rather simple description for this braid variety: it is the open set foliated by the smooth hypersurfaces $(z_1+z_3+z_1z_2z_3)=a$, $a\in\C^*$. For a fixed $a\in\C^*$, the Stein deformation type of the affine surface $\{(z_1+z_3+z_1z_2z_3)=a\}$ is described in \cite[Section 4.1]{CM}.
\hfill$\Box$
\end{ex}

\begin{remark} In the case of positive braids associated to algebraic knots $K\sse\R^3$, the braid varieties can be similarly described symplectically using the arboreal skeleta constructed in \cite{Ca}. In general, following the lines of Example \ref{ex:trefoil}, the braid varieties for $(2,n)$-torus links can be similarly described in terms of affine hypersurfaces.
\end{remark}

Note that we can write
$$X_{0}(\sigma_1^3\cdot \Delta; w_{0})\cong\{(z_1,z_2,z_3,t): (z_1+z_3+z_1z_2z_3)t=1\}\sse\C^3\times\C^*_t,$$
and thus there exists a $\C^*$-action on $X_{0}(\sigma_1^3\cdot \Delta; w_{0})$ whose quotient yields the affine hypersurface $\{z_1+z_3+z_1z_2z_3=1\}\sse\C^3$. The feature of admitting certain (complex) torus actions with interesting quotients is a general property of braid varieties, as we will now see.



\subsection{Torus actions on braid varieties}\label{sect:torus action} Let $[\beta]\in \Br^+_n$ be a positive braid with a fixed positive braid word $\beta=\sigma_{i_1}\cdots\sigma_{i_r}$. Consider the Cartan subgroup $\mathbb{T}\cong(\C^*)^n\sse\GL(n,\C)$ of diagonal matrices, and its quotient $T$ by the subgroup of scalar invertible matrices. In this section we construct an algebraic $T$-action on the braid variety $X_{0}(\beta)$. First, we observe that
\begin{equation}
\label{eq: slide diagonal}
\displaystyle\left(\begin{matrix} t_1 & 0\\ 0 & t_2\\ \end{matrix}\right)\left(\begin{matrix} 0 & 1\\ 1 & z\end{matrix}\right)=\left(\begin{matrix}0 & 1\\
1 & \frac{t_2}{t_1}z\\ \end{matrix}\right)\left(\begin{matrix} t_2 & 0\\ 0 & t_1\\ \end{matrix}\right).
\end{equation}
Let $D_{\bt}=\diag(t_1,\ldots,t_n)\in\mathbb{T}$ be a diagonal matrix. In general, we have $D_{\bt}B_i(z)=B_i\left(\frac{t_{i+1}}{t_i}z\right)D_{s_i(\bt)}$, for $s_i$ the Coxeter projection of $\sigma_i$. Thus
\begin{equation}
\label{eq:slide diagonal 2}
D_{\bt}B_{i_1}(z_1)\cdots B_{i_r}(z_r)=B_{i_1}(c_1z_1)\cdots B_{i_r}(c_rz_r)D_{w(\bt)},
\end{equation}
where $r=\ell(\beta)$, $c_k=t_{w_k(i_{k}+1)}t^{-1}_{w_k(i_{k})}$, $w_k= s_{i_1}\cdots s_{i_{k-1}}$ and $w=w_{r+1}$ is the permutation corresponding to $\beta$. The torus actions we study are defined as follows.

\begin{definition}\label{def:torusaction} Let $\beta$ be a positive $n$-braid word of length $r=\ell(\beta)$. The 
action of the torus $\mathbb{T}\cong(\C^{*})^{n}$ on affine space $\C^{\ell(\beta)}$ is given by
$$
\bt.(z_1,\ldots,z_r):=(c_1z_1,\ldots,c_rz_r),\quad \bt\in\mathbb{T}, \quad(z_1,\ldots,z_r)\in\C^r,
$$
where $c_i$ are defined as above, $i\in[1,r]$. Note that this  $\mathbb{T}$-action preserves the braid variety $X_{0}(\beta)\sse\C^{r}$ thanks to relation \eqref{eq:slide diagonal 2}. Let $T := \mathbb{T}/\C^*_{\diag} \cong (\C^*)^{n-1}$, the quotient of $\mathbb{T}$ by the diagonal subtorus. By definition, the $T$-action $T\times X_{0}(\beta)\lr X_{0}(\beta)$ on the braid variety $X_{0}(\beta)$ is the quotient of the restriction of the above $\mathbb{T}$-action to $X_{0}(\beta)$ by the diagonal subtorus $\C^*_{\diag}$. Note that the $\mathbb{T}$-action descends to the $T$-action quotient since the diagonal subtorus $(t,\ldots,t)\sse\mathbb{T}$ acts trivially on $X_{0}(\beta)$.\hfill$\Box$
\end{definition}

\begin{ex}\label{ex:torus action}
    Let us consider the braid word $\beta = \sigma_1\sigma_2\sigma_2\sigma_1\sigma_2$. If $\bt = (t_1, t_2, t_3) \in (\C^{*})^{3}$ we have
    \[
    \bt.(z_1, z_2, z_3, z_4, z_5) = \left(\frac{t_2}{t_1}z_1, \frac{t_3}{t_1}z_2, \frac{t_1}{t_3}z_3, \frac{t_1}{t_2}z_2, \frac{t_3}{t_2}z_5.\right)
    \]
\end{ex}

\begin{remark}\label{rmk:z weights}
We can read the $t_i/t_j$ factor of each $z_k$-variable from the braid $\beta$ as follows. For the weight of $z_{k}$, consider the strands that are incident on the left to the $k$-th crossing of $\beta$ and follow them until the left border of $\beta$. If the strand incident from the bottom (resp.~the top) to the $k$-th crossing arrives at the $i$-th (resp.~$j$-th) level strand at the leftmost end, then the scalar factor for $z_k$ is $t_i/t_j$. For example, the next figure illustrates that for $z_3$ in Example \ref{ex:torus action} we have $\color{teal}{t_{1}}\color{black}{/}\color{red}{t_3}$.\hfill$\Box$
\begin{center}
    \includegraphics[scale=1]{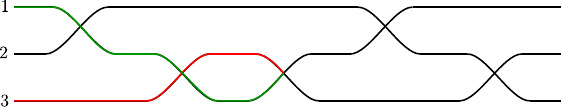}
\end{center}
\end{remark}

The torus action on $\C^r$ in Definition \ref{def:torusaction} depends on the choice of braid word $\beta$. Nevertheless, we have the following result. 

\begin{lemma}
Let $\beta,\beta'$ be two positive presentations of the same braid, i.e.~$[\beta] = [\beta']$. Then, the algebraic isomorphism $X_0(\beta) \cong X_0(\beta')$ defined by formulas \eqref{equation: reidemeister3 matrices} and \eqref{eqn: commutation rel} is $T$-equivariant.
\end{lemma}
\begin{proof}
Let us verify that applying the relation \eqref{equation: reidemeister3 matrices} defines a $T$-equivariant isomorphism. For this, it suffices to consider $n = 3$, $\beta = \sigma_1\sigma_2\sigma_1$ and $\beta' = \sigma_2\sigma_1\sigma_2$. The action of $T$ on $\C^3$ that yields the action on $X_0(\sigma_1\sigma_2\sigma_1)$ is given by:
\begin{equation}\label{eqn:torus action 1}
(t_1, t_2, t_3).(z_1, z_2, z_3) = \left(\frac{t_{2}}{t_{1}}z_1, \frac{t_3}{t_1}z_2, \frac{t_3}{t_2}z_3\right)
\end{equation}
while the $T$-action on $\C^3$ given the action on $X_0(\sigma_2\sigma_1\sigma_2)$ is given by:
\begin{equation}\label{eqn:torus action 2}
(t_1, t_2, t_3).(w_1, w_2, w_3) = \left(\frac{t_3}{t_2}w_1, \frac{t_3}{t_1}w_1, \frac{t_2}{t_1}w_3\right).
\end{equation}

Then \eqref{eqn:torus action 1} and \eqref{eqn:torus action 2} imply that the map $(z_1, z_2, z_3) \mapsto (z_3, z_2-z_1z_3, z_1)$ is $T$-equivariant. The verification that \eqref{eqn: commutation rel} also induces a $T$-equivariant isomorphism is similar. 
\end{proof}

We also have the following result.

\begin{lemma}\label{lem: torus decomposition}
The $T$-action preserves the product decomposition $$X_{0}(\beta\cdot\FT) \cong X_{0}(\beta\cdot \Delta; w_0)\times \C^{\binom{n}{2}}$$ established in Theorem \ref{thm: smooth}. 
\end{lemma}
\begin{proof}
This follows from uniqueness of the LU-decomposition, which is indeed unique if the lower triangular matrix has $1$'s on the diagonal.   
\end{proof}

Let $c(\beta)$ be the number of cycles in the cycle decomposition of the Coxeter projection $\pi(\beta)\in S_n$, i.e. the number of cycles of $\beta$ understood as a permutation. The braid $[\beta]\in \Br_n$ closes up (either through the rainbow or $(-1)$-framed closure, see Figure \ref{fig:rainbow_closure} and Section \ref{sect:augmentations} for more details) to a knot in $\R^3$ if and only if $c(\beta)=1$, and there are $(n-1)!$ such permutations $\pi(\beta)\in S_n$. For a braid associated to a knot, we have the following result.

\begin{lemma}\label{lem: free action}
Let $\beta$ be a positive braid word, $[\beta]\in \Br_n^+$, with $c(\beta)=1$. Then the action  
of $T\cong(\C^*)^{n-1}$ on the braid variety $X_{0}(\beta)$ is free.
\end{lemma}
 
\begin{proof}
Let $(z_{1}, \dots, z_{r}) \in X_{0}(\beta)$ and assume $\bt.(z_{1}, \dots, z_{r}) = (z_{1}, \dots, z_{r})$ for some $\bt \in (\C^{*})^{n}$. In particular, we have that $B_{\beta}(z) = B_{\beta}(\bt. z)$. Thanks to Equality \eqref{eq:slide diagonal 2}, we have that $D_{\bt}B_{\beta}(z)D_{w(\bt)}^{-1} = B_{\beta}(z)$. Since $z \in X_{0}(\beta)$, the matrix $B_{\beta}(z)$ is upper triangular, and therefore its diagonal entries must be nonzero, as $\det(B_{\beta}(z)) = \pm 1$. From the equation
$$D_{\bt}B_{\beta}(z)D_{w(\bt)}^{-1} = B_{\beta}(z),$$
it follows that $t_{i}t_{w(i)}^{-1} = 1$ for every $i = 1, \dots, n$. Given that $c(\beta)=1$, we must have that $t_{i} = t_{j}$ for all $i, j$ and the result follows. 
\end{proof}

\begin{cor}
Let $\beta$ be a positive braid word, $[\beta]\in \Br_n^+$, with $c(\beta) = 1$. Then the action of $T \cong (\C^*)^{n-1}$ on $X_0(\beta\cdot \Delta; w_0)$ is free.
\end{cor}
\begin{proof}
Note that $c(\beta) = c(\beta\FT)$ and thus, by Lemma \ref{lem: free action}, the $T$-action on $X_0(\beta\FT)$ is free. The result now follows from Lemma \ref{lem: torus decomposition}.
\end{proof}

\begin{cor}
Let $\beta$ be a positive braid word, $[\beta]\in \Br_n^+$, with $c(\beta)=1$. Then the quotients of the braid varieties $X_{0}(\beta \cdot \FT)/T$ and $X_{0}(\beta \cdot \Delta; w_{0})/T$ are smooth and of dimension $\ell(\beta) + \binom{n}{2} - n +1$ and $\ell(\beta) - n + 1$, respectively.
\end{cor}

\begin{remark}
Similarly to \cite[Corollary 4.8]{GL} one can argue that for $c(\beta)=1$ we have  $$X_0(\beta)=(X_0(\beta)/T)\times T.$$ Indeed, fixing the diagonal entries of $B_{\beta}(z)$ provides the corresponding principal $T$-bundle over $X_0(\beta)/T$ with a section, hence this bundle is trivial.
\end{remark}

The hypothesis $c(\beta)=1$ in Lemma \ref{lem: free action} is needed, as the $T$-actions on the braid varieties will in general fail to be free. For instance, consider the 2-stranded braid $\beta=\sigma_1^4$ and its braid variety $$X_{0}(\beta)\cong\{(z_1,z_2,z_3,z_4)\in\C^4:z_1+z_3(1+z_1z_2)=0\}.$$
The 
$T$-action scales $z_1$ and $z_3$ by $t\in T\cong\C^*$, and scales $z_2$ and $z_4$ by $t^{-1}$. Hence, it has a fixed point $(z_1,z_2,z_3,z_4)=(0,0,0,0)\in X_{0}(\beta)$. The following remark explains how to proceed in the case that $c(\beta)\neq1$.

\begin{remark}\label{rmk: free subtori}
Consider a positive braid word $\beta$ such that $[\beta]\in \Br_n^+$ closes up to a link with $k$ connected components, i.e. $c(\beta)=k$. 
Let $w=w(\beta)$ be the permutation in $S_n$ corresponding to $\beta$. Let $C_1,\ldots,C_k$ be the disjoint cycles in $w$, and  $C_{j}=(a_{j, 1}\dots a_{j, \ell_{j}})$.
Now let $T_c \subseteq T$ be the $(n-k)$-dimensional torus given by the equations $t_{a_{1, \ell_{1}}} = t_{a_{2, \ell_{2}}} = \cdots = t_{a_{k, \ell_{k}}}$. Recall that $T=\mathbb{T}/\C^*$,
so we can  instead consider the torus $\widetilde{T_c}\subseteq \mathbb{T}$ given by the equations $t_{a_{1, \ell_{1}}} = t_{a_{2, \ell_{2}}} = \cdots = t_{a_{k, \ell_{k}}}=1$.
The projection $\widetilde{T_c}\to T_c$ is an isomorphism, and the actions of $T_c$, $\widetilde{T_c}$
 on the braid varieties coincide, so we will not distinguish between these tori.

 The same argument as in the proof of Lemma \ref{lem: free action} shows that $T_c$ acts freely on $X_{0}(\beta)$. Note that we obtain that the quotient braid variety $X_{0}(\beta\cdot \Delta; w_{0})/T_c$ is a smooth variety of dimension $\ell(\beta) - n + k$, and a similar result holds for the quotient $X_{0}(\beta\cdot\FT)/T_c$.

\end{remark}

This concludes the discussion on the torus  action on $X_0(\beta)$. The geometric structures discussed during the article, e.g. \dec s and holomorphic symplectic structures, are compatible with these torus actions, and will be studied for the braid varieties $X_0(\beta)$ and their quotients $X_0(\beta)/T$.


\subsection{Toric charts in braid varieties} 
\label{sec: opening}

In this subsection, we construct a codimension-$0$ toric chart $T_{\tau}\sse X_{0}(\beta \cdot \Delta;w_{0})$ associated to an (arbitrary) ordering $\tau\in S_{l(\beta)}$ of the crossings of the positive braid word $\beta$. For that, consider two $n$-braid words $$\beta=\sigma_{i_1}\sigma_{i_2}\cdot\ldots\cdot\sigma_{i_{k-1}}\sigma_{i_k}\sigma_{i_{k+1}}\cdot\ldots\cdot\sigma_{i_l},\quad  \beta'=\sigma_{i_1}\sigma_{i_2}\cdot\ldots\cdot\sigma_{i_{k-1}}\sigma_{i_{k+1}}\cdot\ldots\cdot\sigma_{i_l},$$
i.e. $\beta'$ is obtained from $\beta$ by removing the $k$th crossing $\sigma_{i_k}$. We will construct a rational map $X_{0}(\beta\cdot\FT) \dashrightarrow X_{0}(\beta'\cdot\FT) \times \C^{*}$ that identifies the latter variety with an explicit open set in $X_{0}(\beta\cdot\FT)$.

We start with the following lemma.

\begin{lemma}
\label{lem:slide triangular}
Let $L$ and $U$ be invertible lower- and upper-triangular matrices, respectively, and $i = 1, \dots, n-1$. Then there exist lower- and upper-triangular matrices $\widetilde{L}$ and $\widetilde{U}$ such that
$$
B_i(z)U=\widetilde{U}B_i\left(\frac{u_{i+1,i+1}z+u_{i,i+1}}{u_{i,i}}\right),\ LB_i(z)=B_i\left(\frac{l_{i+1,i+1}z+l_{i+1,i}}{l_{i,i}}\right)\widetilde{L}.
$$
Moreover, $\widetilde{u}_{i,i+1}=\widetilde{l}_{i+1,i}=0$ and $\widetilde{u}_{k,k} = u_{s_{i}(k), s_{i}(k)}$ for every $k$. 
\end{lemma}

\begin{proof}
We prove the statement for the upper-triangular matrix $U$, the case of $L$ is proven analogously. First, note that
\begin{equation}\label{eqn:BU}
(B_{i}(z)U)_{j,k} = \begin{cases} u_{j,k} & \text{if} \; j \not\in \{i, i+1\}, \\  u_{i+1, k} & \text{if} \; j = i, \\ u_{i,k} + zu_{i+1, k} & \text{if} \; j = i+1, \end{cases}
\end{equation}
and
\begin{equation}\label{eqn:UB}
(\widetilde{U}B_{i}(w))_{j,k} = \begin{cases} \widetilde{u}_{j,k} & \text{if} \; k \not\in \{i, i+1\}, \\ \widetilde{u}_{j, i+1} & \text{if} \; k = i, \\ \widetilde{u}_{j, i} + w\widetilde{u}_{j, i+1} & \text{if} \; k = i+1. \end{cases} 
\end{equation}
Now, assume that we know the matrix $U$ and $z$, and we want to solve for the entries of $\widetilde{U}$ and $w$ in such a way that $B_{i}(z)U = \widetilde{U}B_{i}(w)$. Note that \eqref{eqn:BU} and \eqref{eqn:UB} force $u_{j,k} = \widetilde{u}_{j,k}$ if $j, k \not\in \{i, i+1\}$. In particular, $\tilde{u}_{k,k} = u_{k,k}$ if $j \not= i, i+1$, and the matrix $\widetilde{U}$ is upper triangular except for, perhaps, the $i$ and $i+1$-st row and column. 

Setting $j = i = k$ in \eqref{eqn:UB} and \eqref{eqn:BU} we obtain $u_{i+1, i} = \widetilde{u}_{i, i+1}$. Since $U$ is upper triangular, this forces $\widetilde{u}_{i, i+1} = 0$. Now setting $j = i$ and $k = i+1$ gives $u_{i+1, i+1} = \widetilde{u}_{i,i} + w\widetilde{u}_{i, i+1}$, so $\widetilde{u}_{i,i} = u_{i+1, i+1}$. Similarly, setting $j = i+1, k = i$ we obtain $u_{i,i} + zu_{i+1, i} = \widetilde{u}_{i+1, i+1}$, so the upper triangularity of $U$ gives us $\widetilde{u}_{i+1, i+1} = u_{i,i}$. Note that at this point we have shown that $\widetilde{u}_{k,k} = u_{s_{i}k, s_{i}k}$ for every $k$.

If $k \not\in \{i, i+1\}$ then \eqref{eqn:BU} and \eqref{eqn:UB} give us
\[
\widetilde{u}_{i,k} = u_{i+1, k} \qquad \text{and} \qquad \widetilde{u}_{i+1, k} = u_{i,k} + zu_{i+1, k}.
\]
Similarly, if $j \not\in \{i, i+1\}$ we obtain (setting $k = i$) $\widetilde{u}_{j, i+1} = u_{j,i}$ that the we can use to solve for $\widetilde{u}_{j, i}$ in the equation  $u_{j, i+1} = \widetilde{u}_{j, i} + w\widetilde{u}_{j, i+1}$, that we obtain setting $k = i+1$. Note that at this point we have found all entries $\widetilde{u}_{k,j}$, except for $\widetilde{u}_{i+1, i}$, that we must show is $0$. This is where our choice of $w$ in the statement of the lemma comes into play. Indeed, setting $j = i+1 = k$ we obtain $u_{i,i+1} + zu_{i+1, i+1} = \widetilde{u}_{i+1, i} + w\widetilde{u}_{i+1, i+1}$. Thus, since we know that $\widetilde{u}_{i+1, i+1} = u_{i,i}$ we obtain 
\[
w = \frac{u_{i+1, i+1}z + u_{i,i+1}}{u_{i,i}} \Rightarrow \widetilde{u}_{i+1, i} = 0
\]
so the matrix $\widetilde{U}$ is upper triangular and the lemma is proved. 
\end{proof}

The key algebraic equality that incarnates opening a crossing $\sigma_i$ in a positive braid word, in terms of braid matrices, reads
\begin{equation}
\label{eq: factorization}
B_i(z)=U_i(z)D_i(z)L_i(z),
\end{equation}
where the variable $z \in \C^{*}$ associated to that crossing $\sigma_i$ is now assumed to be non-zero. In this equation, we have used the matrices
\begin{equation}\label{eqn: UDL}
U_i(z):=\left(\begin{matrix} 1 & z^{-1}\\ 0 & 1\end{matrix}\right),\quad D_i(z):=\left(\begin{matrix} -z^{-1} & 0\\ 0 & z\end{matrix}\right),\quad
L_i(z):=\left(\begin{matrix} 1 & 0\\ z^{-1} & 1\end{matrix}\right),
\end{equation}
understood as being the $(2\times 2)$-block matrices placed in $i$-th and $(i+1)$-st row and column. Let us now illustrate how the process of opening a crossing occurs at the level of general braid matrices, as follows. Consider the positive braid word $\beta=\beta_1 \sigma_i \beta_2$ and the braid word $\beta'=\beta_1\beta_2$ obtained by opening (i.e. removing) the explicit crossing $\sigma_i$ between $\beta_1,\beta_2$. In order to apply Equation \ref{eq: factorization} we must assume that the variable $z$ associated to the crossing $\sigma_i$ is non-vanishing, and we always do so. Then we write
$$
B_{\beta} = B_{\beta_1}(z_{1}, \dots, z_{r-1})B_i(z)B_{\beta_2}(z_{r+1}, \dots, z_{\ell})=B_{\beta_{1}}U_i(z)D_i(z)L_i(z)B_{\beta_{2}},
$$
and use both Equation \eqref{eq: slide diagonal} and Lemma \ref{lem:slide triangular} to slide the middle matrices to the sides, $U,D$ to the left and $L$ to the right. This results in a decomposition of the form
$$
B_{\beta}=U'D'B_{\beta_{1}}(z'_{1}, \dots, z'_{r-1})B_{\beta_{2}}(z'_{r+1}, \dots, z'_{\ell})L'=U'D'B_{\beta'}(z'_{1}, \dots, z'_{r-1}, z'_{r+1}, \dots, z'_{\ell})L'
$$

\noindent where $U'$, $L'$ and $D'$ are some explicit upper (lower) unitriangular and diagonal matrices, respectively, and $z'_{1}, \dots, z'_{r-1}, z'_{r+1}, \dots, z'_{\ell}$ are polynomial functions on $z_{1}, \dots, z_{r-1}, z_{r}^{\pm 1}, z_{r+1}, \dots, z_{\ell}$. Note that $B_{\beta}(z)L_{1}$ is upper-triangular for some lower-triangular matrix $L_{1}$ if and only if $B_{\beta'}(z')L'L_{1}$ is upper-triangular. These are the first ingredients for the construction of the rational map $$\Omega_{\sigma_{i}}: X_{0}(\beta\FT) \dashrightarrow X_{0}(\beta'\FT) \times \C^{*}.$$

For the second ingredient, we consider a point $(z_{1}, \dots, z_{\ell}, c_{ij}) \in X_{0}(\beta\cdot\FT)$. By Theorem \ref{thm: smooth}, this is equivalent to $B_{\beta}(z)L(c_{ij})$ being upper triangular. Now we open a crossing, so we assume $z_i\neq0$ is non-vanishing: using the decomposition above we obtained that $B_{\beta'}(z'_{1}, \dots, z'_{r-1})L'L(c_{ij})$ is upper triangular. Since $L'L(c_{ij})$ is lower triangular with $1$'s in the diagonal, we can write $L'L(c_{ij}) = L(c'_{ij})$, where $c'_{ij}$ are polynomial functions on $z_{r}^{-1}, z_{r+1}, \dots, z_{\ell}, c_{ij}$. These polynomial functions are the second ingredient. In summary, we obtain the following rational map.

\begin{definition}\label{def:openingcross}
Consider the positive braid word $\beta=\beta_1 \sigma_i \beta_2$ of length $\ell=l(\beta)$, $\beta'=\beta_1\beta_2$, and suppose that the complex variable $z_i$ associated to the (middle) crossing $\sigma_i$ is non-vanishing. By definition, the rational map $\Omega_{\sigma_{i}}$ associated to opening the crossing $\sigma_i$ is
$$
\Omega_{\sigma_{i}}: X_{0}(\beta\FT) \dashrightarrow X_{0}(\beta'\FT) \times \C^{*}, (z_{1}, \dots, z_{\ell}, c_{ij}) \mapsto (z'_{1}, \dots, z'_{r-1}, z'_{r+1}, \dots, z'_{\ell}, c'_{ij}, z_{r}^{-1}),
$$
where $z_i'\in\C[z_1,\ldots,z_{r-1},z_r^{\pm},z_{r+1},\ldots,z_{\ell}],c_{ij}'\in\C[z_r^{-1},z_{r+1},\ldots,z_\ell,c_{ij}]$ are the polynomial functions defined as above.
\end{definition}

In the same notation and hypothesis as above, we have the following result.

\begin{lemma}
\label{lem: open one crossing matrices}
The rational map
$$\Omega_{\sigma_{i}}: X_{0}(\beta\FT) \dashrightarrow X_{0}(\beta'\FT) \times \C^{*}$$
restricts to an isomorphism between the open locus $\{z_{r} \neq 0\}\sse X_{0}(\beta\cdot\FT)$ and $X_{0}(\beta'\cdot\FT) \times \C^{*}$.
\end{lemma}
\begin{proof}
From the construction, see e.g. Lemma \ref{lem:slide triangular}, if we know $z'_{1}, \dots, z'_{r-1}$, $z'_{r+1}, \dots, z'_{\ell}$ and $z_{r}$ then we can reconstruct $z_{1}, \dots, z_{\ell}$, provided $z_{r} \neq 0$. It remains to show that, if we also know $c'_{ij}$ then we can reconstruct $c_{ij}$ as well. For that, we just notice that we can reconstruct $L'$, and we have the equation $L(c_{ij}) = (L')^{-1}L(c'_{ij})$.
\end{proof}

There are two fundamental properties of these rational maps $\Omega_{\sigma_{i}}$: they can be iterated, and they are compatible with the torus action. This leads to the following result. 

\begin{prop}\label{prop:open crossings}
Let $\beta$ be a positive $n$-braid word. For each ordering  $\tau\in S_{\ell(\beta)}$ of the crossings of $\beta$, there exists an open set $\widetilde{T_{\tau}}\sse X_{0}(\beta\cdot\FT)$ such that:
\begin{itemize}
	\item[(i)] $\widetilde{T_{\tau}}\cong(\C^{*})^{\ell(\beta)} \times X_{0}(\FT) = (\C^{*})^{\ell(\beta)} \times \C^{\binom{n}{2}}$.\\
	
	\item[(ii)] $\widetilde{T_{\tau}}$ is given by the non-vanishing of Laurent polynomials in $z_{r_{1}}, z'_{r_{2}}, z''_{r_{3}}, \dots, z^{(\ell - 1)}_{r_{\ell}}$; these latter variables can be taken as coordinates of the $(\C^{*})^{\ell(\beta)}$-factor.\\
	
	\item[(iii)] $\widetilde{T_{\tau}}$ is stable under the action of $(\C^{*})^{n-1}$ on $X_{0}(\beta \cdot \FT)$.
\end{itemize} 
\end{prop}
\begin{proof} Parts (i) and (ii) follow from the discussion above, applied iteratively. Thus, the only assertion that needs a proof is the stability under the torus action in Part (iii). For that, we need to show that $z_{r_{1}}, z'_{r_{2}}, \dots, z^{(\ell -1)}_{r_{\ell}}$ are all homogeneous under the $(\C^{*})^{n-1}$-action. This is proven in Lemmas \ref{lemma: moving admissible matrices} and \ref{lemma: moving admissible diagonal} below (both lemmas are independent of the intervening material), and their corresponding analogues in the case of lower-triangular matrices.
\end{proof}

Proposition \ref{prop:open crossings} and the relation between the braid varieties $X_{0}(\beta \cdot \FT)$ and $X_{0}(\beta \cdot \Delta; w_{0})$, as established in Theorem \ref{thm: smooth}, imply the following result.

\begin{cor}\label{cor:open crossings} Let $\beta$ be a positive $n$-braid word. For each ordering  $\tau\in S_{\ell(\beta)}$ of the crossings of $\beta$, there exists an open set $T_{\tau}\sse X_{0}(\beta\cdot \Delta; w_{0})$ which is isomorphic to a torus $T_{\tau}\cong(\C^{*})^{\ell(\beta)}$ and stable under the action of $(\C^{*})^{n-1}$ on $X_{0}(\beta \cdot \Delta; w_{0})$.
\end{cor}

The union of the toric charts $T_{\tau}$ in Corollary \ref{cor:open crossings}, as $\tau\in S_{\ell(\beta)}$ ranges through all the possible orderings, does not necessarily cover the entire variety $X_{0}(\beta \cdot \Delta; w_{0})$. Fortunately, we can show that it does cover it up to codimension 2. 

\begin{thm}
\label{thm: codim 2} Let $\beta$ be a positive braid word. The complement
$$X_{0}(\beta \cdot \Delta;w_{0})\setminus\left(\bigcup_{\tau\in S_{\ell(\beta)}}T_{\tau}\right)\sse X_{0}(\beta \cdot \Delta;w_{0})$$
has codimension at least $2$.
\end{thm}
\begin{proof}
Let us prove this by induction on the length $\ell(\beta)\in\N$. The base case, $\ell(\beta) = 0$ holds, as $X_{0}(\Delta; w_{0}) = \{\operatorname{pt}\}$, see Example \ref{ex: braid varieties}. Note that for the case $\ell(\beta) = 1$, $\beta = \sigma_{i}$ for some $i\in[1,n-1]$, and $X_{0}(\beta \Delta; w_{0})$ is defined by the condition that $B_{i}(z)^{-1}$ admits an $LU$-decomposition. (See the proof of Theorem \ref{thm: smooth}.) Note that  $B_{i}(z)^{-1}$ is the identity everywhere except in the $i$ and $i+1$-st row and columns, where it is 
\[
\left(\begin{matrix} -z & 1 \\ 1 & 0\end{matrix}\right).
\]
So that $B_{i}(z)^{-1}$ admitting an $LU$-decomposition is equivalent to the non-vanishing $z \neq 0$ (which is obviously equivalent to the nonvanishing of the principal minors of $B_{i}(z)^{-1}$). Thus $X_{0}(\beta\cdot\Delta; w_{0}) = \C^*$; thus the statement also holds in this case.

For the induction step, we assume the statement to be true for length $\ell\in\N$ and suppose that $\ell(\beta) = \ell + 1$. Let $U_{1} := \{z_{1} \neq 0\}$ and $U_{2} := \{z_{2} \neq 0\}$ and let $\beta', \beta''$ be the braids we obtain by opening the first and second crossings of $\beta$, respectively. In particular, $U_{1} = X_{0}(\beta' \cdot \Delta; w_{0}) \times \C^{*}$ and $U_{2} = X_{0}(\beta'' \cdot \Delta; w_{0}) \times \C^{*}$. By the induction assumption, $U_{1}$ and $U_{2}$ can be covered up to codimension 2 by opening crossings in the positive braids $\beta', \beta''$, respectively. Moreover, the complement of $U_{1} \cup U_{2}$ is $\{z_{1} = 0\} \cap \{z_{2} = 0\}$. By Lemma \ref{lem:dbs open}, $X_0(\beta\Delta;w_0)$ is an open subset in the affine space where $z_i$ are coordinates, so $\{z_{1} = 0\} \cap \{z_{2} = 0\}\cap X_0(\beta\cdot\Delta; w_0)$ is either empty or has codimension $2$ in $X_0(\beta\cdot\Delta; w_0)$ and the required result follows. 
\end{proof}

The toric charts $T_{\tau}\sse X_{0}(\beta\cdot \Delta; w_{0})$ used in Corollary \ref{cor:open crossings} and Theorem \ref{thm: codim 2} are constructed in Proposition \ref{prop:open crossings}, whose proof we now complete.


\subsection{Proof of Proposition \ref{prop:open crossings}} Let us state and prove Lemma \ref{lemma: moving admissible matrices} and Lemma \ref{lemma: moving admissible diagonal}, which will conclude the proof of Proposition \ref{prop:open crossings}. For that, we will need to establish some notation and conventions regarding actions of tori on $\C$-algebras.

Let $R$ be a $\C$-algebra, and assume that a torus $T$ acts on $R$ by algebra automorphisms. We will assume that this action is rational, that is, each element $r \in R$ is contained in a finite-dimensional $T$-stable subspace of $R$. This guarantees that, as a vector space, $R$ is the direct sum of its $T$-weight spaces. Equivalently, put more succinctly, $R$ is graded (as a $\C$-algebra) by the character lattice $\mathfrak{X}(T)$ of $T$.

Now, given $n > 0$, we denote by $M_{n}(R)$ the algebra of $n \times n$-matrices with coefficients in $R$. Note that the action of $T$ on $R$ extends to an action on $M_{n}(R)$. Indeed, if $\bt \in T$ and $U = (u_{jk}) \in M_n(R)$, defining $\bt.U$ via
\[
(\bt.U)_{jk} = \bt.u_{jk}
\]
defines an action of $T$ on $M_{n}(R)$ and it respects the multiplication on $M_{n}(R)$. 

Our preferred torus will be, as in the previous sections, $T = (\C^*)^n/\C^*$, the quotient of $(\C^*)^n$ by its diagonal torus. Its character lattice is $\mathfrak{X}(T) = \{(a_1, \dots, a_n) \in \Z^n \mid \sum a_{i} = 0\}$. For $i\in[1,n]$, we denote by $\be_i$ the vector $(0, \dots, 0, 1, 0, \dots, 0) \in \Z^n$, where the $1$ is on the $i$-th position, so that the differences $\be_i - \be_j$ belong to $\mathfrak{X}(T)$. 

\begin{definition}\label{def:admissible matrix}
Let $T$ be the torus $(\C^{*})^{n-1} = (\C^{*})^{n}/\C^*$, $w \in S_n$ a permutation, and assume that $T$ acts rationally and by algebra automorphisms on a $\C$-algebra $R$. By definition, a matrix $U = (u_{a,k}) \in M_n(R)$ is said to be \emph{$w$-admissible} if $u_{a,k} \in R$ is homogeneous of weight
$$
\we{u_{a,k}} = \be_{w(a)} - \be_{w(k)}
$$
\noindent for every $a, k \leq n$.
\end{definition}

\begin{remark}\label{rmk:admissible} (Characterization of $w$-admissibility)
Consider the torus $\mathbb{T} = (\C^*)^n$. Under the assumptions of Definition \ref{def:admissible matrix}, $\mathbb{T}$ acts on $R$ via the projection $\mathbb{T} \to T$. Thus, as explained in the discussion above, it also acts on $M_{n}(R)$. Independently, since $R$ is a $\C$-algebra, the torus $\mathbb{T}$ embeds into $M_n(R)$ as the torus of diagonal matrices with entries in $\C^*$. As above, for $\bt = (t_1, \dots, t_{n}) \in (\C^{*})^{n}$, let us denote by $D_{\bt}$ the diagonal matrix $\diag(t_1, \dots, t_n)$. Then, $U$ is $w$-admissible if and only if $\bt.U = D_{w(\bt)}UD_{w(\bt)}^{-1}$.

\end{remark}

Before we proceed with Lemma \ref{lemma: moving admissible matrices}, here is an example of a $w$-admissible matrix.

\begin{ex}
Let $w \in S_{n}$ be any permutation and assume $z_{0}$ is an invertible element of weight $\we{z_{0}} = \be_{w(i+1)} - \be_{w(i)}$. Then the matrix $U_{i}(z_{0}) = Id + z_{0}^{-1}E_{i, i+1}$ is $w$-admissible.
\end{ex}

Consider a $w$-admissible matrix $U\in M_n(R)$ and $z$ an element of weight $\we{z} = \be_{w(a)} - \be_{w(b)}+ \be_{w(m)} - \be_{w(k)}$ for some $a, b, m, k \in[1,n]$. Then the element $u_{a,k} + zu_{b, m}$ is homogeneous. The salient property of admissible matrices, which motivates their definition, is that they allow us to construct homogeneous elements for the torus action, as the following result shows.

\begin{lemma}\label{lemma: moving admissible matrices}
Let $w \in S_{n}$ be a permutation, $U^{0}$ be an invertible upper-triangular $w$-admissible matrix, and $\beta = \sigma_{i_{\ell}}\cdots \sigma_{i_{1}}$ a positive braid word. Consider algebraically independent variables $z_{\ell}, \dots, z_{1}$ with weights
$$
\we{z_{k}} = -\be_{w_{k-1}(i_{k} + 1)} + \be_{w_{k-1}(i_{k})}.
$$
\noindent where $w_{d} = ws_{i_{1}}\cdots s_{i_{d}}$, and inductively define (see Lemma \ref{lem:slide triangular}) the upper triangular matrices $U^{1}, \dots U^{\ell}$ and elements $z_{\ell}', \dots, z_{1}' \in R[z_{\ell}, \dots, z_{1}]$ by the equation
$$
B_{i_{d+1}}(z_{d+1})U^{d} = U^{d+1}B_{i_{d+1}}(z'_{d+1}),
$$
Then the following two facts hold:

\begin{itemize}
\item[(a)] The elements $z'_{1}, \dots, z'_{\ell+1}$ are all homogeneous with respect to the torus action and, moreover, $\we{z'_{d}} = \we{z_{d}}$ for every $d = 1, \dots, \ell$.\\

\item[(b)] For every $d = 0, \dots, \ell$, the matrix $U^{d}$ is invertible, upper triangular, $w_{d}$-admissible and has entries in the polynomial ring $R[z_{d-1}, \dots, z_{1}]$.
\end{itemize}
\end{lemma}
\begin{proof}
A computation shows that the matrices $U^{0}, \dots, U^{\ell}$ are invertible and upper triangular. In order to prove the remaining claims, we induct on the length $\ell$, with the base $\ell = 0$ holding by assumption.

For the inductive step, suppose that the statement holds for positive braids of length $\ell$, and consider a positive braid $\beta = \sigma_{i_{\ell + 1}}\sigma_{i_{\ell}}\cdots \sigma_{i_{1}}$ of length $\ell + 1$. Note that the matrices $U^{0}, U^{1}, \dots, U^{\ell}$ and the elements $z'_{1}, \dots, z'_{\ell}$ coincide with those for the braid $\sigma_{i_{\ell}}\cdots \sigma_{i_{1}}$, so we only need to show that the element $z'_{\ell + 1}$ is homogeneous of the same weight as $z_{\ell + 1}$, and that the matrix $U^{\ell + 1}$ is $ws_{i_{1}}\cdots s_{i_{\ell+1}}$-admissible. To ease the notation, we will write $i := i_{\ell + 1}$. 

By the comment preceding the lemma, each of the entries of the matrix $B_{i}(z_{\ell + 1})U^{\ell}$ is homogeneous. The $(i+1, i+1)$-entry of this matrix is  $u^{\ell}_{i+1, i+1}z_{\ell + 1} + u^{\ell}_{i, i+1}$. Dividing by $u^{\ell}_{i,i}$  we obtain that

$$
z'_{\ell + 1} = \frac{u^{\ell}_{i+1, i+1}z_{\ell+1} + u^{\ell}_{i, i+1}}{u^{\ell}_{i,i}}
$$

\noindent is homogeneous. Since the diagonal entries of $U^{\ell}$ have weight $0$ and every entry of $U^{\ell}$ is algebraically independent with $z_{\ell+1}$, we obtain that $z'_{\ell+1}$ is homogeneous of the same weight as $z_{\ell+1}$. Moreover, using again the $w_{\ell}$-admissibility assumption for $U^{\ell}$ we have that every entry of the matrix $U^{\ell}B^{-1}(z'_{\ell+1})$ is homogeneous.

Now $U^{\ell + 1} = B_{i}(z_{\ell + 1})U^{\ell}B^{-1}_{i}(z'_{\ell + 1})$. We check that this matrix is $w_{\ell+1} = w_{\ell}s_{i}$-admissible. Indeed, computing $D_{w_{\ell+1}(\bt)}U^{\ell+1}D_{w_{\ell+1}(\bt)}^{-1}$ we have

$$
\begin{array}{rl}
D_{w_{\ell+1}(\bt)}U^{\ell+1}D_{w_{\ell+1}(\bt)}^{-1} = & (\bt. B_{i}(z_{\ell+1}))D_{w(\bt)}U^{\ell}D_{w(\bt)}^{-1} (\bt.B_{i}^{-1}(z'_{\ell+1})) \\
= & (\bt. B_{i}(z_{\ell+1}))(\bt. U^{\ell})(\bt.B_{i}^{-1}(z'_{\ell+1})) \\
 = & \bt.(U^{\ell+1})
\end{array}
$$

\noindent where the first equality follows from \eqref{eq: slide diagonal}. This concludes the proof thanks to Remark \ref{rmk:admissible}. 
\end{proof}

The proof of the following result is similar to that of Lemma \ref{lemma: moving admissible matrices} and left to the reader.

\begin{lemma}\label{lemma: moving admissible diagonal} 
Let $U\in M_n(R)$ be a $w$-admissible  upper-triangular matrix and $z_0\in R$ homogeneous and invertible with weight $\we{z_0} = -\be_{w(i+1)} + \be_{w(i)}$. Then the matrix $U' = D_{i}(z_0)UD_{i}^{-1}(z_0)$ is $ws_{i}$-admissible.
\end{lemma}

This concludes the necessary ingredients for Proposition \ref{prop:open crossings}, and thus completes our argument for Corollary \ref{cor:open crossings} and Theorem \ref{thm: codim 2}. The following three subsections relate the results and constructions of Subsections \ref{sec: braid varieties}, \ref{sect:torus action} and \ref{sec: opening} to character varieties, through the work of P. Boalch, A. Mellit \cite{Boalch01,Boalch02,Boalch20,Mellit} and others, augmentation varieties, as featured in \cite{CN,Kalman05,Kalman}, and open Bott-Samelson varieties, according to \cite{SW,STZ}.


\subsection{Mellit's chart and sequences of crossings}
\label{sec:mellit chart}
In this subsection we recast a construction from \cite{Mellit} in the light of braid varieties, in particular defining a certain toric chart in $X_0(\beta\Delta; w_0)$, which we refer to as the \emph{Mellit chart}. The main result of the subsection is that the Mellit chart can be obtained by our opening-crossing procedure from Subsection \ref{sec: opening} above. In order to connect to \cite{Mellit}, we need the following preliminary discussion.

Let $w\in S_n$ be a permutation and $C_{w}=\SB w\SB\sse \GL(n,\C)$ the Bruhat cell corresponding to $w$, where $\SB\sse\GL(n,\C)$ is the Borel subgroup of upper-triangular matrices. Recall that the product of any two matrices in $C_u$ and $C_v$ belongs to $C_{uv}$ if $\ell(uv)=\ell(u)+\ell(v)$. Consequently, for any reduced expression $u=s_{i_1}\cdots s_{i_\ell}$, the associated braid matrix $B_u(z_1,\ldots,z_{\ell})$ belongs to the Bruhat cell $C_{u}$. Recall that we interchangeably use the notation $s_i$ and $\sigma_i$ for the Artin generators of the braid group, which is particularly well-suited when comparing to the notation used in \cite{Mellit}.

\begin{prop}
\label{prop:exchange}
Let $u=s_{i_1}\cdots s_{i_\ell}$ be a reduced expression and suppose that $\ell(us_i)=\ell(u)-1$. Then there exists $k\in\N$ such that:
\begin{itemize}
\item[(a)] The matrix $B_u(z_1,\ldots,z_{\ell})B_i(z)$ belongs to the Bruhat cell $C_{u}$ if and only if $z_{k}\neq 0$,\\

\item[(b)] In case $z_k\neq 0$, we can uniquely write $B_u(z_1,\ldots,z_{\ell})B_i(z)=UB_{u}(z'_1,\ldots,z'_{\ell})$
for a certain upper-triangular matrix $U$.
\end{itemize}
\end{prop}

\begin{proof}
Since $\ell(us_i)=\ell(u)-1$, there exists $k\in\N$ such that $us_i=s_{i_1}\cdots\widehat{s_{i_k}}\cdots s_{i_\ell}$ (this is known as {\em exchange property} for the Coxeter group $S_n$). That is, we can write $u=u_1s_{i_k}u_2$ such that $s_{i_k}u_2=u_2s_{i}$, and thus $us_i=u_1s_{i_k}u_2s_i=u_1s_{i_{k}}s_{i_{k}}u_2=u_1u_2$. This implies the following equation for the braid matrices:
$$
B_u(z_1,\ldots,z_{\ell})B_{i}(z)=B_{u_1}(z_1,\ldots,z_{k-1})B_{i_k}(z_k)B_{i_k}(z')B_{u_2}(z'_{k+1},\ldots,z'_{\ell}),
$$
where $z', z'_{k+1},\ldots,z'_{\ell}$ are some functions of $z,z_{k+1},\ldots,z_{\ell}$. If $z_{k}\neq 0$, then we can further write $B_{i_k}(z_k)B_{i_k}(z')=UB_{i_k}(z'')$, so
$$
B_u(z_1,\ldots,z_{\ell})B_{i}(z)=\widetilde{U}B_{u_1}(z'_1,\ldots,z'_{k-1})B_{i_k}(z'')B_{u_2}(z'_{k+1},\ldots,z'_{\ell})=
\widetilde{U}B_u(z'_1,\ldots,z'_{k-1},z'',z'_{k+1},\ldots,z'_{\ell})
$$
and the result is in the Bruhat cell $C_{u}$. If instead $z_{k}=0$, then $B_{i_k}(z_k)B_{i_k}(z')$ is upper-triangular, and $B_u(z_1,\ldots,z_{\ell})B_{i}(z)$ is in the Bruhat cell
$C_{u_1u_2}$, which is disjoint from $C_{u}$.
\end{proof}

\begin{ex}
Consider $\beta_1=s_1s_2s_1s_1$ and $\beta_2=s_1s_2s_1s_2$. Then the braid matrix $B_{\beta_1}(z_1,z_2,z_3,z_4)$ is in the Bruhat cell $C_{s_1s_2s_1}$ if and only if $z_3\neq 0$. In contrast, the braid matrix $B_{\beta_2}(z_1,z_2,z_3,z_4)$ 
is in the Bruhat cell $C_{s_1s_2s_1}$ if and only if $z_1\neq 0$. In both cases, we have a reduced expression $u=s_1s_2s_1$ and a simple reflection $s_1$, resp. $s_2$, satisfying the assumption of Proposition \ref{prop:exchange}.
\end{ex} 

\begin{remark}
\label{rmk: find k}
The index $k\in\N$ from Proposition \ref{prop:exchange} is unique and can be described geometrically, as follows. Draw a braid diagram for $u$, labeling the strands $1$ to $n$ on the right. Since $\ell(us_i)=\ell(u)-1$, the $i$-th and $(i+1)$-st strands intersect somewhere in the diagram for $u$. Given that $u$ is reduced, they intersect exactly once. The index $k$ corresponds to this intersection point.\hfill$\Box$
\end{remark}

Let us now compare our construction to \cite{Mellit}, with $\beta\Delta=s_{i_1}\cdots s_{i_{\ell+\binom{n}{2}}}$ a positive braid. In \cite[Section 5.4]{Mellit}, a sequence of permutations $p_0=1,p_1,\ldots,p_{\ell+\binom{n}{2}}$ is defined according to the following rules:
\begin{itemize}
\item[(a)] If $\ell(p_{k-1}s_{i_k})=\ell(p_{k-1})+1$ then $p_{k}=p_{k-1}s_{i_k}$,
\item[(b)] If $\ell(p_{k-1}s_{i_k})=\ell(p_{k-1})-1$ then $p_{k}=p_{k-1}$. 
\end{itemize}
In the terminology of ibid., this sequence is a {\it walk which never goes down}.

\begin{remark}
The permutations $1=p_0, p_1, \ldots, p_{\ell + \binom{n}{2}}$ may be described in terms of the \emph{Demazure product}, cf.~ Section \ref{section:Dem_prod}. Indeed, using the notation of that section it follows that $p_{j} = \delta(\sigma_{i_1}\cdots \sigma_{i_j})$.
\end{remark}

Let us now describe the toric chart used in \cite{Mellit}.

\begin{definition}[Mellit Chart]\label{def:Mellitchart}
Let $\beta$ be a positive $n$-braid word, the Mellit chart $\mathfrak{M}\sse X_0(\beta\Delta,w_0)$ is defined as the locus of $z_1,\ldots,z_s$ such that
\begin{equation}
\label{eq: mellit chart}
B_{i_1}(z_1)\cdots B_{i_{s}}(z_s)\in C_{p_s}\ \mathrm{for\ all}\ s\le \ell+\binom{n}{2}.
\end{equation}
Note that $\mathfrak{M}\sse X_0(\beta\Delta,w_0)$ is codimension-0 and Zariski open in $X_0(\beta\Delta,w_0)$.
\end{definition}

\begin{remark}
    By definition, the Mellit chart $\mathfrak{M}$ is closely related to the maximal piece in the Deodhar decomposition \cite{Deodhar}.
\end{remark}

At this stage, our Corollary \ref{cor:open crossings} provides many toric charts $T_{\tau}$ for $X_0(\beta\Delta,w_0)$, (surjectively) indexed by orderings $\tau\in S_{\ell(\beta)}$ of the crossings. The toric chart $\mathfrak{M}$ introduced in Definition \ref{def:Mellitchart} is also a subset of $X_0(\beta\Delta,w_0)$, and it is thus natural to ask whether $\mathfrak{M}$ is of the form $T_{\tau}$ and, if so, for which ordering $\tau$ this is the case. This is answered in our next result (and its proof).

\begin{thm}\label{thm:mellit chart} Let $\beta$ be a positive braid word. Then there exists an ordering $\tau(\beta)\in S_{\ell(\beta)}$ of the crossings such that $T_{\tau(\beta)}\sse X_0(\beta\Delta,w_0)$ coincides with the Mellit chart $\mathfrak{M}\sse X_0(\beta\Delta,w_0)$.
\end{thm}

\begin{proof}
The ordering $\tau(\beta)$ in which we open the crossings is as follows. First, we find the smallest  $j$ such that $p_{j-1}=p_j$. This means that 
$p_{j-1}=s_{i_1}\cdots s_{i_{j-1}}$ is a reduced word and  $\ell(p_{j-1}s_{i_j})=\ell(p_{j-1})-1$. The condition \eqref{eq: mellit chart}
holds automatically for $s<j$, and  for $s=j$ we can apply Proposition \ref{prop:exchange}: there exists some $k<j$ such that 
$B_{i_1}(z_1)\cdots B_{i_{j}}(z_j)\in C_{p_j}$ if and only if $z_k\neq 0$. 

It follows from Remark \ref{rmk: find k} that the crossing with index $k$ is in the braid $\beta$, and never in $\Delta$.
We can open this crossing and obtain a new braid $\beta' \Delta$.  By Proposition \ref{prop:exchange}, a point in $X_0(\beta\Delta,w_0)$ is in the Mellit chart if and only if the corresponding point in $X_0(\beta'\Delta,w_0)$ is in the respective chart. This process can be continued iteratively. Eventually, all crossings in $\beta$ will be exhausted, and we reach a reduced expression $\Delta$, which satisfies the defining inclusion \eqref{eq: mellit chart} automatically. 
\end{proof}

\begin{ex}
Consider the positive 3-braid $\beta=\sigma_1\sigma_2\sigma_1$, and thus $\beta \Delta=\sigma_1\sigma_2\sigma_1\sigma_1\sigma_2\sigma_1$. By opening the third crossing from the left $\sigma_1$, we reach the braid word $\sigma_1\sigma_2\sigma_1\sigma_2\sigma_1$. Then we open the first (leftmost) $\sigma_1$ crossing and obtain $\sigma_2 \sigma_1 \sigma_2 \sigma_1$. Finally, opening again the first (leftmost) crossing $\sigma_2$ in the resulting braid (which corresponds to the second crossing in the original braid) we reach the positive braid word $\Delta=\sigma_1\sigma_2\sigma_1$. This sequence of crossings $\tau(\beta)$ yields a toric chart $T_{\tau(\beta)}\sse X_0(\beta\Delta,w_0)$ which coincides with the Mellit chart $\mathfrak{M}\sse X_0(\beta\Delta,w_0)$.
\end{ex}

This concludes our discussion on the Mellit chart and the relation between our Corollary \ref{cor:open crossings} and \cite[Section 5]{Mellit}. Let us shift our focus towards augmentation varieties, a class of algebraic varieties which are central to the study of Legendrian links in contact 3-manifolds.


\subsection{Augmentation varieties as quotient braid varieties}\label{sect:augmentations} In this subsection, we establish a connection between braid varieties and augmentation varieties. The latter are a class of varieties that feature saliently in the study of Floer-theoretic invariants associated to Legendrian links $\La\sse(\R^3,\xi_\st)$. The reader is referred to \cite{Geiges} for the basics of 3-dimensional contact topology, \cite{EtnyreNg} for a survey on Floer-theoretic invariants of Legendrian knots, and \cite{CN,Che,Kalman05,Kalman} for further details.\\

\begin{center}
	\begin{figure}[h!]
		\centering
		\includegraphics[scale=0.6]{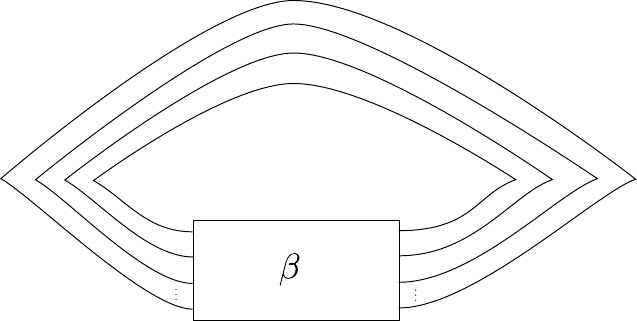}
		\caption{The front projection known as the rainbow closure of $\beta$.}
		\label{fig:rainbow_closure}
	\end{figure}
\end{center}

Let $\beta \in \mathrm{Br}^{+}_{n}$ be a positive braid word and $\Lambda(\beta)\sse(\R^3,\xi_\st)$ the Legendrian link associated to the rainbow framed closure of the braid $\beta$. This is the front diagram for $\Lambda(\beta)$ depicted in Figure \ref{fig:rainbow_closure}, cf.~ \cite{CG,CN}. Let us also choose a collection of marked points $\ft \subseteq \beta$ on the Legendrian link $\La_\beta$, see~e.g. \cite{NgRutherford,NRSSZ}. In our case, the two choices for marked points that we use are:
\begin{itemize}
    \item[(1)] A choice of {\it one} marked point per strand of the braid $\beta$, this collection of marked points will be denoted by $\ft_{s}$.\\

    \item[(2)] A choice of {\it one} marked point per component of the Legendrian link $\Lambda(\beta)\sse(\R^3,\xi_\st)$, this collection will be denoted by $\ft_{c}$.
\end{itemize}

By convention, we place all marked points to the right of all crossings in $\beta$ and before the right cusps. Though not essential, this convention will be useful in simplifying some statements. Note also that $\ft_{c}$ technically depends on a choice of strand per component of $\Lambda(\beta)$, but for the sake of readability we prefer to not include this into our notation. Figure \ref{fig:marked points} depicts two instances of such placing of marked points. 

\begin{center}
	\begin{figure}[h!]
		\centering
		\includegraphics[scale=0.35]{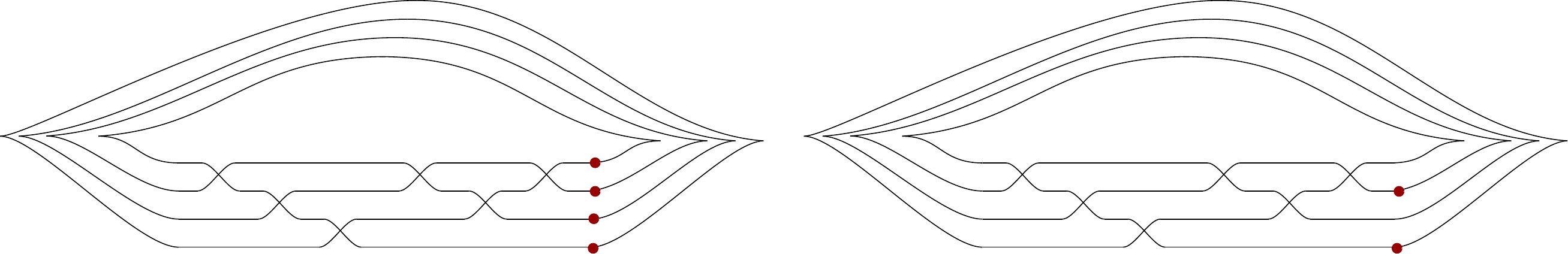}
		\caption{The front ($xz$) projection of the rainbow closure of the braid word $\beta =\sigma_1\sigma_2\sigma_3\sigma_1\sigma_2\sigma_1$, with one marked point per strand (left) and one marked point per component (right). This is a braid word for the half-twist $\Delta_{4}$. Note that the Legendrian condition implies that all crossings are overcrossings. Note also that the marked points are located to the right of all crossings of $\beta$.}
		\label{fig:marked points}
	\end{figure}
\end{center}

Let $\SA(\beta, \ft)$ be the commutative Legendrian Contact DGA of the Legendrian link $\La(\beta)\sse(\R^3,\xi_\st)$ endowed with a set of marked points $\mathfrak{t}\sse\La(\beta)$. The stable tame isomorphism type, and thus the quasi-isomorphism type, of this differential graded algebra (DGA) is an invariant of the Legendrian link $\La(\beta)\sse(\R^3,\xi_\st)$ with marked points $\ft$ up to Legendrian isotopy. It was defined by Y. Chekanov \cite{Che} over $\Z_2$-coefficients and latter lifted to $\Z$-coefficients and marked points \cite{NgRutherford,NRSSZ}, see \cite{EtnyreNg} for a survey. The differential of $\SA(\beta, \ft)$ is given by a count of (pseudo)holomorphic strips whose asymptotics are governed by the Legendrian link $\La\sse(\R^3,\xi_\st)$. In the case of a rainbow closure $\La=\La(\beta)$, the differential is always given by polynomials in the generators, an explicit formula is given in \cite[Section 5]{CN}. In this manuscript, the \emph{augmentation variety} $\Aug(\beta, \ft)$ associated to $(\beta, \ft)$ is defined to be $\Aug(\beta, \ft):=\mbox{Spec }H^0(\SA(\beta, \ft))$, the affine variety associated to the $0$th homology of this DGA. This is an affine algebraic variety defined over $\Z$. The fact that $\SA(\beta, \ft)$ is non-negatively graded implies that the set of $R$-points of $\Aug(\beta, \ft)$ can be identified with $\mbox{Hom}_{\mbox{dg}}(\SA(\beta, \ft),R)$, the set of dg-algebra morphisms from $\SA(\beta, \ft)$ to $R$, where $R$ is taken to be a dg-algebra concentrated in degree 0 and with trivial differential.\\

In the case of Legendrian links $\La\sse(\R^3,\xi_\st)$ associated to positive braids, $\La\simeq\La(\beta)$, augmentation varieties $\Aug(\beta, \ft)$ are closely related to braid varieties. This will follow from the work of T. K\'alm\'an \cite{Kalman}, cf.~also \cite[Section 5]{CN}, as we will now explain.

\begin{thm}\label{thm:aug vs braid}
Let $\beta$ be a positive braid word, $[\beta] \in \mathrm{Br}^{+}_{n}$. The following two statements hold:
\begin{itemize}
\item[(i)] There exists an algebraic isomorphism $\Aug(\beta, \ft_{s}) \cong X_{0}(\beta\cdot\Delta; w_{0})$.\\

\item[(ii)] Let $T_c \subseteq (\C^{*})^{n}$ be the algebraic torus determined by $t_{w^{-1}(i)} = 1$ if the $i$th strand of the braid $\beta$ has a marked point in $\mathfrak{t}_c$ (compare with Remark \ref{rmk: free subtori}). Then there exists an algebraic isomorphism
$$
\Aug(\beta, \ft_{c}) \cong X_{0}(\beta\cdot\Delta;w_{0})/T_c.
$$
\end{itemize}
\end{thm}
\begin{proof}
Let us use the following characterization by T. K\'alm\'an \cite{Kalman05,Kalman} (see also \cite{CN}): if $\beta$ is a positive braid word and $i_{1}, \dots, i_{s}$ are strands that carry a marked point (to the right of every crossing) then the augmentation variety is the affine subvariety of $\C^{\ell(\beta) + \binom{n}{2}}\times (\C^{*})^{s}$ given by the equation

\begin{equation}\label{eq:kalman}
B_{\beta}(z)\left(\begin{matrix} 1 & 0 & \cdots & 0 \\ c_{21} & 1 & \cdots & 0 \\ \vdots & \vdots & \ddots & \vdots \\ c_{n1} & c_{n2} & \cdots & 1\end{matrix}\right)\diag(t_{1}, \dots, t_{n}) \; \; \text{is upper triangular with a prescribed diagonal}
\end{equation}

\noindent where the notation follows the convention that in $\diag(t_{1}, \dots, t_{n})$ we have $t_{i} = 1$ if $i \not= i_{1}, \dots, i_{s}$. For the choice of marked points $\mathfrak{t}_s$, this reduces to $B_{\beta}(z)B_{\Delta}(u)w_{0}$ being upper triangular, which is precisely the definition of $X_{0}(\beta\cdot\Delta; w_{0})$.  This establishes the statement in $(i)$. For the choice of marked points $\mathfrak{t}_c$, as in $(ii)$, Equation \eqref{eq:kalman} reduces to $B_{\beta}(z)B_{\Delta}(u)w_{0}$ being upper triangular with a prescribed diagonal outside of the strands carrying marked points. Since the action of $T_c$ on $X_{0}(\beta\cdot \Delta; w_{0})$ is free (see Remark \ref{rmk: free subtori}) the quotient map $X_{0}(\beta \cdot \Delta; w_{0}) \to X_{0}(\beta\cdot \Delta; w_{0})/T_c$ is a principal $T_c$-bundle. In consequence, $X_{0}(\beta\cdot \Delta; w_{0})/T_c$ is equivalent to the closed subvariety of $X_{0}(\beta\cdot \Delta; w_{0})$ given by prescribing the diagonal elements in $B_{\beta\cdot\Delta}w_{0}$ at entries corresponding to strands not carrying marked points. 
\end{proof}

In contact geometry, opening a crossing from $\beta=\beta_1\sigma_i\beta_2$ to $\beta'=\beta_1\beta_2$ can be realized by an embedded exact Lagrangian cobordism $L_i\sse(\R^3\times\R_t,d(e^t\alpha))$ in the symplectization of $(\R^3,\xi_\st)$, with $\dd L_i=\dd_-L_i\cup\dd_+L_i$ and $\dd_-L_i=\La(\beta')$ and $\dd_+L_i=\La(\beta)$ \cite{ArnoldSing,BST15} (this is correct in the case that the positive braid has a half-twist remaining \cite{CN,EHK}, which will always be the case in our context). It follows from the Floer-theoretic functoriality proven in \cite{EHK,YuPan} that such a Lagrangian cobordism induces an algebraic regular map $\Phi_{L_i}:\Aug(\beta',\mathfrak{t})\lr\Aug(\beta,\mathfrak{t})$ between augmentation varieties. It follows from \cite{CN,EHK} that the ($\Z$-lifted) Floer-theoretical map $\Phi_{L_i}$ agrees with the (quotient of the) map $\Omega_{\sigma_{i}}$ we constructed in Subsection \ref{sec: opening}.

The toric charts we constructed in Corollary \ref{cor:open crossings}, using Proposition \ref{prop:open crossings}, can now be used to give an open cover of the augmentation varieties in Theorem \ref{thm:aug vs braid}, up to codimension $2$, as follows.

\begin{cor}
Let $\beta$ be a positive braid word, $[\beta]\in \Br_n^+$, with $c(\beta)=k$. Geometrically, the Legendrian link has $\La(\beta)$ has $k$ connected components. For each ordering $\tau\in S_{\ell(\beta)}$ of the crossings of $\beta$ there exist codimension-0 toric charts $T^c_{\tau}\sse \Aug(\beta, \ft_{c})$ and $T^s_{\tau}\sse \Aug(\beta, \ft_{s})$, with $T^c_{\tau}\cong(\C^{*})^{\ell(\beta) - n + k}$ and $T^s_{\tau}\cong(\C^{*})^{\ell(\beta)}$ such that the complements
$$\Aug(\beta, \ft_{c})\setminus\left(\bigcup_{\tau\in S_{\ell(\beta)}}T^c_{\tau}\right)\sse \Aug(\beta, \ft_{c}),\qquad \Aug(\beta, \ft_{s})\setminus\left(\bigcup_{\tau\in S_{\ell(\beta)}}T^s_{\tau}\right)\sse \Aug(\beta, \ft_{s})$$
both have codimension at least 2.
\end{cor}
\begin{proof}
In view of Theorem \ref{thm:aug vs braid}, only the statement for $\Aug(\beta, \ft_{c})$ remains unproven. It follows from the $T$-stability of the toric charts on the braid variety $X_{0}(\beta\cdot \Delta; w_{0})$, cf. Corollary \ref{cor:open crossings}.
\end{proof}


\subsection{Open Bott-Samelson varieties}\label{sec: obs}
This section is not required for the rest of the manuscript: it is provided here for contextual completeness with respect to the articles \cite{CG,GSW,SW,STWZ}. The purpose of this section is to relate the braid variety $X_{0}(\beta)$ to the (diagonal) open Bott-Samelson variety $\OBS(\beta)$ associated to the braid $\beta$. This is achieved in Theorem \ref{thm:OBS} below, after a brief reminder on Bott-Samelson varieties.

Consider $G := \operatorname{GL}(n,\C)$, $\cB \subseteq G$ the Borel subgroup of upper-triangular matrices and the flag variety $\Fl := G/\cB$. The projective variety $\Fl$ is the moduli space of complete flags of subspaces in $\C^n$: an element $\cF \in \Fl$ is a flag $\cF = (F_1 \subseteq \cdots \subseteq F_n)$ where $\dim F_{i} = i$. 
Given a flag $\cF\in \Fl$, we can choose a basis $(v_{1}, \dots, v_{n})$ of $\C^n$ such that $F_{j} = \langle v_{1}, \dots, v_{j}\rangle $ for $ j = 1, \dots, n$; we denote by $V_{\cF} \in G$ the matrix whose columns are the vectors $v_{i}$ expressed in the standard basis. Conversely, given a matrix $V \in G$, we can consider a flag $\cF^{V} = (F_{1} \subseteq \cdots \subseteq F_{n})$ where $F_{j}$ is the span of the first $j$ columns of the matrix $V$. In this correspondence, two flags are equal $\cF^{V} = \cF^{V'}$ if and only if their matrices $V,V'$ are related by an upper triangular matrix, i.e. $V = V'U$ for some $U \in \cB$. 

By definition, two flags $\cF,\cF'\in\Fl$ are in relative position $s_{i}$, $i\in[1,n-1]$, if $F_{j} = F'_{j}$ for $j \neq i$ and $F_{i} \neq F'_{i}$. In terms of their matrices, the flags $\cF^{V},\cF^{V'}$ are in relative position $s_{i}$ if and only if there exist upper-triangular matrices $A_{1}$ and $A_{2}$ such that $V' = VA_{1}s_{i}A_{2}$, where $s_{i}$ is understood as a permutation matrix. 

\begin{remark}\label{rmk: rel position}
Since the permutation matrix $s_{i} = B_{i}(0)$ is a braid matrix with the variable set to zero, it follows from Lemma \ref{lem:slide triangular} that the flags $\cF^{V}$ and $\cF^{V'}$ are in relative position $s_{i}$ if and only if there exist an upper-triangular matrix $U$ and $z \in \C$ such that $V' = VUB_{i}(z)$.
\end{remark}

Building on the articles \cite{BM,Deligne}, and the subsequent developments \cite{CG,SW,STWZ,STZ}, we introduce the two algebraic varieties $\OBS(\beta)$ and $\OBS'(\beta)$ as follows.

\begin{definition}\label{def:OBS}
Let $\beta = \sigma_{i_{1}}\cdots \sigma_{i_{\ell}}$ be a positive braid word.

\begin{itemize}
	\item[(i)]  The open Bott-Samelson variety $\OBS(\beta) \subseteq \Fl^{\ell+1}$ associated to $\beta$ is the moduli space of $(\ell+1)$-tuples of flags $(\cF_{0}, \dots, \cF_{\ell})$ such that consecutive flags $\cF_{k-1},\cF_{k}$ are in relative position $s_{i_{k}}$, for each $k\in[1,\ell]$.\\
	
	\item[(ii)] The {\it diagonal} open Bott-Samelson variety $\OBS'(\beta) \subseteq \OBS(\beta)$ is the closed subvariety defined by the additional condition that $\cF_{0} = \cF_{\ell}$.
\end{itemize}
\end{definition}

The diagonal open Bott-Samelson variety $\OBS'(\beta)$ will be related to the braid variety, as we now explain. First, let us construct a map $\pi: G \times X_{0}(\beta) \to \OBS'(\beta)$ as follows. Consider a point $(z_{1}, \dots, z_{\ell}) \in X_{0}(\beta)$ and a matrix $V\in G$, and define $V_{k} := VB_{i_{1}}(z_{1})\cdots B_{i_{k}}(z_{k}) \in G$. The map $\pi$ is then defined by:
$$
\pi: G \times X_{0}(\beta) \lr \OBS'(\beta),\quad \pi(V, z_{1}, \dots, z_{\ell}):= (\cF^{V}, \cF^{V_{1}}, \dots, \cF^{V_{\ell}}).
$$

\noindent It follows from Remark \ref{rmk: rel position} that $\pi(V, z_{1}, \dots, z_{\ell}) \in \OBS(\beta)$ and since $V_{\ell} = VB_{\beta}(z_{1}, \dots, z_{\ell})$, and $(z_{1}, \dots, z_{\ell}) \in X_{0}(\beta)$, we actually have that $\pi(V, z_{1}, \dots, z_{\ell}) \in \OBS'(\beta)$. Thus, the image of $\pi$ belongs to $\OBS'(\beta)\sse\OBS(\beta)$, as written above. This map is, in general, not an isomorphism. Nevertheless, we will now construct a right $\cB$-action on the product $G \times X_{0}(\beta)$, and $\pi$ will descend to an isomorphism on the quotient.

Indeed, consider an upper-triangular matrix $U = U^{0} \in \cB$ and define $z_{1}', \dots, z_{\ell}'$ and $U^{1}, \dots, U^{\ell} \in \cB$ inductively via the equation 
\begin{equation}\label{eqn: Us}
B_{i_{\ell - k}}(z_{\ell - k})U^{k} = U^{k+1}B_{i_{\ell - k}}(z'_{\ell - k}).
\end{equation}
It follows from the equation $U^{\ell}B_{\beta}(z_{1}, \dots, z_{\ell}) = B_{\beta}(z_{1}', \dots, z_{\ell}')U^{0}$ that $(z_{1}, \dots, z_{\ell}) \in X_{0}(\beta)$ if and only if $(z_{1}', \dots, z_{\ell}') \in X_{0}(\beta)$. For each $(V, z_{1}, \dots, z_{\ell}) \in G \times X_{0}(\beta)$ and upper-triangular matrix $U = U^{0} \in \cB$, we define its (right) action by:
$$
(V, z_{1}, \dots, z_{\ell})\cdot U := (VU^{\ell}, z'_{1}, \dots, z'_{\ell}).
$$ 
The usefulness of this right action is manifest in the main result of this  
subsection which reads as follows.

\begin{thm}
\label{thm:OBS} Let $\beta$ be a positive braid word, $G=\GL(n,\C)$ and $\cB\sse G$ the Borel subgroup of upper-triangular matrices. Then
\begin{itemize}
\item[(i)] The right $\cB$-action on $G \times X_{0}(\beta)$ defined above is free.\\

\item[(ii)] The map $\pi: G \times X_{0}(\beta) \lr \OBS'(\beta)$  induces an isomorphism $(G \times X_{0}(\beta))/\cB \cong \OBS'(\beta)$.
\end{itemize}
\end{thm}

\begin{proof}
Let us first prove the freeness of the right $\cB$-action. Indeed, suppose that there exists a fixed point, i.e. there exist $U \in \cB$ and $(V, z_{1}, \dots, z_{\ell}) \in G \times X_{0}(\beta)$ such that
$$(V, z_{1}, \dots, z_{\ell})\cdot U = (V, z_{1}, \dots, z_{\ell}).$$
Since $z'_{j} = z_{j}$ for every $j\in[1,\ell]$, it follows from Equation \eqref{eqn: Us} that the matrices $U, U^{1}, \dots, U^{\ell}$ are pairwise conjugate. In particular, the initial upper-triangular matrix $U$ is conjugate to $U^{\ell}$. Nevertheless, the condition $VU^{\ell} = V$ implies that $U^{\ell} =\mbox{Id}$, and it follows that $U = \mbox{Id}$. The action is thus free.

Second, let us show that the map $\pi$ is surjective, onto the diagonal open Bott-Samelson variety $\OBS'(\beta)$. For that, consider a point $(\cF_{0}, \dots, \cF_{\ell}) \in \OBS'(\beta)$ and let $V \in G$ be any matrix such that $\cF^{V} = \cF_{0}$. Thanks to Remark \ref{rmk: rel position}, we have that there exist upper-triangular matrices $U^{1}, \dots, U^{\ell}$ and $z_{1}, \dots, z_{\ell} \in \C$  such that $$\cF_{k} = \cF^{VU^{1}B_{i_{1}}(z_{1})\cdots U^{k}B_{i_{k}}(z_{k})} \; \text{for all} \; k = 1, \dots, \ell.$$ 

Now, use Lemma \ref{lem:slide triangular} to slide all the upper triangular matrices $U^{2}, \dots, U^{\ell}$ to the left; this yields upper-triangular matrices $\widehat{U}^{1} = U^{1}, \widehat{U}^{2}, \dots, \widehat{U}^{\ell}$ and $\hat{z}_{1}, \dots, \hat{z}_{\ell}$ with the property that, for every $k$:
$$
V\widehat{U}^{1}\cdots \widehat{U}^{\ell} B_{i_{1}}(\hat{z}_{1})\cdots B_{i_{k}}(\hat{z}_{k}) = VU^{1}B_{i_{1}}(z_{1})\cdots U^{k}B_{i_{k}}(z_{k})\widehat{U},
$$
and 
$$
\widehat{U} B_{i_{k+1}}(\hat{z}_{k+1})\cdots B_{i_{\ell}}(\hat{z}_{\ell})=U^{k+1}B_{i_{k+1}}(z_{k+1})\cdots U^{\ell}B_{i_{\ell}}(z_{\ell}),
$$
\noindent where $\widehat{U}$ is an upper-triangular matrix depending on $k$. 
This implies that $\pi(V\widehat{U}^{1}\cdots \widehat{U}^{\ell}, \hat{z}_{1}, \dots, \hat{z}_{\ell}) = (\cF_{0}, \dots, \cF_{\ell})$. It remains to show that $(\hat{z}_{1}, \dots, \hat{z}_{\ell}) \in X_{0}(\beta)$, that is, the matrix $B_{\beta}(\hat{z}_{1}, \dots, \hat{z}_{\ell})$ is upper-triangular. Since $\cF_{0} = \cF_{\ell}$, the matrices $V$ and $V\widehat{U}^{1}\cdots \widehat{U}^{\ell}B_{\ell}(\hat{z}_{1}, \dots, \hat{z}_{\ell})$ differ by an upper-triangular matrix. Since $\widehat{U}^{1},\ldots,\widehat{U}^{\ell}$ are upper-triangular, the result follows. Thus, $\pi$ is surjective.

Third, let us prove that the map $\pi$ is $\cB$-invariant. We need to check that for every $k$ the matrices $VB_{i_{1}}(z_{1})\cdots B_{i_{k}}(z_{k})$ and $VU^{\ell}B_{i_{1}}(z'_{1})\cdots B_{i_{k}}(z'_{k})$ differ by an upper-triangular matrix. It follows from Equation \eqref{eqn: Us} that this matrix is precisely $U^{\ell - k}$, which is upper-triangular. This proves $\cB$-invariance. 

Finally, we must show that if $\pi(V, z_{1}, \dots, z_{\ell}) = \pi(V', z'_{1}, \dots, z'_{\ell})$ then there exists an upper triangular matrix $U$ such that $(V', z'_{1}, \dots, z'_{\ell}) = (V, z_{1}, \dots, z_{\ell})\cdot U$. For that, note that $\cF^{V} = \cF^{V'}$ implies that there exists an upper-triangular matrix, say $U^{\ell}$, such that $V' = VU^{\ell}$. Since $\cF^{VB_{i_{1}}(z_{1})} = \cF^{V'B'_{i_{1}}(z'_{1})}$, there also exists an upper-triangular matrix, say $U^{\ell - 1}$, such that $V'B'_{i_{1}}(z'_{1}) = VB_{i_{1}}(z_{1})U^{\ell - 1}$. In consequence, we obtain the equality $VU^{\ell}B_{i_{1}}(z'_{1}) = VB_{i_{1}}(z_{1})U^{\ell-1}$, and thus $U^{\ell}B_{i_{1}}(z'_{1}) = B_{i_{1}}(z_{1})U^{\ell - 1}$. Note that this is precisely Equation \eqref{eqn: Us}. We iterate this procedure until we find $U^{0}$, which is the required upper-triangular matrix. This concludes the proof of the statement.
\end{proof}

\begin{remark}
By \cite[Section 6.4]{CLSBW} or \cite[Theorem 3.9]{STWZ}, for certain $\beta$, the variety $\OBS'(\beta)$ is closely related to a suitable positroid variety \cite{KLS}, see also \cite{CGGS2}. See \cite{GL}, where the topology of positroid varieties is studied in detail.
\end{remark}

\begin{remark}
As a side note, the homotopy types of the varieties $X_0(\beta)$ and $\OBS'(\beta)$ appear to be related to the spectra constructed in \cite{Kitchloo}. This remains to be explored.
\end{remark}

This concludes our discussion relating braid varieties to open Bott-Samelson varieties. Let us now move to the construction of a holomorphic symplectic structure on braid varieties $X_0(\beta)$.


\section{Holomorphic Symplectic Structure}\label{sec:symplecticform}

This section constructs holomorphic symplectic structures on the quotients $X_0(\beta\Delta;w_0)/T$ of braid varieties, establishing the remainder of Theorem \ref{thm:main1}.(iii). In particular, Theorem \ref{thm:aug vs braid} will imply that the augmentation variety associated to a Legendrian link $\La(\beta)$, $\beta$ a positive braid word and a certain choice of marked points, is holomorphic symplectic. In addition, the toric charts we built in Corollary \ref{cor:open crossings} will actually be exponential pre-Darboux charts (cf.~Example \ref{ex:2form_torus} below) for a closed holomorphic 2-form on $X_0(\beta\Delta;w_0)$, and they will project via the torus quotient to exponential Darboux charts for this holomorphic symplectic structure.
The construction we present draws from the literature on character varieties, where the holomorphic symplectic structures on character varieties have a central role, starting with the Atiyah-Bott-Goldman structures \cite{AB,Goldman84} and continuing with, e.g., the work of P. Boalch and L. Jeffrey \cite{Boalch00,Boalch01,Boalch02,Jeffrey,Mellit}.

\begin{ex}\label{ex:2form_torus}
Let $T=(\C^*)^{n}$ be an algebraic torus with coordinates $(x_1,\ldots,x_{n})\in(\C^*)^{n}$. In general, a holomorphic 2-form $\omega\in\Omega_T^2$ on $T$ is an expression of the form
\[
\omega=\sum_{i<j}f_{ij}(x_1,\ldots,x_{n})dx_idx_j,
\]
where $f_{ij}(x_1,\ldots,x_{n}):T\lr\C$ are holomorphic functions. For instance,
\[
\omega=(x_1x_2)^{-1}dx_1dx_2
\]
is a holomorphic 2-form on $(\C^*)^2$. In general, a set of coordinates $(x_1,\ldots,x_{n})\in(\C^*)^{n}$ are said to be exponential pre-Darboux coordinates for a holomorphic 2-form $\omega\in\Omega_T^2$ if $\omega$ is expressed in this set of coordinates as
$$\omega=\sum_{i<j}C_{ij}\cdot (x_ix_j)^{-1}dx_idx_j,$$
where $C_{ij}\in\C$ are all constant functions. If one defines $X_i:=\log(x_i)$, so that we can formally write $dX_i=d\log(x_i)=x_i^{-1}dx_i$, then exponential pre-Darboux coordinates are such that
$$\omega=\sum_{i<j}C_{ij}\cdot dX_idX_j,$$
i.e.~the coefficients are constant with respect to the expressions $\{dX_1,\ldots,dX_{n}\}$. The 2-form that we now construct
in Subsection~\ref{subsec:closed-2-form}
will be endowed with a set of exponential pre-Darboux coordinates.
\end{ex}

\begin{remark}
We use the term {\it pre-Darboux}, instead of {\it Darboux}, because $\omega$ might not a priori be symplectic and the constants might not define the standard symplectic basis. If $\omega$ is symplectic and $\{X_i = \log(x_i)\}$  are chosen as the standard symplectic basis, then exponential pre-Darboux coordinates coincide with the usual exponential Darboux coordinates.
\end{remark}


\subsection{Construction of a 2-form} 
\label{subsec:closed-2-form}
First, let us review the construction of a 2-form on the braid variety $X_{0}(\beta)$ according to \cite{Jeffrey,Mellit}. For that, let $\theta:=f^{-1}df$ and $\theta^{R}:=df f^{-1}$ denote respectively the left- and right-invariant algebraic 1-forms on the (complex) Lie group $G=\GL(n,\C)$; these 1-forms are valued in the Lie algebra $\mathfrak{g}=\mathfrak{gl}(n)$, and $\theta$ is referred to as the Maurer-Cartan form. 
We have the following facts (see e.g. \cite[Section 4]{Jeffrey},\cite[Section 3]{Mellit}):
\begin{itemize}
\item[(a)] The 3-form $\Omega:=\frac{1}{6}\Tr(\theta\wedge[\theta,\theta])$ is closed and represents a nontrivial class in $H^3(G;\C)\simeq \C$.\\ 

\item[(b)] There is a 2-form $(f|g):=\Tr(\pi_1^*\theta\wedge \pi_2^*\theta^R)=\Tr(f^{-1}df\wedge dg g^{-1})$ on $G\times G$ satisfying the following two ``cocycle conditions":
\begin{equation}
\label{eq: cocycle 1}
d(f|g)=\pi_1^*\Omega-m^*\Omega+\pi_2^*\Omega,
\end{equation}
\begin{equation}
\label{eq: cocycle 2}
(g|h)-(fg|h)+(f|gh)-(f|g)=0,
\end{equation}
where $\pi_1,\pi_2,m:G\times G\to G$ are the two projections and the Lie group multiplication map. 
\end{itemize}

\begin{definition}
Let $X$ be an arbitrary algebraic variety. A map $f: X\to G$ is said to be $\Omega$-trivial if $f^*\Omega=d\omega$ for some 2-form $\omega$ on $X$.
\end{definition}

Suppose that two maps $f:X\to G$ and $g:Y\to G$ are $\Omega$-trivial, then the product
$$
f\cdot g: X\times Y\xrightarrow{f\times g}G\times G\xrightarrow{m} G
$$
is also $\Omega$-trivial. Indeed, if $f^*\Omega=d\omega_X$ and $g^*\Omega=d\omega_Y$ then \eqref{eq: cocycle 1} implies
$$
(f\cdot g)^*\Omega=d(\omega_X+\omega_Y-(f|g)). 
$$
By iterating this construction, we obtain the following result.
\begin{prop}
\label{prop: omega trivial}
Suppose that $f_i:X_i\to G, i\in[1,r]$ are $\Omega$-trivial maps with $f_i^*\Omega=d\omega_i$ and consider the form on $X_1\times \cdots \times X_r$ given by:
\begin{equation}\label{eqn:form 1}
\omega:=\sum \omega_{X_i}-(f_1|f_2)-(f_1f_2|f_3)-\ldots-(f_1\cdots f_{r-1}|f_{r}).
\end{equation}
Then
$$
(f_1\cdots f_r)^*\Omega=d\omega.
$$
\end{prop}

Let us abbreviate:
\begin{equation}\label{eq: def form}
(f_1|f_2|\cdots|f_r) := (f_1|f_2) + (f_1f_2|f_3) + \ldots + (f_1\cdots f_{r-1}|f_r)
\end{equation}
so that \eqref{eqn:form 1} can be more succinctly written as
\begin{equation}\label{eqn:form 2}
\omega = \sum\omega_{X_i} - (f_1|f_2|\cdots|f_r).
\end{equation}
The condition \eqref{eq: cocycle 2} implies that this operation defines an associative convolution $(f_1|f_2|\cdots|f_r)$ on collections of $\Omega$-trivial maps. The following identity  will be useful for us. 
\begin{lemma}
For $f_{i}: X_{i} \to G$, $i = 1, \dots, r$ we have:
\begin{equation}
\label{eq: form merge}
(f_1|\cdots |f_r)=(f_1|\cdots|f_{j}f_{j+1}|\cdots |f_r)+(f_{j}|f_{j+1}).
\end{equation}
\end{lemma}
\begin{proof}
    This follows from \eqref{eq: cocycle 2}. Let us set $f := f_1\cdots f_{j-1}$. It follows from the definition \eqref{eq: def form} that \eqref{eq: form merge} is equivalent to
    \[
    (f|f_{j}) + (ff_{j}|f_{j+1}) = (f|f_{j}f_{j+1}) + (f_{j}|f_{j+1}).
    \]
    This latter identity is a consequence of \eqref{eq: cocycle 2}. 
\end{proof}

\begin{ex}
Suppose that $D_1,\ldots,D_r$ are diagonal matrices. Then $D_i$ and $dD_i$ all commute with each other, and one can prove by induction that
\begin{equation}
\label{eq: form diagonals}
(D_1|\cdots|D_r)=\sum_{i<j}\Tr(d\log D_i\wedge d\log D_j).
\end{equation}
Indeed, if $r=2$ then 
$$
(D_1|D_2)=\Tr(D_1^{-1}dD_1\wedge dD_2 \cdot D_2^{-1})=\Tr(d\log D_1\wedge d\log D_2).
$$
For the step of the induction, we write
$$
(D_1|\cdots|D_{r+1})=(D_1|\cdots|D_r)+(D_1\cdots D_r|D_{r+1})=(D_1|\cdots|D_r)+\Tr\left(d\log(D_1\cdots D_r)\wedge d\log D_{r+1}\right)=
$$
$$
(D_1|\cdots|D_r)+\sum_{i=1}^{r}\Tr(d\log D_i\wedge d\log D_{r+1}). 
$$
\end{ex}

Having summarized the necessary ingredients, let us apply this construction to braid varieties as follows. We can regard the braid matrices $B_i(z)$ as functions $B_i:\C\to G$ where $z$ is the coordinate on $\C$. 
The first key fact is that the maps $B_i:\C\to G$ given by the braid matrices are $\Omega$-trivial, since $B_i^*(\Omega)$ is a 3-form on $\C$ which must vanish.

Similarly, for a braid $\beta=\sigma_{i_1}\cdots \sigma_{i_r}$, we can regard the braid matrix 
$$
B_{\beta}(z_1,\ldots,z_r)=B_{i_1}(z_1)\cdots B_{i_r}(z_r)
$$
as a function $B_{\beta}:\C^r\to G.$
Let us define the following 2-form on $\C^r$:
\begin{multline}
\label{eq: def form beta}
\omega_{\beta}:=(B_{i_1}(z_1)|\cdots|B_{i_r}(z_r))=\\
(B_{i_1}(z_1)|B_{i_2}(z_2))+(B_{i_1}(z_1)B_{i_2}(z_2)|B_{i_3}(z_3))+\ldots+(B_{i_1}(z_1)\cdots B_{i_{r-1}}(z_{r-1})|B_{i_r}(z_{r})).
\end{multline}
Here we keep track of the arguments of different $B_{i_j}$ for the reader's convenience. By Proposition \ref{prop: omega trivial}, we conclude that the map $B_{\beta}:\C^{r}\to G$ is $\Omega$-trivial with primitive $-\omega_{\beta}$. By applying \eqref{eq: cocycle 2} repeatedly, we get the identity
\begin{equation}
\label{eq: omega product}
\omega_{\beta_1\beta_2}=\omega_{\beta_1}+(B_{\beta_1}|B_{\beta_2})+\omega_{\beta_2}.
\end{equation}
For reduced words, this form vanishes.
\begin{lemma}[\cite{Mellit}, Proposition 5.1.5]\label{lem: vanishing form reduced}
Let $\beta \in \Br^+_{n}$ be a reduced positive braid word. Then the 2-form $\omega_{\beta}$ vanishes on $\C^{\ell(\beta)}$.
\end{lemma}

The following example will prove useful.

\begin{ex}
\label{ex: form full twist}
Let $\Delta\in \Br_n^+$ be the positive braid $($word$)$ associated to the half-twist; then by Lemma \ref{lem: vanishing form reduced} we have that the 2-form $\omega_{\Delta}=0$ vanishes on $\C^{\binom{n}{2}}$. Following Lemma \ref{lem: full twist}, we can write
$$
B_{\Delta^2}=B_{\Delta}(c)B_{\Delta}(u)=Lw_0\cdot w_0U=LU,
$$
where two copies of $B_{\Delta}$ depend on two sets of independent variables $c_{ij}$ and $u_{ij}$. The 2-form on $\C^{2\binom{n}{2}}$ associated to $\Delta^2$ then reads:
$$
\omega_{\Delta^2}=\omega_{\Delta}(c)+(B_{\Delta}(c)|B_{\Delta}(u))+\omega_{\Delta}(u)=(B_{\Delta}(c)|B_{\Delta}(u))=
$$
$$
(Lw_0|w_0U)=(L|w_0|w_0|U)=(L|w_0w_0|U)=(L|U).
$$
Here the second equation follows from \eqref{eq: omega product}, and in the second line we use that $w_0$ is constant and $(w_0|f)=(f|w_0)=0$ for any $f$. 
\end{ex}

\begin{lemma}
The restriction of the 2-form $\omega_{\beta}$ to the braid variety $X_{0}(\beta)$ is closed.
\end{lemma}

\begin{proof}
Note that the map $B_{\beta}:X_{0}(\beta)\lr G$  lands in the subgroup of upper-triangular matrices, and the restriction of the 3-form $\Omega$ to the space of upper-triangular matrices vanishes. Therefore, since $d$ commutes with pull-back, we have
$$
d\omega_{\beta}=-B_{\beta}^*\Omega=0,
$$
i.e. $\omega_{\beta}$ is a closed 2-form.
\end{proof}

Consider now the toric charts $T_{\tau}\sse X_{0}(\beta\Delta; w_0) \subseteq X_{0}(\beta\cdot\FT)$ constructed in Corollary \ref{cor:open crossings} and, in particular, the restriction $\omega|_{T_{\tau}}$. Recall the matrices $U_i, D_i, L_i$ defined in \eqref{eqn: UDL}. By Proposition \ref{prop:open crossings} (ii) and Corollary \ref{cor:open crossings}, the coordinates on the torus $T_{\tau}$ are given by the coordinates associated to the $D_{i}$-matrices that appear while opening crossings according to $\tau$. 

\begin{lemma}
\label{lem: form constant}
Let $\beta$ be a positive braid word and $\tau\in S_{\ell(\beta)}$. The restriction of the 2-form $\omega_{\beta\cdot \Delta^2}$ to the toric chart $T_{\tau}\sse X_0(\beta\Delta; w_0) \sse X_{0}(\beta\cdot\FT)$ has constant coefficients in the canonical (exponential) coordinates associated to the $D_i$ matrices.
\end{lemma}

Here {\it constant} coefficients is to be understood in the sense of Example \ref{ex:2form_torus}, i.e.~Lemma \ref{lem: form constant} states that the coordinates associated to the $D_i$ are exponential pre-Darboux coordinates for $\omega_{\beta\cdot \Delta^2}$.

\begin{proof}
By Lemma \ref{lem: full twist}, we can write 
$$
B_{\beta\cdot \Delta^2}=B_{\beta}B_{\Delta^2}=B_{i_1}(z_1)\cdots B_{i_r}(z_r)LU.
$$
By Example \ref{ex: form full twist}, we can also write
$$
\omega_{\beta\cdot \Delta^2}=\omega_{\beta}+(B_{\beta}|B_{\Delta^2})+\omega_{\Delta^2}=
$$
$$
\omega_{\beta}+(B_{\beta}|LU)+(L|U)=
(B_{i_1}(z_1)|\cdots |B_{i_r}(z_r)|L|U).
$$
Next, we need to understand the behavior of the 2-form under opening the crossings according to $\tau$, as this determines the construction of the toric chart $T_{\tau}$. Note that on the variety $X_0(\beta\Delta; w_0) \subseteq X_0(\beta\Delta^{2})$ we have $U = I$, the identity matrix. We break this computation in several steps.

1) By using the decomposition in equation \eqref{eq: factorization} and \eqref{eq: form merge}, we can write
$$
( \cdots|B_{i_s}(z_s)|\cdots )=
$$
$$
( \cdots|U_{i_s}(z_s)|D_{i_s}(z_s)|L_{i_s}(z_s)|\cdots )-
(U_{i_s}(z_s)|D_{i_s}(z_s)|L_{i_s}(z_s)).
$$
Note that $(U_{i_s}(z_s)|D_{i_s}(z_s)|L_{i_s}(z_s))$ is a $2$-form on a $1$-dimensional space (with coordinate $z_{s}$) and therefore vanishes, so
$$
( \cdots|B_{i_s}(z_s)|\cdots )=
( \cdots|U_{i_s}(z_s)|D_{i_s}(z_s)|L_{i_s}(z_s)|\cdots ).
$$
2) Next, we would like to move upper-triangular matrices to the left and lower-triangular matrices to the right as in Lemma \ref{lem:slide triangular}. Assume that $U$ is an upper unitriangular matrix (so $dU$ is strictly upper triangular) then 
$$
(\cdots |B_i(z)|U|\cdots)=(\cdots |B_i(z)U|\cdots )+(B_i(z)|U)=
$$
$$
(\cdots |\widetilde{U}B_i(z')|\cdots )+(B_i(z)|U)=
(\cdots |\widetilde{U}|B_i(z')|\cdots )+(B_i(z)|U)-(\widetilde{U}|B_i(z')).
$$
The terms $(B_i(z)|U),(\widetilde{U}|B_i(z'))$ in fact vanish. Indeed, observe that 
$$
B_i^{-1}(z)dB_i(z)=\left(\begin{matrix} -z & 1\\ 1 & 0\\ \end{matrix}\right)\left(\begin{matrix} 0 & 0\\ 0 & dz\\ \end{matrix}\right)=\left(\begin{matrix} 0 & dz\\ 0 & 0\\ \end{matrix}\right),
$$
while $dU\cdot U^{-1}$ is strictly upper triangular, so 
$$
(B_i(z)|U)=\Tr\left(B_i^{-1}(z)dB_i(z) \wedge dU\cdot U^{-1}\right)=0.
$$
Similarly, 
$$
dB_i(z')\cdot B_i(z')^{-1}=\left(\begin{matrix} 0 & 0\\ 0 & dz'\\ \end{matrix}\right)\left(\begin{matrix} -z' & 1\\ 1 & 0\\ \end{matrix}\right)=\left(\begin{matrix} 0 & 0\\ dz'& 0\\ \end{matrix}\right),
$$
so that 
$$
(\widetilde{U}|B_i(z')) = 
\Tr\left(\widetilde{U}^{-1}d\widetilde{U} \wedge dB_i(z')\cdot B_i(z')^{-1}\right)=
(\widetilde{U}^{-1}d\widetilde{U})_{i, i+1}dz'.$$ 
On the other hand, by Lemma \ref{lem:slide triangular} we get $\widetilde{U}_{i+1,i+1} = 1$ and $\widetilde{U}_{i,i+1}=0$, hence
$d\widetilde{U}_{i+1,i+1} = d\widetilde{U}_{i, i+1} = 0$.
Therefore
$$
(\widetilde{U}^{-1}d\widetilde{U})_{i, i+1} = \sum_{k}(\widetilde{U}^{-1})_{i,k}d\widetilde{U}_{k, i+1}=(\widetilde{U}^{-1})_{i,i}d\widetilde{U}_{i, i+1}+(\widetilde{U}^{-1})_{i,i+1}d\widetilde{U}_{i+1, i+1}=0,
$$
and $(\widetilde{U}|B_i(z'))=0$.
We conclude that 
$$
(\cdots |B_i(z)|U|\cdots )=(\cdots |\widetilde{U}|B_i(z')|\cdots ),
$$
and similarly $(\cdots |D_i(z)|U|\cdots )=(\cdots |\widetilde{U}|D_i(z)|\cdots ).$
The conclusion from this computation is that the 2-form $\omega_{\beta\cdot \Delta^2}$ does not change as we move $U$ to the left. Similarly, it does not change as we move lower-triangular matrices to the right.

3) After opening all crossings, we are left with several upper unitriangular matrices, followed by several diagonal matrices and by several lower unitriangular matrices. 

Let $U$ be an upper unitriangular matrix and $U'$ an upper-triangular matrix, then $dU$ is strictly upper-triangular and $dU'$ is upper-triangular. Therefore $U^{-1}dU$ is strictly upper-triangular and $dU'(U')^{-1}$ is  upper-triangular, hence
\begin{equation}
\label{eq: form upper triangular}
(U|U')=\Tr\left(U^{-1}dU \wedge dU'(U')^{-1}\right)=0.
\end{equation}
Similarly, $(L|L') = 0$ for two lower unitriangular matrices $L,L'$. 

By \eqref{eq: form merge} this means that we can use \eqref{eq: form upper triangular} to consolidate all upper and all lower unitriangular matrices and write 
$$
\omega_{\beta\cdot \Delta^2}=(\widetilde{U}|D_{i_1}|\cdots|D_{i_r}|\widetilde{L}|I).
$$
Since $\widetilde{U}D_{i_1}\cdots D_{i_r}\widetilde{L}$ is upper-triangular, we get $\widetilde{L}=I$. On the other hand, 
by \eqref{eq: form upper triangular} we get
$$
(\widetilde{U}|D_{i_1}\cdots D_{i_r}|I)=(\widetilde{U}|D_{i_1}\cdots D_{i_r})+(\widetilde{U}D_{i_1}\cdots D_{i_r}|I)=0,
$$
Thus, by \eqref{eq: form diagonals} we get
$$
\omega_{\beta\cdot \Delta^2}=(\widetilde{U}|D_{i_1}|\cdots|D_{i_r}|I)=(D_{i_1}|\cdots |D_{i_r})=
\sum_{s<t} \Tr(d\log(D_{i_s})\wedge d\log(D_{i_t})).
$$
By direct computation, using the notation in (\ref{eqn: UDL}) for $D_i$, we have
$$d\log D_i(x)=d\log\left(\begin{matrix} -x^{-1} & 0\\ 0 & x\end{matrix}\right)=\left(\begin{matrix} -x^{-1}dx & 0\\ 0 & x^{-1}dx\end{matrix}\right),$$
for some variable $x$. Therefore, each summand $\Tr(d\log(D_{i_s}(x_s))\wedge d\log(D_{i_t}(x_t)))$ as above is of the form $C_{st}\cdot d\log(x_s)d\log(x_t)$ with $C_{st}$ a constant, for some coordinates $x_s,x_t$ on the torus $T_{\tau}$.
Therefore, $\{x_1,\ldots,x_r\}$ are exponential pre-Darboux coordinates for $\omega_{\beta\cdot \Delta^2}$.
\end{proof}

\begin{cor}
\label{cor: form on cochar}
The form $\omega_{\beta\cdot \Delta^2}$ induces a skew-symmetric bilinear form on the cocharacter lattice of the torus chart $T_{\tau}$.
\end{cor}

As emphasized above, the proof of Lemma \ref{lem: form constant} actually shows that the entries of the $D_i$ matrices associated to (opening the crossings for) $\beta$ are exponential pre-Darboux. We have now discussed closedness of the 2-form $\omega_{\beta\cdot \Delta^2}$ and its expression in the toric charts $T_{\tau(\beta)}$. In order to show that $\omega_{\beta\cdot \Delta^2}$ induces an holomorphic symplectic structure, as stated in Theorem \ref{thm:main1}.(iii), it suffices to show non-degeneracy, which we now address.


\subsection{Non-degeneracy of $\omega_{\beta\cdot \Delta^2}$}\label{sec:nondegenerate form} Let us recall the torus $T_c$ from Remark \ref{rmk: free subtori} that acts freely on the variety $X_0(\beta\cdot\Delta^{2})$. The following result is key for this section: it relates the action of the torus $T_c$ to the $2$-form $\omega_{\beta\cdot\Delta^{2}}$.

\begin{lemma}[\cite{Mellit}, Proposition 5.3.3]\label{lem: equivariant form}
The form $\omega_{\beta\cdot\Delta^2}$ is $T_c$-invariant. Thus, it descends to a $2$-form  $\omega_{\beta\cdot\Delta^{2}/T_c}$ on the quotient $X_0(\beta\Delta; w_0)/T_c$.
\end{lemma}

In this subsection, we will show that $\omega_{\beta\cdot\Delta^{2}/T_c}$ is nondegenerate, and thus holomorphic symplectic, on the space $X_0(\beta\Delta;w_0)/T_c$.  In order to do this, let us consider the Mellit chart $\mathfrak{M}$, as constructed in Theorem \ref{thm:mellit chart}.
By Remark \ref{rmk: free subtori} and Corollary \ref{cor:open crossings}, the torus $T_c$ acts freely on  $\mathfrak{M}$, 
and we will consider restrictions of $\omega_{\beta\cdot\Delta^{2}}$ on $\mathfrak{M}$ and $\omega_{\beta\cdot\Delta^{2}/T_c}$ on $\mathfrak{M}/T_c$ respectively. 

We will first show that the restriction  of $\omega_{\beta\cdot\Delta^2/T_c}$
to $\mathfrak{M}/T_c$ is non-degenerate, and thus (holomorphic) symplectic. Then we prove, in Theorem \ref{thm: symplectic}, that $\omega_{\beta\cdot\Delta^2/T_c}$ induces the holomorphic symplectic structure according to Theorem \ref{thm:main1}.(iii).

Following \cite[Section 6]{Mellit}, we can construct a topological avatar for the torus $\mathfrak{M}$, as follows. Consider a labeled marked surface $(\cS,A,B)$, i.e. an oriented surface $\cS$ with boundary $\partial \cS$ and two sets of points $A:= \{1, 2, \dots, n\}, B:= \{1', 2', \dots, n'\} \subseteq \partial\cS$ such that:

\begin{itemize}
\item[-] Each connected component of $\cS$ has a boundary component. 
\item[-] Each boundary component intersects both $A$ and $B$.
\item[-] The elements of $A$ and $B$ in each boundary component alternate.
\end{itemize}

Let us denote the two Abelian groups $\Lambda := H_{1}(\cS, A)$ and $\Lambda' := H_{1}(\cS, B)$. Since $A$ and $B$ are alternating, there is a perfect pairing $\cdot: \Lambda \otimes \Lambda' \lr \Z$. There is also a map $\rot: \Lambda \lr \Lambda'$, that is induced from the map that, up to homotopy, rotates the boundary components clockwise. This induces a bilinear form $\wt{\omega}_{\cS}$ on the first homology $\Lambda$, given by $\wt{\omega}_{\cS}(\gamma, \gamma') = \gamma \cdot \rot(\gamma')$, and we also consider its anti-symmetrization $\omega_{\cS}$. 

By Corollary \ref{cor: form on cochar}, $\omega_{\beta\cdot \Delta^2}$ induces a form on the cocharacter lattice of $\mathfrak{M}$.
In order to prove the symplecticity of $\omega_{\beta\cdot \Delta^2}$ stated in Theorem \ref{thm:main1}, we use the following result.

\begin{lemma} \emph{(\cite[Section 6.5]{Mellit})} 
There exists a marked surface $(\cS, A, B)$ such that $\Lambda$ is identified with the cocharacter lattice of $\mathfrak{M}$, and the form induced by $\omega_{\beta\cdot \Delta^2}$ on the cocharacter lattice of $\mathfrak{M}$ is identified (up to a nonzero constant factor) with $\omega_{\cS}$.
\end{lemma} 
Note that the surface $\cS$ is homeomorphic to the spectral curve constructed in \cite{CZ}. 
Now, we need two more properties of $(\cS, A, B)$, which follow from the construction in \cite[Section 6.5]{Mellit}. Recall from Remark \ref{rmk: free subtori} that $\pi$ is the permutation corresponding to $\beta$ with disjoint cycles $C_1,\ldots,C_k$, and  $C_{j}=(a_{j, 1}\dots a_{j, \ell_{j}})$.

\begin{itemize}
\item[-] The connected components of $\partial\cS$ correspond to the cycles of $\pi$, i.e. to the components of the closure of the braid $\beta$.
\item[-] Let $C_j$ be the connected component of $\partial\cS$ corresponding to the cycle $(a_{j,1}\dots a_{j,\ell_j})$. Then, the elements of $A = \{1, \dots, n\}$ appearing in $C$ are precisely $a_{j,1}\dots a_{j,\ell_j}$, and they appear in the same order as in the cycle.
\end{itemize}

We will now decompose $\Lambda = H_{1}(\cS, A)$, as follows. First, we have the exact sequence in relative homology
$$
0 \to H_{1}(\cS) \to H_{1}(\cS, A)\buildrel \partial \over \to H_{0}(A) = \Z^{A} \to H_{0}(\cS) \to 0,
$$
\noindent where the image of $\partial$ is spanned by elements of the form $a - b$, where $a, b \in A$ belong to the same connected component of $\cS$.  For each such $a, b$, we choose a path from $a$ to $b$ in $\cS$, and we let $K$ be the span of the classes of these paths in homology. This gives a splitting
$$
H_{1}(\cS, A) = H_{1}(\cS) \oplus K
$$
We construct a basis of $K$ as follows. For simplicity, we will assume that $\cS$ is connected, the general case follows similarly. For each connected component $C_{j}$ of $\partial\cS$, we take the path from $a_{j, i}$ to $a_{j, i+1}$ following $C_{j}$, $j\in[1,\ell_{j} - 1]$. We also take a path $\gamma_{j}$ from $a_{j, \ell_{j}}$ to $a_{j+1, 1}$, $j\in[1,k-1]$. Then we obtain the basis of $K$, see Figure \ref{fig:surface}:
$$K = \Z\{a_{j,i}a_{j,i+1}, \gamma_{j'} \mid j\in[1,k], i\in[1,\ell_{j} - 1], j'\in[1,k-1]\}.$$

\begin{center}
	\begin{figure}[h!]
		\centering
		\includegraphics[scale=0.7]{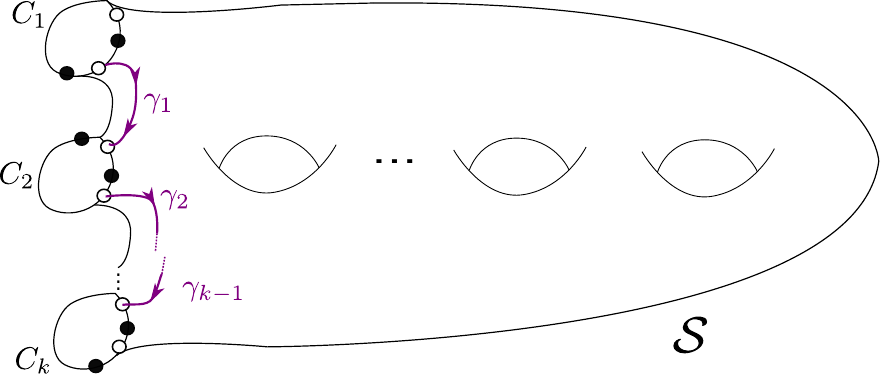}
		\caption{The surface $\cS$, with the marked points in the boundary. Points of $A$ are colored white, and points of $B$ are colored black. For the sake of readability we do not label the paths along the boundary for two consecutive points of $A$.}
		\label{fig:surface}
	\end{figure}
\end{center}

We can further split $H_{1}(\cS)$ as follows. We let $\bar{\cS}$ be the surface obtained from $\cS$ by attaching disks along the boundary components. We have an exact sequence
$$
0 \to H_{2}(\bar{\cS}) \to H_{1}(\partial \cS) \to H_{1}(\cS) \to H_{1}(\bar{\cS}) \to 0
$$

so that $H_{1}(\cS) = H_{1}(\bar{\cS}) \oplus (H_{1}(\partial \cS)/H_{2}(\bar{\cS}))$. Note that a spanning set for $H_{1}(\partial\cS)/H_{2}(\bar{\cS})$ is given by $C_{i} - C_{j}$, where $C_{i}$ and $C_{j}$ are boundary components of the same connected component of $\cS$. Since we are assuming $\cS$ is connected, a basis is given by $C_{i} - C_{i+1}, i\in[1,k-1]$. Moreover, since the elements in $H_{1}(\cS)$ are $\rot$-invariant, the form $\omega_{\cS}$ on $H_{1}(\cS)$ is given by the intersection form. This implies that $\omega_{\cS}|_{H_{1}(\bar{\cS})}$ is the intersection form on $\bar{\cS}$, and therefore is non-degenerate. In addition, $\omega_{\cS}(H_{1}(\bar{\cS}), H_{1}(\partial \cS)/H_{2}(\bar{\cS})) = 0$ and $\omega_{\cS}(H_{1}(\bar{\cS}), K) = 0$. Thus, using the decomposition
$$
\Lambda=H_{1}(\cS, A) = H_{1}(\bar{\cS}) \oplus (H_{1}(\partial \cS)/H_{2}(\bar{\cS})) \oplus K,
$$
the form $\omega_{\cS}$ has the following form
$$
\omega_{\cS} = \left(\begin{matrix} \omega_{\cS}|_{H_{1}(\bar{\cS})} & 0 & 0 \\ 0 & 0 & * \\ 0 & * & * \end{matrix} \right).
$$

We will not find the remaining terms * for $\omega_{\cS}$, we will only do so after passing to the quotient by the action of a torus, as this is all that suffices. We have a map $\psi: \Z^{A} \to H_{1}(\cS, A)$ that to each point $a \in A$ associates the path that follows the boundary component containing $a$ from $a$ to $\rot^{2}(a)$. In other words, it sends $a_{j, i}$ to a path $a_{j,i}a_{j,i+1}$, where $a_{j, \ell_{j} + 1} = a_{j, 1}$. 

 The torus  $T_c$ is the fixed torus for the action of the element $\sigma = (a_{1,\ell_{1}}a_{2,\ell_{2}}\cdots a_{k, \ell_{k}})$. According to \cite{Mellit}, to find the cocharacter lattice of $\mathfrak{M}/T_c$ we need to mod out by the image of $\psi$ on $\sigma$-invariant elements of $A$. Thus, the cocharacter lattice of $\mathfrak{M}/T_c$ can be identified with
$$
\overline{\Lambda}=H_{1}(\bar{\cS}) \oplus (H_{1}(\partial\cS)/H_{2}(\bar{\cS})) \oplus \Z\{\gamma_{1}, \dots, \gamma_{k-1}\}.
$$
There is a natural projection $q:\Lambda\to \overline{\Lambda}$, and the form $\omega_{\cS}$ descends to a form $\omega_{\cS,\overline{\Lambda}}$ on $\overline{\Lambda}$. It agrees with the form $\omega_{\mathfrak{M}/T_c}$ induced by $\omega_{\beta\cdot \Delta^2}$ on the cocharacter lattice of $\mathfrak{M}/T_c$.

In addition, note that we can identify $q(C_{i}) = a_{i, \ell_{i}}a_{i,1}$. Thus, $q(C_{i} - C_{i+1}) = a_{i, \ell_{i}}a_{i,1} - a_{i+1, \ell_{i+1}}a_{i+1, 1}$. Note also that $\omega_{\cS}(\gamma_{i}, \gamma_{j}) = 0$ as $\gamma_{i}\cdot \rot(\gamma_{j}) = 0$ for every $i \neq j$. Moreover, $\gamma_{i}\cdot \rot(a_{i, \ell_{i}}a_{i,1} - a_{i+1, \ell_{i+1}}a_{i+1, 1}) = 0$ while $(a_{i, \ell_{i}}a_{i,1} - a_{i+1, \ell_{i+1}}a_{i+1, 1})\cdot \gamma_{i} = 2$. Thus, $\omega_{\cS,\overline{\Lambda}}(\gamma_{i}, a_{i, \ell_{i}}a_{i,1} - a_{i+1, \ell_{i+1}}a_{i+1, 1}) = 2$. Similarly, we can see that $\omega_{\cS,\overline{\Lambda}}(\gamma_{i}, a_{i-1, \ell_{i-1}}a_{i-1,1} - a_{i, \ell_{i}}a_{i, 1}) = 1$ and $\omega_{\cS,\overline{\Lambda}}(\gamma_{i}, a_{i+1, \ell_{i+1}}a_{i+1,1} - a_{i+2, \ell_{i+2}}a_{i+2, 1}) = -1$.  It follows that the form $\omega_{\cS,\overline{\Lambda}}$ is given by the following matrix:
$$
\omega_{\cS,\overline{\Lambda}} = \left(\begin{matrix} \omega_{\cS}|_{H_{1}(\bar{\cS})} & 0 & 0 \\ 0 & 0 & -P \\ 0 & P & 0 \end{matrix} \right),
$$
\noindent where $P$ is the $(k-1) \times (k-1)$-matrix
$$
P = \left(\begin{matrix} 2 & -1 & 0 & \cdots & 0 \\ -1 & 2 & -1 & \cdots &  0 \\ 0 & -1 & 2  & \cdots &  0 \\  \vdots & \vdots & \vdots & \ddots & \vdots \\ 0 & 0 & 0 & \cdots & 2 \end{matrix}\right).
$$

\noindent This implies that the form $\omega_{\mathfrak{M}/T_c}$ on the cocharacter lattice of $\mathfrak{M}/T_c$
is non-degenerate, therefore the restriction  of $\omega_{\beta\cdot\Delta^2/T_c}$
to $\mathfrak{M}/T_c$
is non-degenerate as well. Thus the chart ${\mathfrak{M}/T_c}$ is (holomorphic) symplectic. Let us now use the above discussion, and this result for the Mellit chart, to conclude Theorem \ref{thm:main1}.(iii).
 
\begin{thm}
\label{thm: symplectic}
Let $\beta \in \mathrm{Br}^{+}_{n}$ be a positive braid (word). Then, the 2-form  $\omega_{\beta\cdot\FT}$ induces a 2-form on the augmentation variety $\Aug(\beta, \ft_{c})$ that has maximal rank at every point. Thus, the augmentation variety of any positive braid is holomorphic symplectic.
\end{thm}

\begin{proof}
By Theorem \ref{thm:aug vs braid}, the augmentation variety $\Aug(\beta, \ft_{c})$ can be identified with $X_{0}(\beta\Delta; w_{0})/T_c$. The coefficients of the form $\omega_{\beta\cdot \FT}$ are regular functions on $X_0(\beta\Delta;w_0)$ and by Lemma \ref{lem: equivariant form} the form is $T_c$-invariant. Thus, we have an induced closed $2$-form $ \omega_{\beta\cdot\Delta^{2}/T_c}$ on the augmentation variety, and it is non-degenerate if and only if its determinant does not vanish anywhere. Let us first prove that it is non-degenerate on all toric charts. Thanks to the discussion above on the Mellit chart, the form $\omega_{\beta\cdot\Delta^{2}/T_c}$ is non-degenerate on the (quotient) toric chart $\mathfrak{M}/T_c$; 
By Lemma \ref{lem: form constant}, the 2-form 
$\omega_{\beta\cdot\Delta^{2}/T_c}$ has constant coefficients in canonical coordinates in any other chart $\mathfrak{M}'/T_c$ 
obtained from an ordering of the crossings, and, by the above, it is non-degenerate on the intersection with $\mathfrak{M}/T_c$. Thus, the 2-form $\omega_{\beta\cdot\Delta^{2}/T_c}$ is non-degenerate on the entire (other) chart $\mathfrak{M}'/T_c$. Finally, by Theorem \ref{thm: codim 2}, these toric charts cover $\Aug(\beta, \ft_{c})$ up to codimension 2. Hence, the determinant of $\omega_{\beta\cdot\Delta^{2}/T_c}$ is non-zero outside of a codimension 2 locus and hence it is non-zero everywhere. 
\end{proof}

This concludes the proof of Theorem \ref{thm:main1} and establishes that the augmentation variety associated to a positive braid is holomorphic symplectic. The following is an explicit example to help illustrate the computations and arguments above.

\begin{ex}
Consider the case $n=2,\beta=\sigma^2$, so that $X(\beta\Delta; w_0)=X(\sigma^3; w_0)$. A direct computation, similar to Example \ref{ex:trefoil}, shows that 
$$
X(\sigma^3;w_0)=\{(z_1,z_2,z_3):z_1+z_3+z_1z_2z_3=0\}\subset \C^3.
$$
As in Example \ref{ex:trefoil}, we can rewrite the defining equation as $z_1+z_3(1+z_1z_2)=0$ and observe that $1+z_1z_2\neq0$. Indeed, $1+z_1z_2=0$ and $z_1+z_3(1+z_1z_2)=0$ imply $z_1=0$, which would imply $1+z_1z_2=1$, a contradiction with $1+z_1z_2=0$. Since $1+z_1z_2\neq 0$, we can write
\[
z_3=-\frac{z_1}{1+z_1z_2},
\]
and thus
$
X(\sigma^3; w_0)\cong \{(z_1,z_2):1+z_1z_2\neq 0\}\subset \C^2.
$
The torus $T_c$ is trivial in this case. Let us compute the form $\omega_{\beta\Delta^2}$ on $X(\sigma^3; w_0)/T_c=X(\sigma^3; w_0)$. For the matrix
\[
M=B(z_1)B(z_2)=\left(
\begin{matrix}
1 & z_2\\
z_1 & 1+z_1z_2\\
\end{matrix}
\right),
\]
we compute
\[
M^{-1}dM=
\left(
\begin{matrix}
1+z_1z_2 & -z_2\\
-z_1 & 1\\
\end{matrix}
\right)
\left(
\begin{matrix}
0 & dz_2\\
dz_1 & z_1dz_2+z_2dz_1\\
\end{matrix}
\right)=
\left(
\begin{matrix}
-z_2dz_1 & dz_2-z_2^2dz_1\\
dz_1 & z_2dz_1\\
\end{matrix}.
\right)
\]
The two-form $\omega_{\beta\Delta^2}$ can now be computed as
\[
\omega_{\beta\Delta^2}=(B(z_1)|B(z_2))+(B(z_1)B(z_2)|B(z_3))=
(B(z_1)|B(z_2))+(M|B(z_3))=
\]
\[
=\Tr\left[\left(
\begin{matrix}
0 & dz_1\\
0 & 0\\
\end{matrix}
\right)\left(
\begin{matrix}
0 & 0\\
dz_2 & 0\\
\end{matrix}
\right)\right]+
\Tr\left[\left(
\begin{matrix}
-z_2dz_1 & dz_2-z_2^2dz_1\\
dz_1 & z_2dz_1\\
\end{matrix}
\right)\left(
\begin{matrix}
0 & 0\\
dz_3 & 0\\
\end{matrix}
\right)\right]=
\]
\[
=dz_1dz_2+dz_2dz_3-z_2^2dz_1dz_3.
\]
Let us explicitly show that $\omega_{\beta\Delta^2}$ is symplectic form on $X(\sigma^3; w_0)$, as follows. Since 
\[
dz_3=d\left(-\frac{z_1}{1+z_1z_2}\right)=\frac{-dz_1(1+z_1z_2)+z_1(z_1dz_2+z_2dz_1)}{(1+z_1z_2)^2}=\frac{-dz_1+z_1^2dz_2}{(1+z_1z_2)^2},
\]
we can further write
\[
\omega_{\beta\Delta^2}=dz_1dz_2-\frac{dz_2dz_1}{(1+z_1z_2)^2}-\frac{z_1^2z_2^2dz_1dz_2}{(1+z_1z_2)^2}=\frac{1+2z_1z_2+z_1^2z_2^2-1-z_1^2z_2^2}{(1+z_1z_2)^2}dz_1dz_2.
\]
This simplifies to the expression
\begin{equation}
\label{eq: form Hopf}
\omega_{\beta\Delta^2}=\frac{2dz_1dz_2}{1+z_1z_2}=\frac{2dz_1dz_2}{w},\quad\mathrm{where}\quad w:=1+z_1z_2.
\end{equation}
We conclude that $\omega$ is holomorphic symplectic on the open subset $\{w\neq 0\}\subset \C^2$
which is isomorphic to $X(\sigma^3; w_0)$ as explained above.\\

Let us now find explicit exponential Darboux coordinates. The two ways of opening crossings in $\beta$ correspond to two toric charts $T_1:=\{z_1\neq 0,w\neq 0\}$ and $T_2:=\{z_2\neq 0,w\neq 0\}$ in $X(\sigma^3; w_0)$. Identity \eqref{eq: form Hopf} implies that
\[
\omega_{\beta\Delta^2}=\frac{2dz_1dw}{z_1w}=\frac{-2dz_2dw}{z_2w}.
\]
It therefore follows that $\{z_1,w\}$ are exponential Darboux coordinates in $T_1$, and $\{z_2,w\}$ are exponential Darboux coordinates in $T_2$. This is indeed in agreement with Lemma \ref{lem: form constant} above. The corresponding skew-symmetric form on the cocharacter lattice of both tori is given by the matrix
$
\left(
\begin{matrix}
0 & 2\\
-2 & 0\\
\end{matrix}
\right),
$
up to reordering coordinates. 

Finally, the surface $\cS$ in this case is an annulus. It has two boundary components and one point from $A$ and one point from $B$ on each component. The relative homology $\Lambda=H_1(\cS,A)$ has rank 2 and is generated by an absolute cycle $\gamma$ along the core of the annulus, and a relative cycle $\gamma'$ connecting the two points in $A$. With an appropriate choice of orientations, the intersection form on $\Lambda$ is given by the matrix
$
\left(
\begin{matrix}
0 & 1\\
-1 & 0\\
\end{matrix}
\right),
$
which is half the skew-symmetric form in the cocharacter lattice. Note that in this case we cannot use $H_1(\cS)$ or $H_1(\cS,\partial \cS)$, as these lattices have rank 1, while our tori are two-dimensional. This explains the need of introducing marked points $A$ and $B$, cf.~also \cite[Section 3]{CW}. From the viewpoint of cluster algebras, the variable $z_1$ is mutable and corresponds to the absolute cycle $\gamma$, and the variable $w$ is frozen and corresponds to the relative cycle $\gamma'$, up to signs. See \cite{CW,CGGLSS} for more details and \cite{Scroggin} for more examples and computations of the form $\omega_{\beta\Delta^2}$ for 2-stranded braids.
\end{ex}

\begin{remark}
In our more recent work \cite{CGGLSS}, we construct cluster structures on braid varieties and in \cite[Section 9.2]{CGGLSS} we show that the Gekhtman-Shapiro-Vainshtein cluster $2$-form for the corresponding cluster structure on $X_0(\beta\Delta; w_0)$ coincides with $\omega_{\beta\cdot\Delta^{2}}$. The variety $\Aug(\beta, \ft_{c}) = X_0(\beta\Delta;w_0)/T_c$ is an even-dimensional quotient of the cluster variety $X(\beta\Delta;w_0)$ which has really full rank by \cite[Section 8.1]{CGGLSS}, and the torus $T_c$ acts by so-called cluster automorphisms, cf. \cite[Section 5.1]{LamSpeyer}. The non-degeneracy of $\omega_{\beta\cdot\FT/T_c}$ can then also be deduced from the results of \cite[Section 5.5]{LamSpeyer}. This argument using \cite{CGGLSS} and \cite{LamSpeyer} is logically independent of the one given in this section, and \cite{CGGLSS} appeared after the present article.
\end{remark}

\begin{remark}
It is likely that the above setup can be shown to fit within the context of P.~Boalch's work \cite{Boalch01,Boalch02,Boalch20}, of which we learnt after this manuscript first appeared. In particular, the holomorphic symplectic structure constructed above might likely coincide with some of the holomorphic symplectic structures he builds on wild character varieties by using quasi-Hamiltonian $G$-spaces with $G$-valued moment maps. (Potentially, the moment map is given by the action of the marked points in $\mathfrak{t_c}$.) Moving onward, we hope to better understand their work and connect it to the results above.
\end{remark}

This subsection concludes the first part of the article, and we now move forward to discuss correspondences between braid varieties and the diagrammatic calculus we develop for their study.


\section{The Combinatorics of Weaves}\label{sec:CombinatoricsWeaves}

This section discusses weaves, based on \cite{CZ}, and connects them to braid varieties. In short, weaves are a diagrammatic calculus that can be used to study the braid varieties $X_0(\beta)$, describing toric charts, regular functions and other relevant geometric structures on them. The present section focuses on the combinatorial aspects of these diagrams; in particular, this formalizes the weave category $\mathfrak{W}_n$ discussed in Section \ref{sec:intro}. We use these  weaves in Section \ref{sec: alg weaves}, where we prove that a weave between two positive braids $\beta_1$ and $\beta_2$ yields a correspondence between the braid varieties $X_0(\beta_1)$ and $X_0(\beta_2)$ (as stated in Theorem \ref{thm:main2}). We refer the reader to \cite{CZ} for the original definition of weaves as well as the contact and symplectic geometry motivation behind them, cf.~also \cite{CW}.
	

\subsection{Weaves}

Weaves are diagrams introduced in the work of the first author and E. Zaslow \cite{CZ}. They are defined on any smooth surface $\Sigma$ but, in the present manuscript, we restrict ourselves to the diffeomorphism type of the plane $\Sigma=\R^2$. In appearance, these diagrams are similar to the planar diagrams appearing in Soergel calculus \cite{EK,EW}; there are nevertheless key distinctions. We refer to our diagrams as {\it weaves}, as they are a particular instance of the symplectic constructions in \cite{CZ}.
\begin{center}
	\begin{figure}[h!]
		\centering
		\includegraphics[scale=0.7]{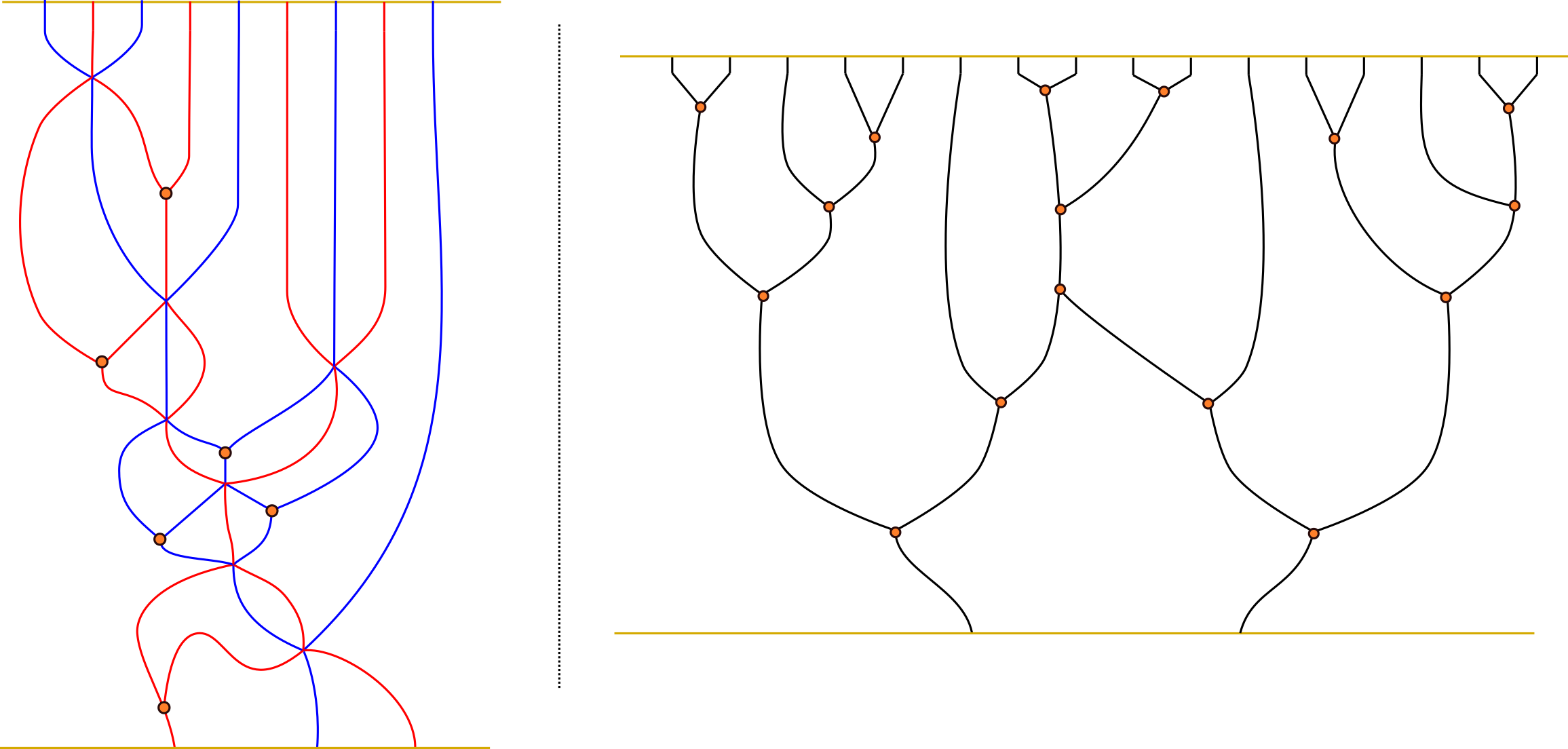}
		\caption{(Left) A 3-weave from $\beta_2=(\sigma_1\sigma_2)^4\sigma_1\in \Br^+_3$ down to $\beta_1=\sigma_2\sigma_1\sigma_2\in Br^+_3$. The blue color indicates a transposition label $s_1\in S_3$ and the red color indicates the transposition label $s_2\in S_3$. (Right) A 2-weave from $\beta_2=\sigma_1^{16}\in \Br_2^+$ to $\beta_1=\sigma_1^{2}\in \Br_2^+$, all black edges are labeled with the unique transposition $s_1\in S_2$. Trivalent vertices are emphasized in orange in both weaves.}
		\label{fig:WeaveExample}
	\end{figure}
\end{center}

\begin{definition}\label{def:weave}
Let $\beta_1,\beta_2$ be two positive $n$-braid words. By definition, a weave $\mathfrak{w}$ of degree $n$ from $\beta_2$ to $\beta_1$, denoted $\mathfrak{w}: \beta_2 \to \beta_1$, is the image of a continuous map $$\mathfrak{w}: \bigcup_{i=1}^{n-1} G_i\lr \R\times[1,2],$$

where each $G_i$, $i\in[1,n-1]$ is a trivalent graph and the following conditions are satisfied:

\begin{itemize}
\item[(i)] The restriction $\mathfrak{w}|_{G_i}:G_i\lr\R\times[1,2]$ is a topological embedding for all $i\in[1,n-1]$, which is a smooth embedding away from the trivalent vertices of the graph $G_i$.\\

\item[(ii)] The images $\mathfrak{w}(G_i)$ and $\mathfrak{w}(G_{i+1})$ are only allowed to intersect at trivalent vertices, $i\in[1,n-2]$, and the planar edges around this intersection point must alternatingly belong to $G_i$ and $G_{i+1}$. In addition, for $|i-j|\geq2$ the intersections between $\mathfrak{w}(G_i)$ and $\mathfrak{w}(G_j)$  are transverse, and these are not allowed to intersect at trivalent vertices.\\

\item[(iii)] In a neighborhood of $\R\times\{j\}\sse\R\times[1,2]$, $j=1,2$, the image $im(\mathfrak{w})$ is given by $l(\beta_j)$ vertical lines, such that the $k$th line belongs to $G_{\sigma^{(j)}_{i_k}}$, where $\sigma^{(j)}_{i_k}$ is the $k$th crossing of $\beta_j$.
\end{itemize}
\end{definition}

See Figure \ref{fig:WeaveExample} for two explicit examples with $n=2,3$. The image $im(\ww)$ of a weave $\ww$ is often referred to as a weave itself and denoted $\ww$, to ease notation. The intersection of a weave $\ww$ with a small neighborhood of $\R\times\{2\}$, resp.~of $\R\times\{1\}$, is said to be the top of the weave, resp.~its bottom.\\

\noindent Following \cite[Section 4]{CZ}, we also introduce a notion of {\em weave equivalence}, represented by the local moves in Figure \ref{fig:cz equivalence}. That is, by definition, two weaves $\ww_1,\ww_2$ are said to be (weave) equivalent if they differ by a sequence consisting of moves from Figure \ref{fig:cz equivalence}. See also \cite[Section 3.1]{CW} and \cite[Section 4.2]{CGGLSS}.

\begin{figure}[ht!]
\includegraphics[scale=0.7]{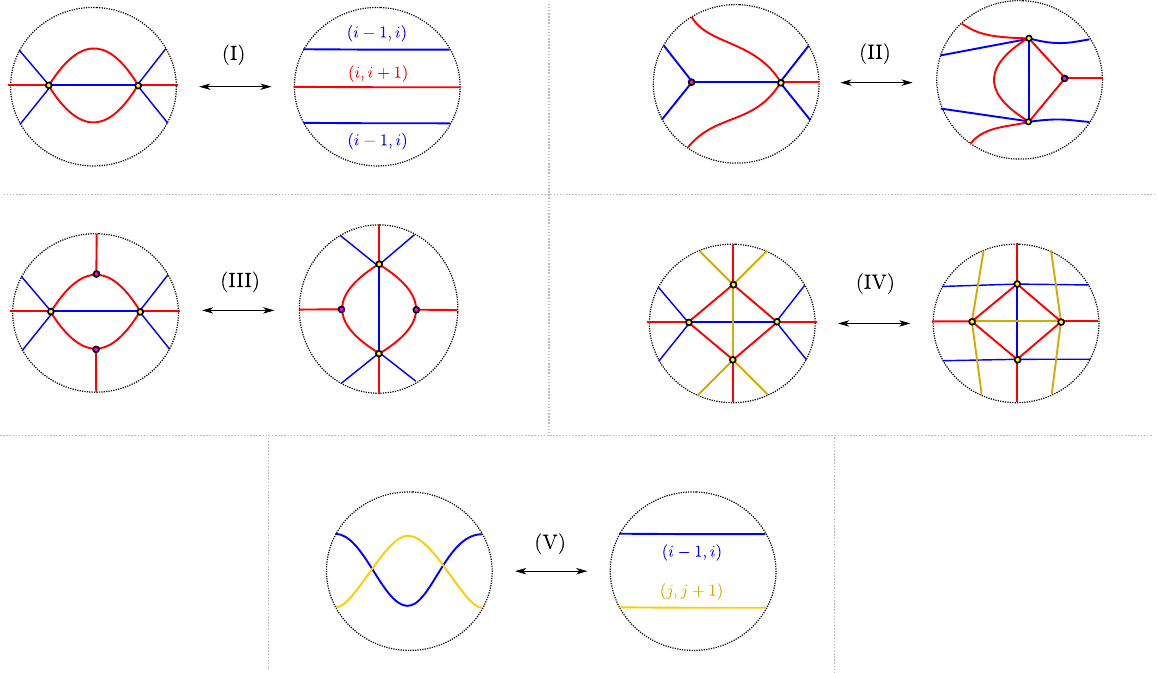}
\caption{Weave equivalences, after \cite[Theorem 1.1]{CZ}.}
\label{fig:cz equivalence}
\end{figure}

By definition, there is also an additional move called a {\em weave mutation}, after \cite[Section 4.8]{CZ}, which is {\it not} considered as an equivalence. Weave mutation is depicted in Figure \ref{fig:cz mutation}.

\begin{center}
\begin{figure}[ht!]

		\begin{tikzpicture}[baseline=0,scale=1.6]
			\draw [blue] (-0.5,0) -- (0.5,0);
			\foreach \i in {0,1}
			{
				\draw [blue] (45-\i*90:1) -- (0.5,0);
				\draw [blue] (135+\i*90:1) -- (-0.5,0);
			} 
		\end{tikzpicture} \quad\quad$\longleftrightarrow$ \quad \quad 
		\begin{tikzpicture}[baseline=0,scale=1.6]
			\draw [blue] (0,-0.5) -- (0,0.5);
			\foreach \i in {0,1}
			{
				\draw [blue] (135-\i*90:1) -- (0,0.5);
				\draw [blue] (-135+\i*90:1) -- (0,-0.5);
			} 
		\end{tikzpicture}\caption{Weave mutation, after \cite[Theorem 4.21]{CZ}. This is {\it not} an equivalence.}
\label{fig:cz mutation}
\end{figure}
	\end{center}

The definition of weaves and weave equivalence in \cite{CZ} are manifestly rotationally symmetric. In this paper we would like to break this symmetry by choosing a generic vertical direction 
and reading a weave top to bottom, allowing only certain local models to appear in such scanning. Similarly to Definition \ref{def:weave}.(iii) above, a generic  horizontal cross-section at the $j$th level of this type of weave is then a sequence of colored points in $\mathfrak{w}$ which we interpret as a braid word
$$\beta_j(\mathfrak{w})=s^{(j)}_{i_1}s^{(j)}_{i_2}\cdots s^{(j)}_{i_{\ell(\beta_j)}}\in \Br_n^+.
$$
This particular type of weave $\mathfrak{w}$ can then be understood as a ``movie" of different braid words:
$$\mathfrak{w}:=(\beta_0(\mathfrak{w})\to\beta_1(\mathfrak{w})\to\cdots\to \beta_{\ell(\mathfrak{w})}(\mathfrak{w})).$$
The initial braid word $\beta_0(\mathfrak{w})$ being read at the (top) horizontal cross-section $\R\times\{2\}$, and the last braid word $\beta_{\ell(\mathfrak{w})}(\mathfrak{w})$ is read at the (bottom) horizontal cross-section $\R\times\{1\}$. The number $\ell(\mathfrak{w})\in\N$ will be referred to as the length of the weave $\Sigma$.


\begin{definition}
Let $\beta_1,\beta_2$ be positive $n$-braid words. A weave $\mathfrak{w}$ of degree $n$ from $\beta_2$ to $\beta_1$ is said to be {\it sliced} if its cross-sections change top to bottom according to one of the following six situations, depicted in Figure \ref{fig:WeaveModels}:

\begin{center}
	\begin{figure}[h!]
		\centering
		\includegraphics[scale=0.7]{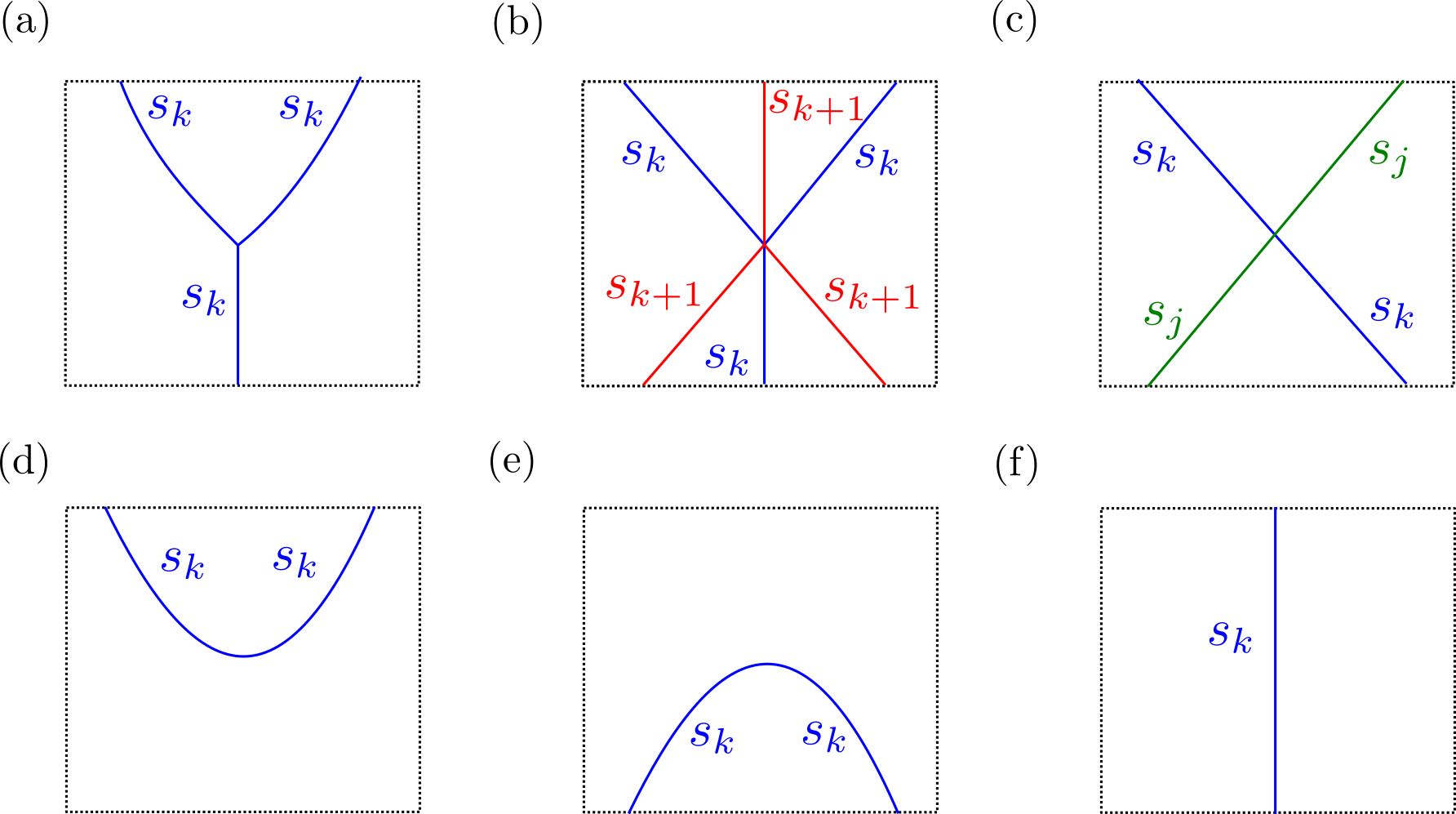}
		\caption{The six local models for sliced $n$-weave. In $(a), (c), (d), (e)$, and $(f)$, we have $j,k\in[1,n-1]$. In $(b)$, we have $k\in[1,n-2]$, and $|j-k|\geq2$. The inverse of the local model in $(b)$, with $s_{k+1}s_ks_{k+1}$ on top and $s_ks_{k+1}s_k$ on the bottom, is also allowed.}
		\label{fig:WeaveModels}
	\end{figure}
\end{center}

\begin{itemize}
	
	\item[(a)] Two consecutive edges labeled with the {\it same} transposition $s_k$ come together, and continue moving down as one unique edge, also labeled with $s_k$, $k\in[1,n-1]$. This is referred to as a {\it trivalent vertex}, and correspond to the model around (the image of) a trivalent vertex of the graph $G_k$ in Definition \ref{def:weave}. Algebraically, we represent this local model by $s_ks_k\to s_k$.\\

	\item[(b)] Three consecutive edges labeled by $s_k,s_{k+1},s_k$ come together, and continue moving down as three edges but now labeled $s_{k+1},s_{k},s_{k+1}$. This is referred to as a {\it hexavalent vertex}, and correspond to the model around an intersection point of the (images of the) graphs $\mathfrak{w}(G_k)\cap \mathfrak{w}(G_{k+1})$ in Definition \ref{def:weave}. Algebraically, we represent this local model by  $s_ks_{k+1}s_k\to s_{k+1}s_ks_{k+1}$. In addition, we also allow the same move, but reversed: $s_{k+1}s_ks_{k+1}\to s_ks_{k+1}s_k$, with $s_{k+1}s_ks_{k+1}$ on top and $s_ks_{k+1}s_k$ at the bottom.\\
	
	\item[(c)] Two consecutive edges labeled with two different transpositions $s_k,s_j$, with $|j-k|\geq2$, come together, and continue moving down as two edges, now labeled by $s_j,s_k$. This is referred to as a {\it 4-valent vertex}, and correspond to the model around a (transverse) intersection point of $\mathfrak{w}(G_k)$ and $\mathfrak{w}(G_j)$ in Definition \ref{def:weave}. Algebraically, we represent this local model by $s_ks_j\to s_js_k$.\\
	
	\item[(d)] Two consecutive edges labeled with the {\it same} transposition $s_k$ come together, merge and there is no edge continuing down. This is referred to as a {\it cup}, and we represent this local model by $s_ks_k\to 1$.\\
	
	\item[(e)] The inverse of the move in (d), where two consecutive edges are created as moving downwards from the empty set. This is referred to as a {\it cap}, and we represent this local model by $1\to s_ks_k$.\\
	
	\item[(f)] There is an edge labeled by $s_k$ and it continues moving down as the same edge labeled by $s_k$, i.e. nothing occurs. This local model is represented algebraically by $s_k\to s_k$.\\
\end{itemize}
By definition, we require that all 3-,4- and 6-vertices, cups and caps appear at different heights, and all horizontal tangencies are isolated. Note that 4-valent and hexavalent vertices represent (the Coxeter projection of the) braid relations. Finally, the following two are special types of sliced weaves that we use:
\begin{itemize}
	\item[(ii)] By definition, a {\it simplifying weave} is a sliced weave with no caps; thus the only allowed local models are (a),(b),(c),(d) and (f), not (e).\\
	
	\item[(ii)] By definition, a {\it Demazure weave} is a sliced weave with no cups nor caps; thus the only allowed local models are (a),(b),(c) and (f), not (d),(e).
\end{itemize}
\end{definition}

Note that a sliced weave is simplifying if and only if the length of a braid word is not increasing as we scan down the weave with horizontal cross-sections. The reasons behind the choice of the name Demazure weaves will be explained in Section \ref{section:Dem_prod}. Demazure weaves prominently feature in \cite{CGGLSS}. In this article, all weaves we discuss are sliced and thus from now onwards {\it weave} will refer to {\it sliced weave} unless otherwise indicated.

\begin{remark}
A cautious reader might have noticed that some local weave pictures are allowed by the general setup of \cite{CZ} but do not directly appear in our list (a)-(f) in Figure \ref{fig:WeaveModels}. The upside-down trivalent vertices, i.e. the horizontal flip of model (a), given by $s_k\to s_ks_k$, can be constructed using the above trivalent vertices and caps, see Section \ref{sec: nonstandard 3v}. Similarly, one may encounter a 6-valent vertex
with $a$ incoming and $(6-a)$ outgoing edges for any $0\le a\le 6$. All these can be modeled using the usual 6-valent vertices, cups and caps, possibly in several different ways. We declare all such weaves (fixing $a$ and the coloring of edges at the top) equivalent. The same applies to ``non-standard" 4-valent vertices, see sections \ref{sec: nonstandard 4v} and \ref{sec: nonstandard 6v} below for details.
\end{remark}

\begin{remark}
For context with \cite[Section 7.1.2]{CZ}, we note that Demazure weaves $\ww:\beta_2\to\beta_1$ are free, in that their fronts can be realized by embedded exact Lagrangian cobordisms from the Legendrian associated to $\beta_1$ to the Legendrian associated to $\beta_2$. Indeed, this is implied by the fact that the three models $(a),(b),(c)$ above are decomposable Lagrangian cobordisms. For $(b),(c)$ this follows from the fact that they are traces of Legendrian isotopies, e.g.~$(b)$ is the Lagrangian trace of the Legendrian Reidemeister III move. For $(a)$, this follows from the generic Legendrian perturbation of the $D_4^-$ front, as drawn in \cite[Figure 36]{CZ}, or by comparison to the pinching saddle cobordism, as established in \cite[Prop. 3.1]{Hughes2}. In particular, a Demazure weave $\ww:\beta\to\Delta$ yields a unique embedded exact Lagrangian filling of the Legendrian associated to $\beta\Delta$, as the Legendrian unlink, associated to $\Delta^2$, admits a unique embedded exact Lagrangian filling.
\end{remark}
 
 
\subsection{Equivalence of Demazure weaves}\label{sec: movies} Let us now introduce a series of situations, all representing an equivalence between two weaves $\ww_1,  \ww_2$ whose top and bottom ends coincide, i.e. $\beta_0(\ww_1)=\beta_0(\ww_2)$ and $\beta_{\ell(\ww_1)}(\ww_1)=\beta_{\ell(\ww_2)}(\ww_2)$. 
The majority of equivalences we describe compare two local models, and an equivalence between two different weaves $\ww_1, \ww_2$ will be obtained by applying several of the local equivalences listed here. We focus on Demazure weaves and their equivalences which only pass through Demazure weaves. In this section we translate the moves of Figure \ref{fig:cz equivalence} to our formalism.

\begin{remark}
A cautious reader may choose to call the equivalence relation in this section {\it Demazure equivalence}. In principle, it might be possible that two Demazure weaves are not equivalent through Demazure weaves, but are equivalent through the more general weave equivalences from Figure \ref{fig:cz equivalence}. We have not investigated this problem. This fine point is irrelevant for the results of this paper, and we use the same notion of equivalence to simplify the exposition.
\end{remark}

\begin{remark}\label{rmk: horizontal ref demazure}
In what follows, we also require that the horizontal reflection of each of the upcoming relations explained in \ref{sec: height}--\ref{sec: Zam} below is also a relation.
\end{remark}


\subsubsection{Changing the height of vertices.}\label{sec: height} We allow to change relative heights of any pair of crossings in a weave provided that they are not connected by an edge and there are no crossings between them. (This is commonly called the interchange law in the context of 2-categories.)

\subsubsection{Canceling pairs of 4- and 6-valent vertices}\label{sec: cancellation} The following weaves are declared to be equivalent:
\begin{center}
\includegraphics[scale=0.3]{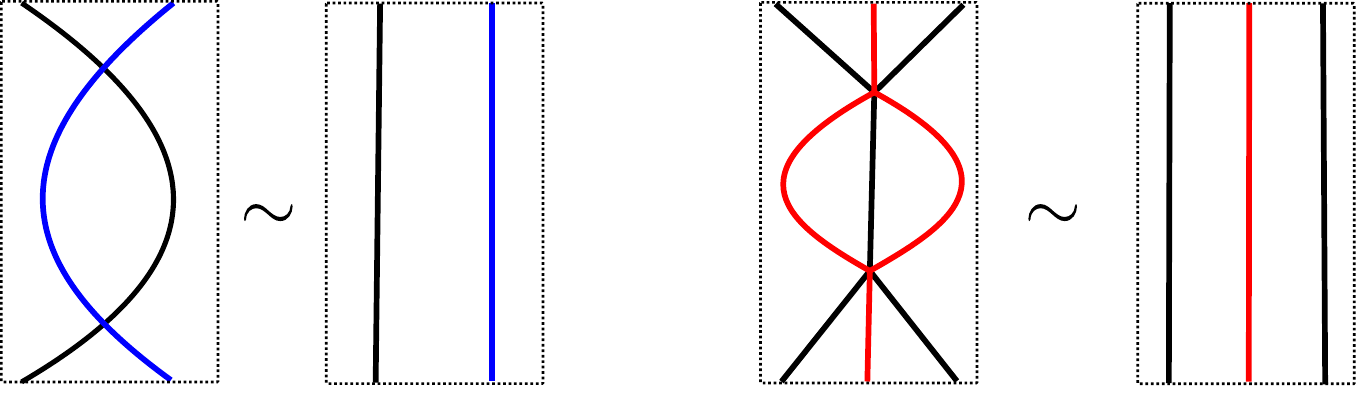}
\end{center}
This corresponds to moves (I) and (V) from Figure \ref{fig:cz equivalence}.
From the algebraic perspective, i.e. studying the braids in the horizontal cross-sections, this is the diagrammatic incarnation of the fact that the two moves $s_ks_{k+1}s_k\to s_{k+1}s_ks_{k+1}$ and $s_{k+1}s_ks_{k+1}\to s_ks_{k+1}s_k$, and the two moves $s_is_j\to s_js_i$ and $s_js_i\to s_is_j$, $|i-j|\geq2$, are inverse to each other. That is, performing a Reidemeister III move and then its inverse is considered to be (equivalent to) the trivial weave. Similarly, performing a commutation move in the braid group, and then the same move in reverse, is also considered to be (equivalent to) the trivial weave. In the notation above, we are declaring the weave $\ww_1=s_{k+1}s_ks_{k+1}\to s_ks_{k+1}s_k\to s_{k+1}s_ks_{k+1}$ to be equivalent to the constant weave $\ww_2=s_{k+1}s_ks_{k+1}$, and the weave $\ww_1=s_{i}s_j\to s_js_i\to s_is_j$ to be equivalent to the constant weave $\ww_2=s_is_j$.


\subsubsection{Commutation with distant colors}\label{sec: distant} We declare that an edge of the weave labeled with a color (i.e. a transposition) which is distant to the rest of the colors at a given vertex can be moved past this vertex. That is, we declare that the following weaves are equivalent:

\begin{center}
\includegraphics[scale=0.2]{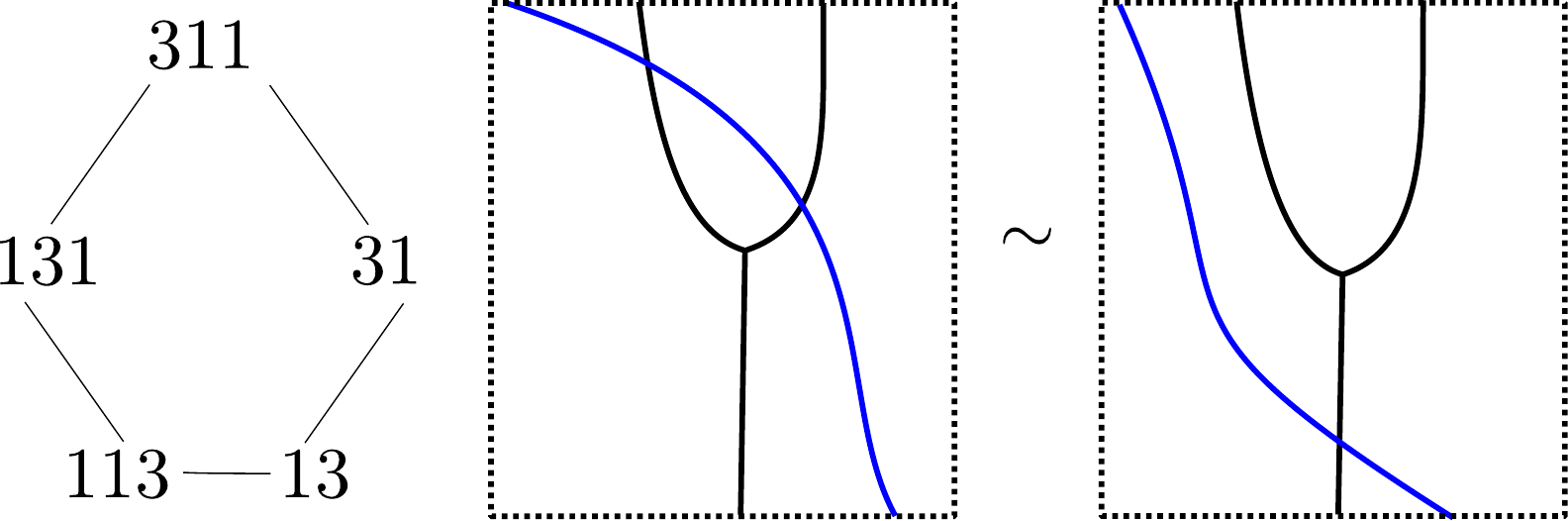}
\end{center}

\begin{center}
\includegraphics[scale=0.2]{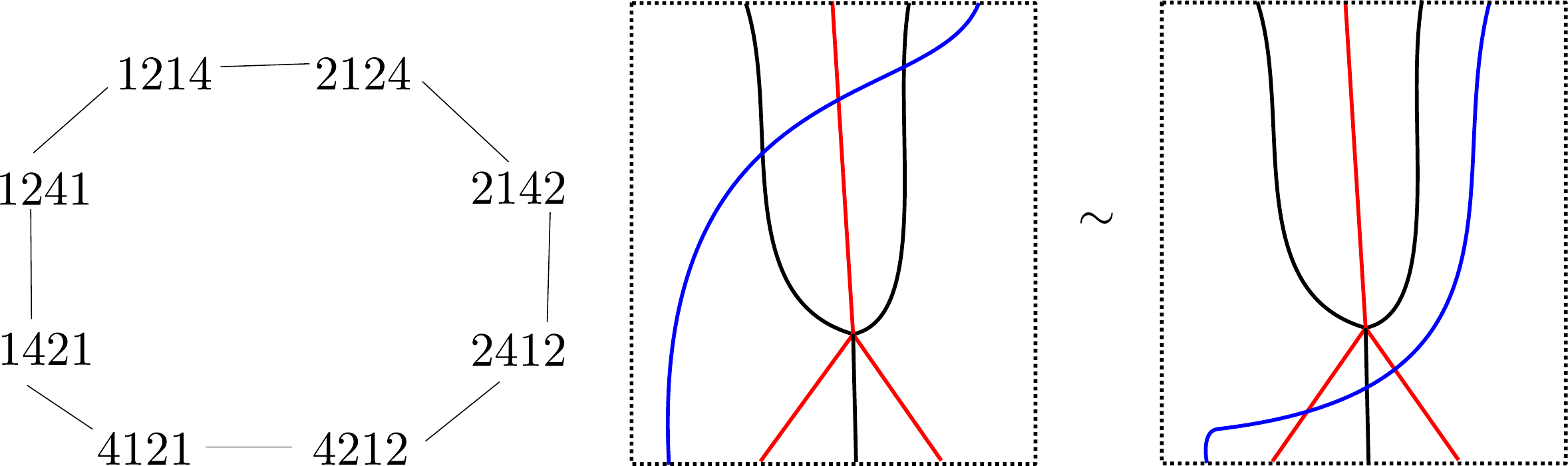}
\end{center}

Similarly, we declare that three lines with pairwise distant colors can be rearranged according to the weave equivalence depicted above. As illustrated in the equivalences above, the particular sequences of braid moves that we are declaring to be equivalent are read from taking horizontal cross-sections in the above diagrams; we will thus not necessarily indicate them any longer.


\subsubsection{1212- and 2121-relations}
\label{sec: 1212}

We require that the following two ways of getting from $\sigma_1\sigma_2\sigma_1\sigma_2$, denoted 1212 for simplicity, to $\sigma_1\sigma_2\sigma_1$, i.e. 121, are equivalent:

\begin{center}
\includegraphics[scale=0.3]{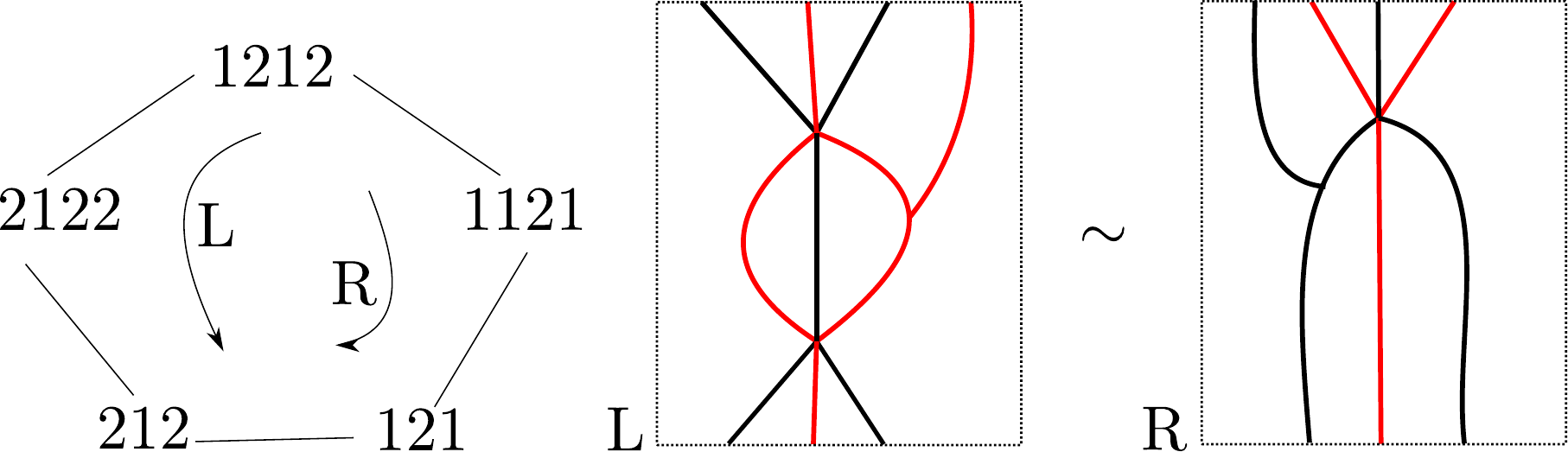}
\end{center}
This corresponds to the move (II) from Figure \ref{fig:cz equivalence}.

We also impose equivalences for other interpretations of the move (II) from Figure \ref{fig:cz equivalence} using Demazure weaves, corresponding to other paths around the pentagon on the left of the above figure. Namely, we require that the two ways of getting from 1121 to 212 are equivalent, and that the two ways of getting from 2122 to 121 are equivalent, and so forth. The equivalence of the two ways of getting from 1121 to 212 corresponds to the equivalence of the following two simplifying weaves:

\begin{center}
\includegraphics[scale=0.3]{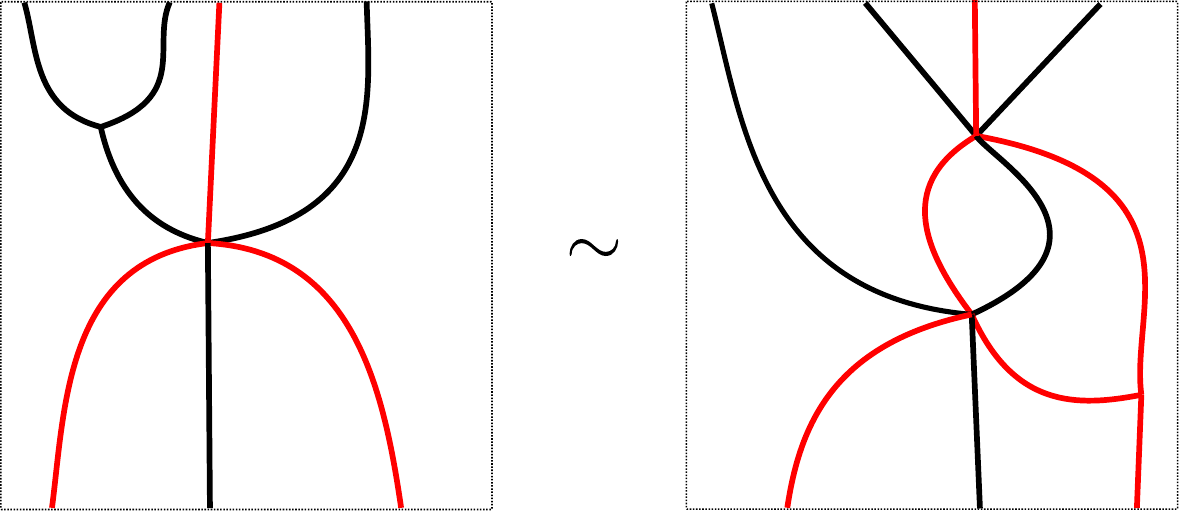}
\end{center}

We also require that the two ways of getting from 1211 to 212 are equivalent, which is the same as requiring that the two ways of getting from 2121 to 212 are equivalent and so on. The weaves are obtained from the ones above by the symmetry along the vertical line:

\begin{center}
\includegraphics[scale=0.3]{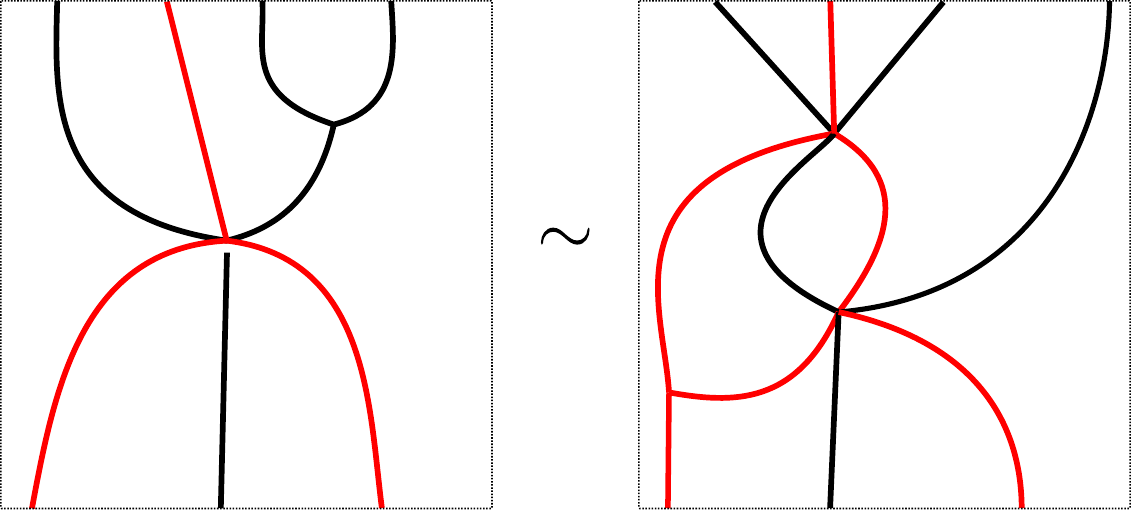}
\end{center}

There are also similar relations where adjacent colors are interchanged (e.g. black and red), and for any pair of adjacent colors, which we do not draw here.
\subsubsection{Cycles for 12121}
\label{sec: 12121}

As an example for the previous relation, we observe that there are many paths in the Demazure graph from 12121 to 212, related by consecutive application of the 1212-relation:

\begin{center}
\includegraphics[scale=0.2]{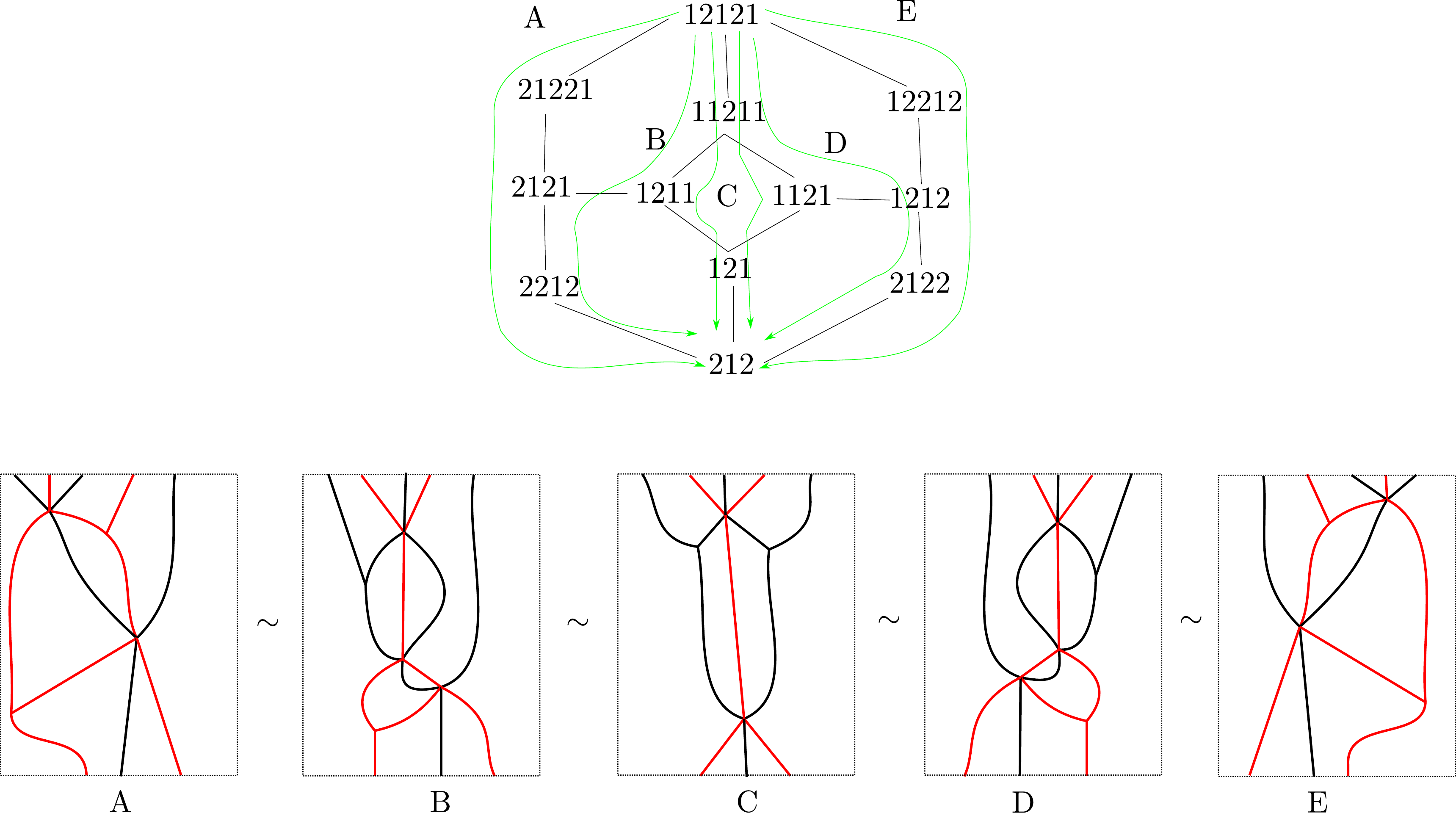}
\end{center}

$$
(A)\ 12121\to 21221\to 2121\to 2212\to 212
$$
$$
(B)\ 12121\to 11211\to 1211\to 2121\to 2112\to 212
$$
$$
(C)\ 12121\to 11211\to 1211\to 121\to 212 \sim 12121\to 11211\to 1121\to 121\to 212
$$
$$
(D)\ 12121\to 11211\to 1121\to 1212\to 2122\to 212
$$
$$
(E)\ 12121\to 12212\to 1212\to 2122\to 212
$$
Note that the equivalence between (A) and (E) corresponds to the move (III) from Figure \ref{fig:cz equivalence}.


\subsubsection{Zamolodchikov relation}
\label{sec: Zam}

Diagrammatically, the Zamolodchikov relation is the equivalence of the following diagrams, relating various braid words for the longest element $w_0\in S_4$:

\begin{center}
\includegraphics[scale=0.2]{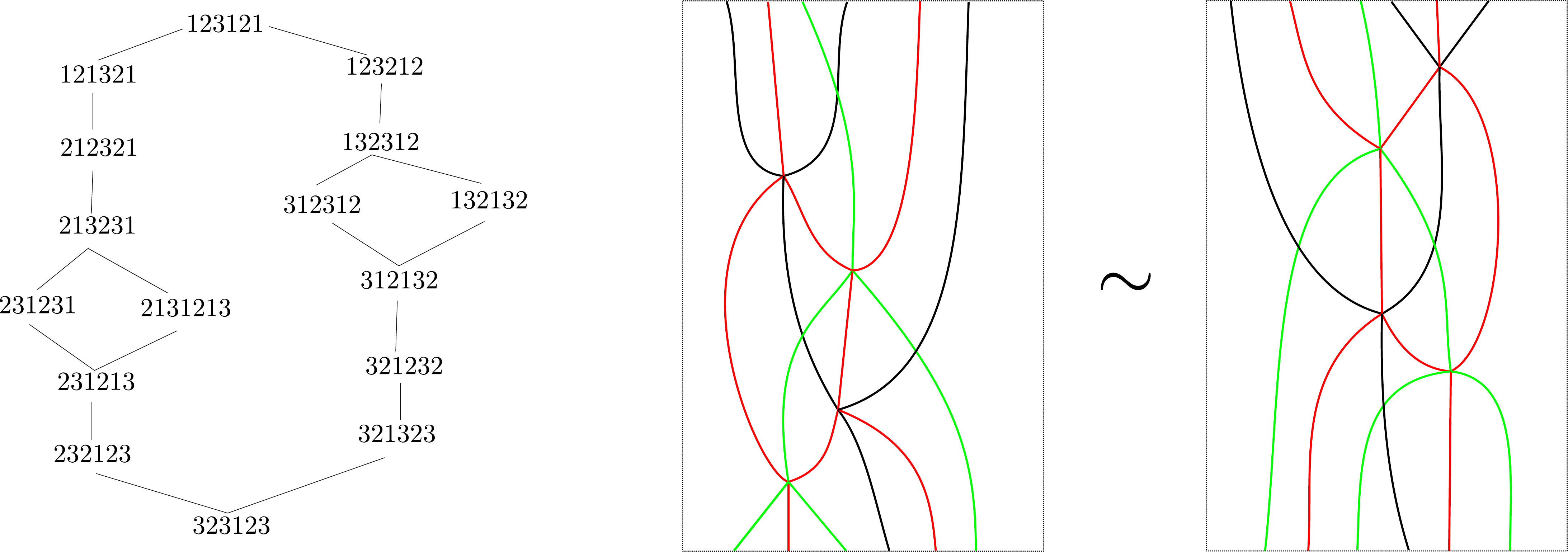}
\end{center}

This corresponds to the move (IV) from Figure \ref{fig:cz equivalence}.


\subsubsection{Mutations}
\label{sec: mutation}

In contrast with Soergel calculus, we do not declare the two ways of getting from $s_is_is_i$ to $s_i$ via $s_is_i$ to be equivalent. They are related by the following special type of move, which we call a weave {\em mutation}:

\begin{center}
\includegraphics[scale=0.2]{mutation.pdf}
\end{center}

This concludes the list of diagrammatic equivalences (\ref{sec: isotopies})-(\ref{sec: Zam}), and the mutation non-equivalence (\ref{sec: mutation}).


\subsection{Equivalence of simplifying weaves}
\label{sec: cup cap equivalence}

In this section we define an equivalence relation for simplifying weaves. The complete list of equivalences includes:
\begin{enumerate}
\item All of the equivalences for Demazure weaves from Section \ref{sec: movies}.
\item Changing the relative height of cups and vertices, see Section \ref{sec:height ii}.
\item Additional moves with cups listed in sections \ref{sec: nonstandard 3v}(a), \ref{sec: nonstandard 6v}(a) and \ref{sec: nonstandard 4v}(a). 
\end{enumerate}

\begin{remark}
One can check that the additional equivalence relations for simplifying weaves do not change the total number of cups. Therefore, two Demazure weaves are equivalent through simplifying weaves if and only if they are equivalent through Demazure weaves.
\end{remark}

By reflecting these relations along a horizontal axis, we get a similar notion of equivalence for weaves with caps, see \ref{sec: nonstandard 3v}(b), \ref{sec: nonstandard 6v}(b) and \ref{sec: nonstandard 4v}(b).  There is an additional ``zig-zag" move in section \ref{sec: isotopies} which allows one to create or delete a cup and a cap simultaneously.  

\begin{remark}
Just as in Remark \ref{rmk: horizontal ref demazure}, we will require that the horizontal reflection of the relations  \ref{sec:height ii}--\ref{sec: nonstandard 4v} are also relations.
\end{remark}

\subsubsection{Changing the height of vertices II}\label{sec:height ii} We allow to change relative heights of any  crossing with a cup or cap in a weave, provided that they are not connected by an edge and there are no crossings, cups, or caps between them.

\subsubsection{Non-standard trivalent vertices}
\label{sec: nonstandard 3v}

\begin{itemize}
\item[(a)] We can consider a trivalent vertex with 3 inputs and 0 outputs,
 defined by either of the pictures:
\begin{center}
\includegraphics[scale=0.4]{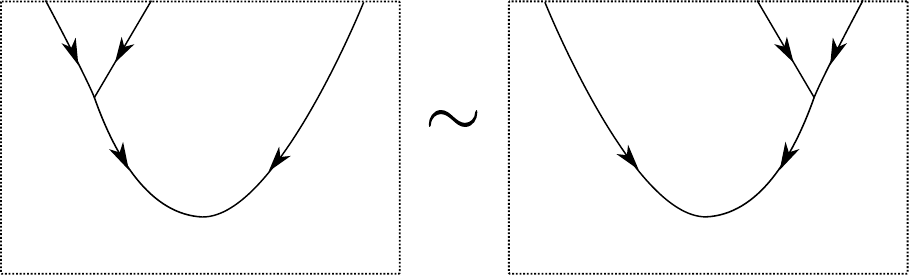}
\end{center}
We require that the two pictures are equivalent, note that both weaves are simplifying.\\

\item[(b)] We can define an upside-down trivalent vertex in the following ways which are required to be equivalent:
\begin{center}
\includegraphics[scale=0.3]{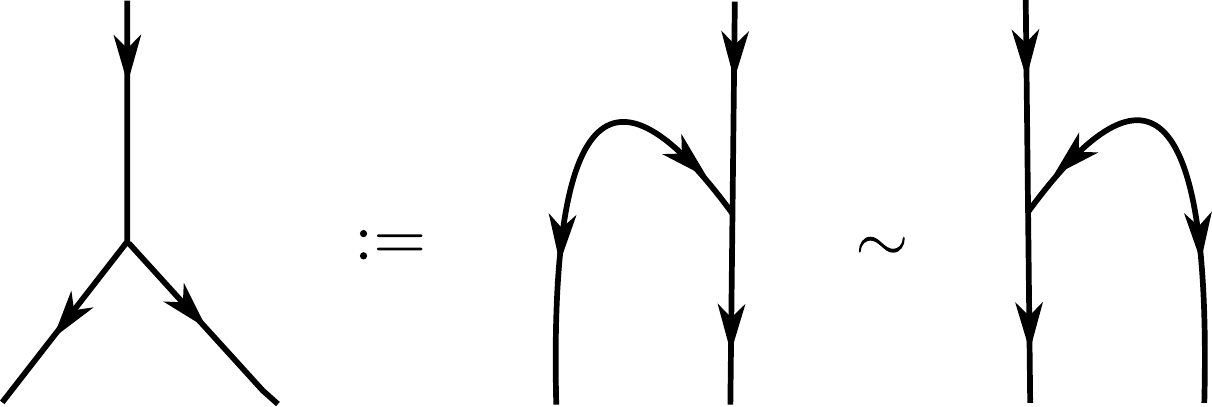}
\end{center}
Since both weaves include a cap, these are not simplifying, and we do not allow upside-down trivalent vertices in simplifying weaves.

A horizontal reflection of these relations would express a "standard" trivalent vertex using an upside-down trivalent vertex and a cup. There is a similar picture to \ref{sec: nonstandard 3v}(a) with 0 inputs and 3 outputs.  
\end{itemize}


\subsubsection{Non-standard 6-valent vertices}
\label{sec: nonstandard 6v}

The following relations correspond to different ways to a look at a single 6-valent vertex. We require that all of them are equivalent, provided that the numbers of inputs and outputs are fixed.

\begin{itemize}
    \item[(a)] We illustrate simplifying weaves with 4 inputs and 2 outputs, and 5 inputs and 1 output, defined using standard 6-valent vertex and cups. The former in fact implies the latter. 
\begin{center}
\includegraphics[scale=0.25]{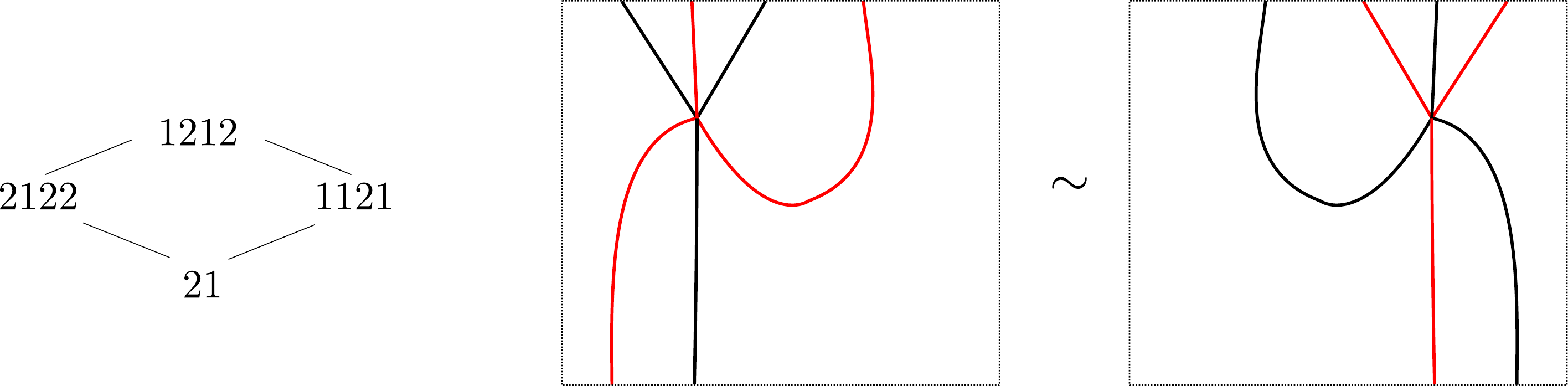}
\end{center}

\begin{center}
\includegraphics[scale=0.2]{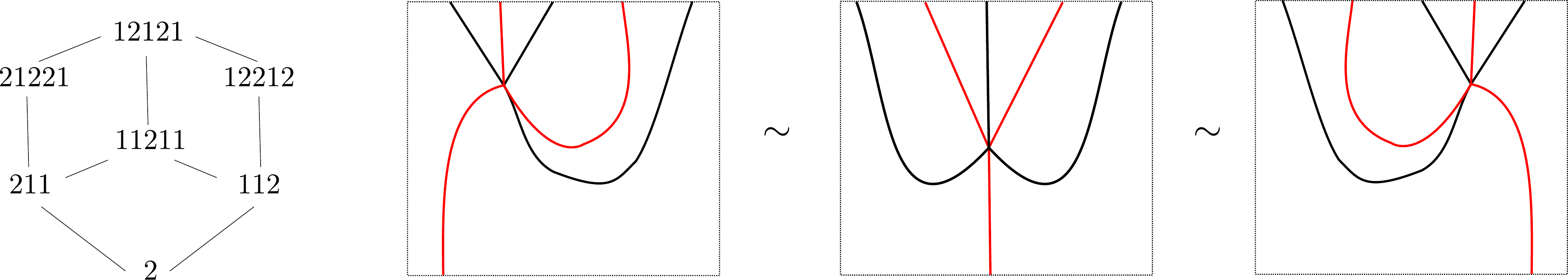}
\end{center}

Finally, we can define a 6-valent vertex with 6 inputs and 0 outputs, and the following diagram shows that all ways to do so are equivalent:
\begin{center}
\begin{tikzpicture}
\draw (0,4) node {121212};
\draw (-3,3) node {212212};
\draw (-1,3) node {112112};
\draw (1,3) node {122122};
\draw (3,3) node {121121};
\draw (-2,2) node {2112};
\draw (0,2) node {1122};
\draw (2,2) node {1221};
\draw (-1,1) node {22};
\draw (1,1) node {11};
\draw (0,0) node {$\emptyset$};
\draw (-1,0.8)--(-0.2,0.2);
\draw (1,0.8)--(0.2,0.2);
\draw (-2,1.8)--(-1.2,1.2);
\draw (-0.2,1.8)--(-0.8,1.2);
\draw (0.2,1.8)--(0.8,1.2);
\draw (2,1.8)--(1.2,1.2);
\draw (-3,2.8)--(-2.2,2.2);
\draw (-1.2,2.8)--(-1.8,2.2);
\draw (-0.8,2.8)--(-0.2,2.2);
\draw (0.8,2.8)--(0.2,2.2);
\draw (1.2,2.8)--(1.8,2.2);
\draw (3,2.8)--(2.2,2.2);
\draw (-0.2,3.8)--(-0.8,3.2);
\draw (0.2,3.8)--(0.8,3.2);
\draw (-0.5,4)--(-3,3.2);
\draw (0.5,4)--(3,3.2);

\end{tikzpicture}
\end{center}

\item[(b)] The symmetric pictures with 2 inputs and 4 outputs, 1 input and 5 outputs, and 0 inputs and 6 outputs, are obtained by reflection in the horizontal axis and using caps. These are not simplifying.
\end{itemize}

\subsubsection{Non-standard 4-valent vertices}
\label{sec: nonstandard 4v}
\begin{itemize}
    \item[(a)] Similarly, we can define  non-standard 4-valent vertices by simplifying weaves with 3 inputs and 1 output, or 4 inputs and 0 outputs.

\begin{center}
\includegraphics[scale=0.35]{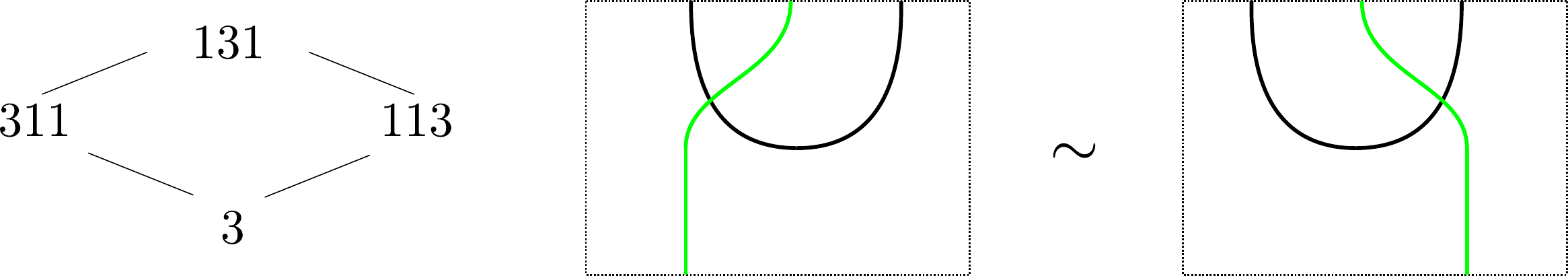}
\end{center}

    \item[(b)] By reflecting these, we get weaves  with 1 input and 3 outputs (or 0 inputs and 4 outputs) which use caps. These are not simplifying.
\end{itemize}

\subsubsection{Planar isotopies}\label{sec: isotopies} A weave in the plane is, in particular, a planar diagram. We declare planar isotopic diagrams to define equivalent weaves. In particular, we need to require the following zigzag relation with canceling pairs of caps and cups:

\begin{center}
\includegraphics[scale=0.4]{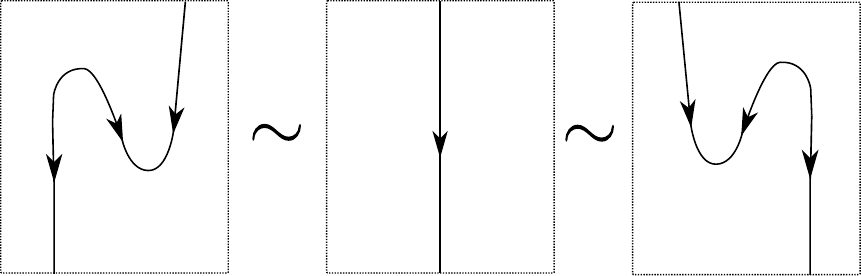}
\end{center}

Algebraically, the weave $\ww_1=s_k\to s_ks_ks_k\to s_k\cdot 1=s_k$, where first a cap creates $1\to(s_ks_k)$ to the left of the initial $s_k$, and then a cup erases the rightmost to $s_k$ via $(s_ks_k)\to1$, is equivalent to the constant weave $\ww_2=s_k$.

\begin{prop}
\label{prop: isotopy}
Assume the zigzag relation and the equivalence relations \ref{sec: nonstandard 3v}-\ref{sec: nonstandard 4v}. Then any two planar isotopic weaves are equivalent.
\end{prop}

\begin{proof}
By \cite[Proposition 3.2]{EK} it is sufficient to prove that every vertex is cyclic, that is, invariant under the 360 degree rotation. For a trivalent vertex, we use the definition of the upside down trivalent vertex and relations \ref{sec: nonstandard 3v} to show that a 60 degree rotation changes either of the trivalent vertices to another one of the same type, e.g.:
\begin{center}
\includegraphics[scale=0.7]{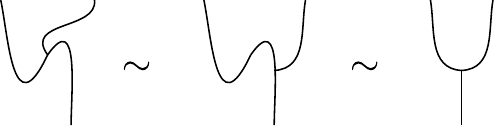}
\end{center}
This implies that a trivalent vertex is invariant under 120 degrees rotation, and hence invariant under 360 degree rotation. Similarly, we can use the relations \ref{sec: nonstandard 6v} to show that the 6-valent vertex is invariant under rotation by 60 degrees, and hence by 360 degrees:
\begin{center}
\includegraphics[scale=0.5]{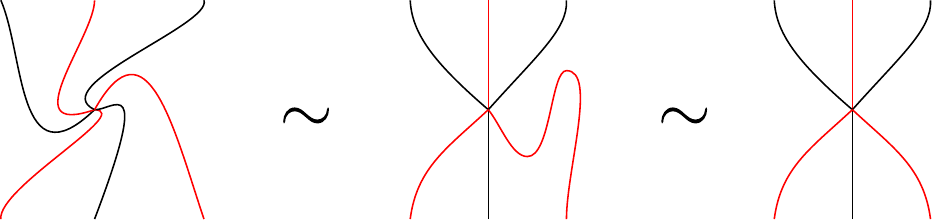}
\end{center}
The proof for a 4-valent vertex is similar. We refer to \cite[Section 3]{EK} and references therein for more details on cyclicity and isotopy invariance.
\end{proof}

\subsubsection{Rotational invariance}

We expect that, similarly to the proof of Proposition \ref{prop: isotopy} and the results of \cite{EK}, the equivalence relations above are rotationally invariant. That is, any rotation of an equivalence relation follows from the relations. We do not need and do not prove it here, but give a couple of examples which illustrate this point:


\begin{center}
\includegraphics[scale=0.25]{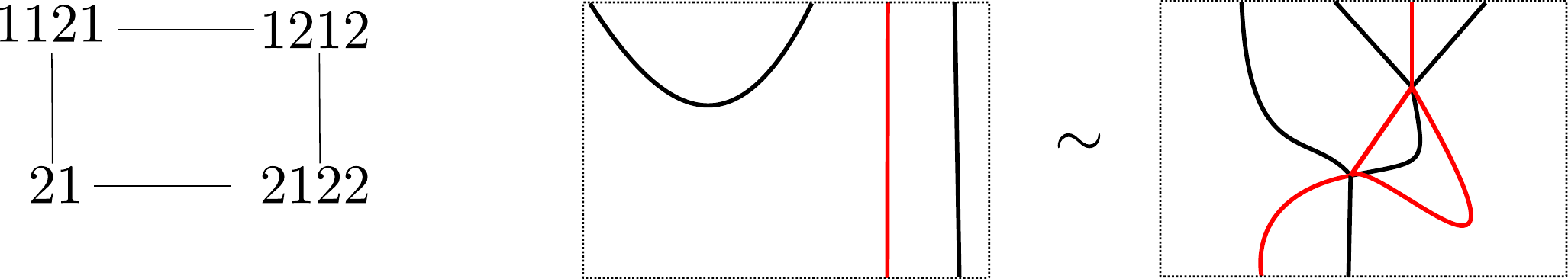}
\end{center}
\begin{center}
\includegraphics[scale=0.25]{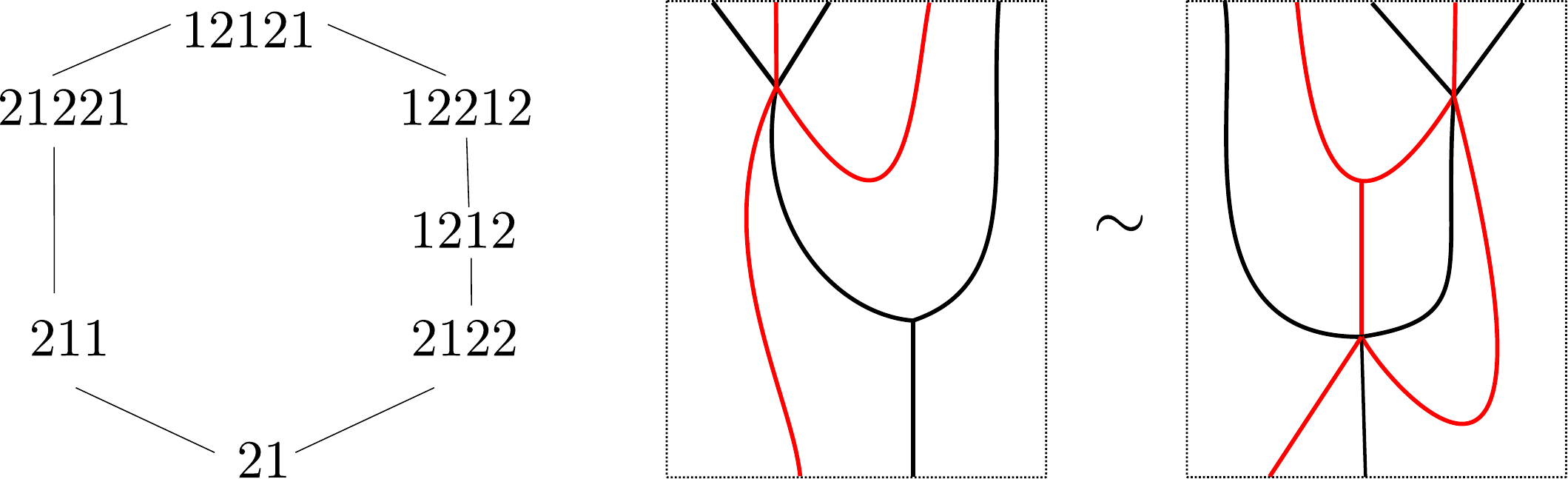}
\end{center}

Note that these two pictures, as planar graphs, are similar to the ones that we already considered (cancellation $121\to 212\to 121$ and $1212$-move), but in this case they are drawn differently. 

Both of these are actually consequences of the above relations. The former follows from the equivalence  \ref{sec: nonstandard 6v}. For the latter, we can consider the diagram:

\begin{center}
\begin{tikzpicture}
 \draw (0,6) node {12121};
 \draw (-2,4) node {21221};
 \draw (0,4) node {11211};
 \draw (2,4) node {12212};
 \draw (-2,2) node {211};
 \draw (0,2) node {1121};
 \draw (2,2) node {1212};
 \draw (0,0) node {21};
 \draw (2,0) node {2122};
 \draw (-0.2,5.8)--(-2,4.2);
 \draw (0,5.8)--(0,4.2);
 \draw (0.2,5.8)--(2,4.2);
 \draw (-2,3.8)--(-2,2.2);
 \draw (2,3.8)--(2,2.2);
 \draw (0,3.8)--(0,2.2);
 \draw (-0.2,3.8)--(-1.8,2.2);
 \draw (-2,1.8)--(-0.2,0.2);
 \draw (0,1.8)--(0,0.2);
 \draw (2,1.8)--(2,0.2);
 \draw (0.5,2)--(1.5,2);
 \draw (0.5,0)--(1.5,0);
\end{tikzpicture}
\end{center}

All cycles in this diagram are covered by the above relations.


\subsection{Demazure product and Demazure weaves} \label{section:Dem_prod}

We will use the notion of Demazure product of a word (equivalently, of an expression) in the alphabet of simple reflections $\left\{s_i\right\}$. This terminology is introduced in \cite{KM2}, but the notion goes back at least to \cite{Demazure}. We refer the reader to \cite[Section 2.2]{DHM} for a detailed discussion on this notion and its relation to 0-Hecke algebras over $\F_2$.

The \emph{Demazure product} of a word $Q = s_{i_1} s_{i_1} \ldots s_{i_l(\beta)}$, denoted by $\delta(Q)$, is the largest element of $S_n$ in the Bruhat order, such that $Q$ contains some reduced expression of this element as a subword. 
This element is well-defined. It admits an equivalent inductive definition by the folllowing rule:

\[
\delta(s_i):=s_i,\ \delta(Qs_i):=\begin{cases}
\delta(Q)s_i & \text{if}\  \ell(\delta(Q)s_i)=\ell(\delta(Q))+1\\
\delta(Q) & \text{if}\ \ell(\delta(Q)s_i)=\ell(\delta(Q))-1.\\
\end{cases}
\]

It can be verified from either definition that for any $Q, Q'$ we have
\begin{equation} \label{eq:Demazure_moves}
\begin{split}
\delta(Q s_i s_i Q') = \delta(Q s_i Q');\\
\delta(Q s_i s_{i+1} s_i Q') = \delta(Q s_{i+1} s_i s_{i+1} Q');\\
\delta(Q s_i s_j Q') = \delta(Q s_j s_i Q'), \, \ |i-j|\ge 2. 
\end{split}
\end{equation}

Given two permutations $u,v\in S_n$, we define their star product $u\star v$ as the Demazure product of the concatenation of an arbitrary reduced expression of $u$ and an arbitrary reduced expression of $v$. By construction, we have
$$
(u\star v)\star w=\delta(uvw)=u\star (v\star w).
$$

By the Demazure product of an element $\beta$ of the positive braid monoid we mean \[\delta(\beta) := \delta(s_{i_1} s_{i_2} \ldots s_{i_l(\beta)})
\] 
for a positive braid word $\sigma_{i_1} \sigma_{i_2} \ldots \sigma_{i_l(\beta)}$ for $\beta$. This is well-defined: by equations \eqref{eq:Demazure_moves}, $\delta(\beta)$ does not depend on the choice of positive braid word.

\begin{ex} \label{ex:nilHecke} The inductive definition of the Demazure product of a word, and the definition of the star product of permutations imply that
\[
w\star s_i = \begin{cases} ws_i & \mathrm{if} \; \ell(ws_i) = \ell(w) + 1, \\ w & \mathrm{if} \; \ell(ws_i) = \ell(w) - 1. \end{cases}
\]
Note that the Demazure power of a simple transposition is simply $s_i \star s_i \star \ldots \star s_i = s_i$ for any number of multiples. Note also that for any $w \in S_n$ we have the equality $w \star w_0 = w_0 \star w = w_0$. 
\end{ex}

In fact, the Demazure product is the product in a monoid known as \emph{Coxeter monoid} \cite{T}, \emph{$0-$Hecke monoid} \cite{FG, HST},  \emph{Coxeter $\ast$-monoid}, or \emph{Richardson-Springer monoid} in the literature. Richardson and Springer studied its action on the set of orbits of the flag variety under the action of the fixed point subgroup of an involution on the algebraic group \cite{RS1, RS2}. Norton \cite{Norton} constructed a bijection between the set $S_n$ and the underlying set of this monoid. As the examples above show, the multiplication in the monoid is quite different from the one in the permutation group. More generally, given a positive braid $\beta$, $\delta(\beta)$ does not coincide with the image of $\beta$ under the canonical surjection onto $S_n$. A first relation to the weaves introduced above is given in the following lemma.

\begin{lemma} \label{lemma:demweaves}
Let $\ww$ be a Demazure weave. Then the Demazure product of the associated braid words $\beta_j(\ww)$, $j\in[0,l
(\ww)]$, remains unchanged, i.e. $\delta(\beta_0(\ww))=\delta(\beta_j(\ww))$.
\end{lemma}

\begin{proof}
Equations \eqref{eq:Demazure_moves} imply that $3$-, $6$- and $4$-valent vertices preserve the Demazure product. 
\end{proof}

Lemma \ref{lemma:demweaves} shows that Demazure weaves provide a transparent diagrammatic interpretation of the Demazure product and of the $0-$Hecke monoid. This motivated our nomenclature.


\subsection{Classification of weaves}

We call two  weaves {\em equivalent} if they are related by a sequence of elementary equivalence moves from Section \ref{sec: movies} (with no mutations), and {\em mutation equivalent} if they are related by a sequence that might involve both equivalence moves and mutations.

\begin{thm}
\label{thm: ben}
$($a$)$ Let $\ww_1,\ww_2$ be two weaves such that the source braids of $\ww_1,\ww_2$ coincide and the target braids of $\ww_1,\ww_2$ coincide. If $\ww_1,\ww_2$ only have 6- and 4- valent vertices, then $\ww_1,\ww_2$ are equivalent.

$($b$)$ Let $\ww_1,\ww_2$ be two {\em Demazure}  weaves such that the source braids of $\ww_1,\ww_2$ coincide and the target braids of $\ww_1,\ww_2$ coincide. If the target is reduced, then $\ww_1,\ww_2$ are mutation equivalent.
\end{thm}

\begin{proof}
The theorem follows from the main result of \cite{Elias}, which we briefly recall. For part (a), consider the graph where vertices correspond to braid words and edges to braid moves (that is, 6- or 4-valent vertices). Then the cycles in this graph are generated by commutation with distant colors and Zamolodchikov relations, hence any two paths in this graph are equivalent.

For (b), consider the Hecke-type algebra with generators $T_i$ and relations 
$$
T_i^2=\alpha T_i+\beta,\ T_iT_{i+1}T_{i}=T_{i+1}T_{i}T_{i+1}+\text{lower order terms},\ T_iT_j=T_jT_i, (|i-j|>1).
$$
Using these relations, it can be verified that every product of $T_i$ can be written as a linear combination of reduced expressions, possibly in a non-unique way. This non-uniqueness appears from {\em ambiguities}: applying the relations in different order could yield different results.

B. Elias proved in \cite[Proposition 5.5]{Elias} that (modulo commutation with distant colors) there are exactly 5 types of potential ambiguities that one needs to consider: $iii,ii(i+1)i, i(i+1)ii, i(i+1)i(i+1)i,i(i+1)(i+2)i(i+1)i$, which are nothing but the trivial move, the 5-cycles corresponding to 1121 and 1211 from \ref{sec: 1212}, the cycle from  \ref{sec: 12121} for the word 12121, and the Zamolodchikov relation. Note that the ambiguity $iii$ corresponds to different ways of getting from $iii$ to $i$. There are two such ways without cups, and they are related by the mutation from Section  \ref{sec: mutation}.
\end{proof}

\begin{remark}
\label{rem: lol}
The assumption in (b) that the target is reduced is important. Indeed, the two trivalent vertices $(ss)s\to ss$ and $s(ss)\to ss$ are neither equivalent nor mutation equivalent.
\end{remark}

\begin{remark}
\label{rem: equivalent weaves}
Note that by Theorem \ref{thm: ben}(a), any two simplifying weaves relating two positive braid words for the same braid are equivalent. Thus, we will oftentimes not specify such a weave.
\end{remark}

Let us continue studying conditions for equivalences. 
Suppose that a positive braid word $\beta$ contains a piece $s_ius_j$. Following \cite{Hersh}, the pair of crossings $(s_i,s_j)$ is said to be a {\em deletion pair} if $s_iu=us_j$. (Note that unlike \cite{Hersh}, we do not require $u$ to be a reduced word.) Let us define a relation $\prec$ on the crossings of the braid $\beta$ according to $s_{i} \prec s_{j}$ if $(s_{i}, s_{j})$ form a deletion pair. The following two lemmas are used in the proof of the criterion Theorem \ref{thm: one 3v} below.

\begin{lemma}
\begin{itemize}
\item[(i)] The relation $\prec$ is a partial order on the set of crossings of $\beta$.
\item[(ii)] The set of crossings of $\beta$ is a disjoint union of linearly ordered sets.
\end{itemize}
\end{lemma}

\begin{proof}
Assume that we have a piece of a braid $s_ius_jvs_k$ and $(s_i,s_j)$ and $(s_j,s_k)$ are deletion pairs so that $s_iu=us_j, s_jv=vs_k$. Then
$$
s_i\cdot us_jv=us_js_jv=us_jv\cdot s_k,
$$
and $(s_i,s_k)$ is a deletion pair. This proves (i). To prove (ii), assume  $(s_i,s_j)$ and $(s_i,s_k)$ are deletion pairs, and assume wlog that $s_{j}$ is to the left of $s_{k}$. We must show that $(s_{j}, s_{k})$ is a deletion pair. We have 
$$
s_iu=us_j,\ s_ius_jv=us_jvs_k,
$$
and $s_ius_jv=s_iuvs_k$. Hence $s_jv=vs_k$, and $(s_j,s_k)$ is a deletion pair. The case when $(s_i,s_k)$ and $(s_j,s_k)$ are deletion pairs is analogous.
\end{proof}

We call a deletion pair $(s_i,s_j)$ {\em close} if $s_{i} \prec s_{j}$ is a cover relation, i.e. no crossing in-between $s_{i}$ and $s_{j}$ forms a deletion pair with $s_{i}$ or $s_{j}$. 

\begin{lemma}
\label{lem: deletion pairs}
Suppose that $(s_i,s_j)$ is a close deletion pair, then the following Demazure weaves are equivalent:
\begin{equation}
\label{eq: u deletion pair}
s_ius_j\to s_is_iu\to s_iu\to us_j\ \sim \ s_ius_j\to us_js_j\to us_j.
\end{equation}
\end{lemma}
Note that the condition that the deletion pair is close is necessary, see Remark \ref{rem: lol}.
\begin{proof}
First, note that by Theorem \ref{thm: ben}(a) we can choose a sequence of braid relations relating $s_iu$ and $us_j$ arbitrarily, and all such weaves would be equivalent. Furthermore, we can choose a braid word for $u$ arbitrarily. Indeed, if $u'$ is related to $u$ by braid relations then we get the diagram:
\begin{center}
\begin{tikzcd}
 & & s_ius_j \arrow{dll} \arrow{d} \arrow{drr}& & \\
s_is_iu \arrow{ddd} \arrow{dr} & & s_iu's_j \arrow{dl} \arrow{dr} & & us_js_j \arrow{ddd} \arrow{dl}\\
 & s_is_iu' \arrow{d} & & u's_js_j \arrow{d} & \\
 & s_iu' \arrow{rr} \arrow{dl} & & u's_j \arrow{dr}& \\
 s_iu \arrow{rrrr} & & & & us_j\\
\end{tikzcd}
\end{center}
The quadrilaterals on the left and on the right are isotopies, and the rest are built entirely from braid relations and hence are equivalences by Theorem \ref{thm: ben}(a). Therefore the outside pentagon is equivalent to the inside one.

We now prove the statement by induction on the length of $u$. If $u$ is empty, the statement is clear. Otherwise, by definition of deletion pair we get $s_iu=us_j$. If $u$ ends with $s_j$ then we do not have a close pair, contradiction. Otherwise we need to apply some braid relation to $us_j$ which involves $s_j$. We have the following cases:

1) If $u=vs_k$ and $|k-j|>1$ then $us_j=vs_ks_j=vs_js_k$ while $s_iu=s_ivs_k$, so $vs_j=s_iv$. We get the following diagram:
\begin{center}
\begin{tikzcd}
s_ivs_ks_j \arrow{r} \arrow{d}& vs_js_ks_j \arrow{r} \arrow{d}& \arrow{d} vs_ks_js_j\\
s_ivs_js_k \arrow{r} \arrow{d}& vs_js_js_k \arrow{dr}& vs_ks_j\\
s_is_ivs_k \arrow{r}& s_ivs_k \arrow{r}& vs_js_k \arrow{u}
\end{tikzcd}
\end{center}
The top square is an isotopy, and the pentagon on the right is commutation with distant colors.
By the assumption of induction, two Demazure weaves \eqref{eq: u deletion pair}  corresponding to $s_ivs_j$ are equivalent, which implies that the bottom pentagon is an equivalence as well. 

2) If $u=vs_js_{j+1}$ then $us_j=vs_js_{j+1}s_j=vs_{j+1}s_js_{j+1}$ while $s_iu=s_ivs_js_{j+1}$, so $s_iv=vs_{j+1}$.
We get the following diagram:
\begin{center}
\begin{tikzcd}
s_ivs_js_{j+1}s_j \arrow{r} \arrow{d}& vs_{j+1}s_js_{j+1}s_j \arrow{r} \arrow{d}& \arrow{d} vs_{j}s_{j+1}s_js_j\\
s_ivs_{j+1}s_js_{j+1} \arrow{r} \arrow{d}& vs_{j+1}s_{j+1}s_js_{j+1} \arrow{dr}& vs_js_{j+1}s_j\\
s_is_ivs_js_{j+1} \arrow{r}& s_ivs_js_{j+1} \arrow{r}& vs_{j+1}s_js_{j+1} \arrow{u}
\end{tikzcd}
\end{center}
The top square is an isotopy, and the pentagon on the right is 5-cycle from Section \ref{sec: 1212}. By the assumption of induction, two Demazure weaves \eqref{eq: u deletion pair} corresponding to $s_ivs_{j+1}$ are equivalent, and the bottom pentagon is an equivalence as well.

The case when $u=vs_js_{j-1}$ is analogous.
\end{proof}

Given a Demazure weave $\ww:\beta_2\to\beta_1$, we have an injection $\iota_{\ww}$ from the set of crossings in the bottom $\beta_1$ to the set of crossings in the top $\beta_2$. For a 6-valent vertex  it is a bijection which exchanges left and right crossings, for a 4-valent vertex it is a bijection exchanging crossings, and for a 3-valent vertex the injection sends the crossing in the target to the right crossing in the source. We refer to the crossings not in the image of $\iota_{\ww}$ as {\it missing}. Note that the number of missing crossings equals the number of trivalent vertices and equals $\ell(\beta_2)-\ell(\beta_1)$. The following result is a characterization of the equivalence between Demazure weaves in a special case.
 
\begin{thm}
\label{thm: one 3v}
Let $\beta_1,\beta_2$ be two braid words such that $\ell(\beta_1)=\ell(\beta_2)-1$. Then two Demazure weaves $\ww_1,\ww_2:\beta_2\to\beta_1$ are equivalent if and only if they have the same missing crossing in $\beta_2$.
\end{thm}

\begin{proof}
Since $\ell(\beta_1)=\ell(\beta_2)-1$, any Demazure weave between $\beta_2$ and $\beta_1$ has one trivalent vertex. Let us prove that equivalent weaves have the same missing crossing. It is verified that commutations with distant colors and Zamolodchikov relations induce the same bijections between crossings, so any two weaves with the same source and target and only 6- and 4-valent vertices induce the same bijection. Finally, for the 5-cycle from Section \ref{sec: 1212} we observe that in either weave for 1121 the first crossing is missing, while in either weave for 1211 the third crossing is missing. 

Conversely, assume that the Demazure weaves $\ww_{1}, \ww_{2}: \beta_2 \to \beta_1$ have the same missing crossing. It is sufficient to prove that they can be related by a sequence of cycles from Lemma \ref{lem: deletion pairs} and equivalences. Note that a trivalent vertex corresponds to a close deletion pair. We have the following cases, where the deletion pair is underlined:

\begin{itemize}
	\item[(1)] Assume that $(s_i,s_j)$ is a close deletion pair and we apply a 4-valent vertex to $s_j$:
	$$
	\underline{s_i}u\underline{s_j}s_k=\underline{s_i}us_k\underline{s_j},\ |k-j|>1,
	$$ 
	then $(s_i,s_j)$ is a close deletion pair in the resulting braid, and Lemma \ref{lem: deletion pairs} applies.\\

	\item[(2)] Assume that we apply a 6-valent vertex with $s_j$ on the left:
	$$
	\underline{s_i}u\underline{s_j}s_{j+1}s_{j}=\underline{s_i}us_{j+1}s_j\underline{s_{j+1}}
	$$
	then $(s_i,s_{j+1})$ is a close deletion pair in the resulting braid, and Lemma \ref{lem: deletion pairs} applies.\\

	\item[(3)] Assume that we apply a 6-valent vertex with $s_j$ on the right:
	$$
	\underline{s_i}us_js_{j+1}\underline{s_{j}}=\underline{s_i}u\underline{s_{j+1}}s_js_{j+1}
	$$
	Note that $s_ius_js_{j+1}=us_{j}s_{j+1}s_j=us_{j+1}s_js_{j+1}$ implies $s_iu=us_{j+1}$, and  $(s_i,s_{j+1})$ is again a close deletion pair.\\

	\item[(4)] Finally, assume that we apply a 6-valent vertex with $s_j$ in the middle, then we no longer get a deletion pair. Instead, $u=vs_{j+1}$ and $s_ivs_{j+1}=vs_{j+1}s_{j}$. By considering possible braid moves, we can write $v=ws_j$, then
	$$
	s_iws_js_{j+1}=ws_{j}s_{j+1}s_j=ws_{j+1}s_js_{j+1},\
	$$
	hence $s_iw=ws_{j+1}$.We get the following diagram:
	\begin{center}
		\begin{tikzcd}
		\underline{s_i}ws_js_{j+1}\underline{s_j}s_{j+1} \arrow{r} \arrow{d} \arrow{dr}& \underline{s_i}ws_{j}s_{j}\underline{s_{j+1}}s_j \arrow{r} & \arrow{dl} w\underline{s_{j+1}}s_js_j\underline{s_{j+1}}s_j\\
		\underline{s_i}w\underline{s_{j+1}}s_{j}s_{j+1}s_{j+1}  \arrow{d}& w\underline{s_{j+1}}s_js_{j+1}\underline{s_j}s_{j+1} \arrow{r} \arrow{dl}& ws_js_{j+1}\underline{s_j}\underline{s_j}s_{j+1}\arrow{d}\\
		w\underline{s_{j+1}}\underline{s_{j+1}}s_js_{j+1}s_{j+1}\arrow{r}& ws_{j+1}s_js_{j+1}s_{j+1} \arrow{r}& ws_js_{j+1}s_js_{j+1} 
		\end{tikzcd}
	\end{center}
	Here the squares are isotopies and 5-cycle is an equivalence from Section \ref{sec: 1212}.
\end{itemize}

By combining all these cases (and the ones obtained by changing $j+1$ to $j-1$, or applying braid moves to $s_i$), we can find equivalent weaves $\ww_{1} \sim \ww_{1}'' \circ \ww_{1}'$ and $\ww_{2} \sim \ww_{2}'' \circ \ww_{2}'$ where 

\begin{itemize}
\item $\ww_{1}', \ww_{2}': \beta_{2} \to \beta'$ are weaves between equivalent braid words.
\item $\ww_{1}'', \ww_{2}'': \beta' \to \beta_{1}$ are weaves obtained by finding a close deletion pair in $\beta'$ and applying either weave from Lemma \ref{lem: deletion pairs}, followed by a sequence of braid moves.
\end{itemize}

Note that $\ww_{1}' \sim \ww_{2}'$, so it is enough to check that $\ww_{1}'' \sim \ww_{2}''$. Since $\ww_1, \ww_2$ have the same missing crossing in $\beta_2$, $\ww_{1}''$ and $\ww_{2}''$ have the same missing crossing in $\beta'$. Thus, $\ww_{1}''$ and $\ww_{2}''$ use the same close deletion pair in $\beta'$, so the result now follows from Lemma \ref{lem: deletion pairs}. 
\end{proof}

\begin{remark} \label{remark: injection_issue}
Although it is natural to consider the above injection and missing crossings for more general weaves, these notions are not invariant under the equivalence relation. Indeed, one can check that the two paths in the 5-cycle for 1121 yield two different injections on crossings (with the same image), and the different paths for 12121 have different missing vertices.
\end{remark}

For a positive braid word $\beta$, we define the {\em mutation graph} of $\beta$ to be a graph with vertices given by the equivalence classes of Demazure weaves $\ww:\beta\Delta\to\Delta$, from $\beta\Delta$ to $\Delta$, i.e.~$\ww\in\mbox{Hom}_{\mathfrak{W}_n}(\beta\Delta,\Delta)$, and edges corresponding to mutations. 
Note that, by Theorem \ref{thm: ben}(b), any two equivalence classes of Demazure weaves in $\mbox{Hom}_{\mathfrak{W}_{n}}(\beta\Delta, \Delta)$ are related by mutations. 

\begin{conj}
Suppose we oriented each mutation in the direction $(ss)s \to s(ss).$ For any positive braid $\beta$, this orientation descends to the mutation graph of $\beta$. With this orientation, the mutation graph has no oriented cycles.
\end{conj}

The conjecture is motivated by \cite{BDP}, where a similar statement was proven for the exchange graphs for quivers and cluster algebras; see also \cite{BY}.


\subsection{Examples} Let us study two explicit examples in detail, illustrating the material and results presented above.

\begin{ex} {\bf 2-strand braids}.
\label{subseq:2strand}
{\rm A braid on two strands is an element of $\mbox{Br}_2$: we denote by $\sigma$ the unique Artin generator of this group, and by $s$ the corresponding Coxeter generator $(12) \in S_2$. Each positive braid $\beta \in \mbox{Br}_2$ has a unique braid word, which has the form $\sigma^l, \, \ l \geq 0$, and note that $\Delta = \sigma$. By abuse of notation, we will also write this word as $s^l$. We refer to the braid $\sigma^l$ as the \emph{$(2, l)$-torus braid}, since its (rainbow) closure is the $(2, l)$-torus link.

We have no braid moves in $\mbox{Br}_2,$ so each weave $\ww \in \Hom_{\mathfrak{W}_2}(\beta,\beta')$ contains only trivalent vertices, cups and caps (and no $6$- or $4$-valent vertices). 
Each Demazure weave $\ww \in \Hom_{\mathfrak{W}_2}(\beta \cdot \Delta, \Delta)$ is naturally a rooted binary tree as it contains only trivalent vertices. By construction, all such binary trees with $l(\beta) + 1$ leaves are mutualy non-equivalent, but they are all related by mutations. If we orient each mutation $(ss)s \to s(ss),$ the oriented mutation graph will coincide with the classical Hasse graph of the Tamari lattice. It is known to be the $1$-skeleton of a combinatorial polytope: the $(l(\beta) - 1)$-dimensional associahedron, see e.g. \cite{Reading}. We can summarize this discussion as follows.

\begin{lemma}\label{lem:2n}
The mutation graph of the $(2, l)$ torus braid is the $1$-skeleton of the $(l - 1)$-dimensional associahedron.
\end{lemma}

We can also understand each Demazure weave $\ww \in \Hom_{\mathfrak{W}_2}(s^l \cdot \Delta, \Delta)$ as a sequence of openings of crossings in the braid $s^l \cdot \Delta = s^{l + 1}.$ As we understand trivalent vertices $ss \to s$ as openings of the left crossing, $\ww$ is actually a sequence of openings of crossings in $\beta$; the only crossing of $\Delta$ is the crossing of the concave end of $\ww.$ Naturally, the sequence of crossings being opened can be seen as a permutation in $S_l.$ The Tamari lattice is known to be both a sublattice and a lattice quotient of the weak order on permutations, see \cite{Reading}. Note that a permutation is the same as a maximal chain in the Boolean lattice $2^{[l]}$ of the subsets of the set of crossings of $\beta.$

Finally, another way to look at Demazure weaves $\ww \in \Hom_{\mathfrak{W}_2}(s^l \cdot \Delta, \Delta)$ is to consider them as monotone paths along the edges in the $l-$dimensional cube, with $2-$dimensional faces representing elementary moves (equivalences or mutations) between weaves. We illustrate this on the example of the $(2, 3)$-torus braid $\beta = sss$ in Figure \ref{cube}. Each edge of the cube is oriented downwards and corresponds to one trivalent vertex in a weave. Equivalently, it corresponds to opening a single crossing in $\beta$. Each vertex represents a horizontal cross-section away from the vertices of a Demazure weave $\ww \in \Hom_{\mathfrak{W}_2}(ssss, s)$; equivalently, it corresponds to a braid word obtained from $\beta$ by the opening of some crossings. The underlined letters represent crossings that have been opened. For each edge of the weave in a horizontal slice, we can trace back its parents in $ssss$; these parents are in parentheses. The cube has the unique top vertex representing the braid $ssss$, and the unique bottom vertex representing $s$. Each Demazure weave can be seen as a monotone path along the edges from the top vertex to the bottom vertex.

\begin{figure}[h!]
        \centering

\begin{tikzpicture}[scale=0.25]
 \fill[yellow!30](-8, 6) to (-8, 16) to (0, 10) to (0, 0) to cycle;
 \fill[blue!30](-2, 24) to (-8, 16) to (0, 10) to (6, 18) to cycle;
\draw (0, 0) -- (0, 10) -- (6, 18) -- (6, 8) -- (0, 0);
\draw (0, 10) -- (-8, 16) -- (-8, 6) -- (0, 0);
\draw (-8, 16) -- (-2, 24) -- (6, 18);
\draw[dashed] (-8, 6) -- (-2, 14) -- (6, 8);
\draw[dashed] (-2, 14) -- (-2, 24); 
\draw (0, 0) node[below] {(\underline{sss}s)};
\draw (0, 10) node[below left] {(\underline{s}s)(\underline{s}s)};
\draw (-8, 6) node [left] {(\underline{ss}s)s};
\draw (6, 8) node [right] {s(\underline{ss}s)};
\draw (-8, 16) node [left] {(\underline{s}s)ss};
\draw (6, 18) node[right] {ss(\underline{s}s)};
\draw (-2, 24) node[above] {ssss};
\draw (-2, 14)  node[above right] {s(\underline{s}s)s};
\draw [->] (-15, 24) -- (-15, 0) node [midway, left, sloped, rotate = 90] {\mbox{opening crossings}};
\draw [<->] (-8, -5) -- (6, -5) node [midway, above] {\mbox{mutations}} ;
\end{tikzpicture}

\caption{The Hasse graph of the Boolean lattice $2^{[3]}$.  The top vertex is the initial braid word $\beta \cdot \Delta = s^3 \cdot s = s^4$, the bottom vertex represents $\Delta = s$. Demazure weaves $\ww  \in \Hom_{\mathfrak{W}_2}(ssss, s)$ correspond to monotone paths from the top vertex to the bottom vertex.} \label{cube}
\end{figure}

The yellow face in the cube in Figure \ref{cube} illustrates the weave mutation given by
\[
\begin{xy} 0;<0.25pt,0pt>:<0pt,-0.25pt>::
(-150,100) *+{\begin{tikzpicture}[scale=0.2]
\draw[thick] (0, -1) -- (0, 1) -- (-3, 4);
\draw[thick] (0, 1) -- (3, 4);
\draw[thick] (-1, 2) -- (1, 4);
\draw[thick] (0, 3) -- (-1, 4);
\draw (-3.6, 4.7) node {\small{((s}};
\draw (-1.5, 4.7) node {\small{(s}};
\draw (1.4, 4.7)  node {\small{s))}};
\draw (3.2, 4.7) node {\small{s)}};
\end{tikzpicture}
}="1",
(250,100) *+{\begin{tikzpicture}[scale=0.2]
\draw[thick] (0, -1) -- (0, 1) -- (-3, 4);
\draw[thick] (0, 1) -- (3, 4);
\draw[thick] (1, 2) -- (-1, 4);
\draw[thick] (0, 3) -- (1, 4);
\draw (-3.3, 4.7) node {\small{(s}};
\draw (-1.5, 4.7) node {\small{((s}};
\draw (1.2, 4.7)  node {\small{s)}};
\draw (3.4, 4.7) node {\small{s))}};
\end{tikzpicture}
}="2",
"1", {\ar@{<->} "2"},
\end{xy}
\]

The blue face is the only face that does not represent a mutation. Two monotone paths related by the flip in this face correspond to two different possibilities to draw the same weave in such a way that each horizontal cross-section contains at most one trivalent vertex:

\[
\begin{xy} 0;<0.25pt,0pt>:<0pt,-0.25pt>::
(-150,100) *+{\begin{tikzpicture}[scale=0.2]
\draw[thick] (0, -1) -- (0, 1) -- (-5, 6);
\draw[thick] (0, 1) -- (5, 6);
\draw[thick] (-4, 5) -- (-3, 6);
\draw[thick] (2, 3) -- (-1, 6);
\draw (-5.5, 6.7) node {\small{((s}};
\draw (-2.8, 6.7) node {\small{s)}};
\draw (-1.3, 6.7)  node {\small{(s}};
\draw (5.4, 6.7) node {\small{s))}};
 \end{tikzpicture}
}="31",
(250,100) *+{\begin{tikzpicture}[scale=0.2]
\draw[thick] (0, -1) -- (0, 1) -- (-5, 6);
\draw[thick] (0, 1) -- (5, 6);
\draw[thick] (-2, 3) -- (1, 6);
\draw[thick] (4, 5) -- (3, 6);
\draw (-5.5, 6.7) node {\small{((s}};
\draw (1.2, 6.7) node {\small{s)}};
\draw (2.7, 6.7)  node {\small{(s}};
\draw (5.4, 6.7) node {\small{s))}};
 \end{tikzpicture}
}="32",
"31", {\ar@{=} "32"},
\end{xy}
\]

Two weaves are related by a single mutation if they are related by a polygonal flip in a non-blue face. In Figure\ref{cube}, mutations $(ss)s \to s(ss)$ correspond to replacements of two ``left'' sides of a square by its two ``right'' sides. The mutation graph is the $1$-skeleton of $2$-dimensional associahedron, that is, a pentagon. It is drawn on Figure \ref{pentagon} as the Hasse graph of the Tamari lattice of rooted binary trees with $4$ leaves. This concludes this example, focused on $2$-stranded braids. The study of $n$-stranded braids and their weaves is, in general, more intricate (and interesting as well). This is illustrated in the next example.

\begin{figure}
     \centering
       
\begin{xy} 0;<0.25pt,0pt>:<0pt,-0.25pt>::
(-550,300) *+{
\begin{tikzpicture}[scale=0.2]
\draw[thick] (0, -1) -- (0, 1) -- (-3, 4);
\draw[thick] (0, 1) -- (3, 4);
\draw[thick] (-1, 2) -- (1, 4);
\draw[thick] (-2, 3) -- (-1, 4);
\draw (-3.8, 4.7) node {\small{(((s}};
\draw (-1.2, 4.7) node {\small{s)}};
\draw (1.2, 4.7)  node {\small{s)}};
\draw (3.2, 4.7) node {\small{s)}};
\end{tikzpicture}
} ="0",
(-150,100) *+{\begin{tikzpicture}[scale=0.2]
\draw[thick] (0, -1) -- (0, 1) -- (-3, 4);
\draw[thick] (0, 1) -- (3, 4);
\draw[thick] (-1, 2) -- (1, 4);
\draw[thick] (0, 3) -- (-1, 4);
\draw (-3.6, 4.7) node {\small{((s}};
\draw (-1.5, 4.7) node {\small{(s}};
\draw (1.4, 4.7)  node {\small{s))}};
\draw (3.2, 4.7) node {\small{s)}};
\end{tikzpicture}
}="1",
(250,100) *+{\begin{tikzpicture}[scale=0.2]
\draw[thick] (0, -1) -- (0, 1) -- (-3, 4);
\draw[thick] (0, 1) -- (3, 4);
\draw[thick] (1, 2) -- (-1, 4);
\draw[thick] (0, 3) -- (1, 4);
\draw (-3.3, 4.7) node {\small{(s}};
\draw (-1.5, 4.7) node {\small{((s}};
\draw (1.2, 4.7)  node {\small{s)}};
\draw (3.4, 4.7) node {\small{s))}};
\end{tikzpicture}
}="2",
(0,500) *+{\begin{tikzpicture}[scale=0.2]
\draw[thick] (0, -1) -- (0, 1) -- (-3, 4);
\draw[thick] (0, 1) -- (3, 4);
\draw[thick] (-2, 3) -- (-1, 4);
\draw[thick] (2, 3) -- (1, 4);
\draw (-3.5, 4.7) node {\small{((s}};
\draw (-0.8, 4.7) node {\small{s)}};
\draw (0.7, 4.7)  node {\small{(s}};
\draw (3.4, 4.7) node {\small{s))}};
 \end{tikzpicture}
}="3",
(550,300) *+{\begin{tikzpicture}[scale=0.2]
\draw[thick] (0, -1) -- (0, 1) -- (-3, 4);
\draw[thick] (0, 1) -- (3, 4);
\draw[thick] (1, 2) -- (-1, 4);
\draw[thick] (2, 3) -- (1, 4);
\draw (-3.3, 4.7) node {\small{(s}};
\draw (-1.3, 4.7) node {\small{(s}};
\draw (0.7, 4.7)  node {\small{(s}};
\draw (3.6, 4.7) node {\small{s)))}};
 \end{tikzpicture}
}="4",
"0", {\ar "1"}, "0", {\ar "3"},
"1", {\ar "2"},
"2", {\ar "4"},
"3", {\ar "4"},  
\end{xy}

\caption{The mutation graph of the $(2, 3)$ torus braid $sss$. All mutations are oriented in the direction $(ss)s \to s(ss)$. It coincides with the Hasse graph of the Tamari lattice.} \label{pentagon}
\end{figure}
}
\end{ex}


\begin{ex} {\bf The $(3, 2)$ torus braid}.
\label{ex: misha}
{\rm
Consider the $(3, 2)$ torus braid $\beta = \sigma_1\sigma_2\sigma_1\sigma_2= 1212.$ Figure \ref{1212121} illustrates Demazure weaves 
 $1212\cdot \Delta=1212121 \to 212$ and relations between them. In Figure \ref{1212121}, we allow weaves with trivalent vertices $11\to 1, 22\to 2$ and $6$-valent vertices representing braid moves only in one direction: $121 \to 212$. Edges of the graph in Figure \ref{1212121} represent single moves. We assume that each weave is drawn in such a way that each horizontal cross-section contains at most one vertex. Each vertex on Figure \ref{1212121} represents a horizontal cross-section without vertices of an (a priori, not unique) weave. All edges are oriented downward. The weaves then correspond to monotone paths from the top vertex to the bottom vertex on the figure.

\begin{figure}[h!]
\begin{center}
\begin{tikzpicture}[scale = 2]
  
  \draw (10, 9.1) node (A24) {\small{1221221}};
  \draw (9.1, 7.2) node (A25) {\small{121221}};
  \draw (11, 7.2) node (A29) {\small{122121}};

  \fill[yellow!30](8.7, 6.6) to (7.8, 6) to (7.8, 5) to (9.2, 6) to cycle;
  \fill[yellow!30](11.5, 6.6) to (11.1, 6) to (12.3, 5) to (12.3, 6) to cycle;

\draw (7.8, 5) node (A28) {\small{2121}}; 
  \draw (12.3, 5) node (A23) {\small{1212}};

  \fill[yellow] 
 (10,8.2) to (9.7,7.5) to (10,7) to (10.4,7.5) to cycle;
  \fill[yellow] 
(7.3,4.6) to (8.3,3.5) to (10,3) to (9,4) to cycle;
  \fill[yellow]
(12.7,4.6) to (11,4) to (10,3) to (11.7,3.5) to cycle;

 \draw[dashed] (A25) -- (A24) -- (A29);
\draw[dashed] (A28) -- (8.3, 3.5);
  \draw[dashed] (A23) -- (11.7, 3.5);

  \draw (10,9.6) node (A1) {\small{1212121}};
  \draw (9.1,9.1) node (A2) {\small{2122121}};
  \draw (7.9,8.1) node (A3) {\small{212121}};
  \draw (6.9,7) node (A4){\small{221221}};
  \draw (6.9,6) node (A5){\small{22121}};
  \draw (7.3,4.6) node (A6) {\small{22212}};
  \draw (8.3,3.5) node (A7){\small{2212}};
  \draw (10,3) node (A8){\small{212}};
  
  \draw (10,8.2) node (A9){\small{2122212}};
\draw (9.7,7.5) node (A10) {\small{212212}};
  \draw (10,7) node (A11){\small{21212}};
  \draw (10,5) node (A12){\small{22122}};
  \draw (9,4) node (A13){\small{2212}};
  
\draw (10.4,7.5) node (A14) {\small{212212}};
  
  \draw (10.9, 9.1) node (A15) {\small{1212212}};
  \draw (12.1, 8.1) node (A16){\small{121212}};
  \draw (13.1, 7) node (A17){\small{122122}};
  \draw (13.1, 6) node (A18){\small{12122}};
  \draw (12.7, 4.6) node (A19) {\small{21222}};
  \draw (11, 4) node(A20) {\small{2122}};
  
  \draw (11.7, 3.5) node(A21) {\small{2122}};
  
  \draw (12.3, 6) node (A22) {\small{12212}};
  

  \draw (8.7, 6.6) node (A26) {\small{212221}};
  \draw (7.8, 6) node (A27) {\small{21221}};
  
  \draw (10.2, 6.5) node (A30) {\small{12121}};
 \draw (9.2, 6) node (A31) {\small{21221}};
  
  \draw (11.5, 6.6) node (A32) {\small{122212}};
  \draw (11.1, 6) node (A33) {\small{12212}};


  \draw[dashed] (A26) -- (A27) -- (A28) -- (A31) -- (A26);
  \draw[dashed]  (A32) -- (A22) -- (A23) -- (A33) -- (A32);
  
  \draw (A1) -- (A2) -- (A3) -- (A4) -- (A5) -- (A6) -- (A7) -- (A8);
  \draw (A3) -- (A10);
  \draw (A2) -- (A9) -- (A10) -- (A11) -- (A12) -- (A13) -- (A8);
  \draw (A6) -- (A13);
  \draw (A9) -- (A14) -- (A11);
  \draw (A1) -- (A15) -- (A9);
  \draw (A15) -- (A16)-- (A17) -- (A18) -- (A19) -- (A20) -- (A8);
  \draw (A16) -- (A14);
  \draw (A12) -- (A20);
  \draw (A19) -- (A21);
  \draw (A21) -- (A8);
 \draw[dashed] (A17) -- (A22);
  \draw[dashed] (A18) -- (A23);
  
 \draw[dashed] (A1) -- (A24);
 \draw[dashed] (A25) -- (A26);
  \draw[dashed] (A4) -- (A27);
  \draw[dashed] (A5) -- (A28);
  \draw[dashed] (A29) -- (A30) -- (A31);
  \draw[dashed] (A25) -- (A30);
  \draw[dashed] (A29) -- (A32);
  \draw[dashed] (A30) --(A33);

\end{tikzpicture}
\end{center}
\caption{The top vertex is the initial braid word $\beta \cdot \Delta = s_1 s_2 s_1 s_2 \cdot  s_1 s_2 s_1$, the bottom vertex represents $s_2 s_1 s_2$. Demazure weaves $\beta \cdot \Delta \to \Delta$ with only $6$-valent vertices $s_1 s_2 s_1 \to s_2 s_1 s_2$ and $3$-valent vertices allowed correspond to monotone paths from the top vertex to the bottom vertex.}
 \label{1212121}
\end{figure}

It appears that there is a way to draw the graph as a $1-$skeleton of a $3-$dimensional polytope with 21 facets, although we did not try to find an explicit polytopal realization. The 2-dimensional (polygonal) faces correspond to the elementary moves between the paths. All the 2-dimensional faces are 4- or 8-gons:

\begin{itemize}
\item[(1)] Yellow quadrilaterals correspond to mutations between the pairs of paths from $sss$ to $s$.
\item[(2)] Other quadrilaterals correspond to 
isotopies exchanging the heights of vertices in a weave.
\item[(3)] Octagons correspond to the outer octagons in Section \ref{sec: 12121}, they are formed by paths (A) and (E). Note that the inner vertices and paths do not appear since the moves $212 \to 121$ are not allowed. 
\end{itemize}

We have no words containing $1121$ or $1211$ in our example, so the pentagons from \ref{sec: 1212} do not appear. In order to cover all Demazure weaves, we should allow the moves $212 \to 121$. In Figure \ref{1212121}, we should then replace each octagonal face by 5 faces from \ref{sec: 12121}.

Two weaves are equivalent if the corresponding paths are separated by several white faces, and related by a single mutation  if they are separated by one yellow face and several white faces. If we start from some monotone path, replace its part given by some edges of a 2-dimensional face by all the other edges of this face, and repeat this procedure by modifying paths across 2-dimensional faces until we come back to the original path (making the 360 degrees turn around the vertical axis in the polytope along the way), we go by all edges of the mutation graph of equivalence classes of paths exactly once. The mutation graph is a pentagon. 

Each Demazure weave $\ww: 1212121 = \beta \cdot \Delta \to \Delta = 121$ is equivalent to a Demazure weave $\ww': \beta \cdot \Delta \to 212$ concatenated with a single $6-$valent vertex $212 \to 121.$ Indeed, if the last vertex in $\ww$ is a $6-$valent vertex $\beta_{\ell(\ww) - 1}(\ww) = 212 \to 121 = \beta_{\ell(\ww)}(\ww),$ then we define $\ww'$ to be $\ww$ with this vertex being removed. By construction, $\ww$ is then the concatenation of $\ww'$ with this vertex $\beta_{\ell(\ww) - 1}(\ww) \to \beta_{\ell(\ww)}(\ww).$ Otherwise, we define $\ww'$ to be $\ww$ concatenated with a vertex $121 \to 212.$ Then $\ww$ is equivalent to $\ww'$ concatenated with a vertex $212 \to 121$ via a cancellation move from Section \ref{sec: cancellation}. These arguments show that the mutation graph of Demazure weaves $\beta \cdot \Delta \to \Delta$ is isomorphic to the mutation graph of Demazure weaves $\beta \cdot \Delta \to 212.$ Thus, the former, i.e. the mutation graph of $\beta,$ is also a pentagon.

The appearance of the pentagon is not completely unexpected. Indeed, it coincides with the mutation graph of the torus braid $(2, 3)$.
Since the Legendrian links $\Lambda(3, 2) = \Lambda(2, 3)$ coincide, and the corresponding augmentation varieties are isomorphic. The fact that the mutations graphs of Demazure weaves of these two braids are isomorphic to each other is to be expected.
}
\end{ex}

Let us conclude this subsection on examples with two conjectures. First, inspired by the (Legendrian) equivalence between certain Legendrian $(2,n)$- and $(n,2)$-torus links, and Lemma \ref{lem:2n}, we state the following conjecture.

\begin{conj}
For the $(n, 2)$ torus braid $\beta$, the mutation graph of Demazure weaves $\beta \cdot \Delta \to \Delta$ is the $1$-skeleton of the $(n-1)$-dimensional associahedron.
\end{conj}

Our conjectural 3-dimensional polytope on Figure \ref{1212121} is similar to polytopes from \cite[Figure 1]{McConville} where the vertices encode equivalence classes of reduced expressions of elements in the braid group and edges correspond to braid moves (also oriented from $s_i s_{i+1} s_i$ 
to $ s_{i+1} s_i s_{i+1}$). Reduced expressions are considered to be equivalent if they are related by a sequence of moves  $s_i s_j \to s_j s_i, |i - j| \geq 2$. This equivalence relation is trivial in our $3-$strand case. It would be interesting to construct such polytopes for other braids. 

\begin{remark} The polytopes in \cite{McConville} are the Hasse graphs of \emph{second higher Bruhat orders} introduced by Manin and Schechtman \cite{MS1, MS2}, see also 
\cite{Ziegler}. Given an arbitrary braid $\beta,$ we can consider a similar oriented graph $D_{\beta}$. First, we associate a vertex to the braid $\beta$. We draw edges corresponding to moves 
 $ss \to s,  \, s_i s_{i+1} s_i \to s_{i+1} s_i s_{i+1},$ and $s_i s_j \to s_j s_i, |i - j| \geq 2$. We then contract all edges corresponding to moves $s_i s_j \to s_j s_i, |i - j| \geq 2$. This defines a poset with covering relations defined by edges. An element of the poset is an equivalence class of  positive braid words with  Demazure product $\Delta$, with a certain extra decoration. Words are considered to be equivalent if they are related by a sequence of moves  $s_i s_j \to s_j s_i, |i - j| \geq 2$. The decoration can be understood in terms of subsets of the set of crossings of the braid $\beta$; however, it is nontrivial to give a precise definition because of the issue discussed in Remark \ref{remark: injection_issue}. We can also define the decoration in a non-combinatorial way by using variables from Section \ref{subsec:weqrat}. 

If we forget the decoration, this poset becomes a poset on the set of words with Demazure product $\Delta$. Its analogue for all expressions of $\Delta$ and covering relations $ss \to s$ replaced by $ss \to e$ was defined by Elias \cite{Elias} as an extension of the second higher Bruhat order to necessarily reduced words. It was used in the proof of the main result of the work \cite{Elias}, which we translated to our language as Theorem \ref{thm: ben}. Our weaves thus resemble saturated chains in the second higher Bruhat order, which in turn can be seen as elements of the third higher Bruhat order. However, our equivalence relations differ from the one considered by Manin and Schechtman. Note also that Thomas \cite{Thomas} defined the $0$th Bruhat order to be the Boolean lattice. As we discussed in Example \ref{subseq:2strand}, Demazure weaves in $\mathfrak{W}_2$ can be seen as maximal chains in the $0$th Bruhat order. In the present article, we will not explore the link between weaves and the theory of higher Bruhat orders further.

The graph $D_{\beta}$ is not always a $1$-skeleton of a polytope: e.g. $D_{12122}$ is only a $1$-skeleton of a union of two quadrilaterals. However, we have the following expectation.

\begin{conj}
For an arbitrary positive braid word $\beta,$ the poset complex of the oriented graph $D_{\beta}$  is either a sphere or a ball.
If it is a sphere, it admits a polytopal realization.
\end{conj}
\end{remark}
\color{black}


\subsection{Triangulations and weaves}\label{sec: admissible}


This subsection provides two types of constructions for weaves, by using certain labeled triangulations, and a relation between them. Specifically, we present the following constructions:

\begin{enumerate}
    \item From an admissible triangulation $\tau$, as in Definition \ref{def:admissible_triang} below, we construct a weave $\ww(\tau)$. There are choices in the construction of $\ww(\tau)$, but any two sets of choices lead to equivalent weaves.\\

    \item From a Demazure triangulation $\underline{\tau}$, as in Definition \ref{def:Demazure_triang} below, we construct a weave $\ww(\underline{\tau})$. There are choices in the construction of $\ww(\underline{\tau})$ and, in contrast to (1) above, different choices might lead to non-equivalent weaves. Nevertheless, any two weaves constructed from the same Demazure triangulation $\underline{\tau}$ are mutation equivalent.\\

    \item In Proposition \ref{prop:subdivide} below we show that any Demazure triangulation can be subdivided to
    an admissible triangulation, and explain how the resulting weaves -- via (1) and (2) above -- are related.
\end{enumerate}

\noindent In this subsection, weaves are not {\it a priori} sliced. The use of the word {\it weave} in this subsection will refer to Definition \ref{def:weave} unless otherwise specified.

\subsubsection{Admissible triangulations and weaves}\label{sssec:admissible_triang} Given a positive braid word $\beta$, let us write the letters (crossings) of $\beta\cdot \Delta$ on the sides of a $(\ell(\beta) + \ell(w_0))-$gon.

\begin{definition}\label{def:admissible_triang}
Let $P\sse\R^2$ be a regular $n$-gon. Consider 
triangulation $\tau$ of $P$ such that the vertices and edges of $P$ are vertices and edges of the triangulation $\tau$, though $\tau$ may contain vertices inside of $P$. By definition, $\tau$ is said to be admissible if each of the edges is oriented and labeled by a permutation $u$ such that every triangle is one of the following two types:
\begin{center}
\begin{tikzpicture}
\begin{scope}[decoration={
    markings,
    mark=at position 0.5 with {\arrow{stealth}}}
    ] 
\draw[postaction={decorate}] (0,0)--(1,2);
\draw[postaction={decorate}] (1,2)--(2,0);
\draw[postaction={decorate}] (0,0)--(2,0);
\draw[postaction={decorate}] (3,0)--(4,2);
\draw[postaction={decorate}] (4,2)--(5,0);
\draw[postaction={decorate}] (3,0)--(5,0);
\end{scope}
\draw (0.3,1.1) node {$s_i$};
\draw (1.7,1.1) node {$s_i$};
\draw (1,-0.2) node {$s_i$};
\draw (3.3,1.1) node {$u$};
\draw (4.7,1.1) node {$v$};
\draw (4,-0.2) node {$uv$};
\end{tikzpicture}
\end{center}
That is, either all three sides of one triangle of $\tau$ are labeled by the same simple reflection $s_i$ and the edges do not form a 3-cycle, or the sides are labeled by permutations $u,v$ and $u\cdot v$ such that $\ell(u\cdot v)=\ell(u)+\ell(v)$ as depicted above (and edge orientations also are as in the figure). By definition, we call such triangles \emph{admissible}.
\end{definition}

\begin{remark}
Note that the edges in $\tau$ are oriented, but the triangles are \emph{not} oriented, and so neither is the triangulation $\tau$ itself. The orientation on edges is used below only to illustrate the way we read and concatenate edge labels: if we follow the edge labeled by $u$ in the direction opposite to its orientation, then we read the label as $u^{-1}$.
\end{remark}

Now, given an admissible triangulation $\tau$ as in Definition \ref{def:admissible_triang}, we can algorithmically construct a weave $\ww(\tau)$ associated to it. For this, we make some additional choices. This is done as follows:

\begin{enumerate}
    \item Choose a reduced expression for the permutation on every edge. For triangles of the second type, we then concatenate the reduced expressions for $u$ and $v$ and get a reduced expression for $uv$. This can be represented by a (piece of a) weave with no vertices at all. Now, this resulting reduced expression for $uv$ is possibly different from the one initially assigned to $uv$. Since two reduced expressions are related by a sequence of braid moves, which are translated to 6- and 4-valent vertices for weaves, we can draw a weave (just with 4- and 6-valent vertices) representing that sequence and 
    connecting these two reduced expressions for $uv$.
    Note that the resulting weave on such a triangle of the second type depends on this choice of a sequence of braid moves. That said, any two such weaves are related by weave equivalences, by Theorem~\ref{thm: ben}.(a).\\

    \item For triangles of the first type, with edges labeled by $s_i$, we associate a weave consisting of a
    single trivalent vertex of the corresponding color $s_i$.

\end{enumerate}

\noindent In summary, the possible choices in $(1)$ give equivalent weaves and there are no choices in $(2)$, for triangles of the first type. Therefore, this assignment of a weave for each triangle in $\tau$ glues up to a weave $\ww(\tau)$ on the entire polygon $P$, well-defined up to weave equivalence. (Cf. Remark \ref{rem: equivalent weaves}.)

We can encode some of the moves between weaves in terms of admissible triangulations, as follows:

\begin{itemize}
    \item[(i)] Given three permutations $u$, $v$ and $w$ such that $\ell(uvw)=\ell(u)+\ell(v)+\ell(w)$, we can make the following moves, which clearly do not change the weave, up to equivalence:

\begin{center}
\begin{tikzpicture}
\begin{scope}[decoration={
    markings,
    mark=at position 0.5 with {\arrow{stealth}}}
    ] 
\draw[postaction={decorate}] (0,0)--(0,2);
\draw[postaction={decorate}] (0,2)--(2,2);
\draw[postaction={decorate}] (0,0)--(2,0);
\draw[postaction={decorate}] (2,2)--(2,0);
\draw[postaction={decorate}] (0,0)--(2,2);
\draw (-0.2,1) node {$u$};
\draw (1,2.2) node {$v$};
\draw (2.2,1) node {$w$};
\draw (1,-0.2) node {$uvw$};
\draw (0.8,1.2) node {$uv$};

\draw (2.5,1) node {$\sim$};

\draw[postaction={decorate}] (3,0)--(3,2);
\draw[postaction={decorate}] (3,2)--(5,2);
\draw[postaction={decorate}] (3,0)--(5,0);
\draw[postaction={decorate}] (5,2)--(5,0);
\draw[postaction={decorate}] (3,2)--(5,0);
\draw (2.8,1) node {$u$};
\draw (4,2.2) node {$v$};
\draw (5.2,1) node {$w$};
\draw (4,-0.2) node {$uvw$};
\draw (4.2,1.2) node {$vw$};
\end{scope}
\end{tikzpicture}
\end{center}

\begin{center}
\begin{tikzpicture}
\begin{scope}[decoration={
    markings,
    mark=at position 0.5 with {\arrow{stealth}}}
    ] 
\draw[postaction={decorate}] (0,0)--(1,2);
\draw[postaction={decorate}] (0,0)--(2,0);
\draw[postaction={decorate}] (1,2)--(2,0);
\draw[postaction={decorate}] (0,0)--(1,2/3); 
\draw[postaction={decorate}] (1,2)--(1,2/3);
\draw[postaction={decorate}] (1,2/3)--(2,0);
\draw (0.5,1.2) node {$u$};
\draw (1.7,1.2) node {$vw$};
\draw (1,-0.2) node {$uvw$};
\draw (1.1,1.2) node {\scriptsize $v$};
\draw (0.5,0.5) node {\scriptsize $uv$};
\draw (1.5,0.5) node {\scriptsize $w$};

\draw (2.5,1) node {$\sim$};

\draw[postaction={decorate}] (3,0)--(4,2);
\draw[postaction={decorate}] (3,0)--(5,0);
\draw[postaction={decorate}] (4,2)--(5,0);
\draw (3.5,1.2) node {$u$};
\draw (4.7,1.2) node {$vw$};
\draw (4,-0.2) node {$uvw$};
\end{scope}
\end{tikzpicture}
\end{center}

\noindent For unlabeled triangulations, these are precisely the \emph{Pachner moves} (also known as \emph{bistellar flips}) in dimension $2$. The result of Pachner \cite{Pachner} states that all triangulations of a polygon are related by such moves (the general version of this result holds for piecewise linear manifolds and bistellar flips in higher dimensions).\\

\item[(ii)] If we have four permutations $u,v,w,t$ such that $uv=tw$, then 
$u^{-1}t=vw^{-1}$. Assuming that all these products are reduced, we have a move

\begin{center}
\begin{tikzpicture}
\begin{scope}[decoration={
    markings,
    mark=at position 0.5 with {\arrow{stealth}}}
    ] 
\draw[postaction={decorate}] (0,0)--(0,2);
\draw[postaction={decorate}] (0,2)--(2,2);
\draw[postaction={decorate}] (0,0)--(2,0);
\draw[postaction={decorate}] (2,0)--(2,2);
\draw[postaction={decorate}] (0,0)--(2,2);
\draw (-0.2,1) node {$u$};
\draw (1,2.2) node {$v$};
\draw (2.2,1) node {$w$};
\draw (1,-0.2) node {$t$};
\draw (1,1.2) node {\scriptsize $uv=tw$};

\draw (2.5,1) node {$\sim$};

\draw[postaction={decorate}] (3,0)--(3,2);
\draw[postaction={decorate}] (3,2)--(5,2);
\draw[postaction={decorate}] (3,0)--(5,0);
\draw[postaction={decorate}] (5,0)--(5,2);
\draw[postaction={decorate}] (3,2)--(5,0);
\draw (2.8,1) node {$u$};
\draw (4,2.2) node {$v$};
\draw (5.2,1) node {$w$};
\draw (4,-0.2) node {$t$};
\draw (4,1.3) node {\scriptsize $u^{-1}t=vw^{-1}$};
\end{scope}
\end{tikzpicture}
\end{center}

\noindent Note that in this case we get the equations 
$$
\ell(u)+\ell(v)=\ell(t)+\ell(w),\ \ell(u)+\ell(t)=\ell(v)+\ell(w)
$$
which imply
$$
\ell(u)=\ell(w),\ \ell(v)=\ell(t).
$$

\item[(iii)] We can encode the 1212-move from Section \ref{sec: 1212} as the following move between triangulations:

\begin{center}
\includegraphics[scale=0.9]{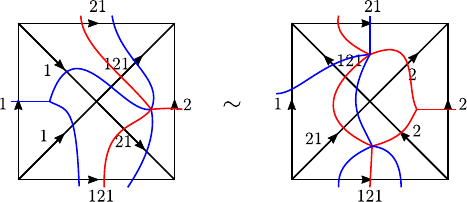}
\end{center}
\end{itemize}

\begin{ex}
\label{ex: braid move triangulation}
The weave corresponding to either of the following two diagrams is a 6-valent vertex: 

\begin{center}
\includegraphics[scale=0.8]{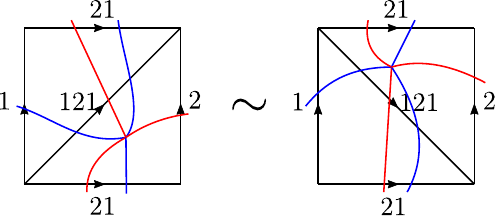}
\end{center}
The choice of a reduced expression ($121$ or $212$) on the diagonal determines the triangle containing this 6-valent vertex. This illustrates $(ii)$ above.
\end{ex}

\begin{remark}
Conversely to the construction above, given a weave we can consider the dual planar graph. It has triangular regions corresponding to 3-valent vertices in the weave, hexagonal regions corresponding to 6-valent vertices, and quadrilateral regions corresponding to 4-valent vertices. By choosing any admissible triangulation of each hexagon and quadrilateral, we get a triangulation of the entire polygon. The choice of the triangulation does not matter - for example, for the hexagon with sides labeled 1,2,1,2,1,2 there are 14 triangulations and 12 of them (those that do not contain triangles formed by three diagonals) are admissible. From Example \ref{ex: braid move triangulation}, any two of them can be related by a sequence of the above moves, and correspond to equivalent weaves. In conclusion, we have a construction starting with a weave and resulting in a triangulation and vice versa. These constructions depend on choices and neither of them is a bijection. For future work, it would be interesting to find a complete set of moves between triangulations such that the corresponding equivalence classes are in bijection with the equivalence classes of weaves.
\end{remark}


\subsubsection{Demazure triangulations}
\label{sec: demazure}
The correspondence between weaves and admissible triangulations in Subsection \ref{sssec:admissible_triang} above is clear combinatorially. Nevertheless, it has a disadvantage: given a triangulation, it is unclear if the corresponding weave is simplifying or Demazure (more precisely, it is unclear whether the underlying colored graph can be drawn as a sliced weave which is simplifying or Demazure, respectively) or, geometrically,  if the corresponding Lagrangian surface is embedded in $\R^4$ (instead of merely immersed). In order to resolve this issue, we now introduce a special class of triangulations -- with different labeling rules -- which we refer to as {\it Demazure} triangulations. We stress that Demazure triangulations are not necessarily admissible triangulations, as defined in Subsection \ref{sssec:admissible_triang} above. That said, Proposition~\ref{prop:subdivide} below explains how to produce admissible triangulations from Demazure triangulations by a subdivision process.\\

Let $\beta$ be a positive $n$-braid word with Demazure product $\delta(\beta)=w_0$ and set $N:=\ell(\beta)+1$.
Consider an $N$-gon $P$ whose vertices are labeled clockwise with integers from $0$ to $N-1$. Label the first $N-1$ sides clockwise by the letters of $\beta$, and label the last (or {\it bottom}) side connecting 
$N-1$ 
and $0$
by $w_0$. By definition, such labeled polygon $P$ is said to be labeled according to $\beta$.

\begin{definition}\label{def:Demazure_triang}
Let $\beta$ be a positive $n$-braid word with Demazure product $\delta(\beta)=w_0$ and $P$ a polygon labeled according to $\beta$. By definition, a {\it Demazure triangulation} $\underline{\tau}(P,\beta)$ is a non-oriented triangulation of $P$ with edges labeled by permutations such that:

\begin{enumerate}
    \item The only vertices of $\underline{\tau}(P,\beta)$ are the vertices of $P$.\\

    \item Every edge is labeled by a permutation as follows. An edge $e$ in $\underline{\tau}(P,\beta)$ divides the boundary $\dd P$ into two connected components, and the labels in the connected component of $\dd P$ that does not contain the $w_0$-label spell a subword $\beta'\sse\beta$. Then the permutation assigned to $e$ is the Demazure product $\delta(\beta')$ of such subword $\beta'$.
\end{enumerate}
\end{definition}

\begin{center}
	\begin{figure}[h!]
		\centering
		\includegraphics[scale=0.65]{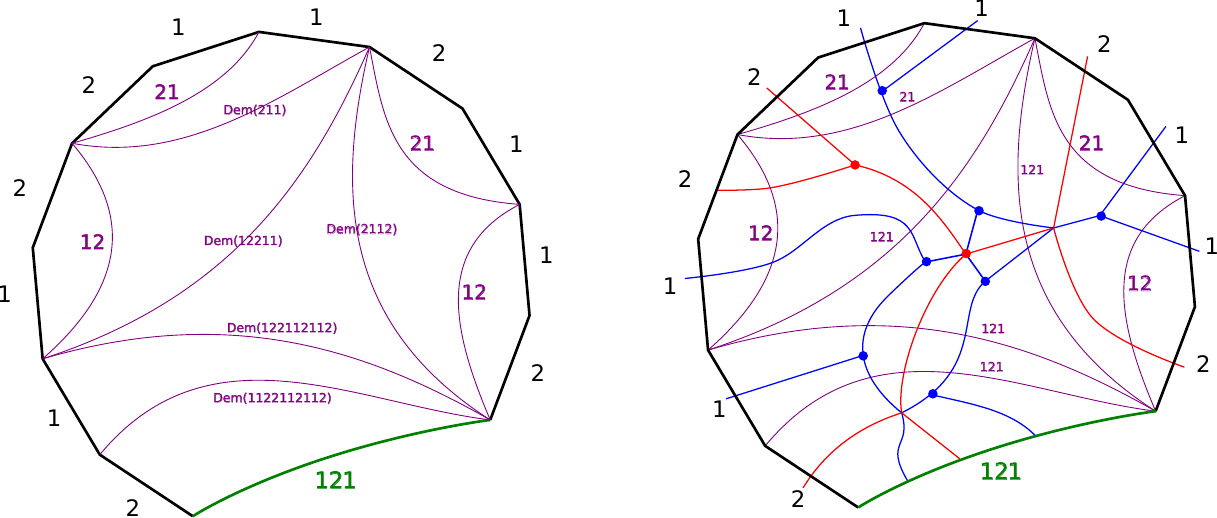}
		\caption{(Left) A Demazure triangulation for the braid word $\beta=\sigma_2\sigma_1^2\sigma_2^2\sigma_1^2\sigma_2\sigma_1^2\sigma_2$. The Demazure product $w_0=s_1s_2s_1$ and its associated edge are depicted in green for visual emphasis.  (Right) A possible weave associated to this Demazure triangulation.}
		\label{fig:Triangulation3graphEx_2}
	\end{figure}
\end{center}

\noindent First, a Demazure triangulation, as in Definition \ref{def:Demazure_triang}, is not necessarily admissible, as in Definition \ref{def:admissible_triang}. Note also that the edges of a Demazure triangulation are diagonals of the polygon $P$ and we sometimes refer to them as diagonals. Second, there are no choices in the definition of a Demazure triangulation beyond the triangulation of the polygon $P$ itself and the braid word for $\beta$. These two uniquely specify the labeling of the edges. Additional choices will be needed when we try to associate a weave to a Demazure triangulation. By definition, any Demazure triangulation of a polygon labeled according to $\beta$ is said to be a Demazure triangulation of $\beta$, and we often simply write $\underline{\tau}(\beta)$ for such a Demazure triangulation.

\begin{ex}
Figure \ref{fig:Triangulation3graphEx_2} (left) illustrates an instance of a Demazure triangulation. The positive 3-braid word is $\beta=\sigma_2\sigma_1^2\sigma_2^2\sigma_1^2\sigma_2\sigma_1^2\sigma_2$, whose Demazure product is indeed $w_0=s_1s_2s_1$.
\end{ex}

\begin{definition}
Given a Demazure triangulation, we define the height of the bottom $w_0$-side to be 0. Given any other side or edge, we define its \emph{height} as the number of edges separating it from the bottom side plus one. 
\end{definition}

\begin{ex}
The following picture illustrates edges in a triangulation labeled by height, where the bottom side is labeled with $w_0$:

\begin{center}
\begin{tikzpicture}
\draw (-1,0)--(1,0)--(2,1)--(1.5,2)--(0,3)--(-1.5,2)--(-2,1)--(-1,0);
\draw (1,0)--(1.5,2)--(-1,0)--(0,3)--(-2,1);
\draw (0,-0.2) node {0};
\draw (0,0.2) node {$w_0$};
\draw (1.7,0.5) node {2};
\draw (-1.7,0.5) node {3};
\draw (1.9,1.5) node {2};
\draw (-1.9,1.5) node {4};
\draw (1,2.5) node {2};
\draw (-1,2.5) node {4};
\draw (1.1,1) node {1};
\draw (0,1) node {1};
\draw (-0.3,1.5) node {2};
\draw (-1.2,1.6) node {3};
\end{tikzpicture}
\end{center}
\end{ex}

By definition, in any triangle in a Demazure triangulation we have two sides of equal height $h$ labeled by some permutations $u,v$ and the third side of height $h-1$ labeled by $u\star v$.

\begin{definition}
Let $\triangle$ be a triangle with sides $u,v$ and $u\star v$. The {\em defect} of $\triangle$ is 
$$\dft(\triangle)=\ell(u)+\ell(v)-\ell(u\star v).$$
\end{definition} 

The following fact relates defects to the length of the boundary braid.

\begin{lemma}
Let $\beta$ be a positive braid word and $\triangle$ a triangle in a Demazure triangulation of $\beta$. Then
$$\sum_{\triangle}\dft(\triangle)=\ell(\beta).$$
\end{lemma}

\begin{proof}
Set $r:=\ell(\beta)$ to ease notation and let us prove the following more general statement: ``Suppose that a diagonal, or the side with vertices $0$ and $N$, encloses a braid $\beta'$ with $k$ crossings, and 
carries the label $u=\delta(\beta')$. Then the sum of defects of the triangles above this diagonal equals $k-\ell(u)$.'' The required statement in the lemma follows from this by setting $u=w_0$, so that we have $r+\ell(w_0)-\ell(w_0)=r$.\\

\noindent To prove this general statement, we use induction in $k\in\N$. Consider the triangle adjacent to the diagonal with $u$, its other sides are labeled $v$ and $w$ such that $v\star w=u$. By the assumption of induction, sum of defects above $v$ equals $k_1-\ell(v)$ and the sum of defects above $w$ equals $k_2-\ell(w)$, with $k_1 + k_2 = k$, so the total sum of defects equals $
k_1-\ell(v)+k_2-\ell(w)+\ell(v)+\ell(w)-\ell(u)=k-\ell(u).
$
\end{proof}

The next result justifies the chosen nomenclature for a {\it Demazure} triangulation. 

\begin{prop}
\label{prop:Demazure triangulation}
Let $\beta$ be a positive braid word with Demazure product $\delta(\beta)=w_0$, $\Delta$ a reduced word for $w_0$, and $\underline{\tau}(\beta)$ a Demazure triangulation associated to $\beta$. Then there exists a non-deterministic algorithm that constructs a Demazure weave $\ww(\underline{\tau}(\beta))\in\Hom_{\mathfrak{W}_n}(\beta,\Delta)$ from the Demazure triangulation $\underline{\tau}(\beta)$.
\end{prop}

Before its proof, we emphasize that there are choices in the construction of a Demazure weave $\ww(\underline{\tau}(\beta))$ from the Demazure triangulation $\underline{\tau}(\beta)$. Different choices for the same $\underline{\tau}(\beta)$ lead to mutation-equivalent, but {\it not} necessarily weave-equivalent, weaves.

\begin{proof}
First, observe that for any two permutations $u,v$ there exists a (non-unique) Demazure weave from the concatenation of any reduced braid words for $u$ and $v$ at the top to any reduced word for $u\star v$ at the bottom. This is immediate by the definition of Demazure product: we can go from $uv$ to $u\star v$ by a sequence of braid relations and moves $s_is_i\to s_i$, which correspond to 6-,4- and 3-valent vertices. Therefore, for each triangle of the Demazure triangulation, we have a Demazure weave from the concatenation of the reduced words for the labels of the sides of height $h$  to the reduced word for the label of the height $h-1$. Note that such Demazure weaves associated to a triangle might not be unique, but at least one exists. Given a Demazure triangulation, we now construct a Demazure weave as follows.

For each possible height $h$ there is a unique polygonal chain $L_h$ inside the Demazure triangulation, consisting of diagonals of the triangulation of height $h$ and (some) sides of the polygon of height at most $h$, and satisfying the following conditions:

\begin{itemize}
    \item[(i)] Its two endpoints coincide with the endpoints of the bottom side.\\

    \item[(ii)] It contains all diagonals of height $h$ precisely once, and each side of the polygon at most once.
\end{itemize}

For each diagonal or side appearing in $L_h$, we choose a reduced expression of its label; for the bottom side, we choose $\Delta$ as a reduced word for $w_0$, and  for each other side 
its label 
is a letter of 
$\beta$. The clockwise orientation of the boundary of the polygon induces an orientation on the bottom side.
We define $\beta_h$ to be the concatenation of all the words appearing as labels of the line segments appearing in $L_h$, where we orient $L_h$ from the target of the oriented bottom side to its source. Note that we have 
$\beta_0=\Delta$
and 
$\beta_{N-3}=\beta$. By the discussion above, for each height $h$, there exists a (non-unique) Demazure weave from  $\beta_h$ to $\beta_{h-1}$. By choosing such a weave for each height $h$ and 
concatenating them
we obtain a Demazure weave from $\beta$ to $w_0$.
\end{proof}

\begin{figure}[ht!]
\includegraphics[scale=0.8]{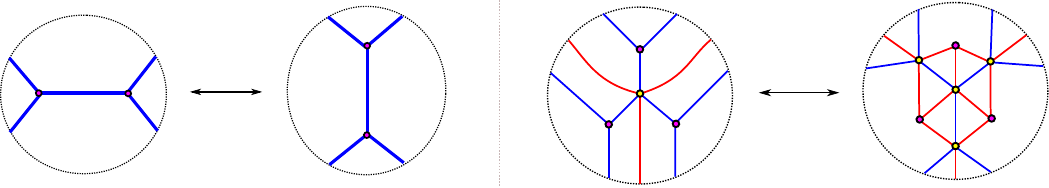}
\caption{(Left) A weave mutation. (Right) A move obtained from composing a sequence of weave equivalences, one weave mutation and another sequence of weave equivalences.}
\label{fig: triangle mutation}
\end{figure}

\begin{remark}\label{rmk:choices_Dem_triang_to_weave}
As stated in the proof above, note that there are several ways to fill a triangle of a Demazure triangulation with a (piece of a) weave. For instance, if all sides are labeled with the same permutation $121$, we have the two options in Figure \ref{fig: triangle mutation} (right), which are mutation equivalent but {\it not} weave-equivalent. In Figure \ref{fig:CorrectRules} we depict (some) possible pieces of weaves that can appear in triangles of a Demazure triangulation with only 2 colors, i.e. for permutations in $s_1,s_2\in S_3$.
\end{remark}

\begin{center}
\begin{figure}[h!]
\includegraphics[scale=0.5]{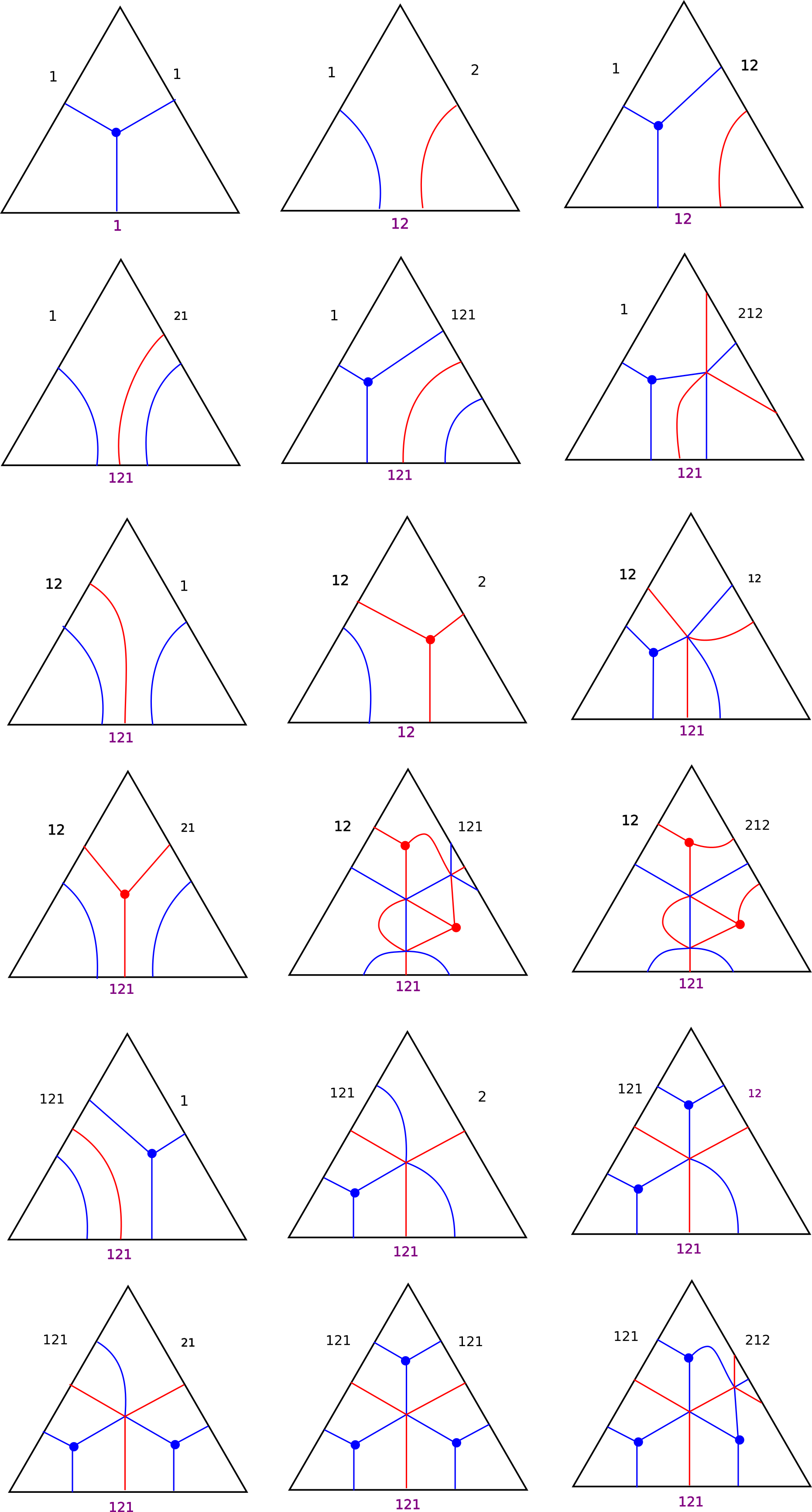}
\caption{Possible weaves associated to triangles of a Demazure triangulation as in the proof of Proposition \ref{prop:Demazure triangulation}, using only two colors $s_1,s_2\in S_3$. Note that fourth (i.e. the first in the second row) and seventh diagrams coincide as (pieces of) weaves; and the same is true for the sixth, ninth and fourteenth diagrams.}
\label{fig:CorrectRules}
\end{figure}
\end{center}

\subsubsection{Relation between Demazure triangulations and admissible triangulations}\label{sssec:triang_relations}
Demazure triangulations, discussed in Subsection \ref{sec: demazure}, relate to admissible triangulations, discussed in \ref{sssec:admissible_triang}, as follows.

To ease notation, we say that a weave is compatible with a Demazure triangulation $\underline{\tau}$ if it can be obtained from $\underline{\tau}$ using the construction in the proof of Proposition \ref{prop:Demazure triangulation}. Note that there are typically different (non-equivalent) weaves compatible with the same Demazure triangulation $\underline{\tau}$, cf.~Remark \ref{rmk:choices_Dem_triang_to_weave}. Similarly, a weave is said to compatible with an admissible triangulation $\tau$ if it can be obtained from $\tau$ using the construction in Subsection \ref{sssec:admissible_triang}. In this case, any two weaves compatible with the same admissible triangulation $\tau$ are weave equivalent.

Let $\beta$ a positive braid word with Demazure product $w_0$ and fix a reduced expression $\Delta$ for $w_0$ as in Proposition~\ref{prop:Demazure triangulation}. 
Consider a Demazure triangulation $\underline{\tau}$ of $\beta$. This is a triangulation $\underline{\tau}$ of a polygon with $l(\beta)+1$ sides, where the summand $1$ accounts for the edge that is labeled with $w_0$. In order to relate it to an admissible triangulation for $\beta$, which triangulates a polygon with $l(\beta)+l(w_0)$ sides, we need to account for this difference in the number of sides. This is achieved as follows:

\begin{definition}\label{def:w0_expansion}
Let $\underline{\tau}(\beta)$ be Demazure triangulation of $\beta$, $\Delta$ a reduced expression for $w_0$ and $\underline{\tau}(\Delta^{op})$ a Demazure triangulation for $\Delta^{op}$. By definition, a $\Delta$-expansion of $\underline{\tau}(\beta)$ is the triangulation of a polygon with $l(\beta)+l(w_0)$ sides obtained by gluing $\underline{\tau}(\beta)$ and $\underline{\tau}(\Delta^{op})$ along their (correspondingly unique) $w_0$-edges. By definition, a $w_0$-expansion is a $\Delta$-expansion, for some (unspecified) reduced expression $\Delta$ of $w_0$.
\end{definition}
We refer to any such triangulation, as in Definition \ref{def:w0_expansion} for some choice of reduced expression for $\Delta$, as a $w_0$-expansion of $\underline{\tau}(\beta)$.

\begin{prop}\label{prop:subdivide}
Let $\beta$ be a positive braid word. Any $w_0$-expansion of a Demazure triangulation $\underline{\tau}(\beta)$ for $\beta$ can be subdivided and oriented to obtain an admissible triangulation $\tau(\beta)$ for $\beta$. In addition, any weave compatible with $\tau(\beta)$ is compatible with $\underline{\tau}(\beta)$.
\end{prop}

\begin{proof}
Consider a triangle in 
a Demazure triangulation with sides $u,v$ and $u\star v=\delta(uv)$. If $u\star v=uv$, then this triangle is admissible. Otherwise, we use induction in $\ell(v)$. Let $v_1$ be the longest prefix of $v$ such that $uv_1$ is reduced. Then we can write $v=v_1sv_2$ such that $\ell(uv_1)=\ell(u)+\ell(v_1)$ but $\ell(uv_1s)<\ell(u)+\ell(v_1)+1$. Therefore we can find a reduced expression $w$ such that $uv_1=ws$, and draw the following diagram:

\begin{center}
\begin{tikzpicture}
\filldraw [lightgray] (0,0)--(3,2/3)--(4,0)--(0,0);
\draw (0,0)--(2,4)--(4,0)--(0,0);
\draw (1,2/3)--(2,8/3)--(3,2/3)--(1,2/3);
\draw (0,0)--(1,2/3);
\draw (0,0)--(2,8/3)--(2,4);
\draw (0,0)--(3,2/3)--(4,0)--(2,8/3);
\draw (1,2.2) node {$u$};
\draw (3,2.2) node {$v$};
\draw (0.5,1/3) node {$w$};
\draw (2,0.75) node {$s$};
\draw (1.5,1.7) node {$s$};
\draw (2.5,1.7) node {$s$};
\draw (2.2,3) node {$v_1$};
\draw (3.5,1/3) node {$v_2$};
\draw (2,-0.2) node {$u\star v$};
\end{tikzpicture}
\end{center}

The unmarked edges are labeled by $ws=uv_1, ws$ and $sv_2$. Now
$$
u\star v=u\star v_1\star s\star v_2=w\star s\star s\star v_2=ws\star v_2,
$$
and by the assumption of induction we can subdivide the marked triangle with sides $ws$ and $v_2$ into admissible ones. In this manner, we subdivide each triangle in an arbitrary Demazure triangulation into admissible triangles. Therefore, we can subdivide each $w_0$-expansion of $\underline{\tau}(\beta)$ (recall that it is glued out of two Demazure triangulations) into a triangulation consisting of admissible triangles, which is then admissible by definition. Let us denote this admissible triangulation by $\tau(\beta)$.

For the statement about the weave, let us describe an arbitrary weave compatible with $\tau(\beta)$ as a sequence of braid words. First, given a subdivided triangle as above, we choose some word for $u, v_1,v_2$ and use a sequence of braid relations 
$$
uv\to uv_1sv_2\to wssv_2
$$
Next, we insert a trivalent vertex in the central triangle and get $wsv_2$, and proceed by induction. This yields one possible sequence of moves computing the Demazure product $u\star v$. This gives a Demazure weave for each triangle in $\underline{\tau}(\beta)$. Gluing them together, we get a Demazure weave $\ww_1$ compatible with $\underline{\tau}(\beta)$, as in the proof of Proposition~\ref{prop:Demazure triangulation}. Doing the same for each triangle in our chosen Demazure triangulation for $\Delta^{op}$, where we fixed the word 
$\Delta^{op}$
on the bottom side, we obtain a Demazure weave $\ww_2$ from $\Delta^{op}$ to 
itself. Now we can glue these two weaves $\ww_1,\ww_2$ together along their $w_0$-edges  and declare that all the consecutive edges that spell $\Delta^{op}$, i.e.~all the edges from $\ww_2$ except its $w_0$-edge, are to be considered as one edge (indeed, since the polygons for $\beta$ and for $\Delta^{op}$ have opposite orientations, the gluing can in fact be interpreted as the concatenation of $\ww_1$ with the half-turn of $\ww_2$, the latter being a Demazure weave from $\Delta$ to itself). By labeling this particular edge with $w_0$, the result of gluing $\ww_1$ and $\ww_2$ and performing this identification gives a weave compatible with $\underline{\tau}(\beta)$. Indeed, it is compatible with $\underline{\tau}(\beta)$ because $\ww_2$ consists only of 4- and 6-valent vertices. In fact, this resulting weave is equivalent to $\ww_1$ by Theorem~\ref{thm: ben}.(a).
\end{proof}


\section{Algebraic Weaves, Morphisms, and Correspondences}
\label{sec: alg weaves}

This section develops the relative geometry of braid varieties, studying morphisms and correspondences between them. These correspondences are defined using weaves, and provide a functor from the category of algebraic weaves to the category of algebraic varieties and their correspondences.  Here and below, a correspondence between $X$ and $Y$ is an algebraic variety $Z$ with two regular maps $Z\to X$ and $Z\to Y$. In general, we do not require that $Z$ is a subset of $X\times Y$. That said, this stronger condition does hold for correspondences associated with simplifying weaves, see Remark~\ref{rem:correspondences}.


\subsection{Correspondences}

In this section, we use horizontal {\it yellow segments} in order to keep track of certain variables, corresponding to the $z_i$-variables in the braid variety. By definition, a horizontal segment inside the domain $\R\times[1,2]$, where a weave is drawn, is any connected segment contained in a line of the form $\R\times\{r\}$, for some real value $r\in[1,2]$. In addition, we also consider particular types of weaves. Altogether, this leads to the following definition.

\begin{definition}\label{def:algebraic_weave} An \emph{algebraic weave} of degree $n$ is a sliced weave $\ww\sse \R\times[1,2]$ of degree $n$ such that:
	\begin{itemize}
		\item[(i)] The edges have been oriented downwards, with the models according to Figure \ref{fig:WeaveModels2} for cups and caps. By convention, diagrams are oriented from top to bottom, from $\R\times\{2\}$ down to $\R\times\{1\}$.\\
		
		\item[(ii)] The weave $\ww$ is decorated with horizontal yellow rays, as follows. By definition, \emph{yellow rays} are horizontal rays of the form $(-\infty,b]\times\{r\}\sse\R\times[1,2]$, for some $b\in\R$ and $r\in(1,2)$, such that the yellow ray starts at a trivalent vertex, or at the bottom of a cup, or at the top of a cap. The first three diagrams in Figure \ref{fig:WeaveModels2}, excluding the rightmost picture, depict the three possible starts of a yellow ray. In other words, the starting point $(b,r)\in\R\times[1,2]$ of a yellow ray must either be a trivalent vertex, the lowest point of a cup or the highest point of a cap.\\

        \item[(iii)] The weave $\ww$ is such that any horizontal line $\R\times\{h\}$, for some $h\in[1,2]$, contains at most one of the following types of points: a vertex of $\ww$, the lowest point of a cup or the highest point of a cap.\\

        \noindent In particular, all vertices, cups and caps of $\ww$ have different heights, and yellow rays never pass through another vertex, cup or cap (in addition to the starting point), yellow rays are all parallel to each other and are all transverse to the edges of $\ww$. Therefore, the only local models involving an intersection between a yellow ray and $\ww$ are as depicted in Figure \ref{fig:WeaveModels2}.
  
\end{itemize}      

By definition, a (transverse) intersection point of a yellow ray with a weave edge, distinct from the starting point of the yellow ray, will be referred to as a {\it virtual vertex}. A virtual vertex is drawn in the rightmost diagram of Figure \ref{fig:WeaveModels2}.
\end{definition}

\begin{remark}
Throughout this section, we refer to the valency of a vertex in the original weave $\ww$, without accounting for any additional valency due to yellow rays. In particular, trivalent vertices will be still called trivalent despite an additional edge starting at them.
\end{remark}

\begin{center}
	\begin{figure}[h!]
		\centering
		\includegraphics[scale=0.7]{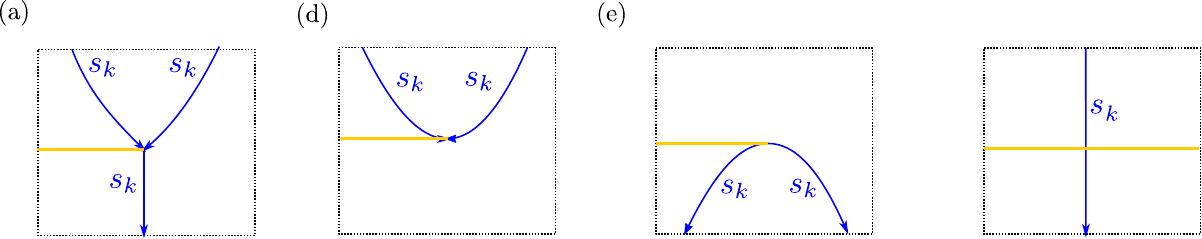}
		\caption{Local models for algebraic weaves, compare with Figure \ref{fig:WeaveModels}. The starting point of a yellow ray must be at a trivalent vertex, a cup or a cap. These cases are respectively labeled with $(a)$, $(d)$ and $(e)$. A virtual vertex is shown on the right.}
		\label{fig:WeaveModels2}
	\end{figure}
\end{center}

By Definition \ref{def:algebraic_weave}, the new local models for algebraic weaves involving yellow rays are those in Figure \ref{fig:WeaveModels2}. The weave edges of the original weave $\ww$ are subdivided by the virtual vertices into smaller segments, and yellow rays are subdivided into intervals, which we often refer to as {\it yellow segments}.\\

\begin{center}
\begin{figure}
    \centering
    \includegraphics{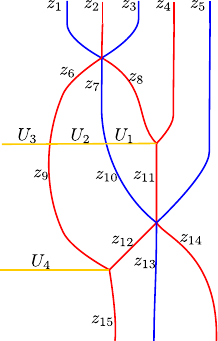}
    \caption{The space $\mathbb{V}^{\ww}$ for this algebraic weave $\ww$ is $\C^{15} \times \left(\C^{\binom{3}{2}} \times (\C^{\ast})^{3}\right)^{4}$. Indeed, there are 15 weave segments, each labeled with a variable $z_i\in\C$, and 4 yellow segments, each labeled with an invertible triangular matrix $U_j$. The correspondence $\mathcal{M}(\ww)$ associated to this algebraic weave is a closed subvariety of $\mathbb{V}^{\ww}$.}
    \label{fig:affine space weave}
\end{figure}
\end{center}

Let 
$\ww:\beta_2\to\beta_1$
be an algebraic weave from 
$\beta_2$, at the top, to 
$\beta_1$
, at the bottom. We now construct a correspondence between the two braid varieties $X_0(\beta_1)$ and $X_0(\beta_2)$.
To each segment of an edge labeled by $i$ we associate a variable $z$ and the braid matrix $B_i(z)$. Segments of a weave edge separated by a virtual vertex carry different variables. In addition, each yellow segment (a segment of a yellow ray) is labeled by an invertible upper triangular matrix whose entries are considered variables as well. All these variables and matrices can be considered as coordinates in the space
$$\mathbb{V}^{\ww}:=\C^{\mathrm{weave\ segments}}\times \left(\C^{\binom{n}{2}}  \times (\C^{\ast})^{n}\right)^{\mathrm{yellow\ segments}},$$
where we are identifying the space of invertible upper triangular $n \times n$-matrices with $\C^{\binom{n}{2}} \times (\C^{\ast})^{n}$. See Figure \ref{fig:affine space weave} for an example. The correspondence associated to an algebraic weave $\ww$ is a closed subvariety of $\mathbb{V}^{\ww}$. This correspondence is defined using the following notion of monodromy.

\begin{definition}\label{def:monodromy}
Let $\ww$ be an algebraic weave and $\tau:[0,1]\to \R\times[1,2]$ a regular parametrization of an oriented embedded path transverse to both $\ww$ and its yellow rays. By definition, the {\em monodromy} of the weave $\ww$ along $\tau$, also referred to as the monodromy of $\tau$, is the ordered product 
of the following matrices:
\begin{itemize}
	\item[(i)] $B_i(z)$, if the path $\tau$ crosses an edge labeled by $i$ with variable $z$ {\em from left to right},
	\item[(ii)] $B_i(z)^{-1}$, if the path $\tau$ crosses an edge labeled by $i$ with variable $z$ {\em from right to left},
	\item[(iii)] $U$, if the path $\tau$ crosses a yellow segment colored by $U$ {\em from top to bottom},
	\item[(iv)] $U^{-1}$, if the path $\tau$ crosses a yellow segment colored by $U$ {\em from bottom to top}.
\end{itemize}
In detail, let $\{t_1,\ldots,t_f\}\in[0,1]$ with $t_1<\ldots<t_f$ be the times such that $\tau(t_i)$ intersects either the weave or a yellow ray, and let $M(t_i)$ be the matrix associated to that intersection point and its intersection sign, according to $(i)$ through $(iv)$ above. Then the monodromy of $\tau$ is the product $M(t_1)\cdot\ldots\cdot M(t_f)$.
\end{definition}

\begin{definition}\label{def:corrweave}
Let $\ww$ be an algebraic weave. By definition, the correspondence variety $\SM(\ww)$ associated to the weave $\ww$ is the affine algebraic subvariety of $\mathbb{V}^{\ww}$
cut out by the following two conditions:
\begin{enumerate}
    \item The monodromy of a closed loop around a neighborhood of every vertex of $\ww$ is the identity.

    \item The monodromy of a closed loop around a neighborhood of every virtual vertex is the identity.
\end{enumerate}
\end{definition}

\noindent The two conditions $(1)$ and $(2)$ in Definition \ref{def:corrweave} can be written in terms of polynomial equations on the $z$-variables and the coefficients of the invertible upper-triangular matrices. Therefore $\SM(\ww)\sse\mathbb{V}^{\ww}$ is a closed affine algebraic subvariety.

\subsection{Properties of correspondences} From the ambient space $\mathbb{V}^{\ww}$ associated to $\ww:\beta_2\to\beta_1$
, we have two natural projections

\[
\C^{\ell(\beta_1)}\leftarrow \mathbb{V}^{\ww}\rightarrow \C^{\ell(\beta_2)}.
\]

The projection 
$\mathbb{V}^{\ww}\rightarrow\C^{\ell(\beta_2)}$
is given by reading the labels $z_j$ associated to the weave segments at the top boundary of $\ww$, corresponding to crossings of 
$\beta_2$.
Similarly, the projection 
$\mathbb{V}^{\ww}\rightarrow\C^{\ell(\beta_1)}$
is given by reading the labels $z_j$ associated to the weave segments at the bottom boundary of $\ww$, which correspond to crossings of 
$\beta_1$.
By considering the braid matrices associated to 
$\beta_1$ and $\beta_2$,
we obtain the corresponding maps
\[
\GL(n)\xleftarrow{B_{\beta_1}} \mathbb{V}^{\ww}\xrightarrow{B_{\beta_2}}  \GL(n).
\]
These two maps can be thought of as monodromies along the left to right horizontal paths near the top boundary of $\ww$, in the case of 
$B_{\beta_2}$,
and near the bottom boundary of $\ww$, for 
$B_{\beta_1}$.
More generally, if a horizontal slice of $\ww$ spells out a braid word $\beta$, the corresponding braid matrix defines a map $B_{\beta}:\mathbb{V}^{\ww}\to \GL(n)$.

\begin{definition}
Let 
$\ww:\beta_2\to\beta_1$
be an algebraic weave of degree $n$ and $\pi \in S_n$ a permutation. 
The closed subvariety $
\SM(\ww, \pi) \subseteq \SM(\ww)\subseteq \mathbb{V}^{\ww}$ is given by the condition that 
$B_{\beta_1}\pi$
is upper triangular. 
There exists a natural map 
$\SM(\ww, \pi) \to X_0(\beta_1, \pi)$
given by projecting to the labels associated to the bottom boundary segments of $\ww$, which spell 
$\beta_1$.
\end{definition}

\begin{prop}\label{prop:stacking-correspondences} Let $\ww_1:\beta_1\to\beta_0$ and $\ww_2:\beta_2\to\beta_1$ be algebraic weaves and 
$\ww:=\ww_1\circ\ww_2$
their composition. Then the following hold:

\begin{itemize}
    \item[(a)] Let $\pi\in S_n$ be a permutation, which we represent by a homonymous permutation matrix $\pi\in\GL_n(\C)$. Suppose that the matrix $B_{\beta_1}\cdot\pi$ is upper-triangular. Then $B_{\beta_2}\cdot\pi$ is upper-triangular and $\SM(\ww_2,\pi)$ is a correspondence between
$X_0(\beta_1; \pi)$ and $X_0(\beta_2; \pi)$.\\

    \item[(b)] The composition of weaves corresponds to the following diagram:
$$ 
\begin{tikzcd}
 & & \SM(\ww,\pi) \arrow{dl} \arrow{dr} & & \\
 & \SM(\ww_1, \pi) \arrow{dl} \arrow{dr}  & & \SM(\ww_2, \pi) \arrow{dl} \arrow{dr}  & \\
X(\beta_0; \pi) &  & X(\beta_1; \pi) & & X(\beta_2; \pi),
\end{tikzcd} 
$$
In addition, the middle square is Cartesian. In other words, $\SM(\ww, \pi)$ is a convolution of correspondences $\SM(\ww_1,\pi)$ and $\SM(\ww_2, \pi)$.\\
\end{itemize}
\end{prop}

\begin{proof}
For Part (a), Definition \ref{def:corrweave} implies that the monodromy around any closed loop is the identity. The monodromy around the closed loop encircling the whole weave with $\beta_2$ on the top and $\beta_1$ on the bottom must then be the identity. The monodromy around this particular loop equals $B_{\beta_2}B_{\beta_1}^{-1}\widetilde{U}^{-1}$, where $\widetilde{U}$ is the product of the upper-triangular matrices assigned to the yellow segments to the left of $\ww$. Therefore we have the equality $B_{\beta_2}=\widetilde{U}B_{\beta_1}$. By the discussion above, projecting to the labels associated to the top boundary of $\ww$ defines a map $\SM(\ww, \pi) \to X_0(\beta_2, \pi)$.\\

\noindent For (b), recall that the composition 
$\ww=\ww_1\circ\ww_2$
of weaves is defined by vertical stacking, with $\ww_2$ on top and $\ww_1$ at the bottom. 
Therefore, we can concatenate the labels in $\mathbb{V}^{\ww_1}$ and $\mathbb{V}^{\ww_2}$ if they agree along $\beta_1$. In that case, there are natural maps $\SM(\ww)\lr\SM(\ww_i)$, $i=1,2$, given by restriction, because the monodromy conditions in $\ww_1$ and $\ww_2$ are independent. The $z$-variables for labels along $\beta_1$ form a space $\C^{l(\beta_1)}$ and restriction to the top, resp.~bottom, gives a map $\SM(\ww_1)\lr\C^{l(\beta_1)}$, resp.~$\SM(\ww_2)\lr\C^{l(\beta_1)}$. It follows from above that we obtain a Cartesian square:

\begin{center}
\begin{tikzcd}
 & \SM(\ww) \arrow{dl} \arrow{dr} & \\
\SM(\ww_1) \arrow{dr}& & \SM(\ww_2) \arrow{dl}.\\
 & \C^{\ell(\beta_1)} & \\
\end{tikzcd}
\end{center}

By Part (a), the condition that $B_{\beta_0}\pi$ is upper-triangular implies that both 
$B_{\beta_1}\pi$ and $B_{\beta_2}\pi$ are upper-triangular, and thus we also have the same Cartesian diagram now incorporating $\pi$.
\end{proof}

\begin{remark}
Note that flipping an algebraic weave $\ww:\beta_2\to\beta_1$ upside down and reversing orientations on the edges corresponds to switching $\beta_1$ and $\beta_2$ and transposing the associated correspondence.
\end{remark}

Proposition \ref{prop:stacking-correspondences}.(a) for the case where $\pi = 1$ is the identity gives a correspondence $\SM(\ww, 1)$ between the braid varieties $X_{0}(\beta_2)$ and $X_{0}(\beta_1)$. Note that $\SM(\ww, 1)$ is a closed subvariety of $\SM(\ww)$ and, in general, $\SM(\ww, 1) \neq \SM(\ww)$. By Proposition \ref{prop:stacking-correspondences}.(b), it suffices to describe these correspondences for elementary weaves in order to understand them for general algebraic weaves. These correspondences, in the case of elementary weaves, are described as follows:
\begin{itemize}
\item[(1)] For a trivalent vertex colored by $i$, the correspondence $\SM(\ww,\pi)$ embeds into $X(\beta_2;\pi)$ as the open locus $\{z_{1} \neq 0\}$
and projects onto $X(\beta_1; \pi)$ with fibers $\P^1\setminus\{0,\infty\}=\C^*$. In terms of matrices,
we have the identity
$$
B_i(z_1)B_i(z_2)=\left(\begin{matrix}-z_1^{-1} & 1 \\ 0 & z_1\end{matrix}\right)B_i(z_2+z_1^{-1}).
$$

\item[(2)] For 6-valent and 4-valent vertices, the corresponding braid varieties $X(\beta_2; \pi)$ and $X(\beta_1; \pi)$ are isomorphic, and $\SM(\ww, \pi)$ realizes this isomorphism. In terms of matrices, this corresponds to the identities
$$
 B_i(z_1)B_{i+1}(z_2)B_i(z_3)=B_{i+1}(z_3)B_i(z_2-z_1z_3)B_{i+1}(z_1),
$$
$$
B_i(z_1)B_j(z_2)=B_j(z_2)B_i(z_1)\ (|i-j|>1).
$$

\item[(3)] For a cup colored by $i$, the correspondence $\SM(\ww, \pi)$ embeds into $X(\beta_2; \pi)$ as the closed locus $\{z_{1} = 0\}$ 
and projects onto $X(\beta_1; \pi)$ with fibers $\P^1\setminus\{\infty\}=\C$. In terms of matrices,
we have the identity
$$
B_i(0)B_i(z)=\left(\begin{matrix} 1 & z\\ 0 & 1\\ \end{matrix}\right).
$$
For a cap, we just use the transposed correspondence.

\item[(4)] A virtual vertex corresponds to the identity $$B_i(z)U=\widetilde{U}B_i(z')$$ from Lemma \ref{lem:slide triangular}. In particular, we have $z' =\frac{u_{i+1,i+1}z+u_{i,i+1}}{u_{i,i}}$. Here $U$ and $\widetilde{U}$ are the labels for the segments of the yellow ray to the right and to the left of the virtual vertex, respectively:
\begin{center}
\begin{tikzpicture}[scale=0.8]
\draw  (1,3)--(1,0);
\draw [Dandelion,line width=0.9] (-1,1.5)--(3,1.5);
\draw (1.2,2.5) node {\scriptsize $z$};
\draw (1.2,0.5) node {\scriptsize $z'$};
\draw (0,1.75) node {\scriptsize $\widetilde{U}$};
\draw (2.5,1.7) node {\scriptsize $U$};
\end{tikzpicture}
\end{center}
\end{itemize}

These four rules are justified by the following result.

\begin{prop}\label{prop:correspondence} In the construction of the correspondence variety $\SM(\ww)$:
\begin{itemize}
	\item[(a)] The invertible triangular matrices labeling yellow segments are uniquely determined by the variables on the edges.
	
	\item[(b)] The output variables of each 3-, 6-, or 4-valent vertex are determined by the input variables.
\end{itemize} 
\end{prop}

\begin{proof}
It follows from the proof of Lemma \ref{lem:slide triangular} that the equation $B_i(z)U=\widetilde{U}B_i(z')$ uniquely determines $\widetilde{U}$ and $z'$ for given $z$ and $U$. This establishes Part (a) near a virtual vertex. It remains to consider the yellow segments near trivalent vertices, cups, and caps. We verify both Part (a) and Part (b) in the necessary cases, as follows:

\begin{enumerate}
    \item For a 6-valent vertex, we have that $B_i(z_1)B_{i+1}(z_2)B_i(z_3)=B_{i+1}(w_1)B_i(w_2)B_{i+1}(w_3)$ implies $w_1=z_3, w_2=z_2-z_1z_3, w_3=z_1$, so the output variables are determined by the input ones. The proof for a 4-valent vertex is similar. Note that there are no yellow segments in this case of 4- and 6-valent vertices, so it is only to do with Part (b).\\

    \item For a 3-valent vertex, we have an equation $B_i(z_1)B_i(z_2)=UB_i(w)$ which can be written as
$$
\left(\begin{matrix}1 & z_2\\ z_1 & 1+z_1z_2 \end{matrix} \right)=\left(\begin{matrix}0 & 1\\ 1 & z_1 \end{matrix} \right)\left(\begin{matrix}0 & 1\\ 1 & z_2 \end{matrix} \right)=\left(\begin{matrix}a & b\\ 0 & c \end{matrix} \right)\left(\begin{matrix}0 & 1\\ 1 & w \end{matrix} \right)=\left(\begin{matrix}b & a+bw\\ c & cw \end{matrix} \right).
$$ 
This equality implies $b=1,c=z_1,w=(1+z_1z_2)/c=z_2+z_1^{-1}$ and $a=z_2-bw=-z_1^{-1}$. In particular, $z_1$ must be nonzero.\\

    \item For a cup, we have $B_i(z_1)B_i(z_2)=U$ and similarly $z_1=0$ and $U$ is determined by $z_2$. The case of a cap follows analogously.
\end{enumerate}
\end{proof}

By combining these facts, we obtain the following result:

\begin{thm}\label{def:correspondence_simplifying_weave}
Let $\ww:\beta_2\to\beta_1$ be a simplifying algebraic weave with $m$ cups and $r$ trivalent vertices. Then:
\begin{enumerate}
    \item There exists an isomorphism $$\SM(\ww,\pi)\cong\C^{m}\times (\C^{*})^r\times X_0(\beta_1; \pi)$$
    such that the map $\SM(\ww,\pi)\to X_0(\beta_1,\pi)$ is given by the projection to the third factor.\\

    \item The map $\SM(\ww,\pi)\to X_0(\beta_2,\pi)$ is injective.
\end{enumerate}
\end{thm}

\begin{proof}
The map to $X_0(\beta_2;\pi)$ is injective by Proposition \ref{prop:correspondence}. This proves Part $(2)$. For Part $(1)$ we read our weave inductively from bottom to top. At the bottom, the bottom edges of $\ww$ spell the braid word $\beta_1$, and the corresponding $z$-variables parametrize a point in $X_0(\beta_1; \pi)$.
As we move up, we encounter the following cases:
\begin{itemize}
\item[(i)] If we cross a 6-valent vertex, similarly to Proposition \ref{prop:correspondence} the variables $z_1,z_2,z_3$ above the vertex are determined by the variables $w_1,w_2,w_3$ below it.\\

\item[(ii)] If we cross a 3-valent vertex $v$, we get an identity $B_i(z_1)B_i(z_2)=UB_i(z_3)$. We can choose $z_1\in \C^*$ arbitrarily, then by Proposition \ref{prop:correspondence} we have $z_2=z_3-z_1^{-1}$ and $U$ is determined by $z_1$ and $z_3$. The $z$-variables right below the yellow ray starting at $v$ and the matrix $U$ uniquely determine the $z$-variables above the yellow ray and the upper-triangular matrices on the yellow ray.\\

\item[(iii)] If we cross a cup, we get an identity $B_i(0)B_i(z_2)=U$. The choice of $z_2$ in $\C$ is arbitrary and, similarly to the previous case, $U$ propagates to the left in a unique way.\\

\item[(iv)] The case of a 4-valent vertex is immediate.
\end{itemize}

\noindent Therefore, each trivalent vertex contributes with a $\C^*$-factor, each cup contributes with a $\C$-factor and neither 4-valent nor 6-valent vertices contribute additional factors. This gives an isomorphism as in Part $(1)$ and, by construction, it satisfies that the map $\SM(\ww,\pi)\to X_0(\beta_1,\pi)$ is given by the projection to the third factor, projecting away the $\C$ and $\C^*$-factors from the cups and trivalents.
\end{proof}

\begin{remark}
\label{rem:correspondences}
If $\ww: \beta_2 \to \beta_1$ is a simplifying algebraic weave, then the maps $\SM(\ww, \pi) \to X_0(\beta_1; \pi)$ and $\SM(\ww, \pi) \to X_0(\beta_2; \pi)$ identify $\SM(\ww, \pi)$ with a subvariety of the product $X(\beta_1; \pi) \times X(\beta_2; \pi)$ and we obtain a correspondence in the sense of \cite{Manin68}.
\end{remark}

\begin{cor}
\label{cor: open}
Let $\ww:\beta_2\to\beta_1$ be a Demazure weave with $r$ trivalent vertices. Then $$\SM(\ww,\pi)=(\C^{*})^r\times X(\beta_1; \pi),$$ and the map $\SM(\ww,\pi)\to X(\beta_2; \pi)$ is an open embedding.
\end{cor}

Corollary \ref{cor: open} follows from \ref{def:correspondence_simplifying_weave} because we have $\ell(\beta_1)+r=\ell(\beta_2)$, and thus $\SM(\ww,\pi)$ and $X_0(\beta_2,\pi)$ have the same dimension. 

We now state the invariance of the correspondences $\SM(\ww)$ under weave equivalence, which will be proven in Section \ref{section:proof_thm_equivalent_weaves}:

 \begin{thm}
\label{thm: equivalent weaves}
Let $\ww_1,\ww_2$ be equivalent Demazure weaves between $\beta_2$ and $\beta_1$, i.e. $\ww_1,\ww_2$ are related by a sequence of elementary moves $($not mutations$)$. Then, their associated correspondences $\SM(\ww_1)$ and $\SM(\ww_2)$ are isomorphic. Furthermore, there exists such an isomorphism that induces an isomorphism between $\SM(\ww_1,\pi)$ and $\SM(\ww_2,\pi)$ for all permutations $\pi\in S_n$.
\end{thm}

\begin{remark}
It is shown in \cite{CZ} that two Legendrian weaves related by an elementary move (or compositions of thereof) yield Hamiltonian isotopic Lagrangian projections, and also yield the same maps between the corresponding Legendrian Contact DGAs. Theorem \ref{thm: equivalent weaves} is an algebraic analogue of this statement.
\end{remark}

\subsubsection{An aside on flag moduli} We could have followed \cite[Section 5]{CZ} and have also defined the following correspondence $\SM_{\OBS}(\ww)$, called the {\it flag moduli space} of $\ww$ in \cite[Section 5]{CZ}. This flag moduli is defined as follows. To each region of $(\R\times[1,2])\setminus\ww$ we associate a flag in $\C^n$, if $\ww$ goes between $n$-braids, and two regions separated by a line colored by $i$ have flags in relative position $s_i$. The flags separated by a yellow segment are required to coincide. Recall the definition of the open Bott-Samelson variety from Section \ref{sec: obs}, cf. Definition \ref{def:OBS}. There are two natural projections $\SM_{\OBS}(\ww)\to \OBS(\beta_0), \SM_{\OBS}(\ww)\to \OBS(\beta_1)$, so that $\SM_{\OBS}(\ww)$ is a correspondence between $\OBS(\beta_0)$ and $\OBS(\beta_1)$. We can also define $\SM_{\OBS'}(\ww) \subseteq \SM_{\OBS}(\ww)$ as the closed subvariety given by the additional condition that the flag corresponding to the unbounded region on the far left of the weave coincides with the flag corresponding to the unbounded region on the  far right. The variety $\SM_{\OBS'}(\ww)$ is a correspondence between $\OBS'(\beta_0)$ and $\OBS'(\beta_1)$. In this setting, in line with Theorem \ref{thm:OBS}, we can conclude the following.

\begin{prop}
\label{prop: MOBS}
Let $G = \GL(n)$ and $\SB \subseteq G$ the Borel subgroup of upper-triangular matrices. There is a free action of $\SB$ on $G \times \SM(\ww)$ that preserves $G \times \SM(\ww, 1)$ and we have isomorphisms

$$\SM_{\OBS}(\ww)\cong(G\times \SM(\ww))/\SB, \qquad \SM_{\OBS'}(\ww) \cong (G \times \SM(\ww, 1))/\SB.$$
\end{prop}

\begin{proof}
An element in $G=\GL(n,\C)$ corresponds to the choice of a basis in one of the regions on the plane. Given a point in $\SM(\ww)$, we can define a basis in every other region, and the trivial monodromy condition ensures that this assignment is well-defined. The flags in regions are induced by these bases. The action of $\SB$ changes the basis in the rightmost region, but does not affect the flag in it. Similarly to the proof of Theorem \ref{thm:OBS}, we can propagate this action to the left and obtain the required isomorphism. 
\end{proof}

\subsection{Opening crossings}

Let us now shift the focus to studying the relation between these correspondences and \emph{opening} crossings of a positive braid; the latter having been a crucial ingredient in Sections \ref{sec: braid varieties} and \ref{sec:symplecticform}. 

\begin{definition}
\label{def: open weave}
Let $\beta$ be a positive braid word on $n$ strands and $\sigma=\sigma_i$ a letter in $\beta$, and let $\beta'$ be the result of removing $\sigma$ from $\beta$. We define an equivalence class of Demazure weaves from $\beta\Delta$ to $\beta'\Delta$ as
the composition of the following three weaves:
\begin{itemize}
\item[(a)] Move $\Delta$ next to $\sigma_i$ and change the braid word for $w_0$
to one which starts from $\sigma_i$. This only uses braid relations, or, equivalently, 6- and 4-valent vertices,
\item[(b)] Apply the trivalent vertex $\sigma_i\sigma_i\to \sigma_i$,
\item[(c)] Move  
$\Delta$ 
back to the end of the word.
\end{itemize}
We will call any such weave an {\it opening weave} for $(\beta,\sigma)$. Any choice of braid relations in (a) and (c) yields equivalent weaves.
\end{definition}

Let us remark that the element $\Delta$ is {\it not} central in the braid group, and care is needed in Steps (a)  and (c) of Definition \ref{def: open weave}: if $\beta = \gamma_1 \sigma_i \gamma_2$ then the procedure in Definition \ref{def: open weave} is

$$
\gamma_1\sigma_i\gamma_2\Delta \to \gamma_1\sigma_i\Delta\gamma_2' \to \gamma_1\sigma_i\Delta'\gamma_2' \to \gamma_1\Delta'\gamma_2' \to \gamma_1\Delta\gamma_2' \to \gamma_1\gamma_2\Delta
$$

\noindent where $\Delta'$ is a minimal braid lift of a reduced expression of $w_0$ that starts with $\sigma_i$ (and is related to $\Delta$ by a sequence of braid moves), the opening of the crossing $\sigma_i$ is performed in the third step, and all other arrows only involve braid moves or, equivalently, $4$- and $6$-valent vertices. Let us now give a concrete example of this procedure.

\begin{ex}
(a) Suppose that $\beta=1212$ and we want to open the second crossing in $\beta\cdot \Delta=1212\underline{121}$. The above moves have the following form, where we have underlined $\Delta$ and $\Delta'$:
$$
1212\underline{121}=121\underline{212}1=12\underline{121}21\to 12\underline{212}21\to 1\underline{212}21\to 
1\underline{121}21=11\underline{212}1=112\underline{121}.
$$ 
The corresponding weave has the form
\begin{center}
\includegraphics[scale=0.2]{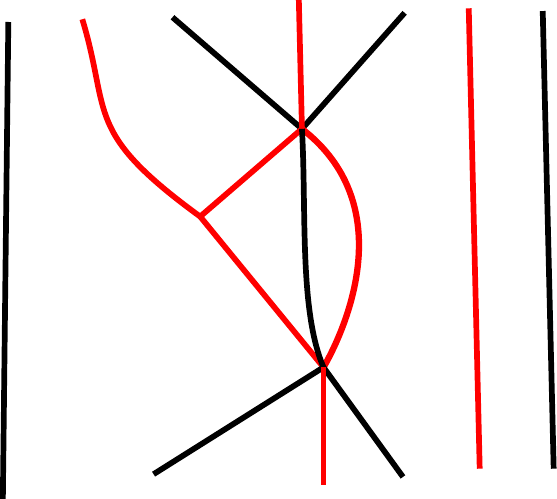}
\end{center}
(b) For another example, suppose that $\beta = 12112$ and we want to open the second crossing in $\beta\cdot\Delta = 12112\underline{121}$. The above moves have the following form, where we have underlined $\Delta$ and $\Delta'$:
\[
\begin{array}{l}
12112\underline{121} = 1211\underline{212}1 = 121\underline{121}21 \rightarrow 121\underline{212}21 = 12\underline{121}221 \rightarrow 12\underline{212}221 \rightarrow 1\underline{212}221 \rightarrow \\
\rightarrow 1\underline{121}221 = 11\underline{212}21 \rightarrow 11\underline{121}21 = 111\underline{212}1 = 1112\underline{121}.
\end{array}
\]
The corresponding weave has the form
\begin{center}
\includegraphics[scale=0.7]{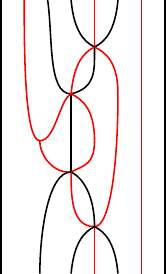}
\end{center}
From these examples, we can see that some steps required to move $\Delta$ next to $\sigma_i$ in Definition \ref{def: open weave} simply require us to use associativity of braid words, without using braid relations, and can be interpreted as the identity. These moves are marked with an equality sign in both examples above.
\end{ex}

\begin{lemma}\label{lem: open weave}  
Let $\sigma_i$ be a letter in $\beta$, and let $\ww$ be an opening weave for $(\beta,\sigma_i)$. Then the correspondence $\SM(\ww)$ agrees with the graph of the rational map $\Omega_{\sigma_i}$ from Definition \ref{def:openingcross}.
\end{lemma}

\begin{proof}
Observe that the trivalent vertex $\sigma_i\sigma_i\to \sigma_i$ corresponds to opening the {\em left} crossing $\sigma_i$. Indeed, applying Lemma \ref{lem: open one crossing matrices} yields a sequence of matrix identities
$$
B_i(z_1)B_i(z_2)=\left(\begin{matrix}-z_1^{-1} & 1 \\ 0 & z_1\end{matrix}\right)\left(\begin{matrix}1 & 0 \\ z_1^{-1} & 1\end{matrix}\right)B_i(z_2)=
\left(\begin{matrix}-z_1^{-1} & 1 \\ 0 & z_1\end{matrix}\right)B_i(z_2+z_1^{-1})
$$
followed by pushing the upper-triangular matrix to the left. This agrees with the correspondence associated to the trivalent vertex (see also the proof of Proposition \ref{prop:correspondence}). It is a direct verification that opening a crossing commutes with braid relations in (a) and (c) not involving this crossing, and the result follows. 
\end{proof}

As a result,  opening all crossings in a braid $\beta$, in some order, corresponds to a Demazure weave. Interestingly, the converse is also true, up to equivalence relation on weaves. 

\begin{thm}
\label{th: weave opening}
Let $\ww:\beta\Delta\to\Delta$ be a Demazure weave. Then $\ww$ is equivalent to a weave obtained by opening crossings in some order.
\end{thm}

\begin{proof}
Similarly to the proof of Theorem \ref{thm: one 3v}, any Demazure weave between braids $\beta$ and $\beta'$ such that $\ell(\beta)=\ell(\beta')+1$ is equivalent to a weave corresponding to opening a crossing in $\beta$ followed by some braid moves. Let us prove the statement of the theorem by induction on the length of $\beta$. When $\ell(\beta) = 0$, we have a weave from $\Delta$ to $\Delta$; since $\Delta$ is reduced, the cancellation relation in \ref{sec: cancellation} and the Zamolodchikov relation in Section \ref{sec: Zam} guarantee that all weaves $\Delta \to \Delta$ are equivalent to the identity weave. Given a weave from $\beta\Delta$ to $\Delta$, choose a slice 
$\beta'$ right below the first trivalent vertex. By the above argument the weave is equivalent to opening a crossing in $\beta$ (which results in a braid $\beta''\Delta$) followed by some braid moves to $\beta'$, and followed by the rest of the weave. By the assumption of induction, the weave from $\beta''\Delta$ to $\Delta$ is equivalent to opening crossings in $\beta''$ in some order.
\end{proof}



\begin{cor}
Let $\ww$ be a Demazure weave between $\beta\Delta$ and $\Delta$. Then the open chart $\SM(\ww, w_0)\hookrightarrow X_0(\beta\Delta,w_0)$ coincides with one of the toric charts from Section \ref{sec: opening}.
\end{cor}

\begin{proof}
By Theorem \ref{th: weave opening} the weave $\ww$ is equivalent to the weave $\ww'$ obtained by opening crossings in some order. By Theorem \ref{thm: equivalent weaves} the open charts in $X_0(\beta\Delta, w_0)$ corresponding to $\ww$ and $\ww'$ coincide.
\end{proof}


\subsection{Proof of Theorem \ref{thm: equivalent weaves}} \label{section:proof_thm_equivalent_weaves}

Let us prove Theorem \ref{thm: equivalent weaves}. In order to do so, we directly check each elementary move from Section \ref{sec: movies} separately. Cancellation of 4- and 6-valent vertices and commuting with distant colors are clear, and we do not include them in the list. Similarly, all the  ways to resolve 12121 are related to each other by a sequence of 1212-moves, as explained in Section \ref{sec: 12121}, so it is sufficient to check the latter. Below are the  remaining verifications needed for proving Theorem \ref{thm: equivalent weaves}.


\subsubsection{Changing the heights of vertices}
Changing the height of vertices does not change the the graph, but can change the yellow segments. 
Specifically, we need to understand how to slide yellow segments past 3-, 4- and 6-valent vertices, cups and caps. Here are the cases:

\begin{center}
\begin{tikzpicture}[scale=0.8]
\draw (0,4)--(1,2)--(2,4);
\draw  (1,2)--(1,0);
\draw [Dandelion,line width=0.8] (-1,3)--(3,3);
\draw [Dandelion,line width=0.8] (-1,2)--(1,2);
\draw (-0.2,4) node {\scriptsize $z_1$};
\draw (0.3,2.5) node {\scriptsize $s^{-1}z_1$};
\draw (2,2.5) node {\scriptsize $sz_2+k$};
\draw (2.2,4) node {\scriptsize $z_2$};
\draw (2.5,1) node {\scriptsize $s(z_2+z_1^{-1})+k$};
\draw (2.2, 3.2) node {\scriptsize $U$};
\draw (-0.2, 3.2) node {\scriptsize $U''$};
\draw (-0.5, 2.2) node {\scriptsize $A$};
\end{tikzpicture}
\qquad
\begin{tikzpicture}[scale=0.8]
\draw (0,4)--(1,2)--(2,4);
\draw  (1,2)--(1,0);
\draw [Dandelion,line width=0.8] (-1,1)--(3,1);
\draw [Dandelion,line width=0.8] (-1,2)--(1,2);
\draw (-0.2,4) node {\scriptsize $z_1$};
\draw (2.2,4) node {\scriptsize $z_2$};
\draw (2,1.5) node {\scriptsize $z_2+z_1^{-1}$};
\draw (2.5,0.5) node {\scriptsize $s(z_2+z_1^{-1})+k$};
\draw (-0.5, 2.2) node {\scriptsize $B$};
\draw (-0.5, 1.2) node {\scriptsize $U'$};
\draw (2.7, 1.2) node {\scriptsize $U$};
\end{tikzpicture}
\end{center}

The most interesting case is sliding through a 3-valent vertex. In this case we have identity
\begin{equation}
\label{eq: yellow 3v}
\left(\begin{matrix}
0 & 1\\
1 & z_1\\
\end{matrix}\right)
\left(\begin{matrix}
0 & 1\\
1 & z_2\\
\end{matrix}\right)
\left(\begin{matrix}
a & b\\
0 & c\\
\end{matrix}\right)=
\left(\begin{matrix}
0 & 1\\
1 & z_1\\
\end{matrix}\right)
\left(\begin{matrix}
c & 0\\
0 & a\\
\end{matrix}\right)
\left(\begin{matrix}
0 & 1\\
1 & \frac{b+cz_2}{a}\\
\end{matrix}\right)=
\left(\begin{matrix}
a & 0\\
0 & c\\
\end{matrix}\right)
\left(\begin{matrix}
0 & 1\\
1 & \frac{az_1}{c}\\
\end{matrix}\right)
\left(\begin{matrix}
0 & 1\\
1 & \frac{b+cz_2}{a}\\
\end{matrix}\right).
\end{equation}
Therefore we have a transformation $(z_1,z_2)\to (s^{-1}z_1,sz_2+k)$ where $s=\frac{c}{a}$ and $k=\frac{b}{a}$. Note that $z_1\neq 0$ is equivalent to $s^{-1}z_1\neq 0$. 
The transformation corresponding to the same yellow segment on the right figure sends 
$z_2+z_1^{-1}\to s(z_2+z_1^{-1})+k$, see the first of the identities \eqref{eq: yellow 3v} (or Lemma~\ref{lem:slide triangular}). On the left figure, we also obtain $s(z_2+z_1^{-1})+k$, now as the result of going down through the trivalent vertex: $(sz_2 + k) +  (s^{-1}z_1)^{-1} = s(z_2 + z_1^{-1}) + k$.

Let us now analyze the labels of the yellow segments on the far left, that we have indicated by $U''$ and $A$ on the left weave, and by $B$ and $U'$ on the right weave. In the left-hand side weave we have:
\[
B_i(z_1)B_i(z_2)U = U'' B_i(s^{-1}z_1)B_i(sz_2 + k) = U''AB_i(s(z_2 + z_1^{-1}) + k).
\]
In the right-hand side weave we have:
\[
B_i(z_1)B_2(z_2)U = BB_i(z_2 + z_1^{-1})U = BU'B_i(s(z_2 + z_1^{-1}) + k).
\]
Comparing, we obtain the equality $U''A = BU'$. It follows that, should there be more edges on the left of the weave, the labels of these edges remain constant below both yellow lines, as needed.

For a 6-valent vertex we have
$$
B_i(z_1)B_{i+1}(z_2)B_i(z_3)\left(
\begin{matrix}
a & b & c\\
0 & d & e\\
0 & 0 & f\\
\end{matrix}
\right)=\left(
\begin{matrix}
f & 0 & 0\\
0 & d & 0\\
0 & 0 & a\\
\end{matrix}
\right)B_i(\widetilde{z_1})B_{i+1}(\widetilde{z_2})B_i(\widetilde{z_3}),
$$
where 
$$
\widetilde{z_1}=\frac{1}{d}(e+z_1f),\ \widetilde{z_2}=\frac{1}{a}(c+z_3e+z_2f),\ \widetilde{z_3}=\frac{1}{a}(b+z_3d).
$$
Similarly,
$$
B_{i+1}(z'_1)B_{i}(z'_2)B_{i+1}(z'_3)\left(
\begin{matrix}
a & b & c\\
0 & d & e\\
0 & 0 & f\\
\end{matrix}
\right)=\left(
\begin{matrix}
f & 0 & 0\\
0 & d & 0\\
0 & 0 & a\\
\end{matrix}
\right)B_{i+1}(\widetilde{z_1}')B_{i}(\widetilde{z_2}')B_{i+1}(\widetilde{z_3}'),
$$
where 
$$
\widetilde{z_1}'=\frac{1}{a}(b+z'_1d),\ \widetilde{z_2}'
=\frac{1}{ad}(cd-be-z'_3bf+z'_2df),\ \widetilde{z_3}'=\frac{1}{d}(e+z'_3f).
$$
Now for $(z_1',z_2',z_3')=(z_3,z_2-z_1z_3,z_1)$ we get
$
(\widetilde{z_1}',\widetilde{z_2}',\widetilde{z_3}')=(\widetilde{z_3},\widetilde{z_2}-\widetilde{z_1}\widetilde{z_3},\widetilde{z_1}).
$ We show all these changes of variables in the following figure:
\begin{center}
\includegraphics[scale=1.35]{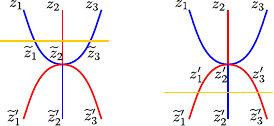}
\end{center}
We leave the check for 4-valent vertices to the reader.

For a cup, we can apply \eqref{eq: yellow 3v} to write $(z_1,z_2)\to (s^{-1}z_1,sz_2+k)$. Note that the cup is defined whenever $z_1=0$, which is equivalent to $s^{-1}z_1=0$, so we can still apply a cup below the yellow segment. We can also check the compatibility of the labels on the yellow segments left of the cup, as follows. If 
$$
B_i(0)B_i(z_2)=A,\ B_i(0)B_i(sz_2+k)=B
$$
for upper-triangular matrices $A,B$, then
we have 
$$
AU=B_i(0)B_i(z_2)U=U''B_i(0)B_i(sz_2+k)=U''B.
$$
The computation for a cap is similar.


\subsubsection{The 1212-relation}

We refer to the notations in Section \ref{sec: 1212}. Two weaves declared to be equivalent have one trivalent vertex each, so the corresponding algebraic weaves have one yellow segment each:

\begin{center}
\includegraphics[scale=0.3]{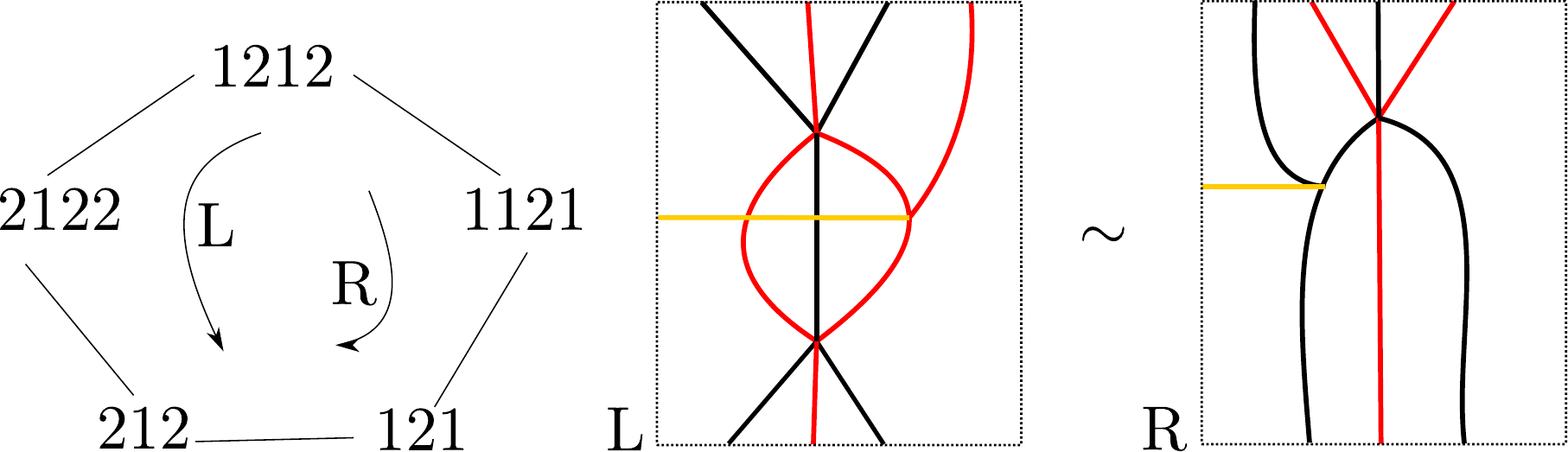}
\end{center}

The path $1212\to 2122\to 212\to 121$ on the left corresponds to  changes of variables
$$
(z_1,z_2,z_3,z_4)\to (z_3,z_2-z_1z_3,z_1,z_4)\to (z_2,-z_2z_1^{-1}+z_3,z_4+z_1^{-1})\to (z_4+z_1^{-1},z_3+z_2z_4,z_2).
$$
Note that the second step corresponds to opening third crossing, which affects all other crossings:
$$
B_2(z_3)B_1(z_2-z_1z_3)B_2(z_1)B_2(z_4)=B_2(z_3)B_1(z_2-z_1z_3)\left(\begin{matrix}
1 & 0 & 0\\
0 & -z_1^{-1} & 1\\
0 & 0 & z_1\\
\end{matrix}
\right)
\left(\begin{matrix}
1 & 0 & 0\\
0 & 1 & 0\\
0 & z_1^{-1} & 1\\
\end{matrix}
\right)B_2(z_4)=
$$
$$B_2(z_3)B_1(z_2-z_1z_3)\left(\begin{matrix}
1 & 0 & 0\\
0 & -z_1^{-1} & 1\\
0 & 0 & z_1\\
\end{matrix}
\right)
B_2({z_1}^{-1} + z_4)=
$$
$$
B_2(z_3)\left(\begin{matrix}
-z_1^{-1} & 0 & 1\\
0 & 1 & z_2-z_1z_3\\
0 & 0 & z_1\\
\end{matrix}
\right)B_1(-z_2z_1^{-1}+z_3)B_2(z_1^{-1}+z_4)=
$$
$$
\left(\begin{matrix}
-z_1^{-1} & 1 & 0\\
0 & z_1 & 0\\
0 & 0 & 1\\
\end{matrix}
\right)B_2(z_2)B_1(-z_2z_1^{-1}+z_3)B_2(z_1^{-1}+z_4).
$$

The yellow segment on the left weave is divided by edges of the weave into three segments corresponding to the upper triangular matrices appearing in this sequence of matrix identities. Namely, if $U, \widetilde{U}, \widetilde{\widetilde{U}}$ are the matrices corresponding to these segments, from right to left, then we have
\[
U = \left(\begin{matrix}
1 & 0 & 0\\
0 & -z_1^{-1} & 1\\
0 & 0 & z_1\\
\end{matrix}
\right), \, 
\widetilde{U} = \left(\begin{matrix}
-z_1^{-1} & 0 & 1\\
0 & 1 & z_2-z_1z_3\\
0 & 0 & z_1\\
\end{matrix}
\right), \,  \widetilde{\widetilde{U}} = \left(\begin{matrix}
-z_1^{-1} & 1 & 0\\
0 & z_1 & 0\\
0 & 0 & 1\\
\end{matrix}
\right).
\]

The right weave $1212\to 1121\to 121$ corresponds to changes of variables
$$
(z_1,z_2,z_3,z_4)\to (z_1,z_4,z_3+z_2z_4,z_2)\to (z_4+z_1^{-1},z_3+z_2z_4,z_2),
$$
so the end result is the same as for the sequence of transformations for the left weave.

For completeness, we also include the computation for some of the other diagrams in Section \ref{sec: 1212}. For 1121 we get two paths
$$
(z_1,z_2,z_3,z_4)\to (z_2+z_1^{-1},z_3,z_4)\to (z_4,z_3-z_4(z_2+z_1^{-1}),z_2+z_1^{-1})
$$
and
$$
(z_1,z_2,z_3,z_4)\to (z_1,z_4,z_3-z_2z_4,z_2)\to (z_3-z_2z_4,z_4-z_1(z_3-z_2z_4),z_1,z_2)\to
$$
$$
 (z_4,-z_1^{-1}(z_4-z_1(z_3-z_2z_4)),z_2+z_1^{-1})
$$
Note that 
$$
-z_1^{-1}(z_4-z_1(z_3-z_2z_4))=z_3-z_4(z_2+z_1^{-1}).
$$
For 1211 we get two paths
$$
(z_1,z_2,z_3,z_4)\to (z_2-z_1z_3,z_3^{-1}z_2,z_4+z_3^{-1})\to 
$$
$$
(z_4+z_3^{-1},z_3^{-1}z_2-(z_4+z_3^{-1})(z_2-z_1z_3),z_2-z_1z_3)
$$
and
$$
(z_1,z_2,z_3,z_4)\to (z_3,z_2-z_1z_3,z_1,z_4)\to (z_3,z_4,z_1-z_4(z_2-z_1z_3),z_2-z_1z_3)\to 
$$
$$(z_4+z_3^{-1},z_1-z_4(z_2-z_1z_3),z_2-z_1z_3).
$$
Note that
$$
z_3^{-1}z_2-(z_4+z_3^{-1})(z_2-z_1z_3)=z_3^{-1}z_2-z_4z_2+z_1z_3z_4-z_3^{-1}z_2+z_1=z_1-z_4(z_2-z_1z_3).
$$
The proof for other pair of adjacent colors is similar.


\subsubsection{The Zamolodchikov relation}

The left diagram in Section \ref{sec: Zam} represents the following path:
$$
123121\to 121321\to 212321\to 213231\to 231213\to 232123\to 323123
$$
which induces the following change of variables:
$$
(z_1,z_2,z_3,z_4,z_5,z_6)\to (z_1,z_2,z_4,z_3,z_5,z_6)\to (z_4,z_2-z_1z_4,z_1,z_3,z_5,z_6)\to $$
$$
 (z_4,z_2-z_1z_4,z_5,z_3-z_1z_5,z_1,z_6)\to
(z_4,z_5,z_2-z_1z_4,z_3-z_1z_5,z_6,z_1)\to 
$$
$$
(z_4,z_5,z_6,\widetilde{z_3},z_2-z_1z_4,z_1)\to  (z_6,z_5-z_4z_6,z_4,\widetilde{z_3},z_2-z_1z_4,z_1).
$$
Here $\widetilde{z_3}=z_3-z_1z_5-z_2z_6+z_1z_4z_6.$
The right diagram represents the following path:
$$
123121\to 123212\to 132312\to 312132\to 321232\to 321323\to 323123
$$
which induces the following change of variables:
$$
(z_1,z_2,z_3,z_4,z_5,z_6)\to (z_1,z_2,z_3,z_6,z_5-z_4z_6,z_4)\to (z_1,z_6,z_3-z_2z_6,z_2,z_5-z_4z_6,z_4)\to 
$$
$$
 (z_6,z_1,z_3-z_2z_6,z_5-z_4z_6,z_2,z_4)\to (z_6,z_5-z_4z_6,\widetilde{z_3},z_1,z_2,z_4)\to
$$
$$
(z_6,z_5-z_4z_6,\widetilde{z_3},z_4,z_2-z_1z_4,z_1)\to (z_6,z_5-z_4z_6,z_4,\widetilde{z_3},z_2-z_1z_4,z_1).
$$

This concludes the proof of Theorem \ref{thm: equivalent weaves}, as required. Hence, we have established invariance of the correspondences $\SM(\ww)$ under equivalence of Demazure weaves.



\subsection{Correspondences for simplifying weaves}

\begin{prop}
\label{prop: cups caps}
Let $\ww_1,\ww_2$ be two equivalent simplifying weaves. Then the associated correspondences $\mathcal{M}(\ww_1)$ and $\mathcal{M}(\ww_2)$ are isomorphic. Furthermore, there exists such an isomorphism that induces an isomorphism between $\SM(\ww_1,\pi)$ and $\SM(\ww_2,\pi)$ for all permutations $\pi\in S_n$.
\end{prop}

Proposition \ref{prop: cups caps} is proved in Sections \ref{sec: non-st_3v_alg}--\ref{sec: non-st_4v_alg} below by verifying that
the non-standard vertices yield well defined correspondences using the diagrams in Sections \ref{sec: nonstandard 3v}.(a), \ref{sec: nonstandard 6v}.(a), and \ref{sec: nonstandard 4v}.(a). That is, that equivalent weaves in these Sections give isomorphic correspondences.
Indeed, by definition of the equivalence for simplifying weaves, Proposition~\ref{prop: cups caps} then follows from Theorem~\ref{thm: equivalent weaves}.

\begin{remark}
Proposition \ref{prop: cups caps} can be also deduced from Proposition \ref{prop: MOBS}  as follows. The flag moduli space $\SM_{\OBS}(\ww)$ can be defined for any weave $\ww$ and is invariant under rotations. The general equivalence moves from Figure \ref{fig:cz equivalence} can be obtained by rotations of Demazure equivalence moves, and hence define isomorphic correspondences by Theorem \ref{thm: equivalent weaves}.
\end{remark}

\subsubsection{Non-standard trivalent vertex}
\label{sec: non-st_3v_alg}

Let us check that the two ways to define an upside down trivalent vertex in Section \ref{sec: nonstandard 3v}.(b) are equivalent. Indeed, the left picture corresponds to the changes of variables 
$$
(z)\to (0,u,z)\to (-u,z+u^{-1})
$$
while the right picture corresponds to
$$
(z)\to (z-w,0,w)\to ((z-w)^{-1},w).
$$
Here the cap on the left produces variables $(0,u)$ while the cap on the right produces variables $(0,w)$. We can identify the two diagrams by setting $w=z+u^{-1}$, $u=-(z-w)^{-1}$. 

Next, we compare two ways in Section \ref{sec: nonstandard 3v}.(a) corresponding to the paths $111\to 11\to \emptyset$. The left one corresponds to a sequence of changes of variables
$$
(z_1,z_2,z_3)\to (z_2+z_1^{-1},z_3)\to \emptyset,
$$
and the cup is well defined if $z_2+z_1^{-1}=0,$ that is, $1+z_1z_2=0$. The right diagram corresponds to 
$$
(z_1,z_2,z_3)\to (z_1+z_2^{-1},z_3+z_2^{-1})\to \emptyset,
$$
and the cup is well defined if $z_1+z_2^{-1}=0$, which leads to the same equation. Note that  $1+z_1z_2=0$ implies that both $z_1$ and $z_2$ are invertible, so that both trivalent vertices are well defined.

\subsubsection{Non-standard 6-valent vertex}
\label{sec: non-st_6v_alg}

Let us check the vertex with 5 inputs and one output from Section \ref{sec: nonstandard 6v}.(a).

One can check that in all weaves we require $z_1=z_2=0.$ Now the movie $12121\to 21221\to 211\to 2$  results in a sequence of changes of variables:
$$
(0,0,z_3,z_4,z_5)\to (z_3,0,0,z_4,z_5)\to (z_3,0,z_5)\to z_3,
$$
the movie $12121\to 11211\to 2$ yields
$$
(0,0,z_3,z_4,z_5)\to (0,z_4,z_3,0,z_5)\to z_3,
$$
and the movie $12121\to 12212\to 112\to 2$ yields
$$
(0,0,z_3,z_4,z_5)\to (0,0,z_5,z_4-z_3z_5,z_3)\to (0,z_4-z_3z_5,z_3)\to z_3.
$$

Now consider 4 inputs and 2 outputs, in both cases, we require $z_1=0$. For the movie
$1212\to 2122\to 21$ we get
$$
(0,z_2,z_3,z_4)\to (z_3,z_2,0,z_4)\to (z_3+z_2z_4,z_2),
$$
while the movie $1212\to 1121\to 21$ yields
$$
(0,z_2,z_3,z_4)\to (0,z_4,z_3+z_2z_4,z_2)\to (z_3+z_2z_4,z_2).
$$

\subsubsection{Non-standard 4-valent vertex}
\label{sec: non-st_4v_alg}

In both cases from Section \ref{sec: nonstandard 4v}.(a), we have $z_1=0$, and 
$131\to 311\to 3$ corresponds to $(0,z_2,z_3)\to (z_2,0,z_3)\to z_2$, while
$131\to 113\to 3$ corresponds to $(0,z_2,z_3)\to (0,z_3,z_2)\to z_2$.

This completes the proof of Proposition~\ref{prop: cups caps}. \hfill$\Box$

\subsubsection{Isotopies}

Finally, let us check the zigzag relation. On the left we have
$(z)\to (0,0,z)\to (z)$ while on the right we have 
$$
(z)\to (z-u,0,u)\to (u)
$$
which is well defined if $z-u=0$, so that $z=u$.

Since the weaves (and the associated correspondences) and the equivalence relations for non-standard 6- and 4-valent vertices from Sections \ref{sec: nonstandard 6v}.(b) and \ref{sec: nonstandard 4v}.(b) are reflections across the horizontal axis of those from Sections \ref{sec: nonstandard 6v}.(a) and \ref{sec: nonstandard 4v}.(a),  the calculations in Sections \ref{sec: non-st_6v_alg} and \ref{sec: non-st_4v_alg} show that such equivalent weaves also give isomorphic correspondences. Proposition \ref{prop: isotopy} then implies the following.

\begin{cor}
Let $\ww_1$ and $\ww_2$ be two planar isotopic weaves. Then the associated correspondences $\SM(\ww_1)$ and $\SM(\ww_2)$ are isomorphic.
\end{cor}



\subsection{Mutation equivalence and rational maps} \label{subsec:weqrat}

The previous subsections have discussed weave equivalence thoroughly. In this subsection, we address weave mutations. First, note that any Demazure weave $\ww$ from $\beta_2$ to $\beta_1$ defines a rational map $\Phi_{\ww}$ from $X_0(\beta_2,\pi)$ to $X_0(\beta_1,\pi)$, that is, the variables associated to crossings in $\beta_1$ can be expressed as rational functions in variables associated to crossings in $\beta_2$. This rational map  $\Phi_{\ww}$  is defined on the image of $\SM(\ww, \pi)$, but we can extend it to its maximal domain; we denote such extension by $\widehat{\Phi}_{\ww}$.

\begin{ex}
\label{ex: not weak eq}
The weave $(ss)s\to ss$ corresponds to the rational map $(z_1,z_2,z_3)\mapsto (z_2+z_1^{-1},z_3)$, while the weave $s(ss)\to ss$ corresponds to the rational map $(z_1,z_2,z_3)\mapsto (-z_2-z_1z_2^2,z_3+z_2^{-1})$.\hfill$\Box$
\end{ex}

Recall that two weaves are mutation equivalent if they are related by a sequence of equivalences and mutations. We now explain the natural relation between the maps associated to mutation equivalent weaves.

\begin{thm}
Let $\ww,\ww'$ be two weaves which are mutation equivalent. Then, the corresponding maximal extensions of rational functions  
$\widehat{\Phi}_{\ww}$,  $\widehat{\Phi}_{\ww'}$ coincide.
\end{thm}

\begin{proof}
By Theorem \ref{thm: equivalent weaves} the maps $\Phi_{\ww}$ and $\Phi_{\ww'}$ coincide for equivalent weaves even before mutations. Therefore it is sufficient to check mutations, using Example \ref{ex: not weak eq}. One of the trivalent graphs involved in a mutation corresponds to the rational map 
$$
(z_1,z_2,z_3)\mapsto z_3+(z_2+z_1^{-1})^{-1}=z_3+\frac{z_1}{1+z_1z_2}
$$
while the other corresponds to the rational map
$$
(z_1,z_2,z_3)\mapsto z_3+z_2^{-1}+(-z_2-z_1z_2^2)^{-1}=z_3+\frac{1}{z_2}-\frac{1}{z_2(1+z_1z_2)}=z_3+\frac{z_1}{1+z_1z_2}.
$$
Note that in the first case the map $\Phi_{\ww}$ is defined on the toric chart $\{z_1\neq 0,1+z_1z_2\neq 0\}$ while in the second case it is defined on the chart $\{z_2\neq 0,1+z_1z_2\neq 0\}$, but in both cases it extends to the locus $\{1+z_1z_2\neq 0\}$ and the extensions agree.
\end{proof}

\begin{remark}
Alternatively, we may state that the rational maps  $\Phi_{\ww}$ and $\Phi_{\ww'}$ agree on the intersection of their corresponding domains, hence their maximal extensions must agree too.
\end{remark}


\subsection{Torus actions and augmentation varieties} 
\label{seq: equivatiant weaves}
In this subsection, given a simplifying weave $\ww$ from $\beta_2$ to $\beta_1$, we will construct an action of the torus $T = (\C^*)^n/\C^*$ on the correspondence variety $\SM(\ww)$ so that for every $\pi \in S_n$ both projections $\SM(\ww, \pi) \to X_{0}(\beta_i; \pi)$, $i=1,2$, are $T$-equivariant. In particular, this allows us to define a 
correspondence between augmentation varieties by Theorem \ref{thm:aug vs braid}.

First, we modify the action of $T$ on $X_{0}(\beta;\pi)$ defined in Section \ref{sect:torus action} as follows. Take $\beta = \sigma_{i_{1}}\cdots \sigma_{i_{\ell}} \in \mbox{Br}^{+}_{n}$ and let $w \in S_{n}$ be its corresponding permutation. We define an action of $T$ on $\C^{\ell}$ by
\begin{equation}\label{eq:right torus action}
\bt\cdot_{\beta}(z_{1}, \dots, z_{\ell}) = (d_{1}z_{1}, \dots, d_{\ell}z_{\ell})
\end{equation}

\noindent where $d_{k} = t_{w_{k}^{\rho}(i_{k})}t^{-1}_{w_{k}^{\rho}(i_{k}+1)}$. Here, $w^{\rho}_{\ell - k} = s_{i_{\ell}}\cdots s_{i_{\ell-k+1}} =  (s_{i_{\ell - k+1}}\cdots s_{i_{\ell}})^{-1}$ 
(the superscript $\rho$ stands for \emph{right}, as we read the braid word $\beta$ right-to-left, as opposed to Section \ref{sect:torus action} above). Thanks to \eqref{eq: slide diagonal} we have that $B_{\beta}(\bt\cdot_{\beta}z) = D^{-1}_{w^{-1}(\bt)}B_{\beta}(z)D_{\bt}$, so for every permutation $\pi \in S_{n}$ we have an induced action on $X_{0}(\beta; \pi)$. 

\begin{ex}\label{ex:right torus action}
Let us take the braid word $\beta = \sigma_1\sigma_2\sigma_2\sigma_1\sigma_2$. If $\bt = (t_1, t_2, t_3)$ we have
\[
\bt\cdot_{\beta}(z_1, z_2, z_3, z_4, z_5) = \left(\frac{t_3}{t_1}z_1, \frac{t_2}{t_1}z_2, \frac{t_1}{t_2}z_3, \frac{t_1}{t_3}z_4, \frac{t_2}{t_3}z_5 \right).
\]
Comparing with Example \ref{ex:torus action} above, we see that the action we define here and that we define in Section \ref{sect:torus action} are different, in general. Note, however, that the two actions coincide up to the transposition $t_{2} \leftrightarrow t_{3}$.
\end{ex}
\begin{remark}
More generally, this torus action on $X_{0}(\beta;\pi)$ differs from the the action in Section \ref{sect:torus action} by conjugation by the permutation matrix $w$. 
The action used in Section \ref{sect:torus action} coincides with that considered in \cite{Mellit}, while the action used in this section behaves better under morphisms given by weaves, as we will see below.
\end{remark}

\begin{remark}
    Similarly to Remark \ref{rmk:z weights}, one can read the weight of the $z$-variables from the braid diagram $\beta$. Indeed, to find the weight of $z_{k}$ look at the strands that are incident to the $k$-th crossing of $\beta$ \emph{on the right} and follow them all the way \emph{to the right}. For example, the next figure computes that the weight of $z_3$ in Example \ref{ex:right torus action} above is  $\color{teal}{t_{1}}\color{black}{/}\color{red}{t_2}$.
    \begin{center}
        \includegraphics[scale=1]{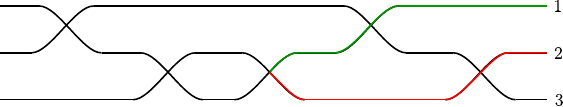}
    \end{center}
\end{remark}


\begin{lemma}\label{lemma:4-6 vertices}
Let $\gamma_{1}, \gamma_{2} \in \Br^{+}_{n}$ and denote $r := \ell(\gamma_{1})$. 
\begin{enumerate}
\item[(1)] Let $\beta_{2} =  \gamma_{1}\sigma_{i}\sigma_{i+1}\sigma_{i}\gamma_{2}$ and $\beta_{1} =\gamma_{1}\sigma_{i+1}\sigma_{i}\sigma_{i+1}\gamma_{2}$. Then, the map $$f: \C^{\ell(\beta_{2})} \to \C^{\ell(\beta_{1})},\quad f(z) = (z_{1}, \dots, z_{r}, z_{r+3}, z_{r+2}-z_{r+1}z_{r+3}, z_{r+1}, z_{r+4}, \dots, z_{\ell})$$ satisfies $f(\bt\cdot_{\beta_{2}}z) = \bt\cdot_{\beta_{1}}f(z)$.\\ 

\item[(2)] Let $\beta_{2} = \gamma_{1}\sigma_{i}\sigma_{j}\gamma_{2}$ and $\beta_{1} = \gamma_{1}\sigma_{j}\sigma_{i}\gamma_{2}$, where $|i - j| > 1$. Then, the map $$f: \C^{\ell(\beta_{2})} \to \C^{\ell(\beta_{1})},\quad f(z) = (z_{1}, \dots, z_{r}, z_{r+2}, z_{r+1}, z_{r+3}, \cdots, z_{\ell})$$ satisfies $f(\bt\cdot_{\beta_{2}}z) = \bt\cdot_{\beta_{1}}f(z)$. 
\end{enumerate}
\end{lemma}

\begin{proof}
This is verified by direct computation.
\end{proof}

If $D = \diag(a_{1}, \dots, a_{n})$ is a diagonal matrix and $w \in S_{n}$, we write $^{w}\!D:= \diag(a_{w^{-1}(1)}, \dots, a_{w^{-1}(n)})$. The following lemma will be used to construct a (well-defined) torus action.

\begin{lemma}\label{lemma:3-valent}
Let $R$ be a $\C$-algebra with a rational $\mathbb{T} = (\C^{*})^{n}$-action by algebra automorphisms. Let $w \in S_n$ be a permutation and $z \in R$ an element of weight $\be_{w(j)} - \be_{w(j+1)}$ for some $j = 1, \dots, n-1$. Let $z_{0} \in R$ be an invertible element of weight $\be_{i+1} - \be_{i}$ and define $z'$ by the equation

$$
B_{j}(z)\,^{w}\!D_{i}(z_{0}) = \, ^{ws_{j}}\!D_{i}(z_{0})B_{j}(z'),
$$

see \eqref{eq: slide diagonal}. Then, the weight of $z'$ is $\we{z'} = \be_{s_{i}w(j)} - \be_{s_{i}w(j+1)}$.
\end{lemma}
We remark that this Lemma, just as  Remark \ref{rmk:admissible}, is valid for an arbitrary rational $\mathbb{T}$-action on a $\C$-algebra $R$, and not just for the action considered in this section or Section \ref{sect:torus action}.  

\begin{proof}
First, note that $z$ having weight $\be_{w(j)} - \be_{w(j+1)}$ is equivalent to saying that for any $\bt \in \mathbb{T}$:

$$
\bt. B_{j}(z) = D_{ws_{j}(\bt)}B_{j}(z)D_{w(t)}^{-1},
$$

cf. Remark \ref{rmk:admissible}. Also, since $D_{i}(z_{0}) = \diag(1, \dots, -z_{0}^{-1}, z_{0}, 1, \dots, 1)$, where $-z_{0}^{-1}$ is in the $i$-th place, we have that $\bt. ^{w}\!D_{i}(z_{0}) = D_{w(\bt)}\,^{w}\!D_{i}(z_{0})D_{s_{i}w(\bt)}^{-1}$. Now we compute

$$
\begin{array}{rl}
\bt.B_{j}(z') & = (\bt.\,^{ws_{j}}\!D_{i}(z_{0})^{-1})(\bt.B_{j}(z))(\bt.\,^{w}\!D_{i}(z_{0})) \\
 & =  (D_{s_{i}ws_{j}(\bt)}\,^{ws_{j}}\!D_{i}(z_{0})^{-1}D_{ws_{j}(\bt)}^{-1})(D_{ws_{j}(\bt)}B_{j}(z)D_{w(\bt)}^{-1})(D_{w(\bt)}\,^{w}\!D_{i}(z_{0}) D_{s_{i}w(\bt)}^{-1}) \\
& = D_{s_{i}ws_{j}(\bt)}B_{j}(z')D_{s_{i}w(\bt)}^{-1}
\end{array}
$$

\noindent and the result follows.
\end{proof}

Finally, the desired statement regarding torus actions on our correspondences reads as follows.

\begin{prop}\label{prop:equivariant weaves}
Let $\ww$ be a simplifying algebraic weave from $\beta_2$ to $\beta_1$. Then, there is an action of the algebraic torus $T = (\C^{*})^{n}/\C^{*}$ on $\SM(\ww)$ such that for every permutation $\pi \in S_{n}$:

\begin{itemize}
\item[(1)] $T$ preserves the correspondence variety $\SM(\ww, \pi)$. 
\item[(2)] The projections $\SM(\ww, \pi) \to X_{0}(\beta_2; \pi)$, $\SM(\ww, \pi) \to X_{0}(\beta_1; \pi)$ are equivariant.
\end{itemize} 
\end{prop}
\begin{proof}
Thanks to Proposition \ref{prop:correspondence} we have $\SM(\ww) \subseteq \C^{\ell(\beta_2)}$, and we have an action of $T$ on $\C^{\ell(\beta_2)}$ given by \eqref{eq:right torus action}. Again by Proposition \ref{prop:correspondence}, this induces an action on $\SM(\ww)$.

Note that, more generally, we have projections $\SM(\ww) \to \C^{\ell(\beta_2)}$, $\SM(\ww) \to \C^{\ell(\beta_1)}$. We will show that both of these maps are $T$-equivariant. This implies (1) and (2) above. By the definition of the $T$-action, the map $\SM(\ww) \to \C^{\ell(\beta_2)}$ is $T$-equivariant. To show that the map $\SM(\ww) \to \C^{\ell(\beta_1)}$ is $T$-equivariant, it suffices to do it for elementary weaves. For four and six-valent vertices, the result follows from Lemma \ref{lemma:4-6 vertices} and Proposition \ref{prop:correspondence}.

Now we move on to three-valent vertices; we have $\beta_{2} = \gamma_{1}\sigma_{i}\sigma_{i}\gamma_{2}$ and $\beta_{1} = \gamma_{1}\sigma_{i}\gamma_{2}$. By Proposition \ref{prop:correspondence} the map $\SM(\ww) \to \C^{\ell(\beta_{1})}$ is given by $z \mapsto (z'_{1},\dots, z'_{r}, z_{r+1} + z_{r}^{-1}, z_{r+2}, \dots, z_{\ell})$, where $z'_{1}, \dots, z'_{r}$ are determined by the equations 
$$
B_{i_{r-d}}(z_{r-d})U^{d} = U^{d+1}B_{i_{r-d}}(z'_{r-d}), \; \; \; U^{0} = U_{i}(z_{r})D_{i}(z_{r}).
$$

Note that the weights of $z_{r+2}, \dots, z_{\ell}$ are clearly preserved under the projection, so for simplicity we may assume that $\gamma_{2} = 1$. We split this into a two-step process, first `sliding $U_i$ to the left' and then `sliding $D_{i}(z_r)$ to the left'. To slide $U_{i}$ to the left, we define $\wt{z}_{1}, \dots, \wt{z}_{r}$ via 
$$
B_{i_{r-d}}(z_{r-d})\wt{U}^{d} = \wt{U}^{d+1}B_{i_{r-d}}(\wt{z}_{r-d}), \; \; \;\wt{U}^{0} = U_{i}(z_{r}).
$$

\noindent And to now slide $D_{i}(z_r)$ to the left, we define $z'_1, \dots, z'_r$ via:
$$
B_{i_{r-d}}(\wt{z}_{r-d})\bar{U}^{d} = \bar{U}^{d+1}B_{i_{r-d}}(z'_{r-d}), \; \; \; \bar{U}^{0} = D_{i}(z_{r}).
$$

Since $U_{i}(z_{r})$ is unitriangular, it follows from Lemma \ref{lemma: moving admissible matrices} that the $T$-weight of $\wt{z}_{r-d}$ coincides with that of $z_{r-d}$ for $d = 0, \dots, r-1$. Now the result follows from Lemma \ref{lemma:3-valent}.

Finally, we check cups: we have $\beta_{2} = \gamma_{1}\sigma_{i}\sigma_{i}\gamma_{2}$ and $\beta_{1} = \gamma_1\gamma_{2}$. The map $\SM(\ww) \to \C^{\ell(\beta_1)}$ is given by $z \mapsto (z_{1}, \dots, z_{r}, z_{r+3}, \dots, z_{\ell})$. Now, since $s_{i}s_{i} = 1$, the result follows.
\end{proof}

Thanks to Proposition \ref{prop:equivariant weaves}, we are able to define correspondences between certain augmentation varieties. Let $\beta_{1}, \beta_{2}$ be braid words, and let $\ft$ be a set of marked points on the strands $1, \dots, n$ satisfying the following conditions:

\begin{itemize}
\item[(i)] There is at most one marked point per strand and, by convention, it is placed to the right of all crossings in both $\beta_{1}$ and $\beta_{2}$ (see Figure \ref{fig:marked points}),
\item[(ii)] Each component of both $\beta_{1}$ and $\beta_{2}$ contains at least one marked point.
\end{itemize}

For example, we can choose $\ft = \ft_{s}$ or $\ft = \ft_{c}$ as in Section \ref{sect:augmentations}. We can then form the augmentation varieties $\Aug(\beta_{1}, \ft)$ and $\Aug(\beta_{2}, \ft)$. Now let $T_{\ft} \subseteq T$ be the torus defined by the equations $t_{i} = 1$ if the $i$-th strand has a marked point. Thanks to (a straightforward generalization of) Theorem \ref{thm:aug vs braid}, we have $\Aug(\beta_{1}, \ft) \cong X_{0}(\beta_{1}\cdot\Delta; w_{0})/T_{\ft}$ and $\Aug(\beta_{2}, \ft) \cong X_{0}(\beta_{2}\cdot\Delta; w_{0})/T_{\ft}$.  In combination with the correspondences above, we then obtain the following result.

\begin{cor}\label{cor:torus aug weave}
Let $\ww$ be a simplifying algebraic weave from $\beta_{2}\cdot \Delta$ to $\beta_{1}\cdot \Delta$. Then, $T_{\ft}$ acts freely on $\SM(\ww)$ and $\SM(\ww, w_{0})/T_{\ft}$ defines a correspondence between $\Aug(\beta_{2}, \ft)$ and $\Aug(\beta_{1}, \ft)$.
\end{cor}


\subsection{Weaves and decompositions}\label{sec:weaves stratifications}

In this subsection, we explain how algebraic weaves can be used to decompose braid varieties; augmentation varieties can be similarly decomposed. For that, recall that a simplifying weave $\ww$ with a braid $\beta_2$ on the top and $\beta_1$ on the bottom defines an injective map 
$$
\SM(\ww, \pi):X_0(\beta_1;\pi)\times \C^{a}\times (\C^*)^b\hookrightarrow X_0(\beta_2;\pi),
$$
where $a$ is the number of cups and $b$ is the number of trivalent vertices. Since each cup decreases the length by 2, and each trivalent vertex by 1, we get the equation $2a+b=\ell(\beta_2)-\ell(\beta_1)$.

We will be interested in simplifying weaves $\ww$  with some braid $\gamma$ on the top and the half twist $\Delta$ on the bottom. Since $X_0(\Delta; w_0)$ is a point, see Example \ref{ex: braid varieties}, we obtain an injective map 
$$
\SM(\ww, w_0):\C^{a}\times (\C^*)^b\hookrightarrow X_0(\gamma;w_0),\ 2a+b=\ell(\gamma)-\binom{n}{2}.
$$

\begin{definition} \label{defn:stratification_by_weaves}
 We say that a collection of simplifying weaves $(\ww_1,\ldots,\ww_k)$ \emph{decomposes} the braid variety $X_0(\gamma,w_0)$ if the images of $\SM(\ww_i)$ do not intersect each other and their union is $X_0(\gamma,w_0)$.
\end{definition}

\begin{remark}
The reason why use the term \emph{decomposition} (as opposed to \emph{stratification}) is that, in some parts of the literature, a condition on a stratification is that the closure of a stratum is a union of strata. This is not the case in, for example, the Deodhar decomposition (cf. \cite{Dudas} or \cite[Section 4.3]{Speyer}), which is a special case of the decompositions we discuss here. 
\end{remark}

\begin{thm}
\label{thm:stratify}
\begin{itemize}
	\item[(a)] 
    Let $\gamma$ be a positive braid word. 
    Then there exists a finite collection of simplifying weaves $(\ww_1,\ldots,\ww_k)$, where each $\ww_i$ has $\gamma$ on the top and the half twist $\Delta$ on the bottom, which decomposes  $X_0(\gamma,w_0)$ in the sense of Definition \ref{defn:stratification_by_weaves}.\\
    
	\item[(b)] Furthermore, given any Demazure weave $\ww$ from $\gamma$ to $\Delta$, there is a decomposition of $X_0(\gamma,w_0)$  by a collection of simplifying weaves $(\ww_1 = \ww, \ww_2, \ldots,\ww_k)$ as in (a), where the correspondence 
	$$\SM(\ww)\cong(\C^*)^{\ell(\gamma)-\binom{n}{2}}$$ is the unique piece of maximal dimension.
\end{itemize}
\end{thm}

\begin{proof} Let us first prove (a) by induction on $\ell(\gamma)\in\N$. If $\gamma$ is reduced, then the matrix $B_{\gamma}(z_1,\ldots,z_{\ell(\gamma)})$ contains 1's corresponding to the permutation matrix for $\gamma$ and independent variables elsewhere. Then $B_{\gamma}(z_1,\ldots,z_{\ell(\gamma)})w_0$ contains $1's$ corresponding to the permutation matrix for $\gamma w_0$, so it is never upper-triangular unless $\gamma w_0=1$. We conclude that $X_0(\gamma;w_0)$ is empty for $\gamma\neq \Delta$ and it is a point for $\gamma=\Delta$. In both cases the variety can be obviously decomposed.

If $\gamma$ is not reduced, then after applying some braid moves we get a braid with two crossings $\sigma_i$ next to each other. Let $z_1$ and $z_2$ be the  variables corresponding to these crossings. If $z_1\neq 0$, we can apply a trivalent vertex and get a braid $\gamma'$, and if $z_1=0$, we can apply a cup and get a braid $\gamma''$. By the assumption of induction, we can decompose $X_0(\gamma';w_0)$ and $X_0(\gamma'';w_0)$ by simplifying weaves. 

For (b), let us decompose $\ww$ into elementary weaves: $\ww^{(1)}$ between $\gamma=\gamma^{(0)}$ and $\gamma^{(1)}$,  $\ww^{(2)}$ between $\gamma^{(1)}$ and $\gamma^{(2)}$ etc. Clearly, we can decompose $X_0(\gamma; w_0)$ as follows:
$$
X_0(\gamma;w_0)=\SM(\ww)\sqcup \left(X_0(\gamma;w_0)\setminus \mathrm{Im\ }\SM(\ww^{(1)})\right)\sqcup \left(\mathrm{Im}\ \SM(\ww^{(1)})\setminus \mathrm{Im}\ \SM(\ww^{(1)}\ww^{(2)})\right)\sqcup \ldots
$$
Let us prove that all these pieces can be further decomposed by simplifying weaves. Indeed, if $\ww^{(i)}$ is a 6- or 4-valent vertex, then $\SM(\ww^{(i)})$ is an isomorphism and 
$$
	\mathrm{Im}\ \SM(\ww^{(1)}\cdots \ww^{(i-1)})=\mathrm{Im}\ \SM(\ww^{(1)}\cdots \ww^{(i)}).
$$
If $\ww^{(i)}$ is a trivalent vertex with variables $z_1$ and $z_2$ then 
$$
\mathrm{Im}\ \SM(\ww^{(1)}\cdots \ww^{(i-1)})\setminus \mathrm{Im}\ \SM(\ww^{(1)}\cdots \ww^{(i)})=\SM(\ww^{(1)}\cdots \ww^{(i-1)})(W_i)
$$
where $W_i$ is the locus  $\{z_1=0\}\subset X_0(\gamma^{(i-1)};w_0)$. In this case we can apply a cup to $\gamma^{(i-1)}$ and obtain 
a new braid $\widetilde{\gamma^{(i)}}$. Then $W_i$ as an image of the correspondence for this cup, and by (a) we can decompose 
$X_0(\widetilde{\gamma^{(i)}};w_0)$ by simplifying weaves. 
\end{proof}

Finally, we obtain the following consequence.

\begin{cor}
\label{cor: non empty}
The braid variety $X_0(\gamma;w_0)$ is not empty if and only if $\gamma$ contains some reduced expression for $w_0$ as a subword, or, equivalently, the Demazure product of $\gamma$ equals $w_0$. In this case, $X_0(\gamma;w_0)$ is an irreducible complete intersection of dimension $\ell(\gamma)-\binom{n}{2}$. 
\end{cor}

\begin{proof}
By \cite[Lemma 3.4]{KM2} a braid word $\gamma$ contains some reduced expression for $w_0$ as a subword if and only if 
$\delta(\gamma)=w_0$.  If $\delta(\gamma)=w_0$ then there is a Demazure weave from $\gamma$ to $w_0$, so $X_0(\gamma;w_0)$ is not empty. By Theorem \ref{thm:stratify}, if $X_0(\gamma;w_0)$ is not empty then there is a simplifying weave from $\gamma$ to $\Delta$, and $\gamma$ contains some reduced expression for $w_0$ as a subword. 

Since $X_0(\gamma;w_0)$ is cut out by $\binom{n}{2}$ equations in the affine space of dimension $\ell(\gamma)$, all its components have dimension  at least $\ell(\gamma)-\binom{n}{2}$. On the other hand, if $\delta(\gamma)=w_0$ then by Theorem \ref{thm:stratify}(b) the braid variety $X_0(\gamma;w_0)$ has unique piece of dimension $\ell(\gamma)-\binom{n}{2}$ and all other pieces have smaller dimension, therefore this variety is an irreducible complete intersection.
\end{proof}

\begin{remark}
In \cite{Mellit} it is proven that the complement to the toric chart in $X_0(\beta\Delta;w_0)$ from Section \ref{sec:mellit chart} can be decomposed into pieces of the form $\C^a\times (\C^*)^b$ with $2a+b=\ell(\beta)$. Similarly to  the proof of Theorem \ref{thm:mellit chart}, one can check that these strata (originally defined in terms of Bruhat cells) can be realized by simplifying weaves.
\end{remark}

The decompositions we presented in Theorem \ref{thm:stratify} are far from unique. However, the number of pieces of given dimension $a+b=\ell(\gamma)-\binom{n}{2}-a$ does not depend of  the decomposition. The topological significance of these numbers is given by the following result.

\begin{lemma}
Suppose that there are $n_a$ pieces of the form  $\C^a\times (\C^*)^b,\ 2a+b=\ell(\gamma)-\binom{n}{2}$ in the decomposition from Theorem \ref{thm:stratify}.
Consider the polynomial
$$
f_{\gamma}(q)=\sum_{a}n_aq^{a}(q-1)^{b}=\sum_{a}n_aq^{a}(q-1)^{\ell(\gamma)-\binom{n}{2}-2a}.
$$
a) The number of points in the variety $X_0(\gamma;w_0)$ over a finite field $\mathbb{F}_q$ equals $f_{\gamma}(q)$.

b) The coefficient in the HOMFLY-PT polynomial of the closure of $\gamma\Delta^{-1}$ of lowest $a$-degree is proportional to $f_{\gamma}(q)$.
\end{lemma}

See \cite{KalmanHOMFLY} for the definition of HOMFLY-PT polynomial, a related computation and more details.

\begin{proof}
Part (a) is clear as Theorem \ref{thm:stratify} can be proved verbatim over any field. Over $\mathbb{F}_q$, the strata $\mathbb{A}^a\times (\mathbb{A}^1\setminus \{0\})^b$ have $q^a(q-1)^b$ points, and the result follows.

For (b), we prove it by induction in $\ell(\gamma)$. If $\gamma$ is reduced then we have two cases:
\begin{itemize}
\item[(i)] If $\gamma=\Delta$, then $\gamma\Delta^{-1}=1$ and the closure of $\gamma\Delta^{-1}$ is the $n$-component unlink. At the same time, $X_0(\Delta;w_0)$ is a point and $f_{\Delta}(q)=1$.

\item[(ii)] If $\gamma\neq \Delta$, then $\gamma\Delta^{-1}$ is the closure of a nontrivial negative permutation braid and the coefficient in the HOMFLY-PT polynomial of lowest $a$-degree vanishes \cite{KalmanHOMFLY}.
At the same time, $X_0(\gamma;w_0)$ is empty and $f_{\gamma}(q)=0$.
\end{itemize}

Now suppose that $\gamma$ is not reduced. It follows from (a) that $f_{\gamma}(q)$ is invariant under braid relations, since so is $X_0(\gamma;w_0)$. Finally, 
if $\gamma=\gamma_1\sigma_i\sigma_i\gamma_2,\gamma'=\gamma_1\sigma_i\gamma_2$ and $\gamma''=\gamma_1\gamma_2$. Then $f_{\gamma}(q)=(q-1)f_{\gamma'}(q)+f_{\gamma''}(q)$ which matches the skein relation for HOMFLY polynomials of the braids $\gamma\Delta^{-1}$, $\gamma'\Delta^{-1}$  and $\gamma''\Delta^{-1}$. By the induction hypothesis, the statement of (b) holds for $\gamma'$ and $\gamma''$. Thus, it also holds for $\gamma$.
\end{proof}


\end{document}